\documentclass[12pt]{amsproc}

\usepackage{fullpage}
\usepackage[utf8]{inputenc}
\usepackage[T1]{fontenc}
\usepackage{graphicx,subfigure}
\usepackage{amsmath,amsfonts,amssymb,mathtools}
\usepackage{stmaryrd}
\usepackage{mathrsfs,dsfont}
\usepackage{amsthm}
\usepackage{mathabx}
\usepackage{tabularx}
\usepackage{float}
\usepackage{amscd}

\usepackage{hyperref}

\usepackage{enumerate}

\usepackage{color}

\newcommand{\E}{\mathbb{E}}
\newcommand{\R}{\mathbb{R}}
\newcommand{\N}{\mathbb{N}}
\newcommand{\D}{\mathcal{D}} 

\newcommand{\IA}{\mathcal{A}}
\newcommand{\IB}{\mathcal{B}}
\newcommand{\IL}{\Lambda}
\newcommand{\Lf}{{\rm L}_F}
\newcommand{\IX}{\mathcal{X}}
\newcommand{\IY}{\mathcal{Y}}
\newcommand{\IW}{\mathcal{W}}
\newcommand{\IP}{\mathcal{P}}

\newcommand{\s}{\rm st}
\newcommand{\e}{\rm e}

\newcommand{\XX}{\mathbb{X}}
\newcommand{\YY}{\mathbb{Y}}

\newcommand{\X}{\mathbf{X}}
\newcommand{\Y}{\mathbf{Y}}
\newcommand{\PP}{\mathbf{P}}

\newtheorem{theo}{Theorem}[section]

\newtheorem{rem}[theo]{Remark}

\newtheorem{propo}[theo]{Proposition}
\newtheorem{lemma}[theo]{Lemma}

\newtheorem{ass}{Assumption}

\usepackage{algorithm}
\usepackage{algorithmic}

\begin{document}

\title{
Analysis of a modified Euler scheme for parabolic semilinear stochastic PDEs
}

\author{Charles-Edouard Br\'ehier}

\address{Univ Lyon, Université Claude Bernard Lyon 1, CNRS UMR 5208, Institut Camille Jordan, 43 blvd. du 11 novembre 1918, F-69622 Villeurbanne cedex, France}
\email{brehier@math.univ-lyon1.fr}

\keywords{Stochastic partial differential equations, Euler schemes, invariant distributions, infinite dimensional Kolmogorov equations, asymptotic preserving schemes, Markov Chain Monte Carlo methods}
\subjclass{60H35;65C30;60H15}

\begin{abstract}
We propose a modification of the standard linear implicit Euler integrator for the weak approximation of parabolic semilinear stochastic PDEs driven by additive space-time white noise. The new method can easily be combined with a finite difference method for the spatial discretization. The proposed method is shown to have improved qualitative properties compared with the standard method. First, for any time-step size, the spatial regularity of the solution is preserved, at all times. Second, the proposed method preserves the Gaussian invariant distribution of the infinite dimensional Ornstein--Uhlenbeck process obtained when the nonlinearity is absent, for any time-step size. The weak order of convergence of the proposed method is shown to be equal to $1/2$ in a general setting, like for the standard Euler scheme. A stronger weak approximation result is obtained when considering the approximation of a Gibbs invariant distribution, when the nonlinearity is a gradient: one obtains an approximation in total variation distance of order $1/2$, which does not hold for the standard method. This is the first result of this type in the literature. A key point in the analysis is the interpretation of the proposed modified Euler scheme as the accelerated exponential Euler scheme applied to a modified stochastic evolution equation. Finally, it is shown that the proposed method can be applied to design an asymptotic preserving scheme for a class of slow-fast multiscale systems, and to construct a Markov Chain Monte Carlo method which is well-defined in infinite dimension. We also revisit the analysis of the standard and the accelerated exponential Euler scheme, and we prove new results with approximation in the total variation distance, which serve to illustrate the behavior of the proposed modified Euler scheme.
\end{abstract}

\maketitle

\section{Introduction}\label{sec:intro}

In the last 25 years, the numerical analysis of stochastic partial differential equations (SPDEs) has been an active field of research. We refer for instance to the monograph~\cite{LPS} for a comprehensive introduction and references therein for historical overview of the field. In this article, we consider a class of parabolic semilinear equations which may be written as
\begin{equation}\label{eq:SPDEintro-field}
\left\lbrace
\begin{aligned}
&\partial_tX(t,\xi)=\partial_\xi\bigl(a(\xi)\partial_\xi X(t,\xi)\bigr)+f(X(t,\xi))+\dot{W}(t,\xi),\quad \forall t>0,\xi\in(0,1),\\
&X(t,0)=X(t,1)=0,\quad \forall t>0,\\
&X(0,\xi)=x_0(\xi),\quad \forall \xi\in(0,1),
\end{aligned}
\right.
\end{equation}
where the unknown $X:(t,\xi)\in [0,T]\times[0,1]\mapsto X(t,\xi)\in\R$ is a random field, homogeneous Dirichlet boundary conditions are imposed, $x_0$ is a given initial value, $a:[0,1]\to(0,\infty)$ and $f:\R\to\R$ are two sufficiently smooth real-valued functions, and $\dot{W}$ is space-time white noise.

It is convenient to interpret the SPDE~\eqref{eq:SPDEintro-field} as a stochastic evolution equation (SEE) in the sense of~\cite{DPZ}
\begin{equation}\label{eq:SPDEintro}
dX(t)=-\IL X(t)dt+F(X(t))dt+dW(t),\quad X(0)=x_0,
\end{equation}
where the unknown $X:t\in[0,T]\mapsto X(t)\in L^2(0,1)$ is a continuous stochastic process with values in the infinite dimensional Hilbert space $H=L^2(0,1)$. The nonlinearity $F:H\to H$ is assumed to be globally Lipschitz continuous. See Section~\ref{sec:setting} for details, in particular for the definition and properties of the linear operator $\IL$. The SEE~\eqref{eq:SPDEintro} is driven by a cylindrical Wiener process. In the sequel, we only deal with SEEs of type~\eqref{eq:SPDEintro}.

To approximate solutions of SPDEs~\eqref{eq:SPDEintro-field} and SEEs~\eqref{eq:SPDEintro}, it is necessary to apply spatial and temporal discretization procedures. On the one hand, the spatial approximation may be performed using either a spectral Galerkin method or a finite differences scheme. To employ the spectral Galerkin technique, one needs to know the eigenvalues and eigenfunctions of the linear operator $\IL$, which is not the case in general (in the case of~\eqref{eq:SPDEintro-field}, this is the case if the function $a$ is constant). The finite differences scheme can be applied in greater generality, therefore this is the method which is chosen in this work. On the other hand, the temporal approximation may be performed for instance using a standard semi-implicit Euler integrator, or an exponential integrator. The application of the exponential Euler scheme also requires to know the eigenvalues and eigenfunctions of the linear operator $\IL$ (or of its spatial approximation using a finite difference method), which is not the case in general. In this work, we study a variant of the standard Euler scheme. Note that the exponential Euler integrator is also treated as a matter of comparison for the proposed method.

Instead of considering a fully discrete scheme, which combines a finite differences scheme and a (modification of the) standard Euler scheme, in this work we analyze only the temporal discretization. Results may be generalized to the fully discrete framework, with the introduction of additional notation, and qualitative properties and quantitative error estimates being uniform with respect to the spatial discretization parameter. Focusing only on the temporal approximation allows us to emphasize how the proposed method overcomes some limitations of the standard scheme. In the sequel, the spatial discretization is thus omitted in the statements and in the proofs.

The standard Euler scheme applied to the SEE~\eqref{eq:SPDEintro} reads
\begin{equation}\label{eq:scheme-standard-intro}
X_{n+1}^{\tau,\s}=(I+\tau\IL)^{-1}\Bigl(X_n^{\tau,\s}+\tau F(X_n^{\tau,\s})+W(t_{n+1})-W(t_n)\Bigr),
\end{equation}
with initial value $X_0^{\tau,s}=x_0$, where the time-step size is denoted by $\tau$ and $t_n=n\tau$. In practice, if a finite difference method is applied for spatial discretization, note that it is sufficient to solve linear systems using a LU decomposition of the resulting linear operator. The standard linear implicit Euler scheme~\eqref{eq:scheme-standard-intro} has been studied extensively in the literature, let us recall that
\begin{itemize}
\item it has strong order of convergence $1/4$, see for instance~\cite{Printems},
\item it has weak order of convergence $1/2$, see for instance~\cite{Debussche:11},
\item when the Lipschitz constant of $F$ is sufficiently small, the invariant distribution of the SEE is approximated  with order of convergence $1/2$, see for instance~\cite{B:2014}.
\end{itemize}
However, the standard Euler scheme~\eqref{eq:scheme-standard-intro} suffers from a major issue: the spatial regularity of the solution is not preserved, more precisely for any fixed value of the time-step size $\tau$ and any integer $n\in\N$, the random variable $X_n^{\tau,\s}$ is more regular than $X(t_n)$, when the regularity is measured either in the sense of H\"older or Sobolev spaces. Furthermore, this difference in the qualitative behavior of the exact and numerical solution has an impact on quantitative error estimates. Indeed, as shown in~\cite{B:2020}, one needs to consider test functions which are at least of class $\mathcal{C}^2$ to obtain a weak order of convergence $1/2$. The distributions of the $H$-valued random variables $X_n^{\tau,\s}$ and $X(t_n)$ are singular, and therefore the distribution of $X(t_n)$ cannot be approximated in the total variation distance sense using the approximation $X_n^{\tau,\s}$ obtained using the standard Euler scheme.

\subsection*{Contributions}

\subsubsection*{The modified Euler scheme}

The objective of this work is to introduce a modified Euler scheme, which can be easily combined with a finite differences method for the spatial approximation, and which overcomes the limitations of the standard Euler scheme mentioned above. The proposed modified Euler scheme is defined as follows:
\begin{equation}\label{eq:scheme-intro}
X_{n+1}^{\tau}=\IA_\tau\bigl(X_n^{\tau}+\tau F(X_n^{\tau})\bigr)+\IB_{\tau,1}\sqrt{\tau}\Gamma_{n,1}+\IB_{\tau,2}\sqrt{\tau}\Gamma_{n,2},
\end{equation}
where $\bigl(\Gamma_{n,1}\bigr)_{n\ge 0}$ and $\bigl(\Gamma_{n,2}\bigr)_{n\ge 0}$ are two independent sequences of equally distributed independent {\it cylindrical Gaussian random variables}, meaning that in distribution $\Gamma_{n,1}=\Gamma_{n,2}=\tau^{-\frac12}\bigl(W(t_{n+1})-W(t_n)\bigr)$ are rescaled increments of the cylindrical Wiener process. The definition of the modified Euler scheme~\eqref{eq:scheme-intro} requires the introduction of three linear operators $\IA_\tau$, $\IB_{\tau,1}$ and $\IB_{\tau,2}$: they are required to satisfy the conditions (see~\eqref{eq:operators} in Section~\ref{sec:scheme-1st})
\begin{equation}\label{eq:operators-intro}
\IA_\tau=(I+\tau\IL)^{-1}~,\quad \IB_{\tau,1}=\frac{1}{\sqrt{2}}(I+\tau\IL)^{-1}~,\quad \IB_{\tau,2}\IB_{\tau,2}^\star=\frac12(I+\tau\IL)^{-1}
\end{equation}
where $L^\star$ denotes the adjoint of a linear operator $L$. It is worth mentioning that the modified Euler scheme~\eqref{eq:scheme-intro} can be seen as a modification of the standard Euler scheme~\eqref{eq:scheme-standard-intro}, with a different treatment of the stochastic term, but the same treatment of the linear operator $\IL$ and of the nonlinear operator $F$. Furthermore, the linear operator $\IB_{\tau,2}$ is not determined uniquely by the third condition~\eqref{eq:operators-intro}: in particular, it is not required to be a self-adjoint operator, instead in practice, when a finite differences method is applied, it is sufficient to compute a Cholesky decomposition of the resulting linear operator. Note that the iterations in the modified Euler scheme~\eqref{eq:scheme-intro} have higher computational cost than the iterations in the standard scheme~\eqref{eq:scheme-standard-intro}. In addition, two cylindrical Gaussian random variables are needed at each iteration of the modified Euler scheme~\eqref{eq:scheme-intro}, instead of one for the standard Euler scheme~\eqref{eq:scheme-standard-intro}. The conditions appearing in~\eqref{eq:operators-intro} are justified in Section~\ref{sec:scheme} and are designed to improve the qualitative and quantitative behavior of the standard scheme, as explained below. Note that the standard Euler scheme~\eqref{eq:scheme-standard-intro} can be written in a form similar to~\eqref{eq:scheme-intro} and would be recovered by setting $\IB_{\tau,2,\s}=\IB_{\tau,1}$ (the definitions of $\IA_\tau$ and $\IB_{\tau,1}$ being unchanged). In other words, the condition
\[
\IB_{\tau,2,\s}\IB_{\tau,2,\s}^\star=\frac12(I+\tau\IL)^{-2}=\frac12 \IA_\tau^2
\]
would be satisfied for the standard Euler scheme instead of the third condition
\[
\IB_{\tau,2}\IB_{\tau,2}^\star=\frac12(I+\tau\IL)^{-1}=\frac12 \IA_\tau
\]
from~\eqref{eq:operators-intro} for the modified Euler scheme. This observation that the powers of the operator $\IA_\tau$ differ in the conditions for $\IB_{\tau,2}$ and $\IB_{\tau,2,\s}$ is crucial to understand how the modified Euler scheme overcomes the limitations of the standard Euler scheme mentioned above.

\subsubsection*{Main qualitative and quantitative results}

We are now in position to state and discuss the main results of this manuscript. We refer to Section~\ref{sec:setting} for precise assumptions on the linear operator $\IL$ and the nonlinearity $F$, and to Section~\ref{sec:results} for rigorous statements of the results. Let us first discuss the qualitative behavior of the modified Euler scheme~\eqref{eq:scheme-intro}. The main result in this direction is the following: for any time-step size $\tau$ and any $n\in\N$, if $F=0$, the distributions of the $H$-valued random variables $X_n^\tau$ and $X(t_n)$ are equivalent, see Theorem~\ref{theo:equivalence}. This result is proved using the Feldman-Hajeck criterion. As a consequence, one then checks that the spatial regularity of the numerical solution $X_n^\tau$ and of the exact solution $X(t_n)$ coincide, see Theorem~\ref{theo:regularity}. The proof is straightforward, since the spatial regularity is determined by the behavior of the stochastic contribution, not by the initial value or the nonlinearity, in the considered framework.

In order to illustrate the qualitative superiority of the modified Euler scheme~\eqref{eq:scheme-intro} over the standard Euler scheme~\eqref{eq:scheme-standard-intro}, let us provide a numerical experiment. The SPDE~\eqref{eq:SPDEintro}, with $a=1$ and $f=0$, is approximated using a finite difference method with mesh size  $h=10^{-3}$. The time-step size is chosen as $\tau=2^{-8}$. The final time is set to $T=1$. Figure~\ref{fig1} (fixed time $T=1$) and Figure~\ref{fig2} (all times $t_n\in[0,1]$) illustrate the preservation of the regularity property (Theorem~\ref{theo:regularity}) for the modified Euler scheme (left figure), compared with the higher regularity obtained when using the standard Euler scheme (right figure). The realizations are sampled using the same Wiener path: this means that the Gaussian random variables satisfy the equality $\Gamma_{n,1}+\Gamma_{n,2}=\sqrt{2}\Gamma_n$ for all $n\ge 0$.

\begin{figure}[h]
\centering
\includegraphics*[width=0.48\textwidth,keepaspectratio]{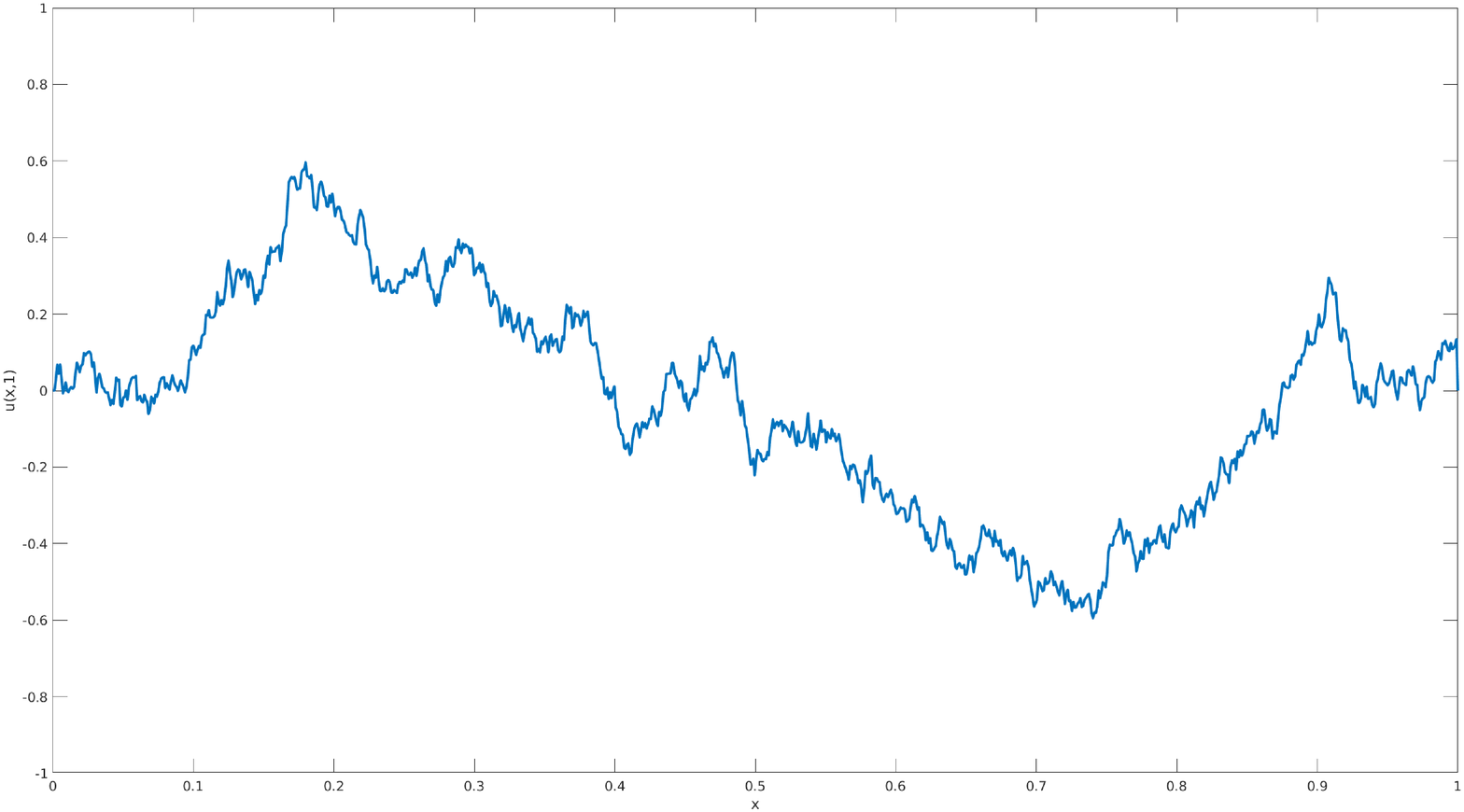}
\includegraphics*[width=0.48\textwidth,keepaspectratio]{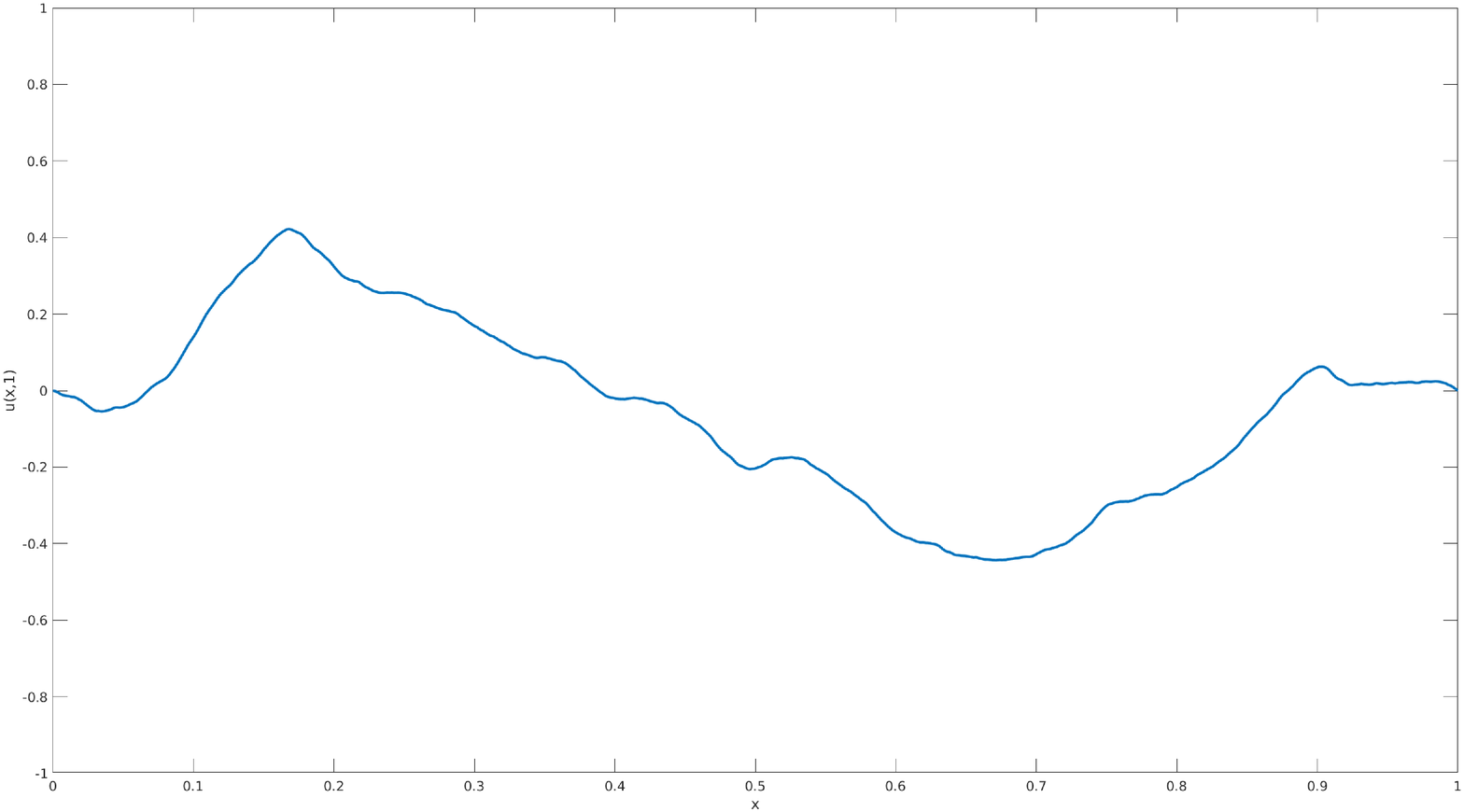}
\caption{Plots of the approximate solution at time $T=1$ obtained using the modified Euler scheme~\eqref{eq:scheme-intro} (left) and the standard Euler scheme~\eqref{eq:scheme-standard-intro} (right). Parameters: $h=10^{-3}$ and $\tau=2^{-8}$.}
\label{fig1}
\end{figure}

\begin{figure}[h]
\centering
\includegraphics*[width=0.48\textwidth,keepaspectratio]{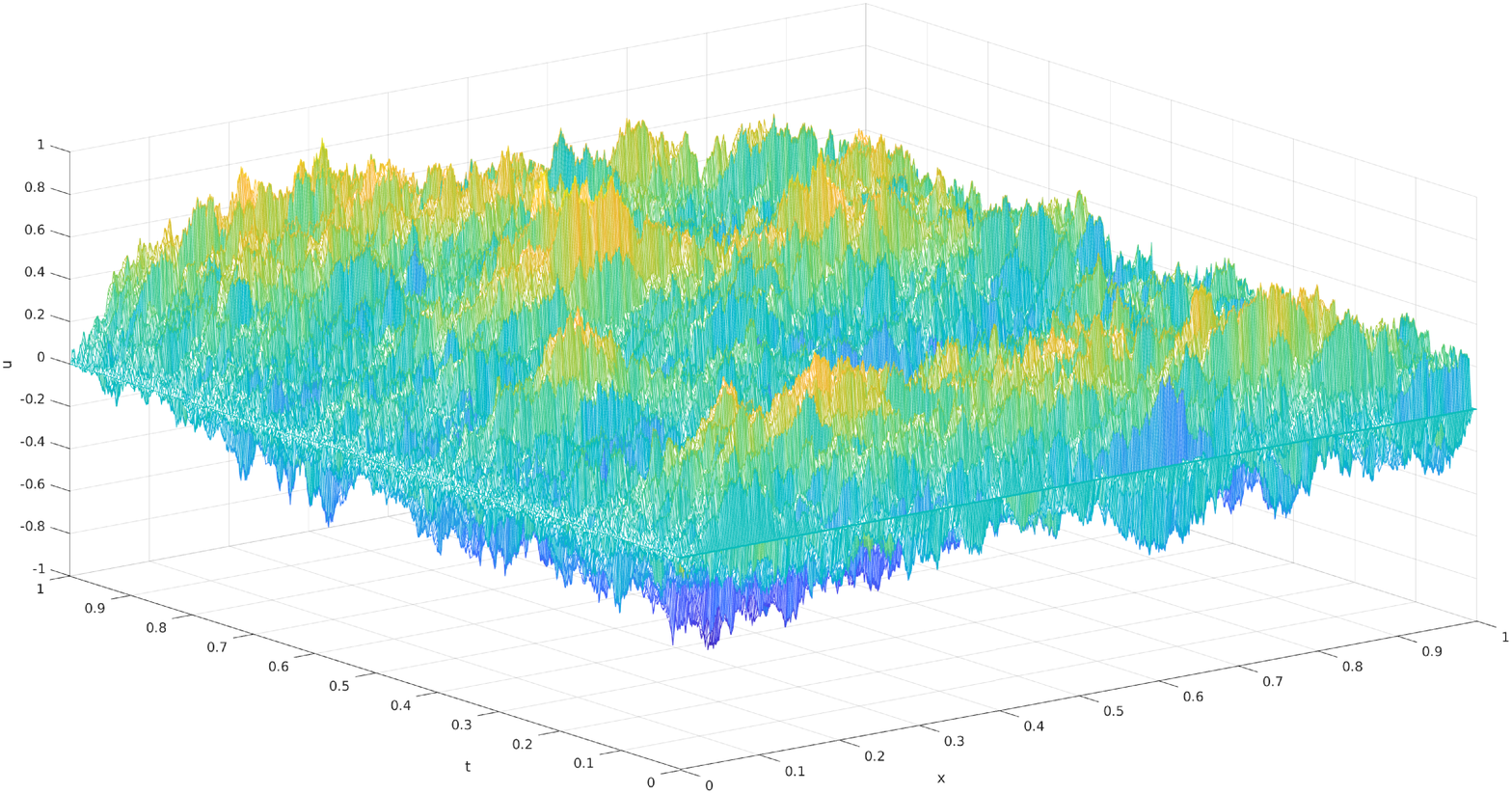}
\includegraphics*[width=0.48\textwidth,keepaspectratio]{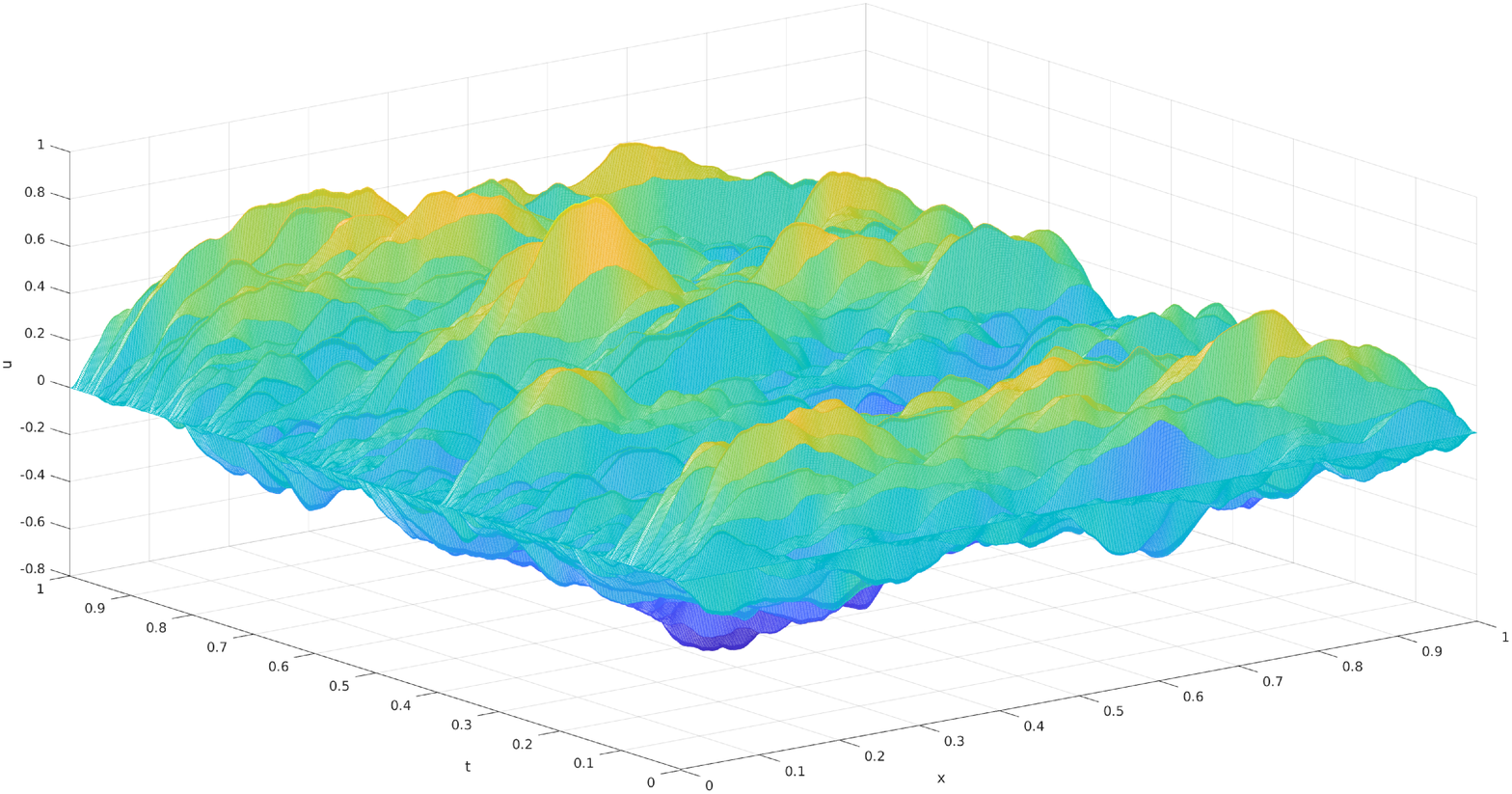}
\caption{Plots of the approximate solution at times $t_n=n\tau\le 1$, obtained using the modified Euler scheme~\eqref{eq:scheme-intro} (left) and the standard Euler scheme~\eqref{eq:scheme-standard-intro} (right). Parameters: $h=10^{-3}$ and $\tau=2^{-8}$.}
\label{fig2}
\end{figure}

For quantitative weak error estimates, two results are stated. The main and most original result of this manuscript is Theorem~\ref{theo:weakinv}. Assume that the mapping satisfies the {\it gradient assumption} $F=-DV$ where $V:H\to\R$ is a given real-valued function and $D$ denotes the Fr\'echet derivative. To ensure ergodicity of the SEE~\eqref{eq:SPDEintro}, assume that the Lipschitz constant of $F$ is sufficiently small (see Assumption~\ref{ass:ergo} below), then the SEE~\eqref{eq:SPDEintro} admits a unique invariant distribution, which is the Gibbs distribution
\[
d\mu_\star(x)=\mathcal{Z}^{-1}\exp\bigl(-2V(x)\bigr)d\nu(x)
\]
where $\nu$ is the centered Gaussian distribution with covariance operator $\frac12\IL^{-1}$, and $\mathcal{Z}$ is a normalization constant. Under the same assumptions, for any time-step size $\tau$, the modified Euler scheme~\eqref{eq:scheme-intro} admits a unique invariant distribution $\mu_\infty^\tau$, and one has the following approximation result (see Theorem~\ref{theo:weakinv})
\[
d_{\rm TV}(\mu_\infty^\tau,\mu_\star)\le C_\delta \tau^{\frac12-\delta}
\]
where $\delta\in(0,\frac12)$ is an arbitrarily small parameter, $C_\delta\in(0,\infty)$, and $d_{\rm TV}$ denotes the total variation distance. In other words, one has weak error estimates
\[
\big|\int\varphi d\mu_\infty^\tau-\int\varphi d\mu_\star\big|\le C_\delta\vvvert\varphi\vvvert\tau^{\frac12-\delta}
\]
where the test functions $\varphi:H\to\R$ only need to be assumed measurable and bounded and $\vvvert\varphi\vvvert=\underset{x\in H}\sup~|\varphi(x)|$.

In addition, when $F=0$, then $\mu_\infty^\tau=\nu$ for any time-step size $\tau$: the Gaussian invariant distribution is preserved in the Ornstein--Uhlenbeck case when using the modified Euler scheme.

Let us mention that the main ingredient for the proof of Theorem~\ref{theo:weakinv} is the interpretation of the modified Euler scheme as the accelerated exponential Euler scheme applied to the modified stochastic evolution equation
\begin{equation}\label{eq:modifiedSPDE-intro}
d\IX_\tau(t)=-\IL_\tau \IX_\tau(t)dt+Q_{\tau}F(\IX_\tau(t))dt+Q_{\tau}^{\frac12}dW(t),\quad \IX_\tau(0)=x_0,
\end{equation}
which depends on two auxiliary linear operators $\IL_\tau$ and $Q_\tau$: we refer to Section~\ref{sec:scheme-3rd} for their definitions. Precisely, for any time-step size $\tau$ and any integer $n\ge 0$, one has equality in distribution $X_n^\tau=\IX_{\tau,n}$, where the sequence $\bigl(\IX_{\tau,n}\bigr)_{n\ge 0}$ is defined by
\begin{equation}\label{eq:scheme-IX-intro}
\IX_{\tau,n+1}=e^{-\tau\IL_\tau}\IX_{\tau,n}+\IL_\tau^{-1}(I-e^{-\tau\IL_\tau})Q_\tau F(\IX_{\tau,n})+\int_{t_n}^{t_{n+1}}e^{-(t_{n+1}-s)\IL_\tau}Q_\tau^{\frac12}dW(s)
\end{equation}
with initial value $\IX_{\tau,0}=x_0$. It is crucial to observe that, when the gradient condition $F=-DV$ is satisfied, the Gibbs distribution $\mu_\star$ is the invariant distribution of the modified SEE~\eqref{eq:modifiedSPDE-intro}, for any value of the time-step size $\tau$. Other technical assumptions and results are needed to obtain Theorem~\ref{theo:weakinv}, we refer to Sections~\ref{sec:results_invar} and~\ref{sec:proofs-weakinv}, in particular to see why the interpretation as an accelerated exponential Euler scheme is helpful to reduce the regularity requirements on the test functions $\varphi$ to prove weak error estimates.

When the gradient condition $F=-DV$ is not satisfied, the SEE~\eqref{eq:SPDEintro} admits a unique invariant distribution $\mu_\infty$, which has no known expression in general. Theorem~\ref{theo:weak-ergo} provides weak error estimates
\[
\big|\int\varphi d\mu_\infty^\tau-\int\varphi d\mu_\infty\big|\le C_\delta\vvvert\varphi\vvvert_2\tau^{\frac12-\delta}
\]
for test functions $\varphi:H\to \R$ which are assumed to be of class $\mathcal{C}^2$ with bounded second order derivatives. This result is similar to the one proved in~\cite{B:2014} for the standard Euler scheme. In addition, for any time $T\in(0,\infty)$, in a general setting, one has weak error estimates
\[
\big|\E[\varphi(X_N^\tau)]-\E[\varphi(X(T))]\big|\le C_\delta(T,x_0)\vvvert\varphi\vvvert_2 \tau^{\frac12-\delta},
\]
with $T=N\tau$, for test functions $\varphi:H\to \R$ which are assumed to be of class $\mathcal{C}^2$ with bounded second order derivatives, see Theorem~\ref{theo:weak}. Like above, this result is similar to the one obtaind in~\cite{Debussche:11} for the standard Euler scheme. Even if Theorems~\ref{theo:weak} and~\ref{theo:weak-ergo} do not show improvements for the modified Euler scheme compared with the standard Euler method, it is worth providing detailed proofs to justify that the modified Euler scheme is not meant to be used only for the approximation of the Gibbs invariant distribution, which would be restrictive. Note that whether the regularity requirement on the test functions $\varphi$ may be weakened in Theorems~\ref{theo:weak} and~\ref{theo:weak-ergo} is an open question.

The main result, Theorem~\ref{theo:weakinv}, is compared with the results which are obtained for the standard Euler scheme~\eqref{eq:scheme-standard-intro} in Section~\ref{sec:results_standard}, and for the accelerated exponential Euler scheme in Section~\ref{sec:results_exponential}.

On the one hand, when $F=0$, the invariant distribution of the standard Euler scheme~\eqref{eq:scheme-standard-intro} is a Gaussian distribution $\nu^\tau$. It is straightforward to check that $\nu^\tau$ and $\nu$ are singular probability distribution for all $\tau>0$, this may be seen as a result of the non preservation of the spatial regularity by the standard Euler method. Using an interpretation of the standard Euler scheme as the accelerated exponential Euler scheme applied to a modified stochastic evolution equation, one proves an approximation result
\[
d_{\rm TV}(\mu_\infty^\tau,\mu_\star^\tau)\le C_\delta\tau^{\frac12-\delta},
\]
in the total variation distance, see Theorem~\ref{theo:weakinv-standard}, when the nonlinearity satisfies the gradient condition $F=-DV$. The main difference with Theorem~\ref{theo:weakinv} is the fact that $\mu_\star^\tau$ is not equal to the Gibbs distribution $\mu_\star$, instead it is a Gibbs distribution
\[
d\mu_\star^\tau(x)=(\mathcal{Z}^\tau)^{-1}e^{-2V(x)}d\nu^\tau(x)
\]
where the reference measure is the Gaussian distribution $\nu^\tau$. Theorem~\ref{theo:weakinv-standard} may not be useful in practice when using the standard Euler scheme since $\mu_\star^\tau$ and $\mu_\star$ are singular probability distributions. However, from a theoretical perspective, this result shows again that the limitations in the performances of the standard Euler scheme are due to the discretization of the stochastic part only. The proof of Theorem~\ref{theo:weakinv-standard}, given in Section~\ref{sec:standard}, employs the same techniques as the proof of Theorem~\ref{theo:weakinv}, and a few more delicate arguments.

On the other hand, let us consider the accelerated exponential Euler scheme, defined by
\begin{equation}\label{eq:exponentialscheme-intro}
X_{n+1}^{\tau,\e}=e^{-\tau\IL}X_n^{\tau,\e}+\IL^{-1}(I-e^{-\tau\IL})F(X_n^{\tau,\e})+\int_{t_n}^{t_{n+1}}e^{-(t_{n+1}-s)\IL}dW(s),\quad X_0^{\tau,\e}=x_0.
\end{equation}
Note that if $F=0$, then $X_n^{\tau,\e}=X(t_n)$ for any time-step size $\tau$ and any integer $n\ge 0$: the discretization in the Ornstein--Uhlenbeck case using the accelerated exponential Euler scheme is exact. As a result, it is not a surprising result that the spatial regularity is preserved when using that scheme. More interestingly, Theorem~\ref{theo:weak-expo} states that one can approximate the distribution of $X(t_n)$ in the total variation distance, at any time and without the requirement that the nonlinearity $F$ satisfies the gradient condition:
\[
\big|\E[\varphi(X_N^{\tau,\e})]-\E[\varphi(X(T))]\big|\le C_\delta(T,x_0)\vvvert\varphi\vvvert_0 \tau^{\frac12-\delta},
\]
with $T=N\tau$, where the test functions $\varphi:H\to\R$ only need to be assumed measurable and bounded. Theorem~\ref{theo:weak-ergo-expo} states a result for the approximation of the invariant distribution $\mu_\infty$ of the SEE~\eqref{eq:SPDEintro},
\[
d_{\rm TV}(\mu_\infty^{\tau,\e},\mu_\infty)\le C_\delta\tau^{\frac12-\delta},
\]
where $\mu_\infty^{\tau,\e}$ is the unique invariant distribution of the accelerated exponential Euler scheme~\eqref{eq:exponentialscheme-intro} with time-step size $\tau$. Compared with the accelerated exponential Euler scheme, approximation results in the total variation distance are obtained for the proposed modified Euler scheme~\eqref{eq:scheme-intro} in a more restrictive setting, namely for the approximation of the Gibbs invariant distribution when the gradient condition $F=-DV$ is satisfied by the nonlinearity. However, the range of application of the modified Euler scheme is larger, since it does not require to known the eigenvalues and the eigenfunctions of the linear operator $\IL$. Theorems~\ref{theo:weak-expo} and~\ref{theo:weak-ergo-expo} are new results on the accelerated exponential Euler scheme. The proof of Theorem~\ref{theo:weak-expo} exhibits the main arguments which are needed in the proof of the main result, Theorem~\ref{theo:weakinv}, and in particular why it is convenient and crucial to interpret the modified Euler scheme~\eqref{eq:scheme-intro} as the accelerated exponential Euler scheme applied to the modified stochastic evolution equation~\eqref{eq:modifiedSPDE-intro}, as explained above.

\subsubsection*{Two applications of the modified Euler scheme}

The last two main contributions of this manuscript are two applications of the modified Euler scheme, which again illustrate its superiority compared with the standard Euler method. First, in Section~\ref{sec:AP}, we study a class of slow-fast stochastic evolution systems
\begin{equation}\label{eq:SPDE-slowfast-intro}
\left\lbrace
\begin{aligned}
d\XX^\epsilon(t)&=-\IL\XX^\epsilon(t)dt+G\bigl(\XX^\epsilon(t),\YY^\epsilon(t)\bigr)\\
d\YY^\epsilon(t)&=-\frac{1}{\epsilon}\IL\YY^\epsilon(t)dt+\frac{\sigma(\XX^\epsilon(t))}{\sqrt{\epsilon}}dW(t),
\end{aligned}
\right.
\end{equation}
depending on a time scale separation parameter $\epsilon$, where $G:H\times H\to H$ and $\sigma:H\to\R$ are bounded and globally Lipschitz continuous mappings. When $\epsilon\to 0$, the averaging principle states the convergence of the slow component $\XX^\epsilon$ to the solution $\overline{\XX}$ of an evolution equation
\[
d\overline{\XX}(t)=-\IL\overline{\XX}(t)dt+\overline{G}(\overline{\XX}(t)),
\]
where the nonlinearity $\overline{G}:H\to H$ is defined as
\[
\overline{G}(x)=\int G\bigl(x,\sigma(x)y\bigr)d\nu(y)=\E[G(x,\frac{\sigma(x)}{\sqrt{2}}\IL^{-\frac12}\Gamma)],
\]
where $\nu$ is the invariant distribution of the Ornstein--Uhlenbeck process, and $\Gamma$ is a cylindrical Gaussian random variable. Using the modified Euler scheme to discretize the fast component, we design an asymptotic preserving scheme
\begin{equation}\label{eq:APscheme-intro}
\left\lbrace
\begin{aligned}
\XX_{n+1}^{\epsilon,\tau}&=\IA_\tau\bigl(\XX_n^{\epsilon,\tau}+\tau G(\XX_n^{\epsilon,\tau},\YY_{n+1}^{\epsilon,\tau})\bigr)\\
\YY_{n+1}^{\epsilon,\tau}&=\IA_{\frac{\tau}{\epsilon}}\YY_n^{\epsilon,\tau}+\sigma(\XX_n^{\epsilon,\tau})\sqrt{\frac{\tau}{\epsilon}}\IB_{\frac{\tau}{\epsilon},1}\Gamma_{n,1}+\sigma(\XX_n^{\epsilon,\tau})\sqrt{\frac{\tau}{\epsilon}}\IB_{\frac{\tau}{\epsilon},2}\Gamma_{n,2}.
\end{aligned}
\right.
\end{equation}
First, it is shown that there exists a limiting scheme, $\XX_n^{\epsilon,\tau}\underset{\epsilon\to 0}\to \XX_0^{0,\tau}$ for any time-step size $\tau$ and any integer $\tau$, where the convergence is understood as convergence in distribution. Second, it is shown that the limiting scheme, given by
\[
\XX_{n+1}^{0,\tau}=\IA_\tau\bigl(\XX_n^{0,\tau}+\tau G(\XX_n^{0,\tau},\sigma(\XX_b^{0,\tau})Q^{\frac12}\Gamma_n)\bigr),
\]
satisfies $\XX_N^{0,\tau}\underset{\tau\to 0}\to \overline{\XX}(T)$, with $T=N\tau$, where the convergence is understood as convergence in distribution. Since for any fixed $\epsilon$, one also has $\XX_N^{\epsilon,\tau}\underset{\tau\to 0}\to \XX^{\epsilon}(T)$, the proposed scheme~\eqref{eq:APscheme-intro} satisfies the asymptotic preserving property, and the time-step size $\tau$ can be chosen independently of the time scale separation parameter $\epsilon$. Note that if the fast component $\YY^\epsilon$ of the system~\eqref{eq:SPDE-slowfast-intro} is discretized using the standard Euler scheme, the resulting limiting scheme is not consistent with the limiting equation given by the averaging principle. After proving that the scheme~\eqref{eq:APscheme-intro} is asymptotic preserving, it is natural to investigate whether the scheme is uniformly accurate: we refer to~\cite{B} for the proof of uniform weak error estimates.

Second, in Section~\ref{sec:MCMC}, another application of the modified Euler scheme is presented: it can be employed as a proposal kernel to apply a Markov Chain Monte Carlo method, based on the Metropolis--Hastings rule for the computation of the acceptance-rejection probability, where the target probability distribution is the Gibbs distribution $\mu_\star$. Theorem~\ref{theo:MCMC} states that for any value of the time-step size $\tau$, the Markov chain defined by
\begin{equation}\label{eq:MCMC-intro}
\left\lbrace
\begin{aligned}
\hat{\X}_{n+1}^{\tau}&=\IA_\tau\X_n^\tau+\sqrt{\tau}\IB_{\tau,1}\Gamma_{n,1}+\sqrt{\tau}\IB_{\tau,2}\Gamma_{n,2}\\
\X_{n+1}^{\tau}&=\mathds{1}_{U_n\le a(\X_n^{\tau},\hat{\X}_{n+1}^{\tau})}\hat{\X}_{n+1}^{\tau}+\mathds{1}_{U_n>a(\X_n^{\tau},\hat{\X}_{n+1}^{\tau})}\X_{n}^{\tau},
\end{aligned}
\right.
\end{equation}
where $U_n$ is a uniformly distributed random variable (independent of the cylindrical Gaussian random variables $\Gamma_{n,1}$ and $\Gamma_{n,2}$), and the acceptance-rejection ratio is defined by
\begin{equation}\label{eq:acceptanceMCMC-intro}
a(x,\hat{x})=\min(1,e^{2(V(x)-V(\hat{x}))})
\end{equation}
is ergodic, with unique invariant distribution $\mu_\star$. In addition, the convergence is exponentially fast. This result would not hold if the standard Euler scheme was used to define the proposal kernel.

\subsection*{Comparison with the literature}

Let us now review the relevant literature in order to illustrate the novelties of this work. We refer to the monograph~\cite{LPS} for an introduction to computational methods for SPDEs. We also refer to the monograph~\cite{Kruse} for a presentation of approximation results for parabolic semilinear SEEs of type~\eqref{eq:SPDEintro}, with a focus on the standard Euler scheme for the temporal discretization and on finite element methods for the spectral discretization.

The standard Euler scheme used to discretize the SPDE~\eqref{eq:SPDEintro-field} (combined with a finite difference method) has been studied by many authors: see for instance~\cite{MR1480861,MR1699161}. When the problem is interpreted as a SEE~\eqref{eq:SPDEintro} (which is the point of view considered in this work), the strong and weak orders of convergence of the standard Euler scheme have been identified in~\cite{Printems} and~\cite{DebusschePrintems:09,Debussche:11} respectively. Theorem~\ref{theo:weak}, in particular, is proved using tools similar to those introduced in~\cite{Debussche:11} for the weak error analysis: analysis of regularity properties for solutions of infinite dimensional Kolmogorov equations (see Sections~\ref{sec:auxiliary-kolmogorov-original} and~\ref{sec:auxiliary-kolmogorov-modified}) and Malliavin calculus techniques. However, details are different since the proof of Theorem~\ref{theo:weak} is based on an original point of view introduced in this article, namely the interpretation of the modified Euler scheme as the accelerated exponential Euler scheme for the modified SEE~\eqref{eq:modifiedSPDE-intro}. For regularity results on infinite dimensional Kolmogorov equations, we refer to the monograph~\cite{Cerrai} and to~\cite{AnderssonHefterJentzenKurniawan}. For other weak approximation results using similar techniques, we refer for instance to~\cite{AnderssonLarsson} (finite element approximation), to~\cite{BrehierDebussche} (standard Euler scheme in the multiplicative noise case) or to~\cite{WangGan}. Other strategies are used to prove weak error results for instance in~\cite{AnderssonKruseLarsson}, in~\cite{Conus-Jentzen-Kurniawan,JentzenKurniawan} or in~\cite{BrehierHairerStuart}. The accelerated exponential Euler scheme has been introduced and studied in~\cite{JentzenKloeden,Jentzen}. Weak convergence results for this method have been proved in~\cite{Wang}.

All the works mentioned above deal with SPDEs~\eqref{eq:SPDEintro-field} and SEEs~\eqref{eq:SPDEintro} under the assumption that the nonlinearity $F$ is globally Lipschitz continuous. In the last decade, there has been a huge interest in the approximation of problems with locally Lipschitz continuous nonlinearities: this class encompasses for instance the Allen--Cahn equation. We refer to~\cite{BrehierGoudenege,CuiHong:19,Cai-Gan-Wang:21} for weak convergence results. In this work, we only focus on the globally Lipschitz continuous case, see Section~\ref{sec:extensions} for a discussion on possible generalization to the locally Lipschitz continuous case.

Concerning the approximation of invariant distributions of SEEs~\eqref{eq:SPDEintro}, under appropriate conditions ensuring ergodicity~\cite{DPZergo}, Theorem~\ref{theo:weak-ergo} is a variant of the results obtained in~\cite{B:2014} and~\cite{BrehierKopec} for the standard Euler scheme and finite element approximation. In~\cite{BrehierKopec}, the error analysis is performed using regularity properties of solutions of infinite dimensional Poisson equations. In~\cite{BV}, a postprocessed version of the standard Euler scheme has been introduced, with the objective to increase the order of convergence of the approximation of the invariant distribution. There are similarities between the method proposed in~\cite{BV} and the modified Euler scheme~\eqref{eq:scheme-intro}, see Remark~\ref{rem:postproc} for details: the two methods preserve the Gaussian distribution $\nu$ in the Ornstein--Uhlenbeck case ($F=0$), however~\cite{BV} does not provide error estimates (with higher weak order of convergence) for general nonlinearities $F$. When the gradient condition $F=-DV$ is satisfied, higher-order schemes for the approximation of the invariant Gibbs distribution $\mu_\star$ can be designed using a preconditioning technique, possibly combined with the postprocessing approach, see~\cite{BDV} (see also~\cite{HairerStuartVoss1,HairerStuartVoss2}). For parabolic semilinear SEEs~\eqref{eq:SPDEintro} with non-globally Lipschitz continuous nonlinearities $F$, we refer to the recent articles~\cite{B:2022,CuiHongSun:21,Chen-Gan-Wang:20}, and to~\cite{BG} for the application of the preconditioning technique in this case. We refer also to the monograph~\cite{MR3967113} and references therein for the analysis of this question for some stochastic Schr\"odinger equations. Results on the approximation of the invariant distribution for viscous stochastic conservation laws are provided in~\cite{BoyavalMartelReygner}. Let us emphasize that all the weak convergence results mentioned above on the approximation of the invariant distribution require test functions which are at least of class $\mathcal{C}^2$.

To the best of our knowledge, Theorems~\ref{theo:weakinv},~\ref{theo:weakinv-standard},~\ref{theo:weak-expo} and~\ref{theo:weak-ergo-expo} are the first results in the literature giving an approximation in the total variation distance for numerical approximation of stochastic evolution equations~\eqref{eq:SPDEintro}. The modified Euler scheme therefore overcomes the limitations identified in~\cite{B:2020} for the standard Euler scheme, when considering the approximation of the Gibbs invariant distribution. In the finite dimensional case, this type of result has been proved in the seminal works~\cite{BallyTalay1},~\cite{BallyTalay2}. The recent preprint~\cite{ChenChiHongSheng} studies a related question for the stochastic partial differential equation: the authors prove an approximation result for the density of the real-valued random variables $X(t,\xi)$ solving~\eqref{eq:SPDEintro-field} for given $t\in(0,\infty)$ and $\xi\in(0,1)$, when using the accelerated exponential Euler scheme. However, considering real-valued Gaussian random variables $X(t,\xi)$ and $H$-valued Gaussian random variables $X(t)$ is a very different matter.

Let us finally discuss how the two applications of the modified Euler scheme mentioned above are related to the literature. On the one hand, the analysis of the asymptotic preserving scheme~\eqref{eq:APscheme-intro} for the multiscale system~\eqref{eq:SPDE-slowfast-intro} is a generalization in an infinite dimensional framework of the recent work~\cite{BR}, where the notion of AP schemes for a class of stochastic differential equations has been introduced. The averaging principle (convergence of $\XX^\epsilon$ to $\overline{\XX}$ when $\epsilon\to 0$) has been studied by many authors, we refer to~\cite{MR2480788} for the convergence and to~\cite{B:2012} for weak error estimates. Instead of using the Heterogeneous Multiscale Method like in~\cite{B:2013}, the AP scheme~\eqref{eq:APscheme-intro} also provides an accurate approximation when $\epsilon$ is not assumed to be small. However, the construction of the AP scheme~\eqref{eq:APscheme-intro} is limited to a fast process which solves an Ornstein--Uhlenbeck dynamics. Compared with~\cite{BR}, in this work we only prove that the scheme is asymptotic preserving and do not investigate whether the scheme is uniformly accurate. We refer to~\cite{B} for the proof of uniform weak error estimates for the AP scheme~\eqref{eq:APscheme-intro} applied to the multiscale system~\eqref{eq:SPDE-slowfast-intro} (when $\sigma$ is constant). On the other hand, in order to sample the Gibbs distribution $\mu_\star$, Markov Chain Monte Carlo (MCMC) methods have been constructed using the preconditioned Crank--Nicolson (pCN) as the proposal kernel. In~\cite{CotterRobertsStuart}, it is proved that applying a proposal kernel using the theta-method with $\theta\neq 1/2$ applied to a preconditioned Ornstein--Uhlenbeck dynamics (which preserves the Gaussian invariant distribution $\nu$) leads to a method which is ill-defined in infinite dimension. It is proved in~\cite{HairerStuartVollmer} that the pCN proposal kernel provides a MCMC sampler which has a spectral gap which is independent of dimension, using the techniques introduced in~\cite{HairerMattinglyScheutzow}. The performance of the pCN MCMC sampler is also analyzed in~\cite{MattinglyPillaiStuart} using the point of view of diffusion limits. In this article, we design a new MCMC sampler~\eqref{eq:MCMC-intro}--\eqref{eq:acceptanceMCMC-intro} where the proposal kernel is the modified Euler scheme, and Theorem~\ref{theo:MCMC} is proved following the approach of~\cite{HairerStuartVollmer}. Our analysis only justifies that the proposed MCMC method is well-defined and applicable, however it does not provide information to choose the auxiliary time-step size $\tau$ in an optimal way. In addition, comparing the performances of the different MCMC samplers is out of scope of this work.

\subsection*{Organization of the manuscript}

This manuscript is organized as follows. Section~\ref{sec:setting} provides the necessary notation and assumptions. The modified Euler scheme is introduced in Section~\ref{sec:scheme}. Equivalent formulations of the modified Euler scheme are given in Sections~\ref{sec:scheme-1st},~\ref{sec:scheme-2nd} and~\ref{sec:scheme-3rd}. The main results are stated in Section~\ref{sec:results}. First, qualitative properties (preservation of the spatial regularity and of the Gaussian invariant distribution in the Ornstein--Uhlenbeck case) are studied in Section~\ref{sec:results_quali}. Second, Section~\ref{sec:results_invar} is devoted to the major result of this article, Theorem~\ref{theo:weakinv}, which gives approximation in the total variation distance with order $1/2$ of the Gibbs invariant distribution $\mu_\star$ when the nonlinearity $F$ is a gradient $-DV$. Third, Section~\ref{sec:results_weak} provides Theorem~\ref{theo:weak}, which states that the weak order of convergence of the method is equal to $1/2$ (for sufficiently smooth test functions) in a general setting. Finally, Sections~\ref{sec:results_standard} and~\ref{sec:results_exponential} are devoted to comparing the modified Euler scheme with the standard Euler method and the accelerated exponential Euler method respectively. Auxiliary results are stated and proved in Section~\ref{sec:auxiliary}. In particular, regularity results for solutions of infinite dimensional Kolmogorov equations are given in Sections~\ref{sec:auxiliary-kolmogorov-modified} and~\ref{sec:auxiliary-kolmogorov-original}. Section~~\ref{sec:proofs} is devoted to proving first Theorem~\ref{theo:weakinv}, second Theorem~\ref{theo:weak}. Sections~\ref{sec:standard} and~\ref{sec:expo} contain the proofs of the results given in Sections~\ref{sec:results_standard} (standard Euler scheme) and~\ref{sec:results_exponential} (accelerated exponential Euler scheme) respectively. Finally, Section~\ref{sec:extensions} provides two applications of the modified Euler scheme. First, an asymptotic preserving scheme is introduced for a class of systems with two time scales, in a regime governed by the averaging principle. Second, a Markov Chain Monte Carlo sampler which is well-defined and has a spectral gap in infinite dimension (Section~\ref{sec:MCMC}, Theorem~\ref{theo:MCMC}) is studied. Moroever, the generalization of the proposed modified Euler scheme for other types of parabolic semilinear SPDEs is discussed in Section~\ref{sec:generalizations}.

\section{Setting}\label{sec:setting}

This section is organized as follows. General notation is first introduced in Section~\ref{sec:notation}. Sections~\ref{sec:IL} and~\ref{sec:F} then provide the abstract conditions on the linear and nonlinear operators $\IL$ and $F$ respectively. Examples operators satisfying the abstract conditions are described in Section~\ref{sec:example}. Some properties of the stochastic evolution equation are provided in Section~\ref{sec:SPDE}. Section~\ref{sec:invariant} deals with the invariant distribution of the stochastic evolution equation. Finally Section~\ref{sec:Galerkin} presents an auxiliary approximation procedure which is used implicitly in the sequel.

\subsection{Notation}\label{sec:notation}

Let us first introduce general notation. The set of integers is denoted by $\N=\{1,2,\ldots\}$, and $\N_0=\{0\}\cup\N$. In the sequel, the values of positive real numbers $C\in(0,\infty)$ may change from line to line.

The time-step size of the numerical schemes is denoted by $\tau$. For all $n\in\N_0$, let $t_n=n\tau$. The  moment and error estimates below are stated for values $\tau\in(0,\tau_0)$, where $\tau_0$ is an arbitrary positive real number, and without loss of generality one may assume that $\tau_0<1$. The values of constants $C$ are allowed to depend on $\tau_0$, but they are independent of $\tau\in(0,\tau_0)$. If $t_1,t_2\ge 0$, set $t_1\wedge t_2=\min(t_1,t_2)$.

The state space is a separable Hilbert space $H$, equipped with inner product and norm denoted by $\langle\cdot,\cdot\rangle$ and $|\cdot|$ respectively. The set of bounded linear operators from $H$ to $H$ is denoted by $\mathcal{L}(H)$, which is a Banach space with the norm $\|\cdot\|_{\mathcal{L}(H)}$ defined by
\[
\|L\|_{\mathcal{L}(H)}=\underset{x\in H\setminus\{0\}}\sup~\frac{|Lx|}{|x|}.
\]
In addition, $\mathcal{L}_2(H)\subset\mathcal{L}(H)$ denotes the set of Hilbert--Schmidt operators from $H$ to $H$. The set $\mathcal{L}_2(H)$ is an Hilbert space, with the norm $\|\cdot\|_{\mathcal{L}_2(H)}$ defined by
\[
\|L\|_{\mathcal{L}_2(H)}^2=\sum_{j\in\N}|L{\bf e}_j|^2,
\]
where $\bigl({\bf e}_j\bigr)_{j\in\N}$ is an arbitrary complete orthonormal system of $H$.

The random variables and the stochastic processes considered in this article are defined on a probability space denoted by $(\Omega,\mathcal{F},\mathbb{P})$. This probability space is equipped with a filtration $\bigl(\mathcal{F}_t\bigr)_{t\ge 0}$ which is assumed to satisfy the usual conditions. The expectation operator is denoted by $\E[\cdot]$.

Let $\bigl(\beta_j\bigr)_{j\in\N}$ denote a sequence of independent standard real-valued Wiener processes, adapted to the filtration $\bigl(\mathcal{F}_t\bigr)_{t\ge 0}$. This means in particular that for each $j\in\N$, $\bigl(\beta_j(t)\bigr)_{t\ge 0}$ is a Gaussian process, such that one has $\beta_j(0)=0$, $\E[\beta_j(t)]=0$ for all $t\ge 0$ and $\E[|\beta_j(t_2)-\beta_j(t_1)|^2]=t_2-t_1$ for all $t_2\ge t_1\ge 0$.

The cylindrical Wiener process $\bigl(W(t)\bigr)_{t\ge 0}$ on $H$ is formally defined as
\begin{equation}\label{eq:Wiener}
W(t)=\sum_{j\in\N}\beta_j(t){\bf e}_j
\end{equation}
where $\bigl({\bf e}_j\bigr)_{j\in\N}$ is an arbitrary complete orthonormal system of $H$. Recall that if $L\in\mathcal{L}_2(H)$, $\bigl(LW(t)\bigr){t\ge 0}$ is a well-defined Gaussian process, with $\E[|LW(t)|^2]=\|L\|_{\mathcal{L}_2(H)}^2t$. However, $\bigl(W(t)\bigr)_{t\ge 0}$ does not take values in $H$: for all $t>0$, $\E[|W(t)|^2]=\infty$, and even $|W(t)|=\infty$ almost surely. We refer to~\cite[Chapter~4]{DPZ} for a description of the cylindrical Wiener processes and of the related theory of stochastic integration in Hilbert spaces. Let us recall a version of the It\^o isometry formula in this setting: if $\Phi:t\in[0,T]\mapsto \Phi(t)\in\mathcal{L}_2(H)$ is a continuous deterministic mapping, the random variable
\[
\int_0^T\Phi(t)dW(t)=\sum_{j\in\N}\int_0^T\Phi(t){\bf e}_jd\beta_j(t)
\]
is a centered $H$-valued Gaussian random variable, with
\[
\E[|\int_0^T\Phi(t)dW(t)|^2]=\int_0^T\|\Phi(t)\|_{\mathcal{L}_2(H)}^{2}dt=\sum_{j\in\N}\int_0^T |\Phi(t){\bf e}_j|^2 dt.
\]
One of the proofs below requires tools from Malliavin calculus~\cite{Nualart}. We do not give precise definitions, instead let us state the notation used in this article and quote the most useful results. If $\Theta$ is an $H$-valued random variable, $\mathcal{D}_s^h\Theta\in H$ is the Malliavin derivative of $\Theta$ at time $s$ in direction $h\in H$. For instance, this means that
\[
\mathcal{D}_s^h\bigl(\int_0^TL(t)dW(t)\bigr)=L(s)h
\]
if $t\in[0,T]\mapsto L(t)\in\mathcal{L}(H)$ is an adapted process. In addition, if $\Theta$ is $\mathcal{F}_t$-measurable, then $\mathcal{D}_s^h\Theta=0$ for all $s>t$. The Malliavin derivative satisfies a chain rule property: if $\Phi:H\to H$ is of class $\mathcal{C}^1$ with bounded derivative, then for all $s\ge 0$ and $h\in H$ one has
\[
\mathcal{D}_s^h\phi(\Theta)=D\phi(\Theta).\mathcal{D}_s^h\Theta.
\]
The same type of notation and results are satisfied for $\R$-valued random variables $\theta$. Finally, one has the following integration by parts formula, which is essential for the proof of weak error estimates, see~\cite{Debussche:11}: if $\theta$ is $\R$-valued random variable and if $(t,s)\mapsto\phi(t,s)\in\R$ is a given deterministic function, for all $j\in\N$, one has
\begin{equation}\label{eq:MalliavinIBP}
\E\bigl[\theta\int_0^t\phi(t,s)d\beta_j(s)\bigr]=\int_0^t\E[\mathcal{D}_s^{e_j}\theta \phi(t,s)d\beta_j(s)].
\end{equation}

In addition, introduce the following notation. If $\varphi:H\to\R$ is a mapping of class $\mathcal{C}^{0}$, $\mathcal{C}^1$ or $\mathcal{C}^2$ respectively, set
\begin{align*}
\vvvert\varphi\vvvert_0&=\underset{x\in H}\sup~|\varphi(x)|,\\
\vvvert\varphi\vvvert_1&=\underset{x,h\in H}\sup~\frac{|D\varphi(x).h|}{|h|},\\
\vvvert\varphi\vvvert_2&=\underset{x,h_1,h_2\in H}\sup~\frac{|D^2\varphi(x).(h_1,h_2)|}{|h_1||h_2|}.
\end{align*}
Note that $\vvvert \varphi\vvvert_0<\infty$ if and only if $\varphi$ is bounded. Similarly, $\vvvert\varphi\vvvert_1<\infty$ and $\vvvert\varphi\vvvert_2<\infty$ when $\varphi$ has a bounded first order derivative, respectively a bounded second order derivative.

The set of bounded and measurable mappings from $H$ to $\R$ is denoted by $\mathcal{B}_b(H)$. For any $\varphi\in\mathcal{B}_b(H)$, set
\[
\vvvert\varphi\vvvert=\underset{x\in H}\sup~|\varphi(x)|.
\]
The total variation distance between two Borel probability distributions $\mu_1$ and $\mu_2$ defined on $H$ is defined by
\[
d_{\rm TV}(\mu_1,\mu_2)=\underset{\varphi\in\mathcal{B}_b(H),\varphi\neq 0}\sup~\frac{\big|\int \varphi d\mu_1-\int \varphi d\mu_2\big|}{\vvvert \varphi\vvvert}. 
\]
Introduce also distances $d_0$ and $d_1$ defined by
\begin{align*}
d_0(\mu_1,\mu_2)&=\underset{\varphi\in\mathcal{C}^0(H),\varphi\neq 0}\sup~\frac{\big|\int \varphi d\mu_1-\int \varphi d\mu_2\big|}{\vvvert \varphi\vvvert_0}\\
d_2(\mu_1,\mu_2)&=\underset{\varphi\in\mathcal{C}^2(H),\varphi\neq 0}\sup~\frac{\big|\int \varphi d\mu_1-\int \varphi d\mu_2\big|}{\vvvert \varphi\vvvert_0+\vvvert\varphi\vvvert_1+\vvvert\varphi\vvvert_2}.
\end{align*}
Observe that the inequality $d_2(\mu_1,\mu_2)\le d_0(\mu_1,\mu_2)$ holds. Recall (see for instance~\cite[Chapter~3,Section~4]{EthierKurtz}) that for any bounded and measurable function $\varphi\in\mathcal{B}_b(H)$, there exists a sequence $\bigl(\varphi_k\bigr)_{k\in\N}$ of bounded and continuous functions which converges boundedly and pointwise to $\varphi$, i.\,e. which satisfies $\underset{k\in\N}\sup~\vvvert\varphi_k\vvvert<\infty$ and $\varphi_k(x)\underset{k\to\infty}\to \varphi(x)$ for all $x\in H$. As a consequence one has the equality
\begin{equation}\label{eq:distances}
d_{\rm TV}(\mu_1,\mu_2)=d_0(\mu_1,\mu_2)
\end{equation}
for all Borel probability distributions $\mu_1$ and $\mu_2$.

Finally, if $X$ is a $H$-valued random variable, the distribution of $X$ is denoted by $\rho_X$: this means that
\[
\E[\varphi(X)]=\int \varphi(x)d\rho_X(x)
\]
for all $\varphi\in\mathcal{B}_b(H)$. If $\mu$ is a Borel probability distribution on $H$, one has
\[
d_{\rm TV}(\rho_X,\mu)=\underset{\varphi\in\mathcal{B}_b(H),\vvvert\varphi\vvvert\le 1}\sup~\big|\E[\varphi(X)]-\int\varphi d\mu\big|=\underset{\varphi\in\mathcal{C}^0(H),\vvvert\varphi\vvvert_0\le 1}\sup~\big|\E[\varphi(X)]-\int\varphi d\mu\big|.
\]

\subsection{Assumptions on the linear operator}\label{sec:IL}

The stochastic evolution equation~\eqref{eq:SPDEintro} is driven by an unbounded self-adjoint linear operator $-\IL:D(\IL)\subset H\to H$, which is assumed to satisfy the following conditions.

\begin{ass}\label{ass:Lambda}
There exists a complete orthonormal system $\bigl(e_j\bigr)_{j\in\N}$ of $H$ and a non-decreasing sequence $\bigl(\lambda_j\bigr)_{j\in\N}$ of positive real numbers, such that
\[
\IL e_j=\lambda_je_j
\]
for all $j\in\N$. In addition, it is assumed that there exists $c_{\IL}\in(0,\infty)$ that $\lambda_j\sim c_{\IL}j^2$ when $j\to\infty$.
\end{ass}

The self-adjoint unbounded linear operator $-\IL$ generates a semigroup which is denoted by $\bigl(e^{-t\IL}\bigr)_{t\ge 0}$. Precisely, for all $t\ge 0$ and $x\in H$, set
\[
e^{-t\IL}x=\sum_{j\in\N}e^{-t\lambda_j}\langle x,e_j\rangle e_j.
\]
For all $t\ge 0$, $e^{-t\IL}$ is a bounded self-adjoint linear operator on $H$, with $\|e^{-t\IL}\|_{\mathcal{L}(H)}\le e^{-\lambda_1t}\le 1$.

In addition, for all $\alpha\in[-1,1]$, define the self-adjoint linear operators $\IL^{\alpha}$ such that
\[
\IL^\alpha e_j=\lambda_j^\alpha e_j
\]
for all $j\in\N$. Equivalently,
\[
\IL^\alpha x=\sum_{j\in\N}\lambda_j^\alpha \langle x,e_j\rangle e_j.
\]
If $\alpha\in[-1,0]$, $\IL^\alpha$ is a bounded linear operator from $H$ to $H$ and the expression above is well-defined for all $x\in H$. For all $\alpha\in[0,1]$, introduce the notation
\[
|x|_\alpha=\bigl(\sum_{j\in\N}\lambda_j^{2\alpha}\langle x,e_j\rangle^2\bigr)^{\frac12}\in[0,\infty],
\]
then $\IL^\alpha$ is an unbounded self-adjoint linear operator with domain $D(\IL^\alpha)=H^\alpha$, defined by
\[
H^\alpha=\{x\in H;~|x|_\alpha<\infty\}.
\]
The definition of $\IL^\alpha$ when $\alpha=1$ coincides with the definition of $\IL$. When $\alpha=0$, $\IL^0$ is the identity operator denoted by $I$, and $|\cdot|_0=|\cdot|$ is the usual norm in the Hilbert space $H$. Note that for all $-1\le \alpha_1\le \alpha_2\le 1$, there exists $C_{\alpha_1,\alpha_2}\in(0,\infty)$ such that $|\IL^{\alpha_1}x|\le C_{\alpha_1,\alpha_2}|\IL^{\alpha_2}x|$ for all $x\in D(\IL^{\alpha_2})$.

One of the main ingredients used in the analysis below is the following smoothing property: for all $\alpha\in[0,1]$, one has
\begin{equation}\label{eq:smoothing}
\underset{t\in(0,\infty)}\sup~t^\alpha \|\IL^\alpha e^{-t\IL}\|_{\mathcal{L}(H)}<\infty.
\end{equation}
The smoothing property~\eqref{eq:smoothing} is often used in the following form in the sequel:
\[
|e^{-t\IL}x|\le C_\alpha t^{-\alpha}|\IL^{-\alpha}x|
\]
for all $t\in(0,\infty)$ and $x\in H$.

In addition, the following property is satisfied: for all $\alpha\in[0,1]$, one has
\begin{equation}\label{eq:regularity}
\underset{t\in(0,\infty)}\sup~\frac{\|\IL^{-\alpha}\bigl(e^{-t\IL}-I\bigr)\|_{\mathcal{L}(H)}}{t^\alpha}<\infty.
\end{equation}
The proofs of the two standard inequalities~\eqref{eq:smoothing} and~\eqref{eq:regularity} are straightforward and are omitted.

\subsection{Assumptions on the nonlinearity}\label{sec:F}

The nonlinear operator $F:H\to H$ is assumed to be globally Lipschitz continuous.
\begin{ass}\label{ass:F}
There exists $\Lf\in(0,\infty)$ such that for all $x_1,x_2\in H$
\[
|F(x_2)-F(x_1)|\le \Lf|x_2-x_1|.
\]
\end{ass}
Assumption~\ref{ass:F} is the minimal condition which ensures the well-posedness of the stochastic evolution equation (see Section~\ref{sec:SPDE}) and which permits the definition of numerical schemes below. In the sequel, Assumption~\ref{ass:F} is always assumed to be satisfied. However note that the analysis of the long-time behavior and the proof of weak error estimates requires to impose additional assumptions on the nonlinearity $F$ which are provided below.

Let us first describe the conditions related to the long time behavior of the stochastic evolution equation. If Assumption~\ref{ass:ergo} below is satisfied, the process is ergodic, see Section~\ref{sec:invariant}.
\begin{ass}\label{ass:ergo}
Let $\Lf$ be defined in Assumption~\ref{ass:F}.

Assume that $\Lf<\lambda_1$, where $\lambda_1=\underset{j\in\N}\min~\lambda_j$ (see Assumption~\ref{ass:Lambda}).
\end{ass}
The ergodicity of the process ensures the existence of a unique invariant distribution denoted by $\mu_\infty$. In general, no expression of $\mu_\infty$ is known, however an expression is available when Assumption~\ref{ass:gradient} below is satisfied.
\begin{ass}\label{ass:gradient}
There exists a function $V:H\to\R$ of class $\mathcal{C}^1$ such that for all $x\in H$, $F(x)=-DV(x)$, where $D$ denotes the Fr\'echet derivative operator. 
\end{ass}

Let us now state the two regularity assumptions on the nonlinearity $F$ which are required to prove the weak error estimates below.
\begin{ass}\label{ass:Fregul1}
For all $\delta\in(0,\frac14)$, there exists $C_{\delta}\in(0,\infty)$ such that for all $x_1,x_2\in H^{\frac{1-\delta}{4}}$, one has
\[
\big|\IL^{-\frac12+\frac{\delta}{4}}\bigl(F(x_2)-F(x_1)\bigr)\big|\le C_{\delta}\bigl(1+|x_1|_{\frac{1-\delta}{4}}+|x_2|_{\frac{1-\delta}{4}}\bigr)\big|\IL^{-\frac14+\delta}(x_2-x_1)\big|.
\]
\end{ass}

\begin{ass}\label{ass:Fregul2}
The nonlinearity $F$ is twice differentiable and there exist $\alpha_F\in[0,1)$ and $C_F\in(0,\infty)$, such that for all $x,h_1,h_2\in H$, one has
\[
|\IL^{-\alpha_F}D^2F(x).(h_1,h_2)|\le C_{F}|h_1| |h_2|.
\]
\end{ass}
Note that by the global Lipschitz continuity of $F$ (Assumption~\ref{ass:F}), when $F$ is differentiable one has
\[
|DF(x).h|\le \Lf|h|
\]
for all $x,h\in H$.

Whereas Assumption~\ref{ass:F} is always assumed to be satisfied in the sequel, the four other assumptions may not always satisfied simultaneously, and which of these assumptions are required to be satisfied is written explicitly for each of the results stated below.

\subsection{Description of an example}\label{sec:example}

The objective of this section is to show that the stochastic partial differential equation~\eqref{eq:SPDEintro-field} firs in the abstract framework described above.

Let $H=L^2(0,1)$, and let $a:[0,1]\to R$ be a smooth mapping, with $\min(a)>0$. The operator $\IL$ is defined by
\[
\IL x(\xi)=-\partial_\xi\bigl(a(\xi)x(\xi)\bigr)
\]
for all $x\in D(\IL)=H_0^1(0,1)\cap H^2(0,1)$ satisfy Assumption~\ref{ass:Lambda}. The choice of the domain is related to the homogeneous Dirichlet boundary conditions imposed in~\eqref{eq:SPDEintro-field}. When $a(\cdot)=1$, $-\IL$ is the standard Laplace operator with homogeneous Dirichlet boundary conditions. In that case, $\lambda_j=(j\pi)^2$ and $e_j=\sqrt{2}\sin(j\pi\cdot)$ for all integers $j\in\N$.

Let us now deal with the nonlinearity $F$ in this setting. Let $f:\R\to\R$ be a mapping of class $\mathcal{C}^2$ with bounded first and second order derivatives. Define the nonlinearity $F:H\to H$ such that
\[
F(x)=f\bigl(x(\cdot)\bigr)
\]
for all $x\in H=L^2(0,1)$. The operator $F$ is called a Nemytskii operator. Let us check that the regularity assumptions from Section~\ref{sec:F} are satisfied in this example.
$\bullet$ Since $f$ is globally Lipschitz continuous, it is straightforward to check that $F$ is also globally Lipschitz continuous.
$\bullet$ Assumption~\ref{ass:F} is satisfied when $\underset{z\in\R}\sup~|f'(z)|<\lambda_1$.
$\bullet$ Assumption~\ref{ass:gradient} is satisfied with
\[
V(x)=-\int_0^1 v(x(\xi))d\xi
\]
where the mapping $v:\R\to\R$ is an antiderivative of $f$, i.\,e. $v'=f$. Indeed, for all $x,h\in H$, one has
\[
DV(x).h=-v'(x(\cdot))h(\cdot)=f(x(\cdot))h(\cdot)=F(x)h
\]

To check that Assumptions~\ref{ass:Fregul1} and~\ref{ass:Fregul2} are satisfied, some auxiliary inequalities are needed. Owing to~\cite[Theorem~16.12]{Yagi}, for all $\alpha\in[0,\frac14)$, one has
\[
H^\alpha=W^{2\alpha,2}(0,1)
\]
and for all $\alpha\in(\frac14,1]$, one has
\[
H^\alpha=W_0^{2\alpha,2}(0,1)=\{x\in W^{2\alpha,2}(0,1);~x(0)=x(1)=0\},
\]
where $W^{2\alpha,2}(0,1)$ are the standard fractional Sobolev spaces. Moreover, the norms $|\cdot|_{\alpha}$ and $|\cdot|_{W^{2\alpha,2}(0,1)}$ are equivalent: for all $\alpha\in[0,1]\setminus\{\frac14\}$, there exists $C_\alpha\in(0,\infty)$ such that
\[
C_\alpha^{-1}|\cdot|_\alpha\le |\cdot|_{W^{2\alpha,2}(0,1)}\le C_\alpha|\cdot|_\alpha.
\]
Let us introduce the Banach spaces $L^\infty(0,1)$ and $L^1(0,1)$, and recall several useful inequalities. First, for all $\delta>0$, there exists $C_\delta\in(0,\infty)$ such that
\[
|x|_{L^\infty(0,1)}\le C_\delta |x|_{W^{\frac12+2\delta}}\le C_\delta |x|_{\frac14+\delta}.
\]
By a duality argument, one then obtains the inequality
\begin{equation}\label{eq:Sobolev}
|\IL^{-\frac14-\delta}x|_{L^2(0,1)}\le |x|_{L^1(0,1)}
\end{equation}
for all $x\in L^1(0,1)$. Moreover, for all $\delta>0$, there exists $C_{\delta}\in(0,\infty)$ such that
\begin{equation}\label{eq:produit}
|\IL^{-\frac14+\delta}(x_1x_2)|_{L^1(0,1)}\le C_{\delta}|\IL^{-\frac14+2\delta}x_1|_{L^2(0,1)}|\IL^{\frac14-\delta}x_2|_{L^2(0,1)}.
\end{equation}
Finally, for all $\delta\in(0,\frac14)$, and for any Lipschitz continuous function $g:\R^2\to\R$, there exists $C_{\delta}(g)\in(0,\infty)$ such that for all $x_1,x_2\in H^{\frac14-\frac{\delta}{4}}$, one has
\begin{equation}\label{eq:compo}
|g(x_1,x_2)|_{\frac14-\frac{\delta}{2}}\le C_{\delta}(g)\bigl(|x_1|_{\frac14-\frac{\delta}{4}}+|x_2|_{\frac14-\frac{\delta}{4}}\bigr).
\end{equation}
We refer to~\cite[Section~3.2]{BrehierDebussche} for the proofs of these inequalities, using properties of the standard fractional Sobolev spaces.

Using the inequalities above, we are now in position to check the remaining assumptions.

$\bullet$ Assumption~\ref{ass:Fregul1} is satisfied. Let $\delta\in(0,\frac14)$ and $x_1,x_2\in H^{\frac14-\delta}$. Observe that
\[
F(x_2)-F(x_1)=(x_2-x_1)g(x_1,x_2)
\]
where for all $z_1,z_2\in \R$ $g(z_1,z_2)=\int_0^1 f'\bigl((1-\theta)z_1+\theta z_2\bigr)d\theta$. Applying succesively~\eqref{eq:Sobolev},~\eqref{eq:produit} and~\eqref{eq:compo} (with a constant $C$ which may vary from line to line and depends on $\delta$), one obtains
\begin{align*}
|\IL^{-\frac12+\frac{\delta}{4}}\bigl(F(x_2)-F(x_1)\bigr)|_{L^2(0,1)}&\le C|\IL^{-\frac14+\frac{\delta}{2}}\bigl(F(x_2)-F(x_1)\bigr)|_{L^1(0,1)}\\
&\le C|\IL^{-\frac14+\delta}(x_2-x_1)|_{L^2(0,1)}|\IL^{\frac14-\frac{\delta}{2}}g(x_1,x_2)|_{L^2(0,1)}\\
&\le C|\IL^{-\frac14+\delta}(x_2-x_1)|_{L^2(0,1)}\bigl(|\IL^{\frac14-\frac{\delta}{4}}x_1|_{L^2(0,1)}+|\IL^{\frac14-\frac{\delta}{4}}x_2|_{L^2(0,1)}\bigr).
\end{align*}

$\bullet$ Assumption~\ref{ass:Fregul2} is satisfied: this is a straightforward consequence of the inequality~\eqref{eq:Sobolev} and of the identity
\[
D^2F(x).(h_1,h_2)=f''(x(\cdot))h_1(\cdot)h_2(\cdot).
\]
One can then choose $\alpha_F=\frac14+\delta$ for arbitrarily small $\delta\in(0,\frac34)$, then the inequality holds with $C_F=C_\varepsilon \underset{z\in\R}\sup~|f''(z)|$.

\subsection{Well-posedness and regularity properties}\label{sec:SPDE}

We are now in position to study the well-posedness property of the stochastic evolution equation
\begin{equation}\label{eq:SPDE}
dX(t)=-\IL X(t)dt+F(X(t))dt+dW(t),\quad X(0)=x_0,
\end{equation}
where the linear operator $\IL$ is introduced in Section~\ref{sec:IL}, the nonlinearity $F$ is introduced in Section~\ref{sec:F}, and the cylindrical Wiener process $\bigl(W(t)\bigr)_{t\ge 0}$ is introduced in Section~\ref{sec:notation}. The initial value $x_0$ is an arbitrary element of $H$. For simplicity, it is assumed that $x_0$ is deterministic, however the extension of the results below to random $\mathcal{F}_0$-measurable initial values with suitable moment bounds is straightforward and is omitted.

An $H$-valued continuous stochastic process $\bigl(X(t)\bigr)_{t\ge 0}$ is called a mild solution of the stochastic evolution equation~\eqref{eq:SPDE} if it satisfies for all $t\ge 0$
\begin{equation}\label{eq:mild}
X(t)=e^{-t\IL}x_0+\int_{0}^{t}e^{-(t-s)\IL}F(X(s))ds+\int_{0}^{t}e^{-(t-s)\IL}dW(s).
\end{equation}
It is convenient to introduce the stochastic convolution $\bigl(W^{\IL}(t)\bigr)_{t\ge 0}$ defined by
\begin{equation}\label{eq:StochasticConvolution}
W^{\IL}(t)=\int_{0}^{t}e^{-(t-s)\IL}dW(s)
\end{equation}
for all $t\ge 0$. Owing to Assumption~\ref{ass:Lambda}, the stochastic convolution defines a $H$-valued Gaussian process. In particular, applying It\^o's isometry formula yields for all $t\ge 0$
\[
\E[|W^{\IL}(t)|^2]=\int_0^t \|e^{-s\IL}\|_{\mathcal{L}_2(H)}^2ds=\sum_{j\in\N}\int_0^t e^{-2s\lambda_j}ds\le \sum_{j\in\N}\frac{1}{2\lambda_j}<\infty.
\]
The following well-posedness result is then obtained applying a standard fixed point argument. We refer for instance to~\cite[Section~7.1]{DPZ}.
\begin{propo}\label{propo:SPDE}
Let Assumptions~\ref{ass:Lambda} and~\ref{ass:F} be satisfied. Then the stochastic evolution equation~\eqref{eq:SPDE} admits a unique mild solution $\bigl(X(t)\bigr)_{t\ge 0}$, for any arbitrary initial value $x_0\in H$.
\end{propo}
The dependence of the mild solution with respect to the initial value $x_0$ is often omitted to simplify the notation. The notation $\E_x[\varphi(X(t))]$ may be used below to compute the expected value of $\varphi(X(t))$ when the solution $X(t)$ has initial value $X(0)=x$.

In addition, the following spatial and temporal regularity properties are satisfied: for all $\alpha\in[0,\frac14)$ and all $T\in(0,\infty)$, there exists $C_\alpha(T)\in(0,\infty)$ such that for all $x_0\in H$ and all $t,t_1,t_2\in(0,T]$, one has
\begin{align*}
\bigl(\E[|X(t)|_\alpha^2]\bigr)^{\frac12}&\le C_\alpha(T)\bigl(1+t^{-\alpha}|x_0|),\\
\bigl(\E[|X(t_2)-X(t_1)|^2]\bigr)^{\frac12}&\le C_\alpha(T)|t_2-t_1|^{\alpha}(1+(t_1\wedge t_2)^{-\alpha}|x_0|).
\end{align*}
The proof of those regularity properties is based on combinations of the smoothing property~\eqref{eq:smoothing} with the error estimate~\eqref{eq:regularity}, using the mild formulation~\eqref{eq:mild}. Since the result is standard, the detailed proof is omitted.

\subsection{Invariant distribution}\label{sec:invariant}

Let us now study the long-time behavior of the mild solution $\bigl(X(t)\bigr)_{t\ge 0}$ given by~\eqref{eq:mild}, when Assumption~\ref{ass:ergo} is satisfied. We refer to the monograph~\cite{DPZergo} for a general presentation of ergodicity results for parabolic semilinear SPDEs.

First, in a general setting, one has the following result.
\begin{propo}\label{propo:invar}
Let Assumption~\ref{ass:ergo} be satisfied. Then the stochastic evolution equation admits a unique invariant distribution $\mu_\infty$, and there exists $C\in(0,\infty)$ such that for any function $\varphi:H\to \R$ of class $\mathcal{C}^1$ with bounded derivative, for all $T\ge 0$ and $x\in H$, one has
\[
\big|\E_x[\varphi(X(T))]-\int\varphi d\mu_\infty\big|\le C\vvvert\varphi\vvvert_1 e^{-(\lambda_1-\Lf)T}(1+|x|).
\]
\end{propo}
As already explained, in general no expression of the invariant distribution $\mu_\infty$ is known, except when Assumption~\ref{ass:gradient} is satisfied.

First, when $F=0$, the invariant distribution of the stochastic convolution $\bigl(W^{\IL}(t)\bigr)_{t\ge 0}$ defined by~\eqref{eq:StochasticConvolution} is the centered Gaussian distribution
\begin{equation}\label{eq:nu}
\nu=\mathcal{N}(0,\frac12\IL^{-1}),
\end{equation}
which is the distribution of the $H$-valued Gaussian random variable $Z=\sum_{j\in\N}\frac{\gamma_j}{\sqrt{2\lambda_j}}e_j$ where $\bigl(\gamma_j\bigr)_{j\in\N}$ is a sequence of independent standard real-valued Gaussian random variables.

Second, when $F=-DV$ has a gradient structure (Assumption~\ref{ass:gradient}), for some $V:H\to \R$, the invariant distribution $\mu_\infty$ is a Gibbs distribution with respect to the reference Gaussian distribution $\nu$. We refer to~\cite[Theorem~8.6.3]{DPZergo}.
\begin{propo}\label{propo:mu_star}
Let Assumptions~\ref{ass:ergo} and~\ref{ass:gradient} be satisfied. Then the invariant distribution $\mu_\infty$ of the stochastic evolution equation~\eqref{eq:SPDE} is equal to the Gibbs distribution $\mu_\star$ defined by
\begin{equation}\label{eq:mu_star}
d\mu_\star(x)=\mathcal{Z}^{-1}e^{-2V(x)}d\nu(x)
\end{equation}
with normalization constant $\mathcal{Z}=\int e^{-2V(x)}d\nu(x)\in(0,\infty)$. Moreover, one has the bound
\begin{equation}\label{eq:mu_star-bound}
\int |x|d\mu_\star(x)<\infty.
\end{equation}
\end{propo}

\subsection{Spectral Galerkin approximation}\label{sec:Galerkin}

In order to justify most of the arguments provided below, it is convenient to introduce an auxiliary finite dimensional approximation procedure. In particular, in the finite dimensional framework, the derivatives may be interpreted as Fr\'echet derivatives, all the linear operators are bounded, and the solutions of the auxiliary Kolmogorov partial differential equations can be understood in a classical sense. On the contrary, differentiability conditions in infinite dimension may require more care, and giving meaning to solutions of infinite dimensional Kolmogorov equations is more involved.

The auxiliary approximation procedure is standard in the literature. In this work, one can employ a spectral Galerkin approximation: for all $J\ge 1$, let $P^J$ be the orthogonal projection operator defined by
\[
P^Jx=\sum_{j=1}^{J}\langle x,e_j\rangle e_j
\]
and introduce the stochastic evolution equation
\[
dX^J(t)=-\IL X^J(t)dt+P^JF(X^J(t))dt+P^JdW(t),\quad X^J(0)=P^Jx_0.
\]
In fact, $\bigl(X^J(t)\bigr)_{t\ge 0}$ is solution of a stochastic differential equation with values in the finite dimensional space $H^J={\rm span}(e_1,\ldots,e_J)$.

To obtain the results stated below, it suffices to combine two arguments:
\begin{itemize}
\item proving moment and error estimates for $X^J$ which are uniform with respect to $J\in\N$,
\item letting $J\to\infty$.
\end{itemize}
To simplify the notation, in the sequel, the dimension $J$ is omitted. All the proofs of moment and error estimates should be understood as being performed for $X^J$, with bounds independent of $J$. Similary, the regularity properties for solutions of Kolmogorov equations should also be understood in an approximate finite dimensional framework, with bounds independent of $J$. This standard convention is used everywhere in the sequel.

\section{Description of the modified Euler scheme}\label{sec:scheme}

The objective of this section is to provide the definition of the proposed modified Euler scheme. Three equivalent formulations of the integrator are given below, however these formulations serve different purposes: the practical implementation is performed using the first one, whereas the second and third ones are employed to prove moment bounds and error estimates.

To define the numerical integrator, the following definition is convenient. A random variable $\Gamma$ is called a \emph{cylindrical Gaussian random variable} if
\[
\Gamma=\sum_{n\in\N}\gamma_n e_n
\]
where $\bigl(e_n\bigr)_{n\in\N}$ is the complete orthonormal system of $H$ given in Assumption~\ref{ass:Lambda}, and $\bigl(\gamma_n\bigr)_{n\in\N}$ are independent real valued standard Gaussian random variables ($\E[\gamma_n]=0$, $\E[\gamma_n^2]=1$ for all $n\in\N$ and $\E[\gamma_n\gamma_m]=0$ for all $n\neq m\in\N$). Note that increments of the cylindrical Wiener processes $\Delta W_n=W(t_{n+1})-W(t_n)$ for $n\in\N_0$ satisfy the equality in distribution $\Delta W_n=\sqrt{\tau}\Gamma_n$, where $\bigl(\Gamma_n\bigr)_{n\in\N}$ is a sequence of independent cylindrical Gaussian random variables. Observe that cylindrical Gaussian random variables $\Gamma$ do not take values in $H$: $\E[|\Gamma|^2]=\infty$, and even $|\Gamma|^2=\infty$ almost surely. However, if $L\in\mathcal{L}_2(H)$ is an Hilbert--Schmidt operator, then $L\Gamma$ is a well-defined $H$-valued centered Gaussian random variable with $\E[|L\Gamma|^2]=\|L\|_{\mathcal{L}_2(H)}^2$.

In order to explain the construction of the proposed modified Euler integrator, let us introduce the standard linear implicit Euler scheme. Let $\bigl(\Gamma_n\bigr)_{n\in\N_0}$ be a sequence of independent cylindrical Gaussian random variables, then set for all $n\in\N_0$
\begin{equation}\label{eq:scheme-standard}
X_{n+1}^{\tau,\s}=\IA_\tau\Bigl(X_n^{\tau,\s}+\tau F(X_n^{\tau,\s})+\sqrt{\tau}\Gamma_n\Bigr),
\end{equation}
with initial value $X_0^{\tau,\s}=x_0$, where the bounded linear operator $\IA_\tau$ is defined by
\begin{equation}\label{eq:IA}
\IA_\tau=(I+\tau\IL)^{-1}.
\end{equation}
The integrator~\eqref{eq:scheme-standard} formally satisfies the equality
\[
X_{n+1}^{\tau,\s}=X_n^{\tau,\s}-\tau\IL X_{n+1}^{\tau,\s}+\tau F(X_n^{\tau,\s})+\Gamma_n,
\]
which justifies to interpret~\eqref{eq:scheme-standard} as a semi-implicit Euler scheme, where the linearity is treated implicitly and the nonlinearity is treated explicitly. The formulation~\eqref{eq:scheme-standard} is more suitable since $\IA_\tau$ is a bounded linear operator, whereas $\IL$ is unbounded. In addition, $\IA_\tau\Gamma_n$ is a well-defined $H$-valued Gaussian random variable: indeed $\IA_\tau$ is an Hilbert--Schmidt operator for all $\tau>0$, with
\[
\|\IA_\tau\|_{\mathcal{L}_2(H)}^2=\sum_{j\in\N}|\IA_\tau e_j|^2=\sum_{j\in\N}\frac{1}{(1+\lambda_j\tau)^2}<\infty.
\]
As a consequence, for any initial value $x_0\in H$, one has $X_n^{\tau,\s}\in H$ for all $n\in\N$.

Properties of the standard Euler scheme~\eqref{eq:scheme-standard} are recalled in Section~\ref{sec:standard}.

\subsection{Definition of the modified Euler scheme}\label{sec:scheme-1st}

We are now in position to define the modified Euler scheme. Let $\bigl(\Gamma_{n,1}\bigr)_{n\in\N_0}$ and $\bigl(\Gamma_{n,2}\bigr)_{n\in\N_0}$ be two independent sequences of independent cylindrical Gaussian random variables. Set for all $n\in\N_0$
\begin{equation}\label{eq:scheme}
X_{n+1}^{\tau}=\IA_\tau\bigl(X_n^{\tau}+\tau F(X_n^{\tau})\bigr)+\IB_{\tau,1}\sqrt{\tau}\Gamma_{n,1}+\IB_{\tau,2}\sqrt{\tau}\Gamma_{n,2},
\end{equation}
with initial value $X_0^\tau=x_0$, where the linear operators $\IA_\tau$, $\IB_{\tau,1}$ and $\IB_{\tau,2}$ are assumed to satisfy
\begin{equation}\label{eq:operators}
\IA_\tau=(I+\tau\IL)^{-1},\quad \IB_{\tau,1}=\frac{1}{\sqrt{2}}(I+\tau\IL)^{-1},\quad \IB_{\tau,2}\IB_{\tau,2}^\star=\frac12(I+\tau\IL)^{-1},
\end{equation}
where $L^\star$ is the adjoint of a linear operator $L$. As already explained above, the random variable $\IB_{\tau,1}\Gamma_{n,1}$ is a well-defined $H$-valued Gaussian random variable since $\IB_{\tau,1}=\frac{1}{\sqrt{2}}\IA_\tau$ is an Hilbert--Schmidt linear operator. If the linear operator $\IB_{\tau,2}$ satisfies the third condition in~\eqref{eq:operators}, then $\IB_{\tau,2}$ is also an Hilbert--Schmidt linear operator (under Assumption~\ref{ass:Lambda}), thus $\IB_{\tau,2}\Gamma_n$ is also well-defined. Indeed, one has
\[
\|\IB_{\tau,2}\|_{\mathcal{L}_2(H)}^2=\sum_{j\in\N}|\IB_{\tau,2} e_j|^2=\sum_{j\in\N}|\IB_{\tau,2}^\star e_j|^2=\sum_{j\in\N}\langle \IB_{\tau,2}\IB_{\tau,2}^\star e_j,e_j\rangle=\sum_{j\in\N}\frac{1}{2(1+\lambda_j\tau)}<\infty.
\]

The motivations for imposing the conditions~\eqref{eq:operators} for the linear operators are the following (more details are given below). First, if $F=0$, then the proposed scheme~\eqref{eq:scheme} preserves the invariant distribution $\nu$ of the stochastic evolution equation~\eqref{eq:SPDE}. Second, contrary to other methods (such as an exponential Euler scheme) which satisfy the first requirement, the proposed scheme can be implemented without knowing the eigenvalues $\bigl(\lambda_j\bigr)_{j\in\N}$ and the eigenfunctions $\bigl(e_n\bigr)_{j\in\N}$ of $\IL$.

The proposed scheme~\eqref{eq:scheme} is a modification of the standard Euler scheme~\eqref{eq:scheme-standard}, in particular note that $\IA_\tau$ is given by~\eqref{eq:IA} in both cases. However, let us highlight the major difference between the two schemes: the definition of the proposed integrator requires the use of two sequences of independent cylindrical Gaussian random variables. We refer to Remark~\ref{rem:construction} below for an explanation of this requirement.

Note that there exist multiple choices to define linear operators $\IB_{\tau,2}$ such that the third condition in~\eqref{eq:operators} is satisfied. Precisely, all choices of the linear operators such that~\eqref{eq:operators} is fulfilled give sequences of random variables $\bigl(X_n^{\tau}\bigr)_{n\in\N_0}$ which are equal in distribution. This is consistent with the fact that the distribution of a $H$-valued Gaussian random variable $L\Gamma$ only depends on its covariance operator $LL^\star$. A naive choice would be to set
\[
\IB_{\tau,2}x=\sum_{j\in\N}\frac{1}{\sqrt{2(1+\lambda_j\tau)}}\langle x,e_j\rangle e_j
\]
for all $x\in H$: then $\IB_{\tau,2}$ would be the square root of the self-adjoint operator $\frac12\IA_\tau$. However, the definition of $\IB_{\tau,2}$ above would require the knowledge of the eigenvalues $\bigl(\lambda_j\bigr)_{j\in\N}$ and the eigenfunctions $\bigl(e_j\bigr)_{j\in\N}$ of $\IL$. To avoid this requirement, which may be restrictive in pratice, note that it is instead possible to use a Cholesky decomposition of the operator $\frac12\IA_\tau=\frac12(I+\tau\IL)^{-1}$. More precisely, in the context of the example described in Section~\ref{sec:example} corresponding to the stochastic partial differential equation~\eqref{eq:SPDEintro-field}, the implementation of the scheme requires a spatial discretization procedure, which may be performed using a finite differences approximation (with mesh size denoted by $h$), the Cholesky decomposition is then performed at the finite dimensional approximation level. Computing the Cholesky decomposition is generally less expensive than identifying the eigenvalues and the eigenfunctions of the linear operator $\IL_h$ (which is a tridiagonal matrix for the example). In the sequel, the spatial approximation is omitted and we focus only on the temporal discretization. The convergence results below may be generalized at the finite dimensional approximation level, with error bounds independent of $h$ (see Remark~\ref{rem:fe}).

Based on the discussion above, it is clear that the proposed scheme~\eqref{eq:scheme} is a modification of the standard scheme~\eqref{eq:scheme-standard} which has a more expensive implementation, due to the need to compute an additional Gaussian random variable $\IB_{\tau,2}\Gamma_{n,2}$ at each iteration. However, the huge benefits of using the modified Euler scheme~\eqref{eq:scheme} will be stated and illustrated below: the main results are stated in Section~\ref{sec:results}, whereas comparisons with an exponential Euler scheme and the standard Euler scheme are provided in Section~\ref{sec:expo} and~\ref{sec:standard} respectively.

\subsection{Second formulation of the modified Euler scheme}\label{sec:scheme-2nd}

Let us introduce an equivalent formulation of the proposed integrator~\eqref{eq:scheme}, where a single sequence $\bigl(\Gamma_n\bigr)_{n\in\N_0}$ of independent cylindrical Gaussian random variables is needed. This formulation is not used in practice.

Define the self-adjoint linear operator $\IB_\tau$ such that
\[
\IB_\tau^2=\IB_{\tau,1}^2+\IB_{\tau,2} \IB_{\tau,2}^\star=\frac12\bigl(\IA_\tau^2+\IA_\tau\bigr)=\frac12(2I+\tau\IL)(I+\tau\IL)^{-2}
\]
where $\IA_\tau$, $\IB_{\tau,1}$ and $\IB_{\tau,2}$ satisfy the conditions~\eqref{eq:operators}. The linear operator $\IB_\tau$ is given by
\begin{equation}\label{eq:IB}
\IB_\tau x=\sum_{j\in\N}\frac{\sqrt{2+\lambda_j\tau}}{\sqrt{2}~(1+\lambda_j\tau)}\langle x,e_j\rangle e_j
\end{equation}
for all $x\in H$.

For all $n\in\N_0$, set
\begin{equation}\label{eq:scheme2}
\hat{X}_{n+1}^\tau=\IA_\tau\bigl(\hat{X}_n^\tau+\tau F(\hat{X}_n^\tau)\bigr)+\sqrt{\tau}\IB_\tau\Gamma_n
\end{equation}
with initial value $\hat{X}_0^\tau=x_0$, where $\IA_\tau$ and $\IB_\tau$ are given by~\eqref{eq:IA} and~\eqref{eq:IB} respectively, and where $\bigl(\Gamma_n\bigr)_{n\in\N_0}$ is a sequence of independent cylindrical Gaussian random variables. Then one has the following result: the sequences $\bigl(X_n^\tau\bigr)_{n\in\N_0}$ and $\bigl(\hat{X}_n^\tau\bigr)_{n\in\N_0}$ are equal in distribution, for any value $\tau\in(0,\tau_0)$ of the time-step size. This result is a straightforward consequence of the following equality in distribution
\begin{equation}\label{eq:distribution_increments}
\IB_{\tau,1}\Gamma_{n,1}+\IB_{\tau,2}\Gamma_{n,2}=\IB_{\tau}\Gamma_n,
\end{equation}
if $\Gamma_{n,1}$ and $\Gamma_{n,2}$ are two independent cylindrical Gaussian random variables. Indeed, the random variables in the left and the right hand sides of~\eqref{eq:distribution_increments} are centered $H$-valued Gaussian random variables with the same covariance operator. Observe that $\IB_\tau$ is indeed an Hilbert--Schmidt linear operator, so that $\IB_{\tau}\Gamma_n$ is a well-defined $H$-valued Gaussian random variable.

In the sequel, the same notation $\bigl(X_n^\tau\bigr)_{n\in\N_0}$ is used for both formulations~\eqref{eq:scheme} and~\eqref{eq:scheme2} of the modified Euler scheme, since all equalities are understood as equalities in distribution. The second formulation~\eqref{eq:scheme2} is more convenient for the analysis the scheme, however this formulation could be implemented only if the eigenvalues $\bigl(\lambda_j\bigr)_{j\in\N}$ and the eigenfunctions $\bigl(e_j\bigr)_{j\in\N}$ of $\IL$ were known, whereas the first formulation~\eqref{eq:scheme} can be implemented without this requirement as explained in Section~\ref{sec:scheme-1st}.

The formulation~\eqref{eq:scheme2} clearly shows why the proposed scheme is a modification of the standard Euler scheme~\eqref{eq:scheme-standard}: the random variable $\IA_\tau\Gamma_n$ in~\eqref{eq:scheme-standard} is replaced by $\IB_\tau\Gamma_n$ in~\eqref{eq:scheme2}.

Let us now justify why introducing the modification of the standard Euler scheme with $\IB_\tau$ such that~\eqref{eq:IB} holds is relevant.
\begin{propo}\label{propo:invarGauss}
Assume that $F=0$. Let $\IA_\tau$ and $\IB_\tau$ be given by~\eqref{eq:IA} and~\eqref{eq:IB} respectively. Then, for any value $\tau\in(0,\tau_0)$ of the time-step size, the unique invariant distribution of the numerical scheme~\eqref{eq:scheme2} is the Gaussian distribution $\nu$ given by~\eqref{eq:nu}: if $X_0^\tau$ is a random variable with distribution $\nu$, independent of the sequence $\bigl(\Gamma_n\bigr)_{n\in\N_0}$ of cylindrical Gaussian random variables, then the distribution of $X_n^\tau$ is equal to $\nu$ for all $n\in\N_0$.
\end{propo}
Note that the standard Euler scheme~\eqref{eq:scheme-standard} does not preserve the invariant distribution $\nu$ when $F=0$, see Section~\ref{sec:standard} for more details and the issues which are raised by this non-preservation of the invariant distribution. We also refer to Section~\ref{sec:extensions} for two applications of Proposition~\ref{propo:invarGauss}, which justify the superiority of the modified Euler scheme over the standard Euler method, in this context where the process is Gaussian: the definition of asymptotic preserving schemes for a class of multiscale stochastic evolution systems in an averaging regime (Section~\ref{sec:AP}) and the definition of a Markov Chain Monte Carlo proposal kernel (Section~\ref{sec:MCMC}).

\begin{proof}
The proof is straightforward. On the one hand, the unique invariant distribution of~\eqref{eq:scheme2} when $F=0$ is the centered $H$-valued Gaussian random variable with covariance operator equal to
\[
\tau(I-\IA_\tau^2)^{-1}\IB_\tau^2.
\]
On the other hand, using the definitions of $\IA_\tau$ and $\IB_\tau$, one has the identity
\[
I-\IA_\tau^2=2\tau\IL\IB_\tau^2.
\]
Therefore one has
\begin{equation}\label{eq:identity-AB}
\tau\IB_\tau^2(I-\IA_\tau^2)^{-1}=\frac12\IL^{-1}
\end{equation}
which is the covariance of the centered Gaussian distribution $\nu$.
\end{proof}

\subsection{Third formulation of the modified Euler scheme}\label{sec:scheme-3rd}

In this section, we introduce the third formulation of the modified Euler scheme, which is a crucial tool to prove the main results below. This formulation consists in interpreting the modified Euler scheme~\eqref{eq:scheme}, or equivalently its second formulation~\eqref{eq:scheme2}, as the accelerated exponentiel Euler scheme associated with a modified stochastic evolution equation, of the type
\begin{equation}\label{eq:modifiedSPDE}
d\IX_\tau(t)=-\IL_\tau \IX_\tau(t)dt+Q_{\tau}F(\IX_\tau(t))dt+Q_{\tau}^{\frac12}dW(t),
\end{equation}
depending on two self-adjoint linear operators $\IL_\tau$ and $Q_\tau$ defined below. The initial value is $\IX_\tau(0)=x_0$. The mild formulation of the solution $\bigl(\IX_\tau(t)\bigr)_{t\ge 0}$ of the modified equation~\eqref{eq:modifiedSPDE} is given by
\begin{equation}\label{eq:modifiedSPDE-mild}
\IX_\tau(t)=e^{-t\IL_\tau}x_0+\int_{0}^{t}e^{-(t-s)\IL_\tau}Q_\tau F(\IX_\tau(s))ds+\int_0^t e^{-(t-s)\IL_\tau}Q_\tau^{\frac12}dW(s),
\end{equation}
for all $t\ge 0$, and the accelerated exponential Euler scheme is obtained by
\begin{equation}\label{eq:scheme-IX}
\IX_{\tau,n+1}=e^{-\tau\IL_\tau}\IX_{\tau,n}+\IL_\tau^{-1}(I-e^{-\tau\IL_\tau})Q_\tau F(\IX_{\tau,n})+\int_{t_n}^{t_{n+1}}e^{-(t_{n+1}-s)\IL_\tau}Q_\tau^{\frac12}dW(s),
\end{equation}
for all $n\in\N_0$, using the identity $\int_{t_n}^{t_{n+1}}e^{-(t_{n+1}-s)\IL_\tau}ds=\IL_\tau^{-1}(I-e^{-\tau\IL_\tau})$. Comparing~\eqref{eq:scheme2} and~\eqref{eq:scheme-IX}, for any value $\tau\in(0,\tau_0)$ of the time-step size, the equalities in distribution
\[
\bigl(X_n^\tau\bigr)_{n\in\N_0}=\bigl(\hat{X}_n^\tau\bigr)_{n\in\N_0}=\bigl(\IX_{\tau,n}\bigr)_{n\in\N_0}
\]
are satisfied when the following equalities hold:
\begin{align*}
&(I+\tau\IL)^{-1}=e^{-\tau\IL_\tau}\\
&(I+\tau\IL)^{-1}=\IL_\tau^{-1}(I-e^{-\tau\IL_\tau})Q_\tau\\
&\IB_\tau^2=\int_{0}^{\tau}e^{-s\IL_\tau}Q_\tau e^{-s\IL_\tau}ds.
\end{align*}
This leads to define the linear operators in the modified stochastic evolution equation~\eqref{eq:modifiedSPDE} as follows. For all $j\in\N$ and all $\tau\in(0,\tau_0)$, set
\begin{equation}\label{eq:modifiedILQ-eigen}
\begin{aligned}
\lambda_{\tau,j}&=\frac{\log(1+\tau\lambda_j)}{\tau}>0\\
q_{\tau,j}&=\frac{\log(1+\tau\lambda_j)}{\lambda_j\tau}>0,
\end{aligned}
\end{equation}
and define the self-adjoint linear operators $\IL_\tau$, $Q_\tau$ and $Q_\tau^{\frac12}$ by
\begin{equation}\label{eq:modifiedILQ}
\begin{aligned}
\IL_\tau x&=\sum_{j\in\N}\lambda_{\tau,j}\langle x,e_j\rangle e_j,\\
Q_\tau x&=\sum_{j\in\N}q_{\tau,j}\langle x,e_j\rangle e_j,\\
Q_\tau^{\frac12}&=\sum_{j\in\N}\sqrt{q_{\tau,j}}\langle x,e_j\rangle e_j,
\end{aligned}
\end{equation}
for all $x\in H$. It is straightforward to check that the conditions above are satisfied with these definitions of $\IL_\tau$ and $Q_\tau$.

Note that the linear operators $Q_\tau$ and $\IL_\tau$ commute, and they both commute with the linear operator $\IL$. The linear operators $e^{-t\IL_\tau}$ are defined by
\[
e^{-t\IL_\tau}x=\sum_{j\in\N}e^{-t\lambda_{\tau,j}}\langle x,e_j\rangle e_j
\]
for all $x\in H$, all $t\ge 0$ and all $\tau\in(0,\tau_0)$. Since $\lambda_{\tau,j}\ge 0$ for all $j\in\N$ and all $\tau\in(0,\tau_0)$, for all $t\ge 0$ the linear operator $e^{-t\IL_\tau}$ is bounded, with $\|e^{-t\IL_\tau}\|_{\mathcal{L}(H)}\le 1$.

It is worth mentioning that $\IL_\tau$ is an unbounded operator, with $\lambda_{\tau,j}$ growing like $\log(j)$ when $j\to\infty$, whereas $\lambda_j$ grows like $j^2$. This major difference in the behaviors of $\IL_\tau$ and $\IL$ leads to technical difficulties in the analysis below. Observe also that $Q_\tau^{\frac12}$ is an Hilbert--Schmidt linear operator for any $\tau\in(0,\tau_0)$: $\sum_{j\in\N}q_{\tau,j}<\infty$. As a consequence, the Gaussian random variables $\int_0^t e^{-(t-s)\IL_\tau}Q_\tau^{\frac12}dW(s)$ and $\int_{t_n}^{t_{n+1}}e^{-(t_{n+1}-s)\IL_\tau}Q_\tau^{\frac12}dW(s)$ appearing in the mild formulation~\eqref{eq:modifiedSPDE-mild} and in the associated scheme~\eqref{eq:scheme-IX} are well-defined with values in $H$. More precisely, it is straightforward to check that the following well-posedness result holds. The details of the proof are omitted.
\begin{propo}\label{propo:modified-wellposed}
Let Assumptions~\ref{ass:Lambda} and~\ref{ass:F} be satisfied. For any $\tau\in(0,\tau_0)$, let $\IL_\tau$ and $Q_\tau$ be defined by~\eqref{eq:modifiedILQ}. Then, for any initial value $x_0\in H$, the modified stochastic evolution equation~\eqref{eq:modifiedSPDE} admits a unique global mild solution $\bigl(\IX_\tau(t)\bigr)_{t\ge 0}$, satisfying~\eqref{eq:modifiedSPDE-mild}.

Moreover, if Assumption~\ref{ass:ergo} is satisfied, the modified stochastic evolution equation~\eqref{eq:modifiedSPDE} admits a unique invariant distribution $\mu_{\tau,\infty}$.
\end{propo}
Note that Proposition~\ref{propo:modified-wellposed-bound} stated and proved below (see Section~\ref{sec:auxiliary}) gives a refined version of Proposition~\ref{propo:modified-wellposed}, with bounds which are uniform with respect to $\tau$.

When $\tau\to 0$, it is observed that for any fixed $j\in\N$, one has $\lambda_{\tau,j}\to \lambda_j$ and $q_{\tau,j}\to 1$. As a consequence, it is expected that $\IX_\tau(t)$ converges to $X(t)$, at least in distribution, for all $t\ge 0$. Proving this convergence result is part and giving the rate of convergence with respect to $\tau$ are part of the proof of Theorem~\ref{theo:weak}. The proof requires precise error estimates for the errors $\IL_\tau x-\IL_\tau x$ and $Q_\tau x-x$: see Lemma~\ref{lem:Q_tauIL_tau-error} in Section~\ref{sec:auxiliary-operators}.

The third formulation~\eqref{eq:scheme-IX} of the modified Euler scheme is crucial in the analysis, but it is not needed for the implementation of the scheme (which is performed using the initial formulation~\eqref{eq:scheme}). In particular, the linear operators $\IL_\tau$, $e^{-\tau\IL_\tau}$ or $Q_\tau$ do not need to be computed.

One of the main properties of the modified stochastic evolution equation~\eqref{eq:modifiedSPDE} is the following result concerning its invariant distribution, when the conditions of Proposition~\ref{propo:mu_star} are fulfilled.
\begin{propo}\label{propo:mu_star_modified}
Let Assumptions~\ref{ass:ergo} and~\ref{ass:gradient} be satisfied. For any value $\tau\in(0,\tau_0)$ of the time-step size, the unique invariant distribution $\mu_{\tau,\infty}$ (see Proposition~\ref{propo:modified-wellposed}) of the modified stochastic evolution equation~\eqref{eq:modifiedSPDE} is equal to the Gibbs distribution $\mu_\star$ given by~\eqref{eq:mu_star}.
\end{propo}
Observe that if $F=0$, this result is consistent with Proposition~\ref{propo:invarGauss} above (when $F=0$, the accelerated exponential Euler scheme gives $\IX_{\tau,n}=\IX_\tau(t_n)$ for all $n\in\N$ by construction). The crucial feature of the modified equation~\eqref{eq:modifiedSPDE} associated with the modified Euler scheme~\eqref{eq:scheme} which justifies Proposition~\ref{propo:mu_star_modified} is the presence of the operators $Q_\tau$ and $Q_\tau^{\frac12}$ in front of the nonlinearity and of the Wiener process respectively. The proof of Proposition~\ref{propo:mu_star_modified} is similar to the proof of Proposition~\ref{propo:mu_star}: it consists in applying a spectral Galerkin approximation procedure (see Section~\ref{sec:Galerkin}), in identifying the invariant distribution for the finite dimensional approximation as a Gibbs distribution with respect to the finite dimensional approximation of $\nu$, and in taking the limit. The details are omitted. Note that in general the result of Proposition~\ref{propo:mu_star_modified} does not hold when $F$ does not have the gradient structure given by Assumption~\ref{ass:gradient}, hence the need to treat separately the gradient and the general case below.

\subsection{Additional remarks}\label{sec:scheme-rem}

Before proceeding with the statement and proofs of the main results concerning the modified Euler scheme, let us state two remarks concerning the construction of the integrator in its first formulation~\eqref{eq:scheme}, in particular to explain how the conditions~\eqref{eq:operators} for the linear operators $\IB_{\tau,1}$ and $\IB_{\tau,2}$ are found, and why two sequences of cylindrical Gaussian random variables appear in the formulation~\eqref{eq:scheme}. Remarks~\ref{rem:construction} and~\ref{rem:postproc} do not play any role in the analysis below.

\begin{rem}\label{rem:construction}
Assume that $F=0$. Introduce the self-adjoint linear operators
\[
\tilde{\IA}_\tau=\tilde{\IB}_\tau=(I+2\tau\IL)^{-1/2}.
\]
Then it is straightforward to check that the identity
\[
2\tau\IL\tilde{\IB}_\tau^2=1-\tilde{\IA}_\tau^2
\]
is satisfied. Introduce the auxiliary numerical scheme defined by
\begin{equation}\label{eq:scheme-aux}
\tilde{X}_{n+1}^{\tau}=\tilde{\IA}_\tau\tilde{X}_{n}^{\tau}+\sqrt{\tau}\tilde{\IB}_\tau\tilde{\Gamma}_n=\tilde{\IA}_\tau\bigl(\tilde{X}_n^\tau+\sqrt{\tau}\tilde{\Gamma}_n\bigr),
\end{equation}
with initial value $\tilde{X}_0^{\tau}=x_0$, where $\bigl(\tilde{\Gamma}_n\bigr)_{n\in\N_0}$ is a sequence of independent cylindrical Gaussian random variables. The identity above shows that the auxiliary scheme~\eqref{eq:scheme-aux} preserves the Gaussian invariant distribution $\nu$, for any value $\tau\in(0,\tau_0)$ of the time-step size. However, this auxiliary scheme is not suitable for a general practical implementation since computing $\tilde{\IA}_\tau$ and $\tilde{\IB}_\tau$ would require the knowledge of the eigenvalues $\bigl(\lambda_j\bigr)_{j\in\N}$ and of the eigenfunctions $\bigl(e_j\bigr)_{j\in\N}$ of $\IL$.

The formulation~\eqref{eq:scheme} of the modified Euler scheme is obtained setting
\[
X_n^{\tau}=\tilde{X}_{2n}^{\frac{\tau}{2}}
\]
with $\Gamma_{n,1}=\tilde{\Gamma}_{2n}$ and $\Gamma_{n,2}=\tilde{\Gamma}_{2n+1}$. This means that, formally, when $F=0$, the modified Euler scheme with time-step size $\tau$ is obtained by composing two steps of the auxiliary scheme~\eqref{eq:scheme-aux} with time-step size $\frac{\tau}{2}$ defined above. Indeed, one then has for all $n\in\N_0$
\begin{align*}
X_{n+1}^\tau&=\tilde{X}_{2n+2}^{\frac{\tau}{2}}\\
&=\tilde{\IA}_{\frac{\tau}{2}}\bigl(\tilde{X}_{2n+1}^{\frac{\tau}{2}}+\sqrt{\frac{\tau}{2}}\tilde{\Gamma}_{2n+1}\bigr)\\
&=\tilde{\IA}_{\frac{\tau}{2}}\Bigl(\tilde{\IA}_{\frac{\tau}{2}}\bigl(\tilde{X}_{2n}^{\frac{\tau}{2}}+\sqrt{\frac{\tau}{2}}\tilde{\Gamma}_{2n}\bigr)+\sqrt{\frac{\tau}{2}}\tilde{\Gamma}_{2n+1}\Bigr)\\
&=\tilde{\IA}_{\frac{\tau}{2}}^2X_n^\tau+\sqrt{\frac{\tau}{2}}\tilde{\IA}_{\frac{\tau}{2}}^2\Gamma_{n,1}+\sqrt{\frac{\tau}{2}}\tilde{\IA}_{\frac{\tau}{2}}\Gamma_{n,2},
\end{align*}
with the identities
\begin{align*}
\tilde{\IA}_{\frac{\tau}{2}}^2\Gamma_{n,1}&=(I+\tau\IL) {-1}=\IA_\tau\\
\frac{1}{\sqrt{2}}\tilde{\IA}_{\frac{\tau}{2}}^2&=\frac{1}{\sqrt{2}}(I+\tau\IL^{-1})=\IB_{\tau,1}\\
\bigl(\frac{1}{\sqrt{2}}\tilde{\IA}_{\frac{\tau}{2}}\bigr)\bigl(\frac{1}{\sqrt{2}}\tilde{\IA}_{\frac{\tau}{2}}\bigr)^\star&=\frac12(I+\tau\IL)^{-1}=\IB_{\tau,2}\IB_{\tau,2}^\star.
\end{align*}
In practice, $\IB_{\tau,2}\neq \frac{1}{\sqrt{2}}\tilde{\IA}_{\frac{\tau}{2}}$ in general, when a Cholesky decomposition is used to implement the scheme~\eqref{eq:scheme}.

The interpretation of the modified Euler scheme~\eqref{eq:scheme} using the auxiliary scheme~\eqref{eq:scheme-aux} gives a justification for the identification of the conditions on the operators $\IB_{\tau,1}$ and $\IB_{\tau,2}$ appearing in~\eqref{eq:operators}. In addition, this interpretation also shows that the cylindrical Gaussian random variables $\Gamma_{n,1}$ and $\Gamma_{n,2}$ may interpreted in terms of increments of the cylindrical Wiener process as follows:
\[
\Gamma_{n,1}=W(t_{n+\frac12})-W(t_n),\quad \Gamma_{n,2}=W(t_{n+1})-W(t_{n+\frac12}),
\]
with $t_n=n\tau$ and $t_{n+\frac12}=t_n+\frac{\tau}{2}=\frac{t_n+t_{n+1}}{2}$. The reason why two sequences $\bigl(\Gamma_{n,1})_{n\in\N_0}$ and $\bigl(\Gamma_{n,2}\bigr)_{n\in\N_0}$ of cylindrical Gaussian random variables appear in the formulation~\eqref{eq:scheme} of the modified Euler scheme is now clear using this interpretation based on the auxiliary scheme~\eqref{eq:scheme-aux}.

Finally, observe that the Gaussian distribution $\nu$ for any value $\tau\in(0,\tau_0)$ is preserved both by the auxiliary scheme~\eqref{eq:scheme-aux} and by the modified Euler scheme~\eqref{eq:scheme} when $F=0$, for any value $\tau\in(0,\tau_0)$ of the time-step size. This property is consistent with the equality $X_n^\tau=\tilde{X}_{2n}^{\frac{\tau}{2}}$ for all $n\in\N_0$.
\end{rem}

\begin{rem}\label{rem:postproc}
The definition of the modified Euler scheme~\eqref{eq:scheme} requires auxiliary linear operators similar to those appearing in the definition of the postprocessed integrator introduced in~\cite{BV}:
\begin{equation}\label{eq:postproc}
\left\lbrace
\begin{aligned}
X_{n+1}^{\tau,{\rm pp}}&=\IA_\tau\Bigl(X_n^{\tau,{\rm pp}}+\tau F\bigl(X_n^{\tau,{\rm pp}}+\sqrt{\tau}\frac12\IA_\tau\gamma_n\bigr)+\sqrt{\tau}\Gamma_n\Bigr),\\
\underline{X}_n^{\tau,{\rm pp}}&=X_n^{\tau,{\rm pp}}+\frac12\mathcal{J}_\tau\sqrt{\tau}\Gamma_n,
\end{aligned}
\right.
\end{equation}
where $\IA_\tau$ is given by~\eqref{eq:IA} and the linear operator $\mathcal{J}_\tau$ is such that $\mathcal{J}_\tau\mathcal{J}_\tau^\star=(I+\frac12\tau\IL)^{-1}$. The initial value is $X_0^{pp}=x_0$, and $\bigl(\Gamma_n\bigr)_{n\in\N}$ are independent cylindrical Gaussian random variables.

Like for the definition of $\IB_{\tau,2}$, there are multiple choices to choose $\mathcal{J}_{\tau}$, and in practice a Cholesky decomposition of $(I+\frac12\tau\IL)^{-1}$ can be employed. The operators $\frac12(I+\tau\IL)^{-1}$ and $(I+\frac12\tau\IL)^{-1}$ appearing in the two Cholesky decompositions to define $\IB_{\tau,2}$ and $\mathcal{J}_{\tau}$ have similar expressions, this observation is justified below.

The postprocessing integrator~\eqref{eq:postproc} is another type of modification of the standard Euler scheme~\eqref{eq:scheme-standard}. It has been introduced in~\cite{BV} to provide a better approximation of the invariant distribution $\mu_\infty$: it is proved that if $F=0$, then $\underline{X}_N^{\tau,{\rm pp}}$ converges in distribution to the Gaussian distribution $\nu$ (defined by~\eqref{eq:nu}) when $N\to\infty$, for any value of the time-step size $\tau\in(0,\tau_0)$ (while one has $X_n^{\tau,{\rm pp}}=X_n^{\tau,\s}$ for all $n\in\N_0$). This property justifies the requirement that $\mathcal{J}_\tau$ is a solution of $\mathcal{J}_\tau\mathcal{J}_\tau^\star=(I+\frac12\tau\IL)^{-1}$, and the similitude with the requirement in~\eqref{eq:operators} for $\IB_{\tau,2}$ such that $\nu$ is also preserved by the modified Euler scheme~\eqref{eq:scheme}.

Using the postprocessed integrator~\eqref{eq:postproc} is computationally less expensive than using the modified Euler scheme~\eqref{eq:scheme}: indeed it is required to compute $\underline{X}_n^{\tau,{\rm pp}}$ only at the last step $n=N$ of the numerical experiment, if the objective is to approximate the invariant distribution only. Note that it is not known whether the postprocessed integrator~\eqref{eq:postproc} leads to improved approximation of the invariant distribution $\mu_\infty$ in general (when $F\neq 0$). In addition, as will be explained below using the modified Euler scheme~\eqref{eq:scheme} results in better qualitative properties than using the standard Euler integrator~\eqref{eq:scheme-standard}, or than using the postprocessed integrator~\eqref{eq:postproc} if $\underline{X}_n^{\tau,{\rm pp}}$ is not computed at all time steps. Computing $\underline{X}_n^{\tau,{\rm pp}}$ at all time steps would result in a scheme with essentially the same computational cost as using the modified Euler scheme~\eqref{eq:scheme}.
\end{rem}

\begin{rem}
The recent article~\cite{ABV} presents another illustration of how the choice of the discretization of the noise may have an impact on the spatial regularity property of the numerical solution. In~\cite{ABV}, the linear part is discretized using explicit-stabilized integrators, instead of an implicit discretization with $\IA_\tau=(I+\tau\IL)^{-1}$, and two methods are proposed. One of the methods has a behavior similar to the one of the modified Euler scheme studied in this article. However, the method from~\cite{ABV} does not preserve the Gaussian invariant distribution $\nu$ in the Ornstein--Uhlenbeck case. 
\end{rem}
\section{Main results}\label{sec:results}

We are in position to state the main results of this article. In this section, the linear operator $\IL$ and the nonlinear operator $F$ satisfy at least Assumptions~\ref{ass:Lambda} and~\ref{ass:F} respectively. Recall that the time-step size is denoted by $\tau$ and satisfies $\tau\in(0,\tau_0)$.

Let us first recall the definition of the modified Euler scheme~\eqref{eq:scheme}, using its second formulation~\eqref{eq:scheme2}: one has
\begin{equation}\label{eq:res-modifiedscheme}
X_{n+1}^\tau=\IA_\tau\bigl(X_n^\tau+\tau F(X_n^\tau)\bigr)+\sqrt{\tau}\IB_\tau\Gamma_n,\quad X_0^\tau=x_0.
\end{equation}
To illustrate the main qualitative and quantitative results concerning the scheme~\eqref{eq:res-modifiedscheme}, it is convenient to introduce two integrators which have been extensively studied in the literature. First, the standard Euler scheme is given by~\eqref{eq:scheme-standard}: one has
\begin{equation}\label{eq:res-standardscheme}
X_{n+1}^{\tau,\s}=\IA_\tau\Bigl(X_n^{\tau,\s}+\tau F(X_n^{\tau,\s})+\sqrt{\tau}\Gamma_n\Bigr),\quad X_0^\tau=x_0.
\end{equation}
Second, the (accelerated) exponential Euler scheme is given as follows: for all $n\in\N_0$,
\begin{equation}\label{eq:res-exponentialscheme}
X_{n+1}^{\tau,\e}=e^{-\tau\IL}X_n^{\tau,\e}+\IL^{-1}(I-e^{-\tau\IL})F(X_n^{\tau,\e})+\int_{t_n}^{t_{n+1}}e^{-(t_{n+1}-s)\IL}dW(s),\quad X_0^{\tau,\e}=x_0.
\end{equation}

The qualitative behavior of the modified Euler scheme~\eqref{eq:res-modifiedscheme} is studied in Subsection~\ref{sec:results_quali}. Error estimates for this new integrator are then stated in Subsection~\ref{sec:results_invar} and~\ref{sec:results_weak}. The most relevant result which justifies the study of the proposed scheme is Theorem~\ref{theo:weakinv} in Subsection~\ref{sec:results_invar}, in a specific context (approximation of the Gibbs invariant distribution under the gradient structure assumption). On the contrary, the results in Subsection~\ref{sec:results_weak} are more standard and are verified for the other schemes, however precise statements and detailed proofs are provided since they show that the new scheme is applicable in a general framework. Old and new results on the standard and exponential Euler schemes are stated in Section~\ref{sec:results_standard} and~\ref{sec:results_exponential} respectively, in order to illustrate the properties of the modified Euler schemes compared with those methods.

\subsection{Qualitative behavior of the modified Euler scheme}\label{sec:results_quali}

Let us first assume that $x_0=0$ and $F=0$: therefore the solution of the stochastic evolution equation~\eqref{eq:SPDE} is the stochastic convolution $\bigl(W^{\IL}(t)\bigr)_{t\ge 0}$ defined by~\eqref{eq:StochasticConvolution}. In that setting, the solution of the modified Euler scheme is given by
\begin{equation}\label{eq:WNtau}
\begin{aligned}
W_N^\tau&=\sum_{n=0}^{N-1}\IA_\tau^{N-n-1}\Bigl(\IB_{\tau,1}\sqrt{\tau}\Gamma_{n,1}+\IB_{\tau,2}\sqrt{\tau}\Gamma_{n,2}\Bigr)\\
&=\sum_{n=0}^{N-1}\IA_\tau^{N-n-1}\IB_\tau\sqrt{\tau}\Gamma_n
\end{aligned}
\end{equation}
for all $N\in\N_0$, using the two equivalent formulations~\eqref{eq:scheme} and~\eqref{eq:scheme2}.

We are in position to state the first main result of this article.
\begin{theo}\label{theo:equivalence}
For all $\tau\in(0,\tau_0)$ and all $N\in\mathbb{N}$, the distributions $\rho(t_N)=\rho_{W^{\IL}(t_N)}$ and $\rho_N^\tau=\rho_{W_N^\tau}$ of the Gaussian random variables $W^{\IL}(t_N)$ and $W_N^\tau$ are equivalent. Moreover, they are both equivalent to the Gaussian distribution $\nu$ given by~\eqref{eq:nu}.
\end{theo}
Note that Theorem~\ref{theo:equivalence} does not hold for the standard Euler scheme~\eqref{eq:scheme-standard}: for all $\tau\in(0,\tau_0)$ and all $N\in\mathbb{N}$, the distributions of $W^{\IL}(t_n)$ and of
\[
W_N^{\tau,\s}=\sum_{n=0}^{N-1}\IA_\tau^{N-n}\sqrt{\tau}\Gamma_n
\]
are singular, see for instance~\cite{B:2020}.

\begin{proof}[Proof of Theorem~\ref{theo:equivalence}]
The proof follows from applying the Feldman-Hajek criterion, see for instance~\cite[Theorem~2.25]{DPZ}.

The covariance operator of the Gaussian distribution $\nu$ is denoted by $\mathcal{Q}=\frac12\IL^{-1}$, and for all $T\in(0,\infty)$, the covariance operator of the Gaussian distribution $\rho(T)=\rho_{W^{\IL}(T)}$ is given by
\[
\mathcal{Q}(T)=\int_0^{T}e^{-2t\IL}dt=\frac12\IL^{-1}\bigl(I-e^{-2T\IL})=\mathcal{Q}\bigl(I-e^{-2T\IL}).
\]
It suffices to check the two following items to check that $\rho(T)$ and $\nu$ are equivalent for all $T\in(0,\infty)$.
\begin{itemize}
\item The ranks $\mathcal{Q}(H)$ and $\mathcal{Q}(T)(H)$ of the operators $\mathcal{Q}$ and $\mathcal{Q}(T)$ respectively are equal to $H^{\frac12}$.
\item The linear operator $\mathcal{R}(T)=\bigl(\mathcal{Q}^{-\frac12}\mathcal{Q}(T)^{\frac12}\bigr)\bigl(\mathcal{Q}^{-\frac12}\mathcal{Q}(T)^{\frac12}\bigr)^\star-I$ is an Hilbert-Schmidt operator: indeed one has
\[
\sum_{j\in\N}|\mathcal{R}(T)e_j|^2=\sum_{j\in\N}e^{-2T\lambda_j}<\infty.
\]
\end{itemize}
For all $\tau\in(0,\tau_0)$ and $N\in\mathbb{N}$, the covariance operator of the Gaussian distribution $\rho_N^\tau=\rho_{W_N^\tau}$ is given by
\begin{align*}
\mathcal{Q}_N^\tau&=\tau\sum_{n=0}^{N-1}\IA_\tau^{N-n-1}\IB_\tau\bigl(\IA_\tau^{N-n-1}\IB_\tau\bigr)^{\star}\\
&=\tau\IB_\tau^2(I-\IA_\tau^2)^{-1}(I-\IA_\tau^{2N})\\
&=\frac{1}{2}\IL^{-1}(I-\IA_\tau^{2N})\\
&=\mathcal{Q}(I-\IA_\tau^{2N}),
\end{align*}
owing to the identity~\eqref{eq:identity-AB}. As above, it suffices to check the two following items to check that $\rho_N^\tau$ and $\nu$ are equivalent.
\begin{itemize}
\item The ranks $\mathcal{Q}(H)$ and $\mathcal{Q}_N^H(H)$ of the operators $\mathcal{Q}$ and $\mathcal{Q}_N^\tau$ respectively are equal to $H^{\frac12}$.
\item The linear operator $\mathcal{R}_N^\tau=\bigl(\mathcal{Q}^{-\frac12}(\mathcal{Q}_\tau)^{\frac12}\bigr)\bigl(\mathcal{Q}^{-\frac12}(\mathcal{Q}_N^\tau)^{\frac12}\bigr)^\star-I$ is an Hilbert-Schmidt operator: one has
\[
\sum_{j\in\N}|\mathcal{R}_N^\tau e_j|^2=\sum_{j\in\N}\frac{1}{(1+\tau\lambda_j)^{2N}}<\infty.
\]
\end{itemize}
The application of the Feldman-Hajek criterion then concludes the proof of Theorem~\ref{theo:equivalence}.
\end{proof}

In the semilinear case ($F\neq 0$), one has the following result, which shows that the modified Euler scheme~\eqref{eq:scheme} preserves the spatial regularity of the solution of the stochastic evolution equation~\eqref{eq:SPDE}, at all times.
\begin{theo}\label{theo:regularity}
Let $F$ satisfy Assumption~\ref{ass:F}, and let $x_0\in H$ be an arbitrary initial value. For all $\alpha\in[0,1)$, $\tau\in(0,\tau_0)$ and $N\ge 1$, the following statements are equivalent:
\begin{enumerate}
\item[(i)] $\E[|X(N\tau)|_\alpha^2]<\infty$,
\item[(ii)] $\E[|X_N^\tau|_\alpha^2]<\infty$,
\item[(iii)] $\alpha\in[0,\frac14)$.
\end{enumerate}
\end{theo}
It is worth mentioning that Theorem~\ref{theo:regularity} does not hold for the standard Euler scheme~\eqref{eq:scheme-standard}: one has
\[
\E[|X_N^{\tau,\s}|_\alpha^2]<\infty
\]
for all $\alpha\in(0,\frac12)$, see~\cite{B:2020}. This means that for a fixed time step size $\tau$, the approximate solution computed using the standard Euler scheme has higher spatial regularity than the exact solution, whereas the approximate solution computed using the modified Euler scheme preserves the spatial regularity property, expressed in terms of the Sobolev-like norms denoted by $|\cdot|_\alpha$. See Figures~\ref{fig1} and~\ref{fig2} for a numerical illustration of the different qualitative behaviors when using the modified Euler scheme and the standard Euler scheme.

The proof of Theorem~\ref{theo:regularity} is straightforward. Note that more precise moment bounds for $X_n^\tau$, which in particular are uniform over $\tau\in(0,\tau_0)$, are stated and proved below (see Lemma~\ref{lem:scheme-bound} in Section~\ref{sec:auxiliary-scheme}) and are instrumental in the error analysis.
\begin{proof}[Proof of Theorem~\ref{theo:regularity}]
First, it is straightforward to check that $\E[|X(N\tau)|^2]<\infty$ and $\E[|X_N^\tau|^2]<\infty$ for all $N\in\N$ and all $\tau\in(0,\tau_0)$. Second, owing to the inequality $\|\IL^{\frac12}e^{-t\IL}\|_{\mathcal{L}(H)}+\|\IL^{\frac12}\IA_\tau\|_{\mathcal{L}(H)}\le Ct^{-\frac12}+C\tau^{-\frac12}$ (using the~smoothing property~\eqref{eq:smoothing}), one obtains
\[
\E[|X(N\tau)-W^{\IL}(N\tau)|_\alpha^2]+\E[|X_N^\tau-W_N^\tau|_\alpha^2]<\infty
\]
for all $\alpha\in(0,\frac12]$. Finally, it remains to check that
\[
\E[|W^{\IL}(T)|_\alpha^2]=\int_0^T\|\IL^{\alpha}e^{-t\IL}\|_{\mathcal{L}_2(H)}^2 dt=\frac12\sum_{j\in\N}\frac{1-e^{-2T\lambda_j}}{2\lambda_j^{1-2\alpha}}
\]
and
\begin{align*}
\E[|W_N^\tau|_\alpha^2]&=\tau\sum_{n=0}^{N-1}\|\IL^{\alpha}\IA_\tau^{n}\IB_\tau\|_{\mathcal{L}_2(H)}^2\\
&=\tau\sum_{n=0}^{N-1}\sum_{j\in\N}\lambda_j^{2\alpha}|\IA_\tau e_j|^{2n}|\IB_\tau e_j|^2\\
&=\tau|\IB_\tau e_j|^2\sum_{j\in\N}\lambda_j^{2\alpha}\frac{1-|\IA_\tau e_j|^{2N}}{1-\|IA_\tau e_j|^2}\\
&=\sum_{j\in\N}\frac{1-\frac{1}{(1+\tau\lambda_j)^{2N}}}{2\lambda_j^{1-2\alpha}}
\end{align*}
are finite if and only if $\alpha<\frac14$, where the identity~\eqref{eq:identity-AB} has been used.

This concludes the proof of Theorem~\ref{theo:regularity}.
\end{proof}

Before proceeding with the statement of the error estimates, let us state the following result concerning the long-time behavior of the modified Euler scheme.
\begin{propo}\label{propo:invar-scheme}
If Assumption~\ref{ass:ergo} is satisfied, then for all $\tau\in(0,\tau_0)$, the modified Euler scheme~\eqref{eq:scheme} admits a unique invariant probability distribution $\mu_\infty^\tau$, which satisfies
\begin{equation}\label{eq:invar-scheme-bound}
\underset{\tau\in(0,\tau_0)}\sup~\int |x|_\alpha^2 d\mu_{\infty}^\tau<\infty.
\end{equation}
for all $\alpha\in[0,\frac14)$.

Moreover, there exists $C\in(0,\infty)$, such that for all functions $\varphi:H\to\R$ of class $\mathcal{C}^1$, for all $x_0\in H$, for all $\tau\in(0,\tau_0)$ and for all $N\in\N$, one has
\begin{equation}\label{eq:invar-scheme-ergo}
\big|\E[\varphi(X_N^\tau)]-\int\varphi d\mu_\infty^\tau\big|\le C\vvvert\varphi\vvvert_1 e^{-\kappa N\tau}(1+|x_0|),
\end{equation}
with $\kappa=\frac{\lambda_1-\Lf}{1+\tau_0\lambda_1}\in (0,\lambda_1-\Lf)$.

Finally, when $F=0$, then $\mu_\infty^\tau=\nu$ is the Gaussian distribution given by~\eqref{eq:nu}, which is the invariant distribution of the stochastic convolution~\eqref{eq:StochasticConvolution}.
\end{propo}
Contrary to the previous results Theorems~\ref{theo:equivalence} and~\ref{theo:regularity}, the first and second parts of Theorem~\ref{propo:invar-scheme} hold also for the standard Euler scheme. Note that the last part of Theorem~\ref{propo:invar-scheme} is given by Proposition~\ref{propo:invarGauss} above, and is not satisfied by the standard Euler scheme.

The proof of Proposition~\ref{propo:invar-scheme} employs standard arguments but requires several technical moment estimates to check that $C$ and $\kappa$ do not depend on $\tau\in(0,\tau_0)$. The proof is thus postponed to Section~\ref{sec:auxiliary-scheme}.

\subsection{Approximation in the total variation distance of the Gibbs invariant distribution}\label{sec:results_invar}

We are in position to state the major result of this article.
\begin{theo}\label{theo:weakinv}
Let the nonlinearity $F$ satisfy Assumptions~\ref{ass:ergo},~\ref{ass:gradient},~\ref{ass:Fregul1} and~\ref{ass:Fregul2}. For all $\delta\in(0,\frac12)$ and $\tau_0\in(0,1)$, there exists $C_{\delta}\in(0,\infty)$ such that for all $\tau\in(0,\tau_0)$ one has
\begin{equation}\label{eq:theo-weakinv_dTVinvar}
d_{\rm TV}(\mu_\infty^\tau,\mu_\star)\le C_\delta\tau^{\frac12-\delta},
\end{equation}
where $\mu_\infty^\tau$ is the invariant distribution of $\bigl(X_n^\tau)_{n\ge 0}$ (Proposition~\ref{propo:invar-scheme}) and $\mu_\star$ is the invariant distribution of $\bigl(X(t)\bigr)_{t\ge 0}$ (Proposition~\ref{propo:invar}), with is the Gibbs distribution given by~\eqref{eq:mu_star} (Proposition~\ref{propo:mu_star}).

Moreover, for all $\delta\in(0,\frac12)$ and $\tau_0\in(0,1)$, there exists $C_{\delta}\in(0,\infty)$ such that for all $x_0\in H^{\frac14-\frac{\delta}{8}}$, all $\tau\in(0,\tau_0)$ and $N\in\N$, with $N\tau\ge 1$, and any bounded measurable function $\varphi\in\mathcal{B}_b(H)$, one has
\begin{equation}\label{eq:theo-weakinv_weakerror}
\big|\E[\varphi(X_N^\tau)]-\int\varphi d\mu_\star\big|\le C_{\delta}\vvvert\varphi\vvvert\Bigl(\tau^{\frac12-\delta}(1+|x_0|_{\frac14-\frac{\delta}{8}}^2)+e^{-\kappa N\tau}(1+|x_0|)\Bigr),
\end{equation}
with $\kappa=\frac{\log(1+\tau_0\lambda_1)}{\tau_0\lambda_1}(\lambda_1-\Lf)\in(0,\lambda_1-\Lf)$.
\end{theo}

Note that the weak error estimate~\eqref{eq:theo-weakinv_weakerror} can be equivalently written as
\begin{equation}\label{eq:theo-weakinv_dTV-time}
d_{\rm TV}(\rho_{X_N^\tau},\mu_\star)\le C_{\delta}\Bigl(\tau^{\frac12-\delta}(1+|x_0|_{\frac14-\frac{\delta}{8}}^2)+e^{-\kappa N\tau}(1+|x_0|)\Bigr),
\end{equation}
where we recall that $\rho_{X_N^\tau}$ denotes the distribution of the $H$-valued random variable $X_N^\tau$. The condition $N\tau\ge 1$ is not restrictive when considering the regime of approximation of the invariant distribution. The condition $x_0\in H^{\frac14-\frac{\delta}{8}}$ for the initial value is not very restrictive. It may be weakened at the price of additional technical arguments, which are omitted in order to focus on the most original points of the approach to prove the main results.

To the best of our knowledge, Theorem~\ref{theo:weakinv} is the first result in the literature where a numerical approximation of the invariant distribution of an infinite dimensional stochastic evolution equation equation is obtained in the total variation distance. Indeed, previous results are obtained in the $d_2$ distance introduced in Section~\ref{sec:notation}, which requires regularity of the function $\varphi$ to obtain a weak error estimate of the type~\eqref{eq:theo-weakinv_weakerror}. For the standard Euler scheme, as explained in~\cite{B:2020}, the total variation distance
\[
d_{\rm TV}(\mu_\infty^{\tau,\s},\mu_\star)=2
\]
does not converge to $0$ when $\tau\to 0$ (where $\mu_\infty^{\tau,\s}$ denotes the invariant distribution of the standard Euler scheme~\eqref{eq:res-standardscheme}). The equality above holds when $F=0$, whereas using the modified Euler scheme one has $\mu_\infty^{\tau}=\nu=\mu_\star$ for all $t\ge 0$ when $F=0$ (see Proposition~\ref{propo:invarGauss}). Theorem~\ref{theo:weakinv} is thus a major improvement over existing results. New results with approximation in the total variation distance for the standard and the exponential Euler schemes are stated in Section~\ref{sec:results_standard} and~\ref{sec:results_exponential} below, and are compared with Theorem~\ref{theo:weakinv-standard}.

As will be explained below, the value $1/2$ for the order of convergence in Theorem~\ref{theo:weakinv} has a natural meaning: it is in correspondance with the temporal H\"older regularity $1/4$ for the solution of~\eqref{eq:SPDE}, and this order coincides with the usual order of convergence in the distance $d_2$ of the numerical scheme in general. It is expected that the value $1/2$ is optimal, but checking this is left open.

\begin{rem}\label{rem:fe}
In practice, a spatial discretization needs to be applied to implement the modified Euler scheme~\eqref{eq:res-modifiedscheme}, using either a finite differences, a finite element, or a spectral Galerkin method. The result of Theorem~\ref{theo:weakinv} needs to be interpreted carefully in this context. Let $h\in(0,h_0)$ denote the spatial discretization parameter (with $h\to 0$): then (for instance) the error estimate~\eqref{eq:theo-weakinv_dTVinvar} is written as
\[
d_{\rm TV}(\mu_\infty^{\tau,h},\mu_\star^h)\le C_\delta\tau^{\frac12-\delta}
\]
where $C_\delta\in(0,\infty)$ is independent of $h$, and where $\mu_\infty^{\tau,h}$ and $\mu_\star^h$ denote the invariant distributions, depending on $h$. Since the supports of these distributions is finite dimensional (with dimension depending on $h$), one has
\[
d_{\rm TV}(\mu_\star^h,\mu_\star)=d_{\rm TV}(\mu_\infty^{\tau,h},\mu_\infty^{\tau})=2
\]
does not converge to $0$ when $h\to 0$. Therefore, the total variation distance $d_{\rm TV}(\mu_\infty^{\tau,h},\mu_\star)$ does not satisfy an error estimate of the type~\eqref{eq:theo-weakinv_dTVinvar}.
\end{rem}

Let us give the most important arguments of the proof of Theorem~\ref{theo:weakinv}. The objective is to prove a weak error estimate of the type~\eqref{eq:theo-weakinv_weakerror}, for bounded and continuous functions $\varphi$ (this is sufficient, see Section~\ref{sec:notation}). First, the modified Euler scheme~\eqref{eq:res-modifiedscheme} is understood using its third formulation~\eqref{eq:scheme-IX}: it is interpreted as the (accelerated) exponential Euler scheme applied to the modified stochastic evolution equation~\eqref{eq:modifiedSPDE}. With this interpretation, the left-hand side of the weak error estimate may be decomposed as
\begin{equation}\label{eq:decomperror}
\E[\varphi(X_N^\tau)]-\int\varphi d\mu_\star=\E[\varphi(\IX_\tau(N\tau))]-\int\varphi d\mu_\star+\E[\varphi(\IX_{\tau,N})]-\E[\varphi(\IX_\tau(N\tau))],
\end{equation}
where $\bigl(\IX_\tau(t)\bigr)_{t\ge 0}$ denotes the solution of~\eqref{eq:modifiedSPDE} and using the identity $\IX_{n,\tau}=X_n^\tau$ for all $n\in\N_0$, see Section~\ref{sec:scheme-3rd}.

On the one hand, the first term on the right-hand side of~\eqref{eq:decomperror} vanishes when $N\to\infty$, owing to Proposition~\ref{propo:mu_star_modified}: under Assumptions~\ref{ass:gradient}, $\mu_\star$ is the unique invariant distribution of the modified equation~\eqref{eq:modifiedSPDE} for all $\tau\in(0,\tau_0)$. This is the first crucial observation which leads to Theorem~\ref{theo:weakinv}. Additional technical arguments are required to proved the error estimate.

On the other hand, the second term on the right-hand side of~\eqref{eq:decomperror} can be treated using the following result.
\begin{propo}\label{propo:utau-regularity}
Let Assumptions~\ref{ass:ergo} and~\ref{ass:Fregul2} be satisfied and $\varphi:H\to \R$ be a bounded and continuous function. For all $\tau\in(0,\tau_0)$, $t\ge 0$ and $x\in H$, set
\[
u_\tau(t,x)=\E_x[\varphi(\IX_\tau(t))].
\]
For all $t>0$, $u_\tau(t,\cdot)$ is differentiable and one has the following estimate: for all $\delta\in(0,\frac12)$, there exists $C_{\delta}\in(0,\infty)$ such that for all $\tau\in(0,\tau_0)$ and for all $t\in(2\tau,\infty)$, one has 
\begin{equation}\label{eq:utau-regularity}
\big|Du_\tau(t,x).h\big|\le C_{\delta} e^{-\kappa t}\vvvert\varphi\vvvert_0 \bigl(1\wedge (t-2\tau)\bigr)^{-\frac12}\Bigl(\sqrt{\tau}|h|+\bigl(1\wedge (t-2\tau)\bigr)^{-\frac12+\delta}|\IL^{-\frac12+\delta}h|\Bigr)
\end{equation}
for all $x,h\in H$, with $\kappa=\frac{\log(1+\tau_0\lambda_1)}{\tau\lambda_1}(\lambda_1-\Lf)>0$.
\end{propo}
As will be explained in Section~\ref{sec:auxiliary-kolmogorov-modified}, the mapping $(t,x)\in(0,\infty)\times H\to \R$ is solution of the Kolmogorov equation
\[
\partial_t u_\tau=\mathcal{L}_\tau u_\tau
\]
with initial value $u_\tau(0,\cdot)=\varphi$, where $\mathcal{L}_\tau$ is the infinitesimal generator of the modified stochastic evolution equation~\eqref{eq:modifiedSPDE}. Proposition~\ref{propo:utau-regularity} has the form of a strong Feller property for the modified equation, the challenge is to obtain estimates which hold uniformly with respect to $\tau\in(0,\tau_0)$ in a certain sense: this is why the condition $t>2\tau$ appears in the statement above, and why an additionnal term $\sqrt{\tau}|h|$ also appears in~\eqref{eq:utau-regularity}.

Proposition~\ref{propo:utau-regularity} is used as follows to prove an upper bound on the second term of the right-hand side of~\eqref{eq:decomperror}: one has the identity
\[
\E[\varphi(\IX_{\tau,N})]-\E[\varphi(\IX_\tau(N\tau))]=\E[u_\tau(0,\IX_{\tau,N})]-\E[u_\tau(N\tau,\IX_{\tau,0})]
\]
and the analysis of the error then follows a usual strategy: the identification of an appropriate continuous time process $\bigl(\tilde{\IX}_\tau(t)\bigr)_{t\ge 0}$ such that $\tilde{\IX}_\tau(t_n)=\IX_{\tau,n}$ for all $n\in\N$, the use of a telescoping sum argument and of It\^o's formula and of the property that $u_\tau$ solves the Kolmogorov equation. Since $\bigl(\IX_{\tau,n}\bigr)_{n\in\N}$ is obtained by applying the (accelerated) exponential Euler scheme to the modified SPDE~\eqref{eq:modifiedSPDE}, the auxiliary process $\bigl(\tilde{\IX}_\tau(t)\bigr)_{t\ge 0}$ is chosen such that a single error term (which vanishes if $F=0$) appears in the resulting expression of the weak error, which can be treated using Proposition~\ref{propo:utau-regularity} above. Note that for all $\delta\in(0,\frac12)$, the right-hand side of~\eqref{eq:utau-regularity} gives a singularity which is integrable. In addition, the form of Proposition~\ref{propo:utau-regularity} is consistent with Assumption~\ref{ass:Fregul1}, giving the required regularity condition on the nonlinearity $F$ in order to obtain the order of convergence $1/2$.

The arguments explained above are only formal, many technical estimates are required to give the proofs of the auxiliary Proposition~\ref{propo:utau-regularity} and then of the main Theorem~\ref{theo:weakinv}. The analysis is postponed to Section~\ref{sec:proofs}.

\begin{rem}
The regularity condition on the nonlinearity $F$ in Assumption~\ref{ass:Fregul1} is used only to obtain the weak rate of convergence $\frac12$ in Theorem~\ref{theo:weakinv}. If $F$ only satisfies Assumptions~\ref{ass:ergo},~\ref{ass:gradient} and~\ref{ass:Fregul2}, the convergence in total variation distance still holds, with order of convergence $\frac14$ instead of $\frac12$: one would get
\[
d_{\rm TV}(\mu_\infty^\tau,\mu_\star)\le C_\delta\tau^{\frac14-\delta}.
\]
Such a result also shows an improvement over the standard linear Euler scheme, when Assumption~\ref{ass:gradient} is satisfied.
\end{rem}

\subsection{Weak error estimates in a general setting}\label{sec:results_weak}

The major result of this article, Theorem~\ref{theo:weakinv}, stated above, gives an error estimate in the total variation distance $d_{\rm TV}$ for the approximation of the Gibbs distribution $\mu_\star$, which is the invariant distribution of the stochastic evolution equation~\eqref{eq:SPDE} when Assumption~\ref{ass:gradient} holds. The next result shows that the modified Euler scheme also provides weak error estimates to approximate the distribution $\rho_{X(T)}$ of the solution of~\eqref{eq:SPDE} at arbitary times $T\in(0,\infty)$, without the requirement that Assumptions~\ref{ass:ergo} and~\ref{ass:gradient} are satisfied. However, the weak error estimates below require the functions $\varphi$ to be of class $\mathcal{C}^2$, equivalently $\rho_{X_N^\tau}$ approximates $\rho_{X(N\tau)}$ when $\tau\to 0$ only in the $d_2$ distance. Whether an approximation result in the total variation distance $d_{\rm TV}=d_0$ can be obtained is left open. Note that this would correspond to weaken the regularity of the function $\varphi$ in the weak error estimates.

\begin{theo}\label{theo:weak}
Let the nonlinearity $F$ satisfy Assumptions~\ref{ass:F},~\ref{ass:Fregul1} and~\ref{ass:Fregul2}. For all $T\in(0,\infty)$, $\delta\in(0,\frac12)$ and $\tau_0\in(0,1)$, there exists $C_{\delta}(T)\in(0,\infty)$ such that for all $\tau=\frac{T}{N}\in(0,\tau_0)$ with $N\in\N$ and all $x_0\in H^{\frac14-\frac{\delta}{8}}$, one has
\begin{equation}\label{eq:theo-weak}
d_{2}(\rho_{X_N^\tau},\rho_{X(T)})\le C_\delta(T)\tau^{\frac12-\delta}\bigl(1+|x_0|_{\frac14-\frac{\delta}{8}}^2\bigr).
\end{equation}
More precisely, for all functions $\varphi:H\to\R$ of class $\mathcal{C}^2$ with bounded first and second order derivatives, all $\tau\in(0,\tau_0)$ and all $N\in\N$, such that $T=N\tau$, and all $x_0\in H^{\frac14-\frac{\delta}{8}}$, one has
\begin{equation}\label{eq:theo-weak-weakerror}
\big|\E[\varphi(X_N^\tau)]-\E[\varphi(X(T))]\big|\le C_{\delta}(T)\tau^{\frac12-\delta}\bigl(\vvvert\varphi\vvvert_1+\vvvert\varphi\vvvert_2\bigr)\bigl(1+|x_0|_{\frac14-\frac{\delta}{8}}^2\bigr).
\end{equation}
\end{theo}

When Assumption~\ref{ass:ergo} is satisfied, the stochastic evolution equation~\eqref{eq:SPDE} admits a unique invariant distribution $\mu_\infty$ (see Proposition~\ref{propo:invar}). The next result provides an error estimate for $d_2(\mu_\infty^\tau,\mu_\infty)$, where $\mu_\infty^\tau$ is the unique invariant distribution of the modified Euler scheme~\eqref{eq:scheme} (see Proposition~\ref{propo:invar-scheme}).

\begin{theo}\label{theo:weak-ergo}
Let the nonlinearity $F$ satisfy Assumptions~\ref{ass:ergo},~\ref{ass:Fregul1} and~\ref{ass:Fregul2}. For all $\delta\in(0,\frac12)$ and $\tau_0\in(0,1)$, there exists $C_{\delta}\in(0,\infty)$ such that for all $\tau\in(0,\tau_0)$ one has
\begin{equation}\label{eq:theo-weak-ergo}
d_{2}(\mu_\infty^\tau,\mu_\infty)\le C_\delta\tau^{\frac12-\delta}.
\end{equation}
More precisely, for all $\delta\in(0,\frac12)$ and $0\le \kappa<\min(\frac{1}{1+\tau_0\lambda_1},\frac{\log(1+\tau_0\lambda_1)}{\tau_0\lambda_1})(\lambda_1-\Lf)\in(0,\lambda_1-\Lf)$, there exists $C_{\delta,\kappa}\in(0,\infty)$ such that for all functions $\varphi:H\to\R$ of class $\mathcal{C}^2$ with bounded first and second order derivatives, all $x_0\in H^{\frac14-\frac{\delta}{8}}$, all $\tau\in(0,\tau_0)$ and all $N\in\N$, such that $N\tau\ge 1$, one has
\begin{equation}\label{eq:theo-weak-ergo-weakerror}
\big|\E[\varphi(X_N^\tau)]-\int\varphi d\mu_\star\big|\le C_{\delta,\kappa}\bigl(\vvvert\varphi\vvvert_1+\vvvert\varphi\vvvert_2\bigr)\Bigl(\tau^{\frac12-\delta}(1+|x_0|_{\frac14-\frac{\delta}{8}}^2)+e^{-\kappa N\tau}(1+|x_0|)\Bigr).
\end{equation}
\end{theo}

Like for Theorem~\ref{theo:weak}, whether an approximation result in the total variation distance, namely an error estimate for $d_{\rm TV}(\mu_\infty^\tau,\mu_\infty)$, can be obtained is left open. Note that Theorem~\ref{theo:weakinv} gives a positive answer when Assumption~\ref{ass:gradient} holds.

Let us now describe the approach to prove Theorem~\ref{theo:weak}, more precisely to prove the weak error estimate~\eqref{eq:theo-weak-weakerror}. Similarly to~\eqref{eq:decomperror}, it is convenient to decompose the error as
\begin{equation}\label{eq:decomperror-gen}
\E[\varphi(X_N^\tau)]-\E[\varphi(X(T))]=\E[\varphi(\IX_\tau(T))]-\E[\varphi(X(T))]+\E[\varphi(\IX_{\tau,N})]-\E[\varphi(\IX_\tau(T))],
\end{equation}
where $\bigl(\IX_\tau(t)\bigr)_{t\ge 0}$ denotes the solution of the modified stochastic evolution equation~\eqref{eq:modifiedSPDE}, and using the identity $\IX_{n,\tau}=X_n^\tau$ for all $n\in\N_0$, see Section~\ref{sec:scheme-3rd}. Like for the proof of Theorem~\ref{theo:weakinv}, the second term in the right-hand side of~\eqref{eq:decomperror-gen} is written as
\[
\E[\varphi(\IX_{\tau,N})]-\E[\varphi(\IX_\tau(T))]=\E[u_\tau(0,\IX_{\tau,N})]-\E[u_\tau(T,\IX_{\tau,0})],
\]
where $u_\tau(t,x)=\E_{x}[\varphi(\IX_\tau(t))]$ (see Proposition~\ref{propo:utau-regularity}). The first term in the right-hand side of~\eqref{eq:decomperror-gen} may be written as
\[
\E[\varphi(\IX_\tau(T))]-\E[\varphi(X(T))]=\E[u_\tau(T,X(0))]-\E[u_\tau(0,X(T))].
\]
Applying It\^o's formula, one would need to establish regularity estimates for the first and the second order derivatives $Du_\tau(t,x)$ and $D^2u_\tau(t,x)$. In order to obtain bounds with constants which do not depend on the time-step size $\tau$, technical arguments are needed, see the statement of Proposition~\ref{propo:utau-regularity} for the first order derivative. To avoid using such technical statements and analysis on the second order derivative, it is more convenient to employ the alternative expression
\[
\E[\varphi(\IX_\tau(T))]-\E[\varphi(X(T))]=\E[u(0,\IX_\tau(T))]-\E[u(T,\IX_\tau(0))],
\]
where $u(t,x)=\E_x[\varphi(X(t))]$, and to use appropriate regularity estimates for the first and second order derivatives $Du(t,x)$ and $D^2u(t,x)$ to obtain an error estimate when $\tau\to 0$, see Section~\ref{sec:auxiliary-kolmogorov-original}. Note that the requirement to assume that $\varphi$ is of class $\mathcal{C}^2$ ($\vvvert\varphi\vvvert_2$ appears in the weak error estimates) is due to the analysis of the error term $\E[u(0,\IX_\tau(T))]-\E[u(T,\IX_\tau(0))]$, whereas the treatment of the other error term can be performed using a version of Proposition~\ref{propo:utau-regularity} and assuming $\varphi$ to be bounded would be sufficient. It is not known whether the regularity condition on $\varphi$ may be relaxed.

Remark that when Assumptions~\ref{ass:ergo} and~\ref{ass:gradient} are satisfied, one has
\[
\underset{T\to\infty}\lim~\bigl(\E[\varphi(\IX_\tau(T))]-\E[\varphi(X(T))]\bigr)=0,
\]
since $\mu_\star$ is the invariant distribution for both stochastic evolution equations~\eqref{eq:SPDE} and~\eqref{eq:modifiedSPDE} with that restrictive condition on the nonlinearity $F$ (see Propositions~\ref{propo:invar} and~\ref{propo:mu_star_modified}). This observation explains why the regularity condition on $\varphi$ can be relaxed in Theorem~\ref{theo:weakinv}.

In order to prove Theorem~\ref{theo:weak-ergo}, it suffices to apply the arguments above with appropriate regularity bounds for the derivatives of the functions $u_\tau(t,\cdot)$ and $u(t,\cdot)$, to obtain upper bounds depending on $\exp(-\kappa t)$. The required bounds will be provided below, however the details of the proof of Theorem~\ref{theo:weak-ergo} will be omitted.

\begin{rem}\label{rem:approaches}
The approach presented above to prove Theorems~\ref{theo:weak} and~\ref{theo:weak-ergo} exploits the interpretation of the modified Euler scheme in terms of the modified stochastic evolution equation~\eqref{eq:modifiedSPDE} explained in Section~\ref{sec:scheme-3rd}. It is possible to prove those results without using that interpretation and using instead the second interpretation~\eqref{eq:scheme2} of the modified Euler scheme (see Section~\ref{sec:scheme-2nd}), in the spirit of~\cite{Debussche:11} and~\cite{B:2014} respectively. Note that these two references do not require Assumption~\ref{ass:Fregul1}. Very similar arguments would be needed to estimates the relevant error terms in the two approaches, however the approach we follow in this work may be simpler to present, and it is worth giving the details of the approach using a modified equation, since it is has not been treated in the literature so far. Our approach may also be used to prove the weak error estimates from~\cite{Debussche:11} and~\cite{B:2014} for the standard Euler scheme, which can also be interpreted as the accelerated exponential Euler scheme applied to an appropriate modified stochastic evolution equation, see Section~\ref{sec:results_standard} below, the details are omitted.
\end{rem}

\subsection{Comparison with the standard Euler scheme}\label{sec:results_standard}

In this section, we compare the results stated above concerning the modified Euler scheme~\eqref{eq:res-modifiedscheme}, with the results obtained for the standard Euler scheme~\eqref{eq:res-standardscheme}.

Let us first study the qualitative behavior of the schemes. As already mentioned and illustrated in Section~\ref{sec:results_quali}, the standard Euler scheme does not preserve the spatial regularity of the solution of the stochastic evolution equation: Theorems~\ref{theo:equivalence} and~\ref{theo:regularity} do not hold for the standard Euler scheme. Precisely, for any fixed $\tau\in(0,\tau_0)$ and any $N\in\N$, $X_N^{\tau,\s}$ takes values in $D(\IL^\alpha)$ for all $\alpha\in(0,\frac12)$, whereas $X(N\tau)$ takes values in $D(\IL^{\alpha})$ if and only if $\alpha\in(0,\frac14)$. In the Gaussian case ($F=0$), the distributions $\rho_{X_N^{\tau,\s}}$ and $\rho_{X(N\tau)}$ of the $H$-valued random variables $X_N^{\tau,\s}$ and $X(N\tau)$ are thus singular. The proposed scheme~\eqref{eq:res-modifiedscheme} thus overcomes the qualitative limitations of the standard scheme~\eqref{eq:res-standardscheme}.

In the ergodic situation, when Assumption~\ref{ass:ergo} is satisfied, an appropriate version of Proposition~\ref{propo:invar-scheme} is satisfied for the standard Euler scheme. On the one hand, the scheme admits a unique invariant distribution denoted by $\mu_\infty^{\tau,\s}$ and variants of the inequalities~\eqref{eq:invar-scheme-bound} and~\eqref{eq:invar-scheme-ergo} hold: see for instance~\cite{B:2014}. On the other hand, when $F=0$, $\mu_\infty^{\tau,\s}$ is a Gaussian distribution $\nu^{\tau}$ (see Equation~\eqref{eq:nu}), which is not equal to the Gaussian distribution $\nu$. Precisely, for all $\tau\in(0,\tau_0)$, $\nu^\tau$ is the centered Gaussian distribution given by
\begin{equation}\label{eq:nutau}
\nu^\tau=\mathcal{N}\bigl(0,\frac12\IL^{-1}(I+\frac{\tau\IL}{2})^{-1}\bigr),
\end{equation}
which is the distribution of the $H$-valued Gaussian random variable $Z^\tau=\sum_{j\in\N}\frac{2\gamma_j}{\sqrt{2\lambda_j(2+\lambda_j\tau)}}e_j$ where $\bigl(\gamma_j\bigr)_{j\in\N}$ is a sequence of independent standard real-valued Gaussian random variables. Consistently with the discussion above concerning the non-validity of Theorem~\ref{theo:equivalence} for the standard Euler scheme, the Gaussian distributions $\nu^\tau$ and $\nu$ are singular for all $\tau\in(0,\tau_0)$. In addition, for all $\tau\neq\tau'\in(0,\tau_0)$, the Gaussian distributions $\nu^{\tau}$ and $\nu^{\tau'}$ are also singular.

Let us now discuss convergence in distribution of $\rho_{X_N^{\tau,\s}}$ to $\rho_{X(N\tau)}$, when $\tau\to 0$ and $N\tau=T$, and of $\mu_\infty^{\tau,\s}$ to $\mu_\infty$ in the ergodic situation. This question has been studied in the literature: the weak order of convergence of the standard Euler scheme is equal to $1/2$. We refer to~\cite{Debussche:11} for a version of Theorem~\ref{theo:weak} and to~\cite{B:2014} for a version of Theorem~\ref{theo:weak-ergo}, for the standard Euler scheme, under weaker assumptions on the nonlinearity $F$ (Assumption~\ref{ass:Fregul1} is not required in those references). We provide neither precise statements nor elements of proofs of these two results. As already mentioned in Remark~\ref{rem:approaches} (Section~\ref{sec:results_weak}), the weak error analysis in~\cite{Debussche:11} and in~\cite{B:2014} exploits different decompositions of the errors, and a proof using an interpretation of the standard Euler scheme in terms of a modified equation may also be employed.

The remaining question is whether a result of the type of Theorem~\ref{theo:weakinv} holds for the standard Euler scheme, when Assumption~\ref{ass:gradient} is satisfied for the nonlinearity $F$. Owing to the discussion above, this is not possible in the case Gaussian case ($F=0$), since the Gaussian distributions $\nu^\tau$ and $\nu$ are singular for all $\tau\in(0,\tau_0)$. To state weak error estimates with order of convergence $1/2$, it is necessary to assume that the function $\varphi$ is of class $\mathcal{C}^2$, see~\cite{B:2020}. Therefore, the qualitative properties of the scheme have an impact on its quantitative analysis. The proposed modified Euler scheme improves both the qualitative and quantitative properties compared with the standard Euler scheme.

To go further in the analysis of the standard Euler scheme, we state the following variant of Theorem~\ref{theo:weakinv}, where the Gibbs distribution $\mu_\star$ needs to be replaced by the Gibbs distribution $\mu_\star^\tau$ defined by
\begin{equation}\label{eq:mu_startau}
d\mu_\star^\tau(x)=(\mathcal{Z}^\tau)^{-1}e^{-2V(x)}d\nu^\tau(x)
\end{equation}
with normalization constant $\mathcal{Z}^\tau=\int e^{-2V(x)}d\nu^\tau(x)\in(0,\infty)$, where the reference Gaussian distribution is $\nu^\tau$ defined by~\eqref{eq:nutau} (instead of $\nu$). The only but crucial difference in the definitions~\eqref{eq:mu_star} and~\eqref{eq:mu_startau} is the choice of the reference Gaussian distribution, equal to $\nu$ or $\nu^\tau$ respectively. One has the following result.

\begin{theo}\label{theo:weakinv-standard}
Let the nonlinearity $F$ satisfy Assumptions~\ref{ass:ergo},~\ref{ass:gradient},~\ref{ass:Fregul1} and~\ref{ass:Fregul2}. Assume also that for all $\delta\in(0,\frac12)$, there exists $C_\delta\in(0,\infty)$ such that for all $x\in H^{2\delta}$, one has
\begin{equation}\label{eq:Fregul3}
|\IL^\delta F(x)|\le C_\delta|\IL^{2\delta}x|.
\end{equation}
For all $\delta\in(0,\frac12)$ and $\tau_0\in(0,1)$, there exists $C_{\delta}\in(0,\infty)$ such that for all $\tau\in(0,\tau_0)$ one has
\begin{equation}\label{eq:theo-weakinv_dTVinvar-standard}
d_{\rm TV}(\mu_\infty^{\tau,\s},\mu_\star^\tau)\le C_\delta\tau^{\frac12-\delta},
\end{equation}
where $\mu_\infty^{\tau,\s}$ is the unique invariant distribution of $\bigl(X_n^{\tau,\s})_{n\ge 0}$ and $\mu_\star^\tau$ is given by~\eqref{eq:mu_star}.

Moreover, for all $\delta\in(0,\frac12)$ and $\tau_0\in(0,1)$, there exists $C_{\delta}\in(0,\infty)$ such that for all $x_0\in H^{\frac14-\frac{\delta}{8}}$, all $\tau\in(0,\tau_0)$ and $N\in\N$, with $N\tau\ge 2\tau_0$, and any bounded measurable function $\varphi\in\mathcal{B}_b(H)$, one has
\begin{equation}\label{eq:theo-weakinv_weakerror-standard}
\big|\E[\varphi(X_N^{\tau,\s})]-\int\varphi d\mu_\star^\tau\big|\le C_{\delta}\vvvert\varphi\vvvert\Bigl(\tau^{\frac12-\delta}(1+|x_0|_{\frac14-\frac{\delta}{8}}^2)+e^{-\kappa N\tau}(1+|x_0|)\Bigr),
\end{equation}
with $\kappa=\frac{\log(1+\tau_0\lambda_1)}{\tau_0\lambda_1}(\lambda_1-\Lf)\in(0,\lambda_1-\Lf)$.
\end{theo}
To the best of our knowledge, Theorem~\ref{theo:weakinv-standard} is a new result. Even if this convergence result has little practical interest since one is interested in the approximation of $\mu_\star$, not of $\mu_\star^\tau$, this statement is another illustration of the limitations of the standard Euler scheme~\eqref{eq:res-standardscheme} and of the superiority of the modified Euler scheme~\eqref{eq:res-modifiedscheme}. In addition, the proof of Theorem~\ref{theo:weakinv-standard} is based on the ideas developed to prove Theorem~\ref{theo:weakinv}. Note that when $F=0$, one has $d_{\rm TV}(\mu_\infty^{\tau,\s},\mu_\star^\tau)=0$, by definition of $\nu^\tau$. However, $d_{\rm TV}(\mu_\star^\tau,\mu_\star)=2$, since the distributions $\nu^\tau$ and $\nu$ are singular.

The statement of Theorem~\ref{theo:weakinv-standard} requires the nonlinearity $F$ to satisfy the additional condition~\eqref{eq:Fregul3}. This is not restrictive in practice, since this condition is also satisfied in the framework of Section~\ref{sec:example}.

Let us explain the strategy of the proof of Theorem~\ref{theo:weakinv-standard}. As already mentioned above, the standard Euler scheme~\eqref{eq:res-standardscheme} can be interpreted as obtained from the application of the accelerated exponential Euler scheme to a modified stochastic evolution equation. Precisely, for all $\tau\in(0,\tau_0)$, introduce the process $\bigl(\IX^{\tau,\s}(t)\bigr)_{t\ge 0}$ which is solution of the modified stochastic evolution equation
\begin{equation}\label{eq:modifiedSPDEstandard}
d\IX^{\tau,\s}(t)=-\IL_\tau \IX^{\tau,\s}(t)dt+Q_{\tau}F(\IX^{\tau,\s}(t))dt+R_{\tau}^{\frac12}dW(t),
\end{equation}
with initial value $\IX^{\tau,\s}(0)=x_0$, where the linear operators $\IL_\tau$ and $Q_\tau$ are given by~\eqref{eq:modifiedILQ}, and the linear operator $R_\tau$ is given by
\begin{equation}\label{eq:Rtau}
R_\tau=Q_\tau(I+\frac{\tau\IL}{2})^{-1}.
\end{equation}
The application of the accelerated exponential Euler scheme to the modified stochastic evolution equation~\eqref{eq:modifiedSPDEstandard} gives
\[
\IX_{n+1}^{\tau,\s}=e^{-\tau\IL_\tau}\IX_n^{\tau,\s}+\IL_\tau^{-1}(I-e^{-\tau\IL_\tau})Q_\tau F(\IX_n^{\tau,\s})+\int_{t_n}^{t_{n+1}}e^{-(t_{n+1}-s)\IL_\tau}R_\tau^{\frac12}dW(s),
\]
with initial value $\IX_0^{\tau,\s}=x_0=X_0^{\tau,\s}$. It is straightforward to check that the identity
\[
\int_{t_n}^{t_{n+1}}e^{-2(t_{n+1}-s)\IL_\tau}dsR_\tau=\tau(I+\tau\IL)^{-2}=\tau\IA_\tau^2
\]
holds for all $\tau\in(0,\tau_0)$, using the definitions of $\IL_\tau$, $Q_\tau$ and $R_\tau$. Therefore $\bigl(X_n^{\tau,\s}\bigr)_{n\in\N_0}$ and $\bigl(\IX_n^{\tau,\s}\bigr)_{n\in\N_0}$ are equal in distribution. Note that the operators $\IL_\tau$ and $Q_\tau$ in the modified stochastic evolution equation~\eqref{eq:modifiedSPDEstandard} are the same as in the interpretation of the modified Euler scheme~\eqref{eq:res-modifiedscheme} in terms of the modified stochastic evolution equation~\eqref{eq:modifiedSPDE} developped in Section~\ref{sec:scheme-3rd}, however $R_\tau\neq Q_\tau$. The invariant distribution of the modified stochastic evolution equation~\eqref{eq:modifiedSPDEstandard} is not known, even if $F$ satisfies Assumption~\ref{ass:gradient} (except when $F=0$), contrary to the situation for the modified Euler scheme. To overcome this issue, let us introduce an additional auxiliary process $\bigl(\IX_\star^{\tau,\s}(t)\bigr)_{t\ge 0}$, which is solution of the modified stochastic evolution equation
\begin{equation}\label{eq:modifiedSPDEstandardstar}
d\IX_\star^{\tau,\s}(t)=-\IL_\tau \IX_\star^{\tau,\s}(t)dt+R_{\tau}F(\IX_\star^{\tau,\s}(t))dt+R_{\tau}^{\frac12}dW(t),
\end{equation}
with initial value $\IX_\star^{\tau,\s}(0)=x_0$. Compared with~\eqref{eq:modifiedSPDEstandard}, the drift $Q_\tau F(\cdot)$ is replaced by $R_\tau F(\cdot)$ in~\eqref{eq:modifiedSPDEstandardstar}. Thanks to this modification, it is straightforward to check that, when $F=-DV$ (Assumption~\ref{ass:gradient}), the invariant distribution of~\eqref{eq:modifiedSPDEstandardstar} is equal to the modified Gibbs distribution $\mu_\star^\tau$ defined by~\eqref{eq:mu_startau}, where the reference measure is the Gaussian distribution $\nu^\tau$ which is invariant when $F=0$.

The weak error in the left-hand side of~\eqref{eq:theo-weakinv_weakerror-standard} can be decomposed as follows (compare with~\eqref{eq:decomperror} for the analysis of the modified scheme): with the notation $t_N=N\tau$, one has
\begin{equation}\label{eq:decomperror-standard}
\begin{aligned}
\E[\varphi(X_N^{\tau,\s})]-\int\varphi d\mu_\star^\tau&=\E[\varphi(\IX_N^{\tau,\s})]-\int\varphi d\mu_\star^\tau\\
&=\E[\varphi(\IX_N^{\tau,\s})]-\E[\varphi(\IX^{\tau,\s}(t_N)]\\
&+\E[\varphi(\IX^{\tau,\s}(t_N))]-\E[\varphi(\IX_\star^{\tau,\s}(t_N)]\\
&+\E[\varphi(\IX_\star^{\tau,\s}(t_N)]-\int\varphi d\mu_\star^\tau.
\end{aligned}
\end{equation}
The next step is the introduction of the auxiliary functions  $u^{\tau,\s}$ and $u_\star^{\tau,\s}$ defined by
\[
u^{\tau,\s}(t,x)=\E_x[\varphi(\IX^{\tau,\s}(t))],\quad u_\star^{\tau,\s}(t,x)=\E_x[\varphi(\IX_\star^{\tau,\s}(t))],
\]
for all $t\ge 0$ and $x\in H$, which are solutions of the Kolmogorov equation associated with the modified stochastic evolution equations~\eqref{eq:modifiedSPDEstandard} and~\eqref{eq:modifiedSPDEstandardstar} respectively. The decomposition of the error above may be rewritten as
\begin{align*}
\E[\varphi(X_N^{\tau,\s})]-\int\varphi d\mu_\star^\tau&=\E[u^{\tau,\s}(0,\IX_N^{\tau,\s})]-\E[u^{\tau,\s}(t_N,\IX_0^{\tau,\s})]\\
&+\E[u_\star^{\tau,\s}(0,\IX^{\tau,\s}(t_N))]-\E[u_\star^{\tau,s}(t_N,\IX^{\tau,\s}(0))]\\
&+\E[u_\star^{\tau,s}(t_N,x_0)]-\int\varphi d\mu_\star^\tau.
\end{align*}
Finally, the three error terms in the right-hand side above are studied using regularity properties of the mappings $u^{\tau,\s}$ and $u_\star^{\tau,\s}$: versions of Proposition~\ref{propo:utau-regularity} hold, see Lemma~\ref{lem:utausutausstar-regularity} in Section~\ref{sec:standard-kolmogorov}. In particular, assuming that $\varphi$ is bounded and continuous is sufficient to prove the weak error estimate with order $1/2$.

The tools in the proof of Theorem~\ref{theo:weakinv-standard} are similar to those employed in the proof of Theorem~\ref{theo:weakinv}, with a few differences, due to the presence of the linear operator $R_\tau$ and $R_\tau^{\frac12}$ instead of $Q_\tau$ and $Q_\tau^{\frac12}$ in the auxiliary modified equations. We refer to Section~\ref{sec:standard} for the statements of the required auxiliary results and the proof of Theorem~\ref{theo:weakinv-standard}. Some proofs are omitted to avoid repeating the same arguments as in Section~\ref{sec:auxiliary}, however the main new and non trivial arguments are treated carefully.

The results of this section show that the modified Euler scheme~\eqref{eq:res-modifiedscheme} proposed in this article is a substantial improvement of the standard Euler scheme~\eqref{eq:res-standardscheme}, both qualitatively and quantitatively. In practice, the cost of each iteration of the modified Euler scheme is more expensive, but the computations are of the same type as for the standard Euler scheme, except for an additional step which requires a Cholesky decomposition. However, the gain provided by the proposed scheme is huge: the properties of the standard scheme cannot be improved by reducing the time-step size.

\subsection{Comparison with the exponential Euler scheme}\label{sec:results_exponential}

In this section, we state new weak error estimates for the (accelerated) exponential Euler scheme defined by~\eqref{eq:res-exponentialscheme} above. Compared with existing results in the literature, the error between the distributions $X_N^{\tau,\e}$ and $X(T)$ is considered in the total variation distance $d_{\rm TV}$, instead of the distance $d_2$. The order of convergence is equal to $1/2$. This means that it is not necessary to assume that the function $\varphi$ is of class $\mathcal{C}^2$ (or $\mathcal{C}^1$) to obtain a weak error estimate for $\E[\varphi(X_N^{\tau,\e})]-\E[\varphi(X(T))]$. Compared with the results stated above for the modified Euler scheme in Section~\ref{sec:results_invar} and~\ref{sec:results_weak}, the approximation in the total variation distance holds at any time $T\in(0,\infty)$, and for the invariant distribution without Assumption~\ref{ass:gradient}.

\begin{theo}\label{theo:weak-expo}
Let the nonlinearity $F$ satisfy Assumptions~\ref{ass:F},~\ref{ass:Fregul1} and~\ref{ass:Fregul2}. For all $T\in(0,\infty)$, $\delta\in(0,\frac12)$ and $\tau_0\in(0,1)$, there exists $C_{\delta}(T)\in(0,\infty)$ such that for all $\tau=\frac{T}{N}\in(0,\tau_0)$ with $N\in\N$ and all $x_0\in H^{\frac14-\frac{\delta}{2}}$, one has
\begin{equation}\label{eq:theo-weak-expo}
d_{\rm TV}(\rho_{X_N^{\tau,\e}},\rho_{X(T)})\le C_\delta(T)\tau^{\frac12-\delta}\bigl(1+|x_0|_{\frac14-\frac{\delta}{8}}^2\bigr).
\end{equation}
More precisely, for all bounded and measurable functions $\varphi:H\to\R$, all $\tau\in(0,\tau_0)$ and all $N\in\N$, such that $T=N\tau$, and all $x_0\in H^{\frac14-\frac{\delta}{8}}$, one has
\begin{equation}\label{eq:theo-weak-weakerror-expo}
\big|\E[\varphi(X_N^{\tau,\e})]-\E[\varphi(X(T))]\big|\le C_{\delta}(T)\tau^{\frac12-\delta}\vvvert\varphi\vvvert\bigl(1+|x_0|_{\frac14-\frac{\delta}{8}}^2\bigr)).
\end{equation}
\end{theo}

When Assumption~\ref{ass:ergo} is satisfied, the exponential Euler scheme~\eqref{eq:res-exponentialscheme} admits a unique invariant distribution denoted by $\mu_\infty^{\tau,\e}$, and one obtains the following result for the approximation of the invariant distribution $\mu_\infty$.
\begin{theo}\label{theo:weak-ergo-expo}
Let the nonlinearity $F$ satisfy Assumptions~\ref{ass:ergo},~\ref{ass:Fregul1} and~\ref{ass:Fregul2}. For all $\delta\in(0,\frac12)$ and $\tau_0\in(0,1)$, there exists $C_{\delta}\in(0,\infty)$ such that for all $\tau\in(0,\tau_0)$ one has
\begin{equation}\label{eq:theo-weak-ergo-expo}
d_{0}(\mu_\infty^{\tau,\e},\mu_\infty)\le C_\delta\tau^{\frac12-\delta}.
\end{equation}
More precisely, for all bounded and measurable functions $\varphi:H\to\R$s, all $x_0\in H^{\frac14-\frac{\delta}{8}}$, all $\tau\in(0,\tau_0)$ and all $N\in\N$, such that $N\tau\ge 1$, one has
\begin{equation}\label{eq:theo-weak-ergo-weakerror-expo}
\big|\E[\varphi(X_N^{\tau,\e})]-\int\varphi d\mu_\star\big|\le C_{\delta}\vvvert\varphi\vvvert\Bigl(\tau^{\frac12-\delta}(1+|x_0|_{\frac14-\frac{\delta}{8}}^2)+e^{-\kappa N\tau}(1+|x_0|)\Bigr),
\end{equation}
with $\kappa=\Lf-\lambda_1$.
\end{theo}

The proof of Theorem~\ref{theo:weak-expo} is given in Section~\ref{sec:expo}. The proof of Theorem~\ref{theo:weak-ergo-expo} is omitted, since it only requires standard additional arguments compared with the proof of Theorem~\ref{theo:weak-expo}.

Note that Theorem~\ref{theo:equivalence} obviously holds when using the accelerated exponential Euler scheme~\eqref{eq:res-exponentialscheme}. Indeed, if $F=0$, then $X_N^{\tau,\e}=W^{\IL}(N\tau)=X(N\tau)$ for all $\tau\in(0,\tau_0)$ and $N\in\N$. Theorem~\ref{theo:regularity} is also verified for that integrator, the details are omitted.

The modified Euler scheme~\eqref{eq:res-modifiedscheme} proposed in this article shares qualitative and quantitative properties with the accelerated exponential Euler scheme~\eqref{eq:res-exponentialscheme}, with a major practical difference: it is not needed to compute exponentials of the type $e^{-\tau\IL}$, instead one only needs to compute LU and Cholesky decompositions. This may be of interest in some situations where using exponential integrators may not be possible.

\section{Auxiliary results}\label{sec:auxiliary}

This section is devoted to state and prove the auxiliary results which are required to prove the main results of this article stated in Section~\ref{sec:results}. Additional results will be required in Sections~\ref{sec:expo} and~\ref{sec:standard}, they will be studied later since they are not required for the analysis of the modified Euler scheme.

Section~\ref{sec:auxiliary-operators} gives several properties of the linear operators $\IL_\tau$ and $Q_\tau$ appearing in the interpretation~\eqref{eq:scheme-IX} of the modified Euler scheme in terms of the modified stochastic evolution equation~\eqref{eq:modifiedSPDE}. Properties of the so-called modified stochastic convolution process are studied in Section~\ref{sec:auxiliary-convolution}. Then well-posedness, moment bounds and long-time behavior of the modified equation are studied in Section~\ref{sec:auxiliary-modified}. The most original and important results of this section deal with the regularity properties of the solutions of the Kolmogorov equations associated with the modified equation~\eqref{eq:modifiedSPDE}, with a careful analysis of the dependence in the bounds with respect to the time-step size $\tau$, see Section~\ref{sec:auxiliary-kolmogorov-modified}. Those results are combined to give a proof of Proposition~\ref{propo:utau-regularity} in Section~\ref{sec:proofs-weakinv}. Then, several results concerning the modified Euler scheme (moment bounds, invariant distribution) are given in Section~\ref{sec:auxiliary-scheme}. Finally, Section~\ref{sec:auxiliary-kolmogorov-original} is devoted to regularity properties of the solutions of the Kolmogorov equation associated with the original stochastic evolution equation~\eqref{eq:SPDE}.

\subsection{Properties of the auxiliary linear operators}\label{sec:auxiliary-operators}

Lemma~\ref{lem:Q_tauIL_tau-bound} below states several bounds on the linear operators $Q_\tau$, $\IL_\tau$ and $e^{-t\IL_\tau}$, which are uniform with respect to the time-step size $\tau\in(0,\tau_0)$, with arbitrary $\tau_0\in(0,1)$. Then Lemma~\ref{lem:Q_tauIL_tau-error} gives error estimates for $Q_\tau-Q$ and $\IL_\tau-\IL$ when $\tau\to 0$, in an appropriate sense.

Recall that $\IL_\tau$ and $Q_\tau$ are defined by~\eqref{eq:modifiedILQ}, and that$\IL_\tau e_j=\lambda_{\tau,j}e_j$ and $Q_\tau e_j=q_{\tau,j}e_j$, for all $j\in\N$, with the eigenvalues $\lambda_{\tau,j}$ and $q_{\tau,j}$ given by~\eqref{eq:modifiedILQ-eigen}, see Section~\ref{sec:scheme-3rd}.

The first result of this section is Lemma~\ref{lem:Q_tauIL_tau-bound}.
\begin{lemma}\label{lem:Q_tauIL_tau-bound}
Let $\tau_0\in(0,1)$. The linear operator $Q_\tau$ is bounded, and for all $\tau\in(0,\tau_0)$ one has
\begin{equation}\label{eq:Q_tau-bound}
\|Q_\tau\|_{\mathcal{L}(H)}=\frac{1+\log(\tau\lambda_1)}{\tau\lambda_1}\le 1.
\end{equation}
Moreover, one has the following spectral gap inequality for the self-adjoint unbounded linear operators $\IL_\tau$: for all $x\in D(A)$ and all $\tau\in(0,\tau_0)$, one has
\begin{equation}\label{eq:IL_tau-gap}
\langle \IL_\tau x,x\rangle\ge \frac{\log(1+\tau_0\lambda_1)}{\tau_0}|x|^2.
\end{equation}

For all $\alpha\in[0,1)$, the semigroup $\bigl(e^{-t\IL_\tau}\bigr)_{t\ge 0}$ satisfies the following properties:
\begin{align}
&\underset{\tau\in(0,\tau_0)}\sup~\underset{t\in(0,\infty)}\sup~\|e^{-t\IL_\tau}\|_{\mathcal{L}(H)}\le 1,\label{eq:IL_tau-bound}\\
&\underset{\tau\in(0,\tau_0)}\sup~\underset{t\in(\tau,\infty)}\sup~(t-\tau)^\alpha \|\IL^\alpha e^{-t\IL_\tau}\|_{\mathcal{L}(H)}<\infty,\label{eq:IL_tau-smoothing}\\
&\underset{\tau\in(0,\tau_0)}\sup~\underset{t\in(0,\tau)}\sup~\tau^\alpha\|\IL^\alpha e^{-t\IL_\tau}Q_\tau\|_{\mathcal{L}(H)}<\infty,\label{eq:IL_tau-smoothing-bis}\\
&\underset{\tau\in(0,\tau_0)}\sup~\|Q_\tau^{-\frac12}e^{-\tau\IL_\tau}\|_{\mathcal{L}(H)}<\infty\label{eq:IL_tau-IL_tauQ_tau1/2},\\
&\underset{\tau\in(0,\tau_0)}\sup~\underset{t\in(0,\infty)}\sup~t^{-\alpha}\|\IL^{-\alpha}(I-e^{-t\IL_\tau})\|_{\mathcal{L}(H)}<\infty.\label{eq:IL_tau-increment}
\end{align}
\end{lemma}

The inequalities~\eqref{eq:IL_tau-smoothing} and~\eqref{eq:IL_tau-smoothing-bis} are a form of the smoothing inequality~\eqref{eq:smoothing}. To obtain bounds which are uniform with respect to the time-step size $\tau$, it is necessary to treat separately the cases $t\in(0,\tau_0)$ and $t\in(\tau,\infty)$: indeed, for all $0<\alpha<\alpha'\le 1$, $\IL^{\alpha'} e^{-\alpha \tau\IL_\tau}=\IL^{\alpha'}(I+\tau\IL)^{-\alpha}$ is not a bounded operator, for any value of $\tau\in(0,\tau_0)$, therefore the inequality~\eqref{eq:IL_tau-smoothing} requires the condition $t>\tau$. A naive and simpler form of the smoothing inequality holds: one has
\[
\underset{\tau\in(0,\tau_0)}\sup~\underset{t\in(0,\infty)}\sup~t^\alpha\|\IL_\tau^\alpha e^{-t\IL_\tau}\|_{\mathcal{L}(H)}<\infty.
\]
However the linear operators $\IL_\tau^\alpha$ and $\IL_\tau^\alpha$ define norms $|\IL_\tau^\alpha\cdot|$ and $|\IL^\alpha\cdot|$ which are not equivalent. This is why the smoothing inequality~\eqref{eq:IL_tau-smoothing} and~\eqref{eq:IL_tau-smoothing-bis} are needed.

The second result of this section is Lemma~\ref{lem:Q_tauIL_tau-error}.
\begin{lemma}\label{lem:Q_tauIL_tau-error}
Let $\tau_0\in(0,1)$. For all $\alpha\in[0,1]$, one has the error bounds
\begin{align}
&\underset{\tau\in(0,\tau_0)}\sup~\tau^{-\alpha}\|\IL^{-\alpha}(Q_\tau-I)\|_{\mathcal{L}(H)}<\infty,\label{eq:Q_tau-error}\\
&\underset{\tau\in(0,\tau_0)}\sup~\tau^{-\alpha}\|\IL^{-1-\alpha}(\IL_\tau-\IL)\|_{\mathcal{L}(H)}<\infty.\label{eq:IL_tau-error}
\end{align}
\end{lemma}

\begin{proof}[Proof of Lemma~\ref{lem:Q_tauIL_tau-bound}]

It is straightforward to check that the mapping $\theta:z\in[0,\infty)\mapsto \frac{\log(1+z)}{z}$ (with $\theta(0)=1$) is non-increasing. Since $q_{\tau,j}=\theta(\tau\lambda_j)$ for all $j\in\N$, one obtains
\[
\|Q_\tau\|_{\mathcal{L}(H)}=\underset{j\in\N}\sup~q_{\tau,j}=\underset{j\in\N}\sup~\theta(\tau\lambda_j)=\theta(\tau\lambda_1)=\frac{\log(1+\tau\lambda_1)}{\tau\lambda_1}.
\]
In addition, $\theta(\tau\lambda_1)\le \theta(0)=1$. This gives~\eqref{eq:Q_tau-bound}. To prove~\eqref{eq:IL_tau-gap}, it suffices to check that
\[
\underset{\tau\in(0,\tau_0)}\inf~\underset{j\in\N}\inf~\lambda_{\tau,j}=\underset{\tau\in(0,\tau_0)}\inf~\lambda_{\tau,1}=\lambda_1\underset{\tau\in(0,\tau_0)}\inf~\theta(\tau\lambda_1)=\lambda_1\theta(\tau_0\lambda_1)=\lambda_{\tau_0,1}.
\]
Let us now prove the properties of the semigroup $\bigl(e^{-t\IL_\tau}\bigr)_{t\ge 0}$. The proof of~\eqref{eq:IL_tau-bound} is straightforward: for all $j\in\N$ and all $\tau\in(0,\tau_0)$, one has $\lambda_{\tau,j}\ge 0$. In order to prove the smoothing inequalities, recall the notation $t_n=n\tau$, and observe that
\[
e^{-t_n\IL_\tau}=\bigl(e^{-\tau\IL_\tau}\bigr)^{n}=(I+\tau\IL)^{-n}
\]
for all $n\in\N$. It is straightforward to check that for all $\tau\in(0,\tau_0)$ and all $\alpha\in[0,1]$, one has
\[
\underset{n\in\N}\sup~t_n^\alpha\|\IL^\alpha(I+\tau\IL)^{-n}\|_{\mathcal{L}(H)}=\underset{n\in\N}\sup~\underset{j\in\N}\sup~ \frac{t_n^{\alpha}\lambda_j^\alpha}{(1+\tau\lambda_j)^n}\le \underset{n\in\N}\sup~\underset{z\in(0,\infty)}\sup~\frac{(nz)^{\alpha}}{(1+z)^{n}}\le 1.
\]
The smoothing inequality~\eqref{eq:IL_tau-smoothing} is then obtained as follows: for all $t\in(\tau,\infty)$, let $n\ge 1$ be the unique integer such that $t\in [t_n,t_{n+1})$. Then one has
\[
\|\IL^\alpha e^{-t\IL_\tau}\|_{\mathcal{L}(H)}\le \|\IL^\alpha e^{-t_n\IL_\tau}\|_{\mathcal{L}(H)}\le t_n^{-\alpha}\le (t-\tau)^{-\alpha},
\]
using the inequality $t_n=t_{n+1}-\tau>t-\tau$.

To prove the second smoothing inequality~\eqref{eq:IL_tau-smoothing-bis}, it suffices to check that for all $\tau\in(0,\tau_0)$ and all $t\in(0,\tau)$, one has
\begin{align*}
\tau^\alpha \|\IL^\alpha e^{-t\IL_\tau}Q_\tau\|_{\mathcal{L}(H)}&=\underset{j\in\N}\sup~\tau^\alpha \lambda_{j}^\alpha e^{-t\lambda_{\tau,j}}q_{\tau,j}\\
&\le \underset{j\in\N}\sup~(\tau\lambda_j)^{\alpha}\frac{\log(1+\tau\lambda_j)}{\tau\lambda_j}\le \underset{z\in(0,\infty)}\sup~\frac{\log(1+z)}{z^{1-\alpha}}<\infty.
\end{align*}
The inequality~\eqref{eq:IL_tau-IL_tauQ_tau1/2} is proved as follows: for all $\tau\in(0,\tau_0)$, one has
\begin{align*}
\|Q_\tau^{-\frac12}e^{-\tau\IL_\tau}\|_{\mathcal{L}(H)}&=\underset{j\in\N}\sup~\frac{e^{-\tau\lambda_{\tau,j}}}{q_{\tau,j}^{\frac12}}\\
&=\underset{j\in\N}\sup~\frac{(\tau\lambda_j)^{\frac12}}{(1+\tau\lambda_j)\sqrt{\log(1+\tau\lambda_j)}}&\le \underset{z\in(0,\infty)}\sup~\frac{z^{\frac12}}{(1+z)\sqrt{\log(1+z)}}<\infty.
\end{align*}
Finally, the inequality~\eqref{eq:IL_tau-increment} is proved as follows: for all $\alpha\in[0,1]$, all $\tau\in(0,\tau_0)$ and all $t\in(0,\infty)$, one has
\[
t^{-\alpha}\|\IL^{-\alpha}(I-e^{-t\IL_\tau})\|_{\mathcal{L}(H)}=\underset{j\in\N}\sup~\frac{1-e^{-t\lambda_{\tau,j}}}{(t\lambda_j)^{\alpha}}\le\underset{j\in\N}\sup~\frac{\lambda_{\tau,j}^\alpha}{\lambda_j^\alpha}=\underset{j\in\N}\sup~\theta(\tau\lambda_j)^\alpha\le 1.
\]
This concludes the proof of Lemma~\ref{lem:Q_tauIL_tau-bound}.
\end{proof}

\begin{proof}[Proof of Lemma~\ref{lem:Q_tauIL_tau-error}]

Note that
\[
\underset{z\in(0,1)}\sup~\frac{|\log(1+z)-z|}{z^2}<\infty,
\]
therefore for all $\alpha\in[0,1]$, one has
\begin{equation}\label{eq:boundproofQILerror}
C_\alpha=\underset{z\in(0,\infty)}\sup~z^{-\alpha}|\frac{\log(1+z)}{z}-1|<\infty.
\end{equation}
As a consequence, for all $\tau\in(0,\tau_0)$, one has
\[
\tau^{-\alpha}\|\IL^{-\alpha}(Q_\tau-I)\|_{\mathcal{L}(H)}=\underset{j\in\N}\sup~(\tau\lambda_j)^{-\alpha}|\frac{\log(1+\tau\lambda_j)}{\tau\lambda_j}-1|\le C_\alpha.
\]
This gives~\eqref{eq:Q_tau-error}.

The inequality~\eqref{eq:IL_tau-error} is then a straightforward consequence of the equality $\IL_\tau-\IL=\IL(Q_\tau-I)$.

This proof of Lemma~\ref{lem:Q_tauIL_tau-error} is thus completed.
\end{proof}

\subsection{Properties of the modified stochastic convolution}\label{sec:auxiliary-convolution}

For all $\tau\in(0,\tau_0)$ and all $t\ge 0$, set
\begin{equation}\label{eq:modifiedStochasticConvolution}
\IW_\tau(t)=\int_0^t e^{-(t-s)\IL_\tau}Q_\tau^{\frac12}dW(s).
\end{equation}
In the sequel, the associated process $\bigl(\IW_\tau(t)\bigr)_{t\ge 0}$ is referred to as the modified stochastic convolution. This process plays a crucial role in the analysis. On the one hand, this is the solution of the modified stochastic evolution equation~\eqref{eq:modifiedSPDE} with initial value $\IX_\tau(0)=0$ when $F=0$. Calling the process $\bigl(\IW_\tau(t)\bigr)_{t\ge 0}$ is consistent with the usual terminology, such that the process $\bigl(W_\tau^{\IL}(t)\bigr)_{t\ge 0}$ defined by~\eqref{eq:StochasticConvolution} is referred to as the stochastic convolution and is solution of the original stochastic evolution equation~\eqref{eq:SPDE} with initial value $X(0)=0$ and $F=0$. On the other hand, one has the following equality: for all $n\in\N_0$, $\IW_\tau(t_n)=W_n^\tau$, where $W_n^\tau$ is defined by~\eqref{eq:WNtau} and $t_n=n\tau$. That equality (in distribution) is due to the third interpretation of the modified Euler scheme, as the accelerated exponential Euler method applied to the modified equation~\eqref{eq:modifiedSPDE}.

Using the properties from Lemma~\ref{lem:Q_tauIL_tau-bound} above, the following moment bounds are obtained for the modified stochastic convolution, with bounds independent on $\tau\in(0,\tau_0)$.
\begin{lemma}\label{lem:modifiedStochasticConvolution}
For all $\tau_0\in(0,1)$ and all $\alpha\in[0,\frac14)$, one has
\begin{equation}\label{eq:lem-modifiedStochasticConvolution-bound}
\underset{\tau\in(0,\tau_0)}\sup~\underset{t\ge 0}\sup~\E[|\IW_\tau(t)|_{\alpha}^2]<\infty.
\end{equation}
\end{lemma}
Note that Lemma~\ref{lem:modifiedStochasticConvolution} is instrumental Sections~\ref{sec:auxiliary-modified} and~\ref{sec:auxiliary-scheme} below.

\begin{proof}[Proof of Lemma~\ref{lem:modifiedStochasticConvolution}]
Let $\alpha\in[0,\frac14)$. Using the It\^o isometry formula, for all $t\ge 0$ and $\tau\in(0,\tau_0)$, one has
\begin{align*}
\E[\IW_\tau(t)|_\alpha^2]&=\E[\big|\int_0^t \IL^\alpha e^{-(t-s)\IL_\tau}Q_\tau^{\frac12}dW(s)\big|^2]\\
&=\int_{0}^{t}\|\IL^\alpha e^{-(t-s)\IL_\tau}Q_\tau^{\frac12}\|_{\mathcal{L}_2(H)}^2ds\\
&=\int_0^t\sum_{j\in\N}\lambda_j^{2\alpha}e^{-2(t-s)\lambda_{\tau,j}}q_{\tau,j} ds\\
&\le \sum_{j\in\N}\frac{\lambda_j^{2\alpha}q_{\tau,j}}{2\lambda_{\tau,j}}\\
&=\sum_{j\in\N}\frac{1}{2\lambda_j^{1-2\alpha}},
\end{align*}
using the identity $q_{\tau,j}=\frac{\lambda_{\tau,j}}{\lambda_j}$, for all $j\in\N$, see~\eqref{eq:modifiedILQ-eigen}. Since $\sum_{j\in\N}\frac{1}{2\lambda_j^{1-2\alpha}}<\infty$ if and only if $\alpha\in[0,\frac14)$, this concludes the proof of the inequality~\eqref{eq:lem-modifiedStochasticConvolution-bound} and of Lemma~\ref{lem:modifiedStochasticConvolution}.
\end{proof}

\subsection{Analysis of the modified equation}\label{sec:auxiliary-modified}

Using Lemma~\ref{lem:Q_tauIL_tau-bound}, it is now straightforward to justify the well-posedness of the modified stochastic evolution equation~\eqref{eq:modifiedSPDE}, and to prove the existence and uniqueness of an invariant distribution $\mu_{\tau,\infty}$ when the ergodicity condition (Assumption~\ref{ass:ergo}) is satisfied.
\begin{propo}\label{propo:modified-wellposed-bound}
Let Assumptions~\ref{ass:Lambda} and~\ref{ass:F} be satisfied, and let the linear operators $\IL_\tau$ and $Q_\tau$ be defined by~\eqref{eq:modifiedILQ}, for all $\tau\in(0,\tau_0)$.

For any initial value $x_0\in H$, the modified stochastic evolution equation~\eqref{eq:modifiedSPDE} admits a unique mild solution $\bigl(\IX_\tau(t)\bigr)_{t\ge 0}$ satisfying~\eqref{eq:modifiedSPDE-mild}, with $\IX_\tau(0)=x_0$. In addition, for all $T\in(0,\infty)$ and all $\alpha\in[0,\frac14)$, one has
\begin{equation}\label{eq:modified-wellposed-bound}
\underset{x_0\in H^\alpha}\sup~\underset{\tau\in(0,\tau_0)}\sup~\underset{t\in[0,T]}\sup~\frac{\E[|\IX_\tau(t)|_\alpha^2]}{1+|x_0|_\alpha^2}<\infty.
\end{equation}

Moreover, if Assumption~\ref{ass:ergo} is satisfied, the modified stochastic evolution equation~\eqref{eq:modifiedSPDE} admits a unique invariant distribution $\mu_{\tau,\infty}$, which satisfies
\begin{equation}\label{eq:modified-wellposed-invar}
\underset{\tau\in(0,\tau_0)}\sup~\int |x|_\alpha^2 d\mu_{\tau,\infty}<\infty.
\end{equation}
for all $\alpha\in[0,\frac14)$. In addition, if Assumption~\ref{ass:ergo} is satisfied, the moment bound~\eqref{eq:modified-wellposed-bound} is uniform with respect to $T\in(0,\infty)$: one has
\begin{equation}\label{eq:modified-wellposed-bound-ergo}
\underset{x_0\in H^\alpha}\sup~\underset{\tau\in(0,\tau_0)}\sup~\underset{t\ge 0}\sup~\frac{\E[|\IX_\tau(t)|_\alpha^2]}{1+|x_0|_\alpha^2}<\infty.
\end{equation}

Finally, there exists $C\in(0,\infty)$, such that for all functions $\varphi:H\to\R$ of class $\mathcal{C}^1$, for all $x_0\in H$, for all $\tau\in(0,\tau_0)$ and for all $t\ge 0$, one has
\begin{equation}\label{eq:modified-wellposed-ergo}
\big|\E[\varphi(\IX_\tau(t))]-\int\varphi d\mu_{\tau,\infty}\big|\le C\vvvert\varphi\vvvert_1 e^{-\kappa t}(1+|x_0|),
\end{equation}
with $\kappa=\frac{\log(1+\tau_0\lambda_1)}{\tau_0\lambda_1}(\lambda_1-\Lf)\in (0,\lambda_1-\Lf)$.
\end{propo}
Note that Proposition~\ref{propo:modified-wellposed-bound} is a refinement of Proposition~\ref{propo:modified-wellposed}, with bounds which are uniform with respect to $\tau\in(0,\tau_0)$.

\begin{proof}[Proof of Proposition~\ref{propo:modified-wellposed-bound}]
First, the existence and uniqueness of the mild solution $\bigl(\IX_\tau(t)\bigr)_{t\ge 0}$ satisfying~\eqref{eq:modifiedSPDE-mild} follows from a straightforward fixed point argument, which is omitted (see the discussion before the statement of Proposition~\ref{propo:modified-wellposed}). Let us now establish the properties of the mild solution.

Assume first that $\alpha=0$. The mild formulation~\eqref{eq:modifiedSPDE-mild} can be written as
\[
\IX_\tau(t)=e^{-t\IL_\tau}x_0+\int_{0}^{t}e^{-(t-s)\IL_\tau}Q_\tau F(\IX_\tau(s))ds+\IW_\tau(t),
\]
where $\IW_\tau(t)$ is the modified stochastic convolution given by~\eqref{eq:modifiedStochasticConvolution}.

Owing to the global Lipschitz continuity of $F$ and to the inequalities~\eqref{eq:IL_tau-bound} and~\eqref{eq:Q_tau-bound}, the application of Minkowskii's inequality and of the It\^o isometry formula yields
\[
\bigl(\E[|\IX_\tau(t)|^2]\bigr)^{\frac12}\le |x_0|^2+C\int_0^t\bigl(\E[|\IX_\tau(s)|^2]\bigr)^{\frac12}ds+\bigl(\E[|\IW_\tau(t)|^2]\bigr)^{\frac12},
\]
where $C$ does not depend on $\tau\in(0,\tau_0)$. Using the moment bound~\eqref{eq:lem-modifiedStochasticConvolution-bound} from Lemma~\ref{lem:modifiedStochasticConvolution} and applying Gronwall's lemma then gives
\[
\underset{x_0\in H}\sup~\underset{\tau\in(0,\tau_0)}\sup~\underset{t\in[0,T]}\sup~\frac{\E[|\IX_\tau(t)|^2]}{1+|x_0|^2}<\infty.
\]
When $\alpha\in(0,\frac14)$, one obtains the inequality
\[
\bigl(\E[|\IX_\tau(t)|_\alpha^2]\bigr)^{\frac12}\le |x_0|_\alpha^2+C\int_0^t \|\IL^\alpha e^{-(t-s)\IL_\tau}Q_\tau\|_{\mathcal{L}(H)}\bigl(\E[|\IX_\tau(s)|^2]\bigr)^{\frac12}ds+\bigl(\E[|\IW_\tau(t)|_\alpha^2]\bigr)^{\frac12}.
\]
Using the smoothing inequalities~\eqref{eq:IL_tau-smoothing} and~\eqref{eq:IL_tau-smoothing-bis}, one obtains for all $t\in[0,T]$
\begin{align*}
\int_0^t \|\IL^\alpha e^{-s\IL_\tau}Q_\tau\|_{\mathcal{L}(H)}ds&\le \int_0^{\tau} \|\IL^\alpha e^{-s\IL_\tau}Q_\tau\|_{\mathcal{L}(H)}ds+\mathds{1}_{t>\tau}\int_{\tau}^t \|\IL^\alpha e^{-s\IL_\tau}Q_\tau\|_{\mathcal{L}(H)}ds\\
&\le C_\alpha\tau^{1-\alpha}+C_\alpha\int_0^{T}s^{-\alpha}ds\le C_\alpha(T),
\end{align*}
where $C_\alpha(T)\in(0,\infty)$ for all $\alpha\in[0,1)$. Using the moment bound~\eqref{eq:lem-modifiedStochasticConvolution-bound} with $\alpha\in(0,\frac14)$ and the moment bound for $\IX_\tau(t)$ when $\alpha=0$ above then concludes the proof of the inequality~\eqref{eq:modified-wellposed-bound}.

Let us now assume that Assumption~\ref{ass:ergo} is satisfied. Let $x_0^1\in H$ and $x_0^2\in H$ be two arbitrary initial values, and introduce the auxiliary processes defined by
\begin{align*}
\IX_\tau^i(t)&=e^{-t\IL_\tau}x_0^i+\int_{0}^{t}e^{-(t-s)\IL_\tau}Q_\tau F(\IX_\tau^i(s))ds+\IW_\tau(t)\\
\IY_{\tau}^i(t)&=\IX_{\tau}^i(t)-\IW_\tau(t),
\end{align*}
for $i=1,2$. Observe that $\IX_\tau^2(t)-\IX_\tau^1(t)=\IY_\tau^2(t)-\IY_\tau^1(t)$ (noise is additive and the two processes are driven by the same Wiener process $W$), and one has for all $t\ge 0$
\begin{align*}
\frac12\frac{d|\IY_{\tau}^2-\IY_\tau^1(t)|^2}{dt}&=\langle \IY_{\tau}^2(t)-\IY_\tau^1(t),\frac{d\IY_{\tau}^2(t)-\IY_\tau^1(t)}{dt}\rangle\\
&=\langle \IY_{\tau}^2(t)-\IY_\tau^1(t),-\IL_\tau (\IY_{\tau}^2(t)-\IY_\tau^1(t)\rangle\\
&+\langle \IY_{\tau}^2(t)-\IY_\tau^1(t),Q_{\tau}\bigl(F(\IY_\tau^2(t))-F(IY_\tau^1(t))\bigr)\rangle\\
&\le -\frac{\log(1+\tau\lambda_1)}{\tau}|\IY_\tau^2(t)-\IY_\tau^1(t)|^2+\Lf\|Q_\tau\|_{\mathcal{L}(H)}|\IY_\tau^2(t)-\IY_\tau^1(t)|\\
&\le -\frac{\log(1+\tau\lambda_1)}{\tau\lambda_1}(\lambda_1-\Lf)\|IY_\tau^2(t)-\IY_\tau^1(t)|^2,
\end{align*}
owing to the inequalities~\eqref{eq:Q_tau-bound} and~\eqref{eq:IL_tau-gap} from Lemma~\ref{lem:Q_tauIL_tau-bound}. In addition, one has
\[
\underset{\tau\in(0,\tau_0)}\inf~\frac{\log(1+\tau\lambda_1)}{\tau\lambda_1}=\frac{\log(1+\tau_0\lambda_1)}{\tau_0\lambda_1}.
\]
Applying Gronwall's lemma, one obtains the inequality
\[
|\IX_\tau^2(t)-\IX_\tau^1(t)|\le e^{-\kappa t}|x_0^2-x_0^1|
\]
with $\kappa=\frac{\log(1+\tau_0\lambda_1)}{\tau_0\lambda_1}(\lambda_1-\Lf)\in (0,\lambda_1-\Lf)$. A similar argument yields the uniform bound~\eqref{eq:modified-wellposed-bound-ergo}, the details are omitted.

Finally, the proof of the existence and uniqueness of the invariant distribution $\mu_{\tau,\infty}$, which satisfies the bound~\eqref{eq:modified-wellposed-invar}, and the proof of the inequality~\eqref{eq:modified-wellposed-ergo}, are standard and the details are omitted.

This concludes the proof of Proposition~\ref{propo:modified-wellposed-bound}.
\end{proof}

\subsection{Kolmogorov equation associated with the modified equation}\label{sec:auxiliary-kolmogorov-modified}

The objective of this section is to state and prove regularity results for the functions $u_\tau$ defined by
\begin{equation}\label{eq:utau}
u_\tau(t,x)=\E_x[\varphi(\IX_\tau(t))],
\end{equation}
for all $t\ge 0$, $x\in H$, and $\tau\in(0,\tau_0)$, where $\varphi$ is a bounded and continuous function from $H$ to $\R$. In the above definition, $\bigl(\IX_\tau(t)\bigr)_{t\ge 0}$ is the unique solution of the modified stochastic evolution equation~\eqref{eq:modifiedSPDE}, with initial value $\IX_\tau(0)=x$. To study the regularity properties of the function $u_\tau$, it is convenient to rely on the convention introduced in Section~\ref{sec:Galerkin}: recall that an auxiliary finite dimensional approximation is applied, in order to justify the regularity properties and the computations, and all the upper bounds do not depend on the auxiliary discretization parameter, which is omitted to simplify the notation.

It is convenient to introduce the family of linear operators $\bigl(\IP_{\tau,t}\bigr)_{t\ge 0}$, such that $u_\tau(t,\cdot)=\IP_{\tau,t}\varphi(\cdot)$ for all $t\ge 0$. The Markov property for the solutions of the modified stochastic evolution equation~\eqref{eq:modifiedSPDE} yields the semigroup property: for all $t,s\ge 0$ and all $\tau\in(0,\tau_0)$, one has
\begin{equation}\label{eq:Ptau}
\IP_{\tau,t+s}\varphi=\IP_{\tau,t}\bigl(\IP_{\tau,s}\varphi\bigr),
\end{equation}
for any $\varphi\in\mathcal{B}_b(H)$.

Under appropriate regularity conditions on the function $\varphi$, the function $(t,x)\in \R^+\times H\mapsto u_\tau(t,x)=\IP_{\tau,t}\varphi(x)$ is solution of the Kolmogorov equation
\begin{equation}\label{eq:IKolmogorov}
\partial_t u_\tau=\mathcal{L}_\tau u_\tau
\end{equation} 
with initial value $u_\tau(0,\cdot)=\varphi$, where the infinitesimal generator $\mathcal{L}_\tau$ of the modified stochastic evolution equation~\eqref{eq:modifiedSPDE} is defined by
\[
\mathcal{L}_\tau\phi(x)=D\phi(x).\bigl(-\IL_\tau x+Q_\tau F(x)\bigr)+\frac12\sum_{j\in\N}D^2\phi(x).(Q_\tau^{\frac12}e_j,Q_\tau^{\frac12}e_j).
\]
Giving a meaning to $\mathcal{L}_\tau\phi(x)$ above in a finite dimensional framework only requires to assume that $\phi$ is twice differentiable. However, since $\IL_\tau$ is an unbounded linear operator, obtaining a bound independent of the auxiliary finite dimensional approximation parameter requires an additional condition on $D\phi(x)$. Observe that $Q_\tau$ is an Hilbert--Schmidt linear operator for any positive $\tau$, so no condition on $D^2\phi(x)$ is needed. To obtain bounds which are also uniform with respect to the time-step size parameter $\tau\in(0,\tau_0)$, precise upper bounds on $Du_\tau(t,\cdot)$ and $Du_\tau^2(t,\cdot)$ are needed. It turns out that to prove Theorem~\ref{theo:weakinv} and Theorem~\ref{theo:weak} in Section~\ref{sec:proofs}, the second order derivative $Du_\tau^2(t,x)$ never appeats in the expressions of the weak error. Therefore we only need to state and prove suitable regularity results for the first order derivative $Du_\tau(t,x)$. Three lemmas are stated below, and combining the three results using the semigroup property~\eqref{eq:Ptau} provides additional results: see the proof of Proposition~\ref{propo:utau-regularity} in Section~\ref{sec:proofs-weakinv}.

Let us now state and prove the three lemmas concerning properties of the first order derivative $Du_\tau(t,x)$. It is challenging and crucial to obtain upper bounds which are uniform with respect to $\tau\in(0,\tau_0)$. This requires to use the properties stated in Lemma~\ref{lem:Q_tauIL_tau-bound} on the semigroup $\bigl(e^{-t\IL_\tau}\bigr)_{t\ge 0}$. The arguments used to prove the three lemmas are also employed in Section~\ref{sec:auxiliary-kolmogorov-original} below, with simpler formulations since the parameter $\tau$ does not appear there. 

Two expressions for $Du_\tau(t,x).h$ are employed below in the proofs. On the one hand, if $u_\tau(0,\cdot)=\varphi$ is of class $\mathcal{C}^1$ with bounded derivative, one has
\begin{equation}\label{eq:Dutau-1}
Du_\tau(t,x).h=\E_x[D\varphi(\IX_\tau(t)).\eta_\tau^h(t)],
\end{equation}
where $\bigl(\eta_\tau^h(t)\bigr)_{t\ge 0}$ is solution of
\begin{equation}\label{eq:etatau}
d\eta_\tau^h(t)=-\IL_\tau\eta_\tau^h(t)dt+Q_\tau DF(\IX_\tau^h(t)).\eta_\tau^h(t)dt,
\end{equation}
with initial value $\eta_\tau^h(0)=h$, see for instance~\cite[Chapter~4]{Cerrai}. On the other hand, if $u_\tau(0,\cdot)=\varphi$ is only assumed to be bounded and continuous, one has
\begin{equation}\label{eq:Dutau-0}
Du_\tau(t,x).h=\frac1t \E_x[\varphi(\IX_{\tau}(t))\int_{0}^{t}\langle Q_\tau^{-\frac12}\eta_\tau^h(s),dW(s)\rangle].
\end{equation}
The expression~\eqref{eq:Dutau-0} is given by a Bismut--Elworthy type formula, see for instance~\cite[Equation~(4.0.2)]{Cerrai} or~\cite[Lemma~7.1.3]{DPZergo}. The validity of the expressions~\ref{eq:Dutau-1} and~\eqref{eq:Dutau-0} is easily checked in the finite dimensional approximation setting (Section~\ref{sec:Galerkin}).

On the one hand, Lemmas~\ref{lem:utau-1} and~\ref{lem:utau-ergo} require $\varphi$ to be of class $\mathcal{C}^1$ and thus use the expression~\eqref{eq:Dutau-1}. On the other hand, Lemma~\ref{lem:utau-0} uses the expression~\eqref{eq:Dutau-0} since it gives an upper bound in terms of $\vvvert\varphi\vvvert_0$. In the proofs of the three auxiliary lemmas below, we only focus on proving upper bounds with constants independent of the time-step size $\tau\in(0,\tau_0)$.

\begin{lemma}\label{lem:utau-1}
Let Assumptions~\ref{ass:Lambda},~\ref{ass:F} and~\ref{ass:Fregul2} be satisfied. For all $t\ge 0$, all $\tau\in(0,\tau_0)$ and all $\varphi:H\to\R$ of class $\mathcal{C}^1$, $u_\tau(t,\cdot)$ is of class $\mathcal{C}^1$, with
\begin{equation}\label{eq:lem-utau1-0}
\vvvert u_\tau(t,\cdot)\vvvert_1\le e^{\Lf t}\vvvert\varphi\vvvert_1.
\end{equation}
Moreover, for all $\alpha\in(0,\frac12)$ and all $T\in(0,\infty)$, there exists $C_\alpha(T)\in(0,\infty)$ such that for all $\tau\in(0,\tau_0)$, all $\varphi:H\to\R$ of class $\mathcal{C}^1$, one has for all $t\in(\tau,T]$ and $x,h\in H$
\begin{equation}\label{eq:lem-utau1-alpha}
\big|Du_\tau(t,x).h\big|\le C_\alpha(T)\vvvert\varphi\vvvert_1\Bigl(\sqrt{\tau}|h|+(t-\tau)^{-\alpha}|\Lambda^{-\alpha}h|\Bigr).
\end{equation}
\end{lemma}
Note that the expression in the right-hand side of~\eqref{eq:lem-utau1-alpha} is directly linked to the smoothing inequalities~\eqref{eq:IL_tau-smoothing} and~\eqref{eq:IL_tau-smoothing-bis}, which as explained above require to treat the cases $t>\tau$ and $t\le \tau$ separately. Without loss of generality, it may be assumed that $T>\tau$. Assumption~\ref{ass:Fregul2} ensures the well-posedness of $\eta_\tau^h(t)$ and differentiability of $u_\tau(t,\cdot)$ for all $t\ge 0$.

\begin{proof}[Proof of Lemma~\ref{lem:utau-1}]
The proof of the inequality~\eqref{eq:lem-utau1-0} is straightforward: using the expression~\eqref{eq:etatau}, one has for all $t\ge 0$
\begin{align*}
\frac12\frac{d|\eta_\tau^h(t)|^2}{dt}&=\langle \eta_\tau^h(t),\frac{d\eta_\tau^h(t)}{dt}\rangle\\
&=-\langle \eta_\tau^h(t),\IL_\tau \eta_\tau^h(t)\rangle+\langle \eta_\tau^h(t),Q_\tau DF(\IX_\tau(t)).\eta_\tau^h(t)\rangle\\
&\le (\Lf-\lambda_1)\frac{\log(1+\tau\lambda_1)}{\tau\lambda_1}|\eta_\tau^h(t)|^2\\
&\le \Lf|\eta_\tau^h(t)|^2,
\end{align*}
owing to the inequalities~\eqref{eq:Q_tau-bound} and~\eqref{eq:IL_tau-gap} from Lemma~\ref{lem:Q_tauIL_tau-bound}. Applying Gronwall's lemma then gives
\[
|\eta_\tau^h(t)|\le e^{\Lf t}|h|
\]
and the inequality~\eqref{eq:lem-utau1-0} follows from the expression~\eqref{eq:Dutau-1}.

In order to prove the inequality~\eqref{eq:lem-utau1-alpha}, let us introduce an auxiliary process defined by
\begin{equation}\label{eq:etatauhat}
\hat{\eta}_\tau^h(t)=\eta_\tau^h(t)-e^{-t\IL_\tau}h
\end{equation}
for all $t\ge 0$. Using the expression~\eqref{eq:Dutau-1}, one obtains
\[
|Du_\tau(t,x).h|\le \vvvert\varphi\vvvert_1 |e^{-t\IL_\tau}h|+\vvvert\varphi\vvvert_1 \bigl(\E_x[|\hat{\eta}_\tau^h(t)|^2]\bigr)^{\frac12}.
\]
On the one hand, owing to the smoothing inequality~\eqref{eq:IL_tau-smoothing}, one has $|e^{-t\IL_\tau}h|\le C_\alpha (t-\tau)^{-\alpha}|\IL^{-\alpha}h|$ for all $t>\tau$. On the other hand, the auxiliary process $\bigl(\hat{\eta}_\tau^h(t)\bigr)_{t\ge 0}$ solves the evolution equation
\begin{equation}\label{eq:etatauhat-equation}
\frac{d\hat{\eta}_\tau^h(t)}{dt}=-\IL_\tau\hat{\eta}_\tau^h(t)+Q_\tau DF(\IX_\tau(t)).\hat{\eta}_\tau^h(t)+Q_\tau DF(\IX_\tau(t)).e^{-t\IL_\tau}h,
\end{equation}
with initial value $\hat{\eta}_\tau^h(0)=0$. Using the same arguments as above, one then obtains the inequality
\[
\frac12\frac{d|\hat{\eta}_\tau^h(t)|^2}{dt}\le \Lf|\hat{\eta}_\tau^h(t)|^2+\Lf |\hat{\eta}_\tau^h(t)| |e^{-t\IL_\tau}h|.
\]
Using Young's inequality and Gronwall's lemma, one thus obtains
\begin{align*}
\bigl(\E_x[|\hat{\eta}_\tau^h(t)|^2]\bigr)^{\frac12}&\le C(T)\int_0^T |e^{-t\IL_\tau}h|^2 dt=C(T)\int_0^\tau|e^{-t\IL_\tau}h|^2dt+C(T)\int_{\tau}^{T}|e^{-t\IL_\tau}h|^2dt\\
&\le C(T)\tau|h|^2+\int_\tau^T(t-\tau)^{-2\alpha}dt|\IL^{-\alpha}h|^2\\
&\le C_\alpha(T)\bigl(\sqrt{\tau}|h|+|\IL^{-\alpha}h|\bigr)^2.
\end{align*}
using the inequality~\eqref{eq:IL_tau-bound}, the smoothing inequality~\eqref{eq:IL_tau-smoothing}, and the condition $\alpha<\frac12$. Gathering the estimates then concludes the proof of the inequality~\eqref{eq:lem-utau1-alpha} and of Lemma~\ref{lem:utau-1}.
\end{proof}

\begin{lemma}\label{lem:utau-ergo}
Let Assumptions~\ref{ass:Lambda},~\ref{ass:F} and~\ref{ass:Fregul2} be satisfied. For all $\tau\in(0,\tau_0)$ and all $\varphi:H\to\R$ of class $\mathcal{C}^1$, one has, for all $t\ge 0$,
\begin{equation}\label{eq:lem-utau-ergo}
\vvvert u_\tau(t,\cdot)\vvvert_1\le e^{-\kappa t} \vvvert\varphi\vvvert_1
\end{equation}
with $\kappa=\frac{\log(1+\tau_0\lambda_1)}{\tau_0\lambda_1}(\lambda_1-\Lf)>0$.
\end{lemma}

\begin{proof}[Proof of Lemma~\ref{lem:utau-ergo}]
Owing to the computations from the proof of Lemma~\ref{lem:utau-1}, and using the lower bound
\[
\underset{\tau\in(0,\tau_0)}\inf~\frac{\log(1+\tau\lambda_1)}{\tau\lambda_1}=\frac{\log(1+\tau_0\lambda_1)}{\tau_0\lambda_1},
\]
one obtains the inequality
\[
\frac12\frac{d|\eta_\tau^h(t)|^2}{dt}\le -\kappa|\eta_\tau^h(t)|^2,
\]
with the value of $\kappa$ given above. Therefore $|\eta_\tau^h(t)|\le e^{-\kappa t}|h|$ for all $t\ge 0$, and then the inequality~\eqref{eq:lem-utau-ergo} is a straightforward consequence of the expression~\eqref{eq:Dutau-1}. This concludes the proof of Lemma~\ref{lem:utau-ergo}.
\end{proof}

\begin{lemma}\label{lem:utau-0}
Let Assumptions~\ref{ass:Lambda},~\ref{ass:F},~\ref{ass:ergo} and~\ref{ass:Fregul2} be satisfied. For all $\tau_0\in(0,1)$, there exists $C\in(0,\infty)$ such that for all $\tau\in(0,\tau_0)$, all $t\in(\tau,\infty)$, and all $\varphi:H\to\R$ of class $\mathcal{C}^0$, $u_\tau(t,\cdot)$ is of class $\mathcal{C}^1$, and one has, for all $t\in(\tau,\infty)$ and all $x,h\in H$,
\begin{equation}\label{eq:lem-utau0}
\vvvert u_\tau(t,\cdot)\vvvert_1\le \frac{C}{(t-\tau)^{\frac12}}\vvvert\varphi\vvvert_0.
\end{equation}

Moreover, for all $t\in(0,\infty)$, all $x\in H$ and all $h\in D(Q_\tau^{-\frac12})$, one has
\begin{equation}\label{eq:utau0-bad}
|D\IP_{\tau,t}\varphi(x).h|=|Du_\tau(t,x).h|\le \frac{\vvvert\varphi\vvvert_0}{\sqrt{t}}|Q_\tau^{-\frac12}h|.
\end{equation}
\end{lemma}
Note that Lemma~\ref{lem:utau-0} is a crucial ingredient to establish the main result of this work, Theorem~\ref{theo:weakinv}, where the error between the invariant distributions $\mu_\infty^\tau$ and $\mu_\infty=\mu_\star$ (when Assumption~\ref{ass:gradient} is satisfied) is considered using the total variation distance $d_{\rm TV}$, see Equation~\eqref{eq:theo-weakinv_dTVinvar}. Assumption~\ref{ass:ergo} may be removed, the constant $C$ would then depend on the final time $T\in(0,\infty)$ with a condition $t\in(\tau,T]$. The result used to prove Theorem~\ref{theo:weakinv} is Proposition~\ref{propo:utau-regularity}, which is proved below by a combination of Lemmas~\ref{lem:utau-1},~\ref{lem:utau-ergo} and~\ref{lem:utau-0}, which all serve different purposes: having estimates in terms of $|\IL^{-\alpha}h|$ with $\alpha\in(0,\frac12)$, in terms of $\exp(-\kappa t)$ with $\kappa>0$ and in terms of $\vvvert\varphi\vvvert_0$.

The inequality~\eqref{eq:utau0-bad} is a first step to prove~\eqref{eq:lem-utau0} and is also employed to deal with some of the error terms in the weak error analysis below. Note that $Q_\tau^{-\frac12}$ is an unbounded linear operator, therefore to obtain the inequality~\eqref{eq:lem-utau0} (which requires a bound with $|h|$ on the right-hand side) one needs an additional argument. Like the proof of Lemma~\ref{lem:utau-1}, one needs to be careful to obtain constants which are independent of $\tau\in(0,\tau_0)$, this explains why the quantity $(t-\tau)^{-1/2}$ appears in the right-hand side of~\eqref{eq:lem-utau0}.

\begin{proof}[Proof of Lemma~\ref{lem:utau-0}]
Recall the identity $Q_\tau^{-1}\IL_\tau=\IL$. Then, for all $t\ge 0$ and all $h\in H$, one obtains
\begin{align*}
\frac12\frac{d|Q_\tau^{-1/2}\eta_\tau^h(t)|^2}{dt}&=-\langle \IL \eta_\tau^h(t),\eta^h(t)\rangle+\langle DF(\IX_{\tau,t}).\eta_\tau^h(t),\eta_\tau^h(t)\rangle\\
&\le -\lambda_1|\eta_\tau^h(t)|^2+\Lf|\eta_\tau^h(t)|^2\\
&\le 0,
\end{align*}
using Assumption~\ref{ass:ergo}. Therefore, for all $\tau\in(0,\tau_0)$, $t\ge 0$ and $h\in H$, one has
\[
|Q_\tau^{-\frac12}\eta_\tau^h(t)|\le |Q_\tau^{-\frac12}h|.
\]
Using It\^o's isometry formula and the expression~\eqref{eq:Dutau-0} of $Du_\tau(t,x).h$, one then obtains the upper bound~\eqref{eq:utau0-bad}, for all $t\in(0,\infty)$. As explained above, an additional argument is required to prove~\eqref{eq:lem-utau0} and remove the unbounded operator $Q_\tau^{-\frac12}$ on the right-hand side.

Let now $t>\tau$. The semigroup property~\eqref{eq:Ptau} yields the identity
\[
u_\tau(t,\cdot)=\IP_{\tau,\tau}\bigl(\IP_{\tau,t-\tau}\varphi\bigr),
\]
and using the expression~\eqref{eq:Dutau-1} gives the equality
\[
Du_\tau(t,x).h=\E_x[D\IP_{\tau,t-\tau}\varphi(\IX_\tau(t)).\eta_\tau^h(\tau)].
\]
Applying the inequality~\eqref{eq:utau0-bad} then gives
\[
|Du_\tau(t,x).h|\le \frac{\vvvert\varphi\vvvert_0}{\sqrt{t-\tau}}\E_x[|Q_\tau^{-\frac12}\eta_\tau^h(\tau)|].
\]
It remains to check that there exists $C\in(0,\infty)$ such that $|Q_\tau^{-\frac12}\eta_\tau^h(\tau)|\le C|h|^2$ for all $\tau\in(0,\tau_0)$ and $h\in H$. Recall that $\eta_\tau^h(\tau)=e^{-\tau\IL_\tau}h+\hat{\eta}_\tau^h(\tau)$, see~\eqref{eq:etatauhat}. On the one hand, one has $|Q_\tau^{-\frac12}e^{-\tau\IL_\tau}h|\le C|h|$ for all $\tau\in(0,\tau_0)$, owing to the inequality~\eqref{eq:IL_tau-IL_tauQ_tau1/2} from Lemma~\ref{lem:Q_tauIL_tau-bound}. On the other hand, using~\eqref{eq:etatauhat-equation} and the identity $Q_\tau^{-1}\IL_\tau=\IL$, one obtains for all $t\ge 0$
\begin{align*}
\frac12\frac{d|Q_\tau^{-\frac12}\hat{\eta}_\tau^h(t)|^2}{dt}&=-\langle \IL\hat{\eta}_\tau^h(t),\hat{\eta}_\tau^h(t)\rangle+\langle DF(\IX_\tau(t)).\hat{\eta}_\tau^h(t),\hat{\eta}_\tau^h(t)\rangle+\langle DF(\IX_\tau(t)).e^{-t\IL_\tau}h,\hat{\eta}_{\tau}^h(t)\rangle\\
&\le -(\lambda_1-\Lf)|\hat{\eta}_\tau^h(t)|^2+\Lf|e^{-t\IL_\tau}h||\hat{\eta}_\tau^h(t)|\\
&\le C|h|^2,
\end{align*}
using the inequality~\eqref{eq:IL_tau-bound} and Young's inequality. Applying Gronwall's lemma then gives the upper bound $|\hat{\eta}_\tau^h(\tau)|\le C|h|$ for all $\tau\in(0,\tau_0)$.

Finally, one obtains the upper bound
\[
|Du_\tau(t,x).h|\le \frac{C\vvvert\varphi\vvvert_0}{\sqrt{t-\tau}}|h|,
\]
which concludes the proof of Lemma~\ref{lem:utau-0}.
\end{proof}

\subsection{Properties of the modified Euler scheme}\label{sec:auxiliary-scheme}

Recall that the modified Euler scheme can be interpreted in three different ways: see~\eqref{eq:scheme},~\eqref{eq:scheme2} and~\eqref{eq:scheme-IX} in Section~\ref{sec:scheme}. Let us first provide moment bounds, using the second formulation~\eqref{eq:scheme2} of the scheme.
\begin{lemma}\label{lem:scheme-bound}
Let Assumptions~\ref{ass:Lambda} and~\ref{ass:F} be satisfied. For all $\tau_0\in(0,1)$ and $T\in(0,\infty)$, one has
\begin{equation}\label{eq:lem-scheme-bound}
\underset{x_0\in H}\sup~\underset{\tau\in(0,\tau_0)}\sup~\underset{n\in\N;~n\tau\le T}\sup~\frac{\E[|X_n^\tau|^2]}{1+|x_0|^2}<\infty.
\end{equation}
Moreover, if Assumption~\ref{ass:ergo} is satisfied, then for all $\tau_0\in(0,1)$ and $T\in(0,\infty)$, one has
\begin{equation}\label{eq:lem-scheme-bound-ergo}
\underset{x_0\in H}\sup~\underset{\tau\in(0,\tau_0)}\sup~\underset{n\in\N}\sup~\frac{\E[|X_n^\tau|^2]}{1+|x_0|^2}<\infty.
\end{equation}
\end{lemma}

\begin{proof}[Proof of Lemma~\ref{lem:scheme-bound}]
Introduce the auxiliary random variables
\[
Y_n^\tau=X_n^\tau-W_n^\tau,
\]
where $W_n^\tau$ is defined by~\eqref{eq:WNtau}, and recall that $W_n^\tau=\IW_\tau(t_n)$. Therefore, $X_n^\tau=Y_n^\tau+W_n^\tau$ for all $n\in\N_0$.

On the one hand, the moment bound~\eqref{eq:lem-modifiedStochasticConvolution-bound} from Lemma~\ref{lem:modifiedStochasticConvolution} gives
\[
\underset{\tau\in(0,\tau_0)}\sup~\underset{n\in\N_0}\sup~\E[|W_n^\tau|^2]<\infty.
\]
On the other hand, for all $n\in\N_0$, one has the identity
\[
Y_{n+1}^\tau=\IA_\tau Y_n^\tau+\tau \IA_\tau F(Y_n^\tau+W_n^\tau),
\]
Writing $F(Y_n^\tau+W_n^\tau)=F(Y_n^\tau+W_n^\tau)-F(W_n^\tau)+F(W_n^\tau)$, and using the inequality $\|\IA_\tau\|_{\mathcal{L}(H)}\le (1+\tau\lambda_1)^{-1}$ and the Lipschitz continuity of $F$ (Assumption~\ref{ass:F}), one obtains
\begin{align*}
\bigl(\E[|Y_{n+1}^\tau|]\bigr)^{\frac12}&\le \frac{1+\tau\Lf}{1+\tau\lambda_1}\bigl(\E[|Y_n^\tau|^2]\bigr)^{\frac12}+\frac{\tau}{1+\tau\lambda_1}\bigl(\E[|F(W_n^\tau)|^2]\bigr)^{\frac12}\\
&\le \frac{1+\tau\Lf}{1+\tau\lambda_1}\bigl(\E[|Y_n^\tau|^2]\bigr)^{\frac12}+C\tau.
\end{align*}
It is then straightforward to conclude the proof of the first inequality~\eqref{eq:lem-scheme-bound}. When Assumption~\ref{ass:ergo} is satisfied, one has $\Lf<\lambda_1$, and one obtains for all $n\in\N$ the upper bound
\[
\bigl(\E[|Y_n^\tau|^2]\bigr)^{\frac12}\le \bigl(\frac{1+\tau\Lf}{1+\tau\lambda_1}\bigr)^n|x_0|^2+C\tau\sum_{k=0}^{\infty}\bigl(\frac{1+\tau\Lf}{1+\tau\lambda_1}\bigr)^k\le |x_0|^2+C\frac{1+\tau\lambda_1}{\lambda_1-\Lf}.
\]
This gives the second inequality~\eqref{eq:lem-scheme-bound-ergo} and concludes the proof of Lemma~\ref{lem:scheme-bound}.
\end{proof}

In this sequel, to perform the weak error analysis, it is convenient to exploit the third interpretation of the scheme given in Section~\ref{sec:scheme}: the modified Euler scheme is obtained by the application of the (accelerated) exponential Euler scheme to the modified stochastic evolution equation~\eqref{eq:modifiedSPDE}, which gives $X_n^\tau=\IX_{\tau,n}$ for all $n\in\N_0$, with
\[
\IX_{\tau,n+1}=e^{-\tau\IL_\tau}\IX_{\tau,n}+\IL_\tau^{-1}(I-e^{-\tau\IL_\tau})Q_\tau F(\IX_{\tau,n})+\int_{t_n}^{t_{n+1}}e^{-(t_{n+1}-s)\IL_\tau}Q_\tau^{\frac12}dW(s).
\]

One of the main ingredients of the proof of Theorems~\ref{theo:weakinv} and~\ref{theo:weak} is the introduction of the auxiliary process $\bigl(\tilde{\IX}_\tau(t)\bigr)_{t\ge 0}$, defined as follows: for all $\tau\in(0,\tau_0)$ and all $t\ge 0$,
\begin{equation}\label{eq:tildeIX}
\tilde{\IX}_\tau(t)=e^{-t\IL_\tau}x_0+\int_{0}^{t}e^{-(t-s)\IL_\tau}Q_\tau F(\IX_{\tau,\ell(s)})ds+\IW_\tau(t),
\end{equation}
where $\IW_\tau(t)=\int_0^t e^{-(t-s)\IL_\tau}Q_{\tau}^{\frac12}dW(s)$ is given by~\eqref{eq:modifiedStochasticConvolution}, $\ell(s)=n$ if $t_n\le s<t_{n+1}$, with $t_n=n\tau$. For every $n\in\N_0$ and all $t\in[t_n,t_{n+1}]$, one has
\begin{equation}\label{eq:dtildeIX}
d\tilde{\IX}_\tau(t)=-\IL_\tau\tilde{\IX}_\tau(t)dt+Q_\tau F(\IX_{\tau,n})dt+Q_\tau^{\frac12}dW(t).
\end{equation}
Finally, note that by construction of the auxiliary process, one has the equality $\tilde{\IX}_\tau(t_n)=\IX_{\tau,n}$ for all $n\in\N_0$.

Lemma~\ref{lem:tildeIX} below gives the main properties of the auxiliary process $\bigl(\tilde{\IX}_\tau(t)\bigr)_{t\ge 0}$ which will be used in Section~\ref{sec:proofs} below.
\begin{lemma}\label{lem:tildeIX}
Let Assumptions~\ref{ass:Lambda} and~\ref{ass:F} be satisfied.

For all $T\in(0,\infty)$, $\alpha\in[0,\frac14)$ and $\tau_0\in(0,1)$, one has
\begin{equation}\label{eq:tildeIX-bound}
\underset{x_0\in H^\alpha}\sup~\underset{\tau\in(0,\tau_0)}\sup~\underset{t\in[0,T]}\sup~\frac{\E[|\tilde{\IX}_\tau(t)|_\alpha^2]}{1+|x_0|_\alpha^2}<\infty
\end{equation}
and
\begin{equation}\label{eq:tildeIX-increment}
\underset{x_0\in H^\alpha}\sup~\underset{\tau\in(0,\tau_0)}\sup~\underset{t\in[\tau,T]}\sup~\frac{\E[\big|\IL^{-\alpha}\bigl(\tilde{\IX}_\tau(t)-\tilde{\IX}_\tau(t_{\ell(t)})\bigr)\big|^2]}{\tau^{2\alpha}(1+|x_0|_\alpha^2)}<\infty.
\end{equation}

Moreover, if Assumption~\ref{ass:ergo} is satisfied, then for all $\tau_0\in(0,1)$ and $\alpha\in[0,\frac14)$, one has
\begin{equation}\label{eq:tildeIX-bound-ergo}
\underset{x_0\in H^\alpha}\sup~\underset{\tau\in(0,\tau_0)}\sup~\underset{t\ge 0}\sup~\frac{\E[|\tilde{\IX}_\tau(t)|_\alpha^2]}{1+|x_0|_\alpha^2}<\infty
\end{equation}
and
\begin{equation}\label{eq:tildeIX-increment-ergo}
\underset{x_0\in H^\alpha}\sup~\underset{\tau\in(0,\tau_0)}\sup~\underset{t\ge\tau}\sup~\frac{\E[\big|\IL^{-\alpha}\bigl(\tilde{\IX}_\tau(t)-\tilde{\IX}_\tau(t_{\ell(t)})\bigr)\big|^2]}{\tau^{2\alpha}(1+|x_0|_\alpha^2)}<\infty.
\end{equation}
\end{lemma}

\begin{proof}[Proof of Lemma~\ref{lem:tildeIX}]
The mild formulation
\[
\tilde{\IX}_\tau(t)=e^{-t\IL_\tau}x_0+\int_0^t e^{-(t-s)\IL_\tau}Q_\tau F(X_{\tau,\ell(s)})ds+\IW_\tau(t)
\]
of the auxiliary process $\tilde{\IX}_\tau$ gives the inequality
\begin{align*}
\bigl(\E[|\tilde{\IX}_\tau(t)|_\alpha^2]\bigr)^{\frac12}&\le |x_0|_\alpha+C\int_0^t \|\IL^\alpha e^{-(t-s)\IL_\tau}Q_\tau\|_{\mathcal{L}(H)}\bigl(1+\bigl(\E[|X_{\ell(s)}|^2]\bigr)^{\frac12}\bigr)ds+C_\alpha
\end{align*}
using the Minkowskii inequality, the inequality~\eqref{eq:IL_tau-bound} and the moment bound~\eqref{eq:lem-modifiedStochasticConvolution-bound} from Lemma~\ref{lem:modifiedStochasticConvolution}. In addition, using the moment bound~\eqref{eq:lem-scheme-bound} and the smoothing inequalities~\eqref{eq:IL_tau-smoothing} and~\eqref{eq:IL_tau-smoothing-bis}, one obtains
\begin{align*}
\int_0^t \|\IL^\alpha e^{-(t-s)\IL_\tau}Q_\tau\|_{\mathcal{L}(H)}\bigl(1+\bigl(\E[|X_{\ell(s)}|^2]\bigr)^{\frac12}\bigr)ds&\le C(T)(1+|x_0|)\int_0^t \|\IL^\alpha e^{-(t-s)\IL_\tau}Q_\tau\|_{\mathcal{L}(H)}ds\\
&\le C(T)(1+|x_0|)\int_0^\tau \|\IL^\alpha e^{-s\IL_\tau}Q_\tau\|_{\mathcal{L}(H)}ds\\
&+C(T)(1+|x_0|)\mathds{1}_{t>\tau}\int_\tau^t \|\IL^\alpha e^{-s\IL_\tau}Q_\tau\|_{\mathcal{L}(H)}ds\\
&\le C_\alpha(T)(1+|x_0|).
\end{align*}
This concludes the proof of the inequality~\eqref{eq:tildeIX-bound}. Note that in particular one obtains for all $\alpha\in[0,\frac14)$ and all $T\in(0,\infty)$ the moment bound
\begin{equation}\label{eq:lem-scheme-bound-alpha}
\underset{x_0\in H^\alpha}\sup~\underset{\tau\in(0,\tau_0)}\sup~\underset{n\in\N;~n\tau\le T}\sup~\frac{\E[|X_n^\tau|_\alpha^2]}{1+|x_0|_\alpha^2}<\infty.
\end{equation}
In order to prove the inequality~\eqref{eq:tildeIX-increment}, observe that for all $n\in\N_0$ and $t\in[t_n,t_{n+1})$, one has
\[
\tilde{\IX}_{\tau}(t)-\tilde{\IX}_\tau(t_n)=\bigl(e^{-(t-t_n)\IL_\tau}-I\bigr)X_n^\tau+\int_{t_n}^{t}e^{-(t-s)\IL_\tau}Q_\tau F(X_n^\tau)ds+\int_{t_n}^{t}e^{-(t-s)\IL_\tau}Q_\tau^{\frac12}dW(s),
\]
therefore using the inequality above, the inequalities~\eqref{eq:Q_tau-bound} and \eqref{eq:IL_tau-bound}, and It\^o's isometry formula, one obtains for all $t\in[t_n,t_{n+1})$ with $t\le T$
\begin{align*}
\bigl(\E[\big|\IL^{-\alpha}\bigl(\tilde{\IX}_\tau(t)-\tilde{\IX}_\tau(t_{\ell(t)})\bigr)\big|^2]\bigr)^{\frac12}&\le C(1+|x_0|_\alpha)\|\IL^{-2\alpha}\bigl(e^{-(t-t_n)\IL_\tau}-I\bigr)\|_{\mathcal{L}(H)}+C\tau(1+|x_0|)\\
&+\bigl(\int_{t_n}^{t}\|\IL^{-\alpha}e^{-(t-s)\IL_\tau}Q_\tau^{\frac12}\|_{\mathcal{L}_2(H)}^2ds\bigr)^{\frac12}.
\end{align*}
Using the inequality~\eqref{eq:IL_tau-increment}, one has $\|\IL^{-2\alpha}\bigl(e^{-(t-t_n)\IL_\tau}-I\bigr)\|_{\mathcal{L}(H)}\le C\tau^{2\alpha}$. In addition, one has the upper bound
\begin{align*}
\int_{t_n}^{t}\|\IL^{-\alpha}e^{-(t-s)\IL_\tau}Q_\tau^{\frac12}\|_{\mathcal{L}_2(H)}^2ds&\le \int_{0}^{\tau}\|\IL^{-\alpha}e^{-s\IL_\tau}Q_\tau^{\frac12}\|_{\mathcal{L}_2(H)}^2ds\\
&\le \sum_{j\in\N}\lambda_j^{-2\alpha}\int_{0}^{\tau}e^{-2s\frac{\log(1+\tau\lambda_j)}{\tau}}ds\frac{\log(1+\tau\lambda_j)}{\lambda_j\tau}\\
&\le \sum_{j\in\N}\lambda_j^{-2\alpha-1}\bigl(1-e^{-2\tau\frac{\log(1+\tau\lambda_j)}{\tau}}\bigr)\\
&\le \sum_{j\in\N}\lambda_j^{-2\alpha-1}\bigl(1-\frac{1}{(1+\lambda_j\tau)^2}\bigr)\\
&\le \sum_{j\in\N}\tau\lambda_j^{-2\alpha}\frac{2+\lambda_j\tau}{(1+\lambda_j\tau)^2}\\
&\le \tau^{4\alpha}\sum_{j\in\N}\frac{(\tau\lambda_j)^{1-4\alpha}}{(1+\tau\lambda_j)^2}\lambda_j^{2\alpha-1}\\
&\le \tau^{4\alpha}\underset{z\in(0,\infty)}\sup~\frac{z^{1-4\alpha}(1+2z)}{(1+z)^2}\sum_{j\in\N}\lambda_j^{2\alpha-1}.
\end{align*}
Note that the inequality $\underset{z\in(0,\infty)}\sup~\frac{z^{1-4\alpha}(1+2z)}{(1+z)^2}\sum_{j\in\N}\lambda_j^{2\alpha-1}<\infty$ holds if and only if the condition $\alpha\in(0,\frac14)$ holds. Gathering the estimates concludes the proof of the inequality~\eqref{eq:tildeIX-increment}.

In order to prove the inequality~\eqref{eq:tildeIX-bound-ergo}, when Assumption~\ref{ass:ergo} is satisfied, it suffices to replace the use of~\eqref{eq:lem-scheme-bound} by~\eqref{eq:lem-scheme-bound-ergo}, and to use the bound
\[
\|e^{-t\IL_\tau}\|_{\mathcal{L}(H)}=e^{-t\lambda_{\tau,1}}\le e^{-t\lambda_1\frac{\log(1+\tau_0\lambda_1)}{\tau_0\lambda_1}},
\]
to obtain bounds which do not depend on the final time $T\in(0,\infty)$. The inequality~\eqref{eq:tildeIX-increment-ergo} is obtained using a similar argument. The details are omitted. This concludes the proof of Lemma~\ref{lem:tildeIX}.
\end{proof}

To conclude this section, it remains to provide the proof of Proposition~\ref{propo:invar-scheme}. Like for the proof of Proposition~\ref{propo:modified-wellposed-bound}, some standard details are omitted.
\begin{proof}[Proof of Proposition~\ref{propo:invar-scheme}]
The objective is to prove the existence and the uniqueness of the invariant distribution $\mu_{\tau,\infty}$ of the modified Euler scheme, when Assumption~\ref{ass:ergo} is satisfied. In this proof, we use the second interpretation~\eqref{eq:scheme2} of the scheme. Note that the inequality~\eqref{eq:lem-scheme-bound-alpha} obtained above ensures the existence of an invariant distribution by the Krylov--Bogoliubov criterion, and also gives the bound~\eqref{eq:invar-scheme-bound}. It thus suffices to check the uniqueness of the invariant distribution and the inequality~\eqref{eq:invar-scheme-ergo}.

Let $x_0^1\in H$ and $x_0^2\in H$ be two arbitrary initial values, and introduce the processes defined by
\[
X_{n+1}^{\tau,i}=\IA_\tau X_n^{\tau,i}+\tau\IA_\tau F(X_n^{\tau,i})+\sqrt{\tau}\IB_\tau \Gamma_n
\]
for $i=1,2$. For all $n\in\N$, one has and
\[
X_{n}^{\tau,2}-X_{n}^{\tau,1}=\IA_\tau\bigl(X_{n-1}^{\tau,2}-X_{n-1}^{\tau,1}\bigr)+\tau\IA_\tau\bigl(F(X_{n-1}^{\tau,2})-F(X_{n-1}^{\tau,1})\bigr).
\]
Using the inequality $\|\IA_\tau\|_{\mathcal{L}(H)}$ and the definition of $\Lf$ (see Assumption~\ref{ass:F}), one obtains for all $n\in\N$
\[
|X_{n}^{\tau,2}-X_n^{\tau,1}|\le \frac{1+\tau\Lf}{1+\tau\lambda_1}|X_{n-1}^{\tau,2}-X_{n-1}^{\tau,1}|\le \bigl(\frac{1+\tau\Lf}{1+\tau\lambda_1}\bigr)^n|x_0^2-x_0^1|.
\]
When Assumption~\ref{ass:ergo} is satisfied, $\frac{1+\tau\Lf}{1+\tau\lambda_1}<1$ for all $\tau\in(0,\tau_0)$, one obtains the uniqueness of the invariant distribution $\mu_{\infty}^{\tau}$. Finally, to obtain the inequality~\eqref{eq:invar-scheme-ergo}, it suffices to use the upper bound
\[
\frac{1+\tau\Lf}{1+\tau\lambda_1}=1-\frac{\tau(\lambda_1-\Lf)}{1+\tau\lambda_1}\le \exp\bigl(-\frac{\tau(\lambda_1-\Lf)}{1+\tau\lambda_1}\bigr)\le \exp\bigl(-\frac{\tau(\lambda_1-\Lf)}{1+\tau_0\lambda_1}\bigr).
\]
This concludes the proof of Proposition~\ref{propo:invar-scheme}.
\end{proof}

\subsection{Kolmogorov equation associated with the original equation}\label{sec:auxiliary-kolmogorov-original}

The objective of this section is to state and prove regularity results for the function $u$ defined by
\begin{equation}\label{eq:u}
u(t,x)=\E_x[\varphi(X(t))],
\end{equation}
for all $t\ge 0$, $x\in H$, where $\varphi$ is a bounded and continuous function from $H$ to $\R$. In the above definition, $\bigl(X(t)\bigr)_{t\ge 0}$ is the unique solution of the stochastic evolution equation~\eqref{eq:SPDE}, with initial value $X(0)=x$. Like in Section~\ref{sec:auxiliary-kolmogorov-modified} above, to study the regularity properties of the function $u$, it is convenient to rely on the convention introduced in Section~\ref{sec:Galerkin}. An auxiliary finite dimensional approximation is applied, in order to justify the regularity properties and the computations, and all the upper bounds do not depend on the auxiliary discretization parameter, which is omitted to simplify the notation.

Note that the results in this section are not required to prove Theorem~\ref{theo:weakinv}, but they are needed to prove Theorem~\ref{theo:weak}. The regularity results are proved using arguments similar to those used in Section~\ref{sec:auxiliary-kolmogorov-modified}, however they are sometimes simpler since the additional parameter $\tau$ is absent.

It is convenient to introduce the family of linear operators $\bigl(P_{t}\bigr)_{t\ge 0}$, such that $u(t,\cdot)=P_{t}\varphi(\cdot)$ for all $t\ge 0$. The Markov property for the solutions of the stochastic evolution equation~\eqref{eq:SPDE} yields the semigroup property: for all $t,s\ge 0$, one has
\begin{equation}\label{eq:P}
P_{t+s}\varphi=P_{t}\bigl(P_{s}\varphi\bigr),
\end{equation}
for any $\varphi\in\mathcal{B}_b(H)$.

Under appropriate regularity conditions on the function $\varphi$, the function $(t,x)\in \R^+\times H\mapsto u(t,x)=P_t\varphi(x)$ is solution of the Kolmogorov equation
\begin{equation}\label{eq:Kolmogorov}
\partial_t u=\mathcal{L}u
\end{equation} 
with initial value $u(0,\cdot)=\varphi$, where the infinitesimal generator $\mathcal{L}$ of the stochastic evolution equation~\eqref{eq:SPDE} is defined by
\[
\mathcal{L}\phi(x)=D\phi(x).\bigl(-\IL x+F(x)\bigr)+\frac12\sum_{j\in\N}D^2\phi(x).(e_j,e_j).
\]
Like for $\mathcal{L}_\tau\phi(x)$ in Section~\ref{sec:auxiliary-kolmogorov-modified}, giving a meaning to $\mathcal{L}\phi(x)$ in a finite dimensional context only requires to assume that $\phi$ is of class $\mathcal{C}^2$. However, to obtain bounds which are independent of the auxiliary spatial discretization parameter, appropriate estimates are needed to deal with the terms $\D\phi(x)\cdot \IL x$ (since $\IL$ is an unbounded operator) and also $\sum_{j\in\N}D^2\phi(x).(e_j,e_j)$.

Like for $Du_\tau(t,x).h$ in Section~\ref{sec:auxiliary-kolmogorov-modified} above, two expressions for $Du(t,x).h$ are employed below in the proofs. On the one hand, if $u(0,\cdot)=\varphi$ is of class $\mathcal{C}^1$ with bounded derivative, one has
\begin{equation}\label{eq:Du-1}
Du(t,x).h=\E_x[D\varphi(X(t)).\eta^h(t)],
\end{equation}
where $\bigl(\eta^h(t)\bigr)_{t\ge 0}$ is solution of
\begin{equation}\label{eq:eta}
d\eta^h(t)=-\IL\eta^h(t)dt+DF(X^h(t)).\eta^h(t)dt,
\end{equation}
with initial value $\eta^h(0)=h$, see for instance~\cite[Chapter~4]{Cerrai}. On the other hand, if $u(0,\cdot)=\varphi$ is only assumed to be bounded and continuous, one has
\begin{equation}\label{eq:Du-0}
Du(t,x).h=\frac1t \E_x[\varphi(X(t))\int_{0}^{t}\langle \eta^h(s),dW(s)\rangle].
\end{equation}
The expression~\eqref{eq:Du-0} is given by a Bismut--Elworthy type formula, see for instance~\cite[Equation~(4.0.2)]{Cerrai} or~\cite[Lemma~7.1.3]{DPZergo}. The validity of the expressions~\ref{eq:Dutau-1} and~\eqref{eq:Dutau-0} is easily checked in the finite dimensional approximation setting (Section~\ref{sec:Galerkin}).

In addition, if $u(0,\cdot)=\varphi$ is of class $\mathcal{C}^2$ with bounded first and second order derivatives, one has
\begin{equation}\label{eq:D2u}
D^2u(t,x).(h_1,h_2)=\E_x[D\varphi(X(t)).\zeta^{h_1,h_2}(t)]+\E_x[D^2\varphi(Xt)).(\eta^{h_1}(t),\eta^{h_2}(t))]
\end{equation}
where $\bigl(\zeta^{h_1,h_2}(t)\bigr)_{t\ge 0}$ is solution of
\begin{equation}\label{eq:zeta}
d\zeta^{h_1,h_2}(t)=-\IL\zeta^{h_1,h_2}(t)dt+DF(X(t)).\zeta^{h_1,h_2}(t)dt+D^2F(X(t)).(\eta^{h_1}(t),\eta^{h_2}(t))dt,
\end{equation}
with initial value $\zeta^{h_1,h_2}(0)=0$, see for instance~\cite[Chapter~4]{Cerrai}.

Lemma~\ref{lem:u-12} requires $\varphi$ to be of class $\mathcal{C}^2$ and thus uses the expressions~\eqref{eq:Du-1} and~\eqref{eq:D2u}. Similarly, Lemma~\ref{lem:u-ergo} requires $\varphi$ to be of class $\mathcal{C}^1$ and thus uses the expression~\eqref{eq:Du-1}. Finally, Lemma~\ref{lem:u-0} uses the expression~\eqref{eq:Du-0} since it gives an upper bound in terms of $\vvvert\varphi\vvvert_0$. The proofs are standard, details are provided for completeness and for comparison with the results in Section~\ref{sec:auxiliary-kolmogorov-modified}. Note also that the three results below may be combined using the semigroup property~\eqref{eq:P}, see Sections~\ref{sec:proofs} and~\ref{sec:expo} for details.

\begin{lemma}\label{lem:u-12}
Let Assumptions~\ref{ass:Lambda},~\ref{ass:F} and~\ref{ass:Fregul2} be satisfied. For all $T\in(0,\infty)$ and all $\alpha\in[0,1)$, there exists $C_\alpha(T)\in(0,\infty)$, such that for all $\varphi:H\to\R$ of class $\mathcal{C}^1$ with bounded derivative, for all $t\in(0,T]$ and all $x,h\in H$, one has
\begin{equation}\label{eq:lem-u-1}
|Du(t,x).h|\le C_\alpha(T)\vvvert\varphi\vvvert_1 t^{-\alpha}|\IL^{-\alpha}h|,
\end{equation}
Moreover, for all $T\in(0,\infty)$ and all $\alpha_1,\alpha_2\in[0,1)$ such that $\alpha_1+\alpha_2<1$, there exists $C_{\alpha_1,\alpha_2}(T)\in(0,\infty)$, such that for all $\varphi:H\to\R$ of class $\mathcal{C}^2$ with bounded first and second order derivatives, one has
\begin{equation}\label{eq:lem-u-2}
|D^2u(t,x).(h_1,h_2)|\le C_{\alpha_1,\alpha_2}(T)\bigl(\vvvert \varphi\vvvert_1+\vvvert \varphi\vvvert_2\bigr) t^{-\alpha_1-\alpha_2}|\IL^{-\alpha_1}h_1| |\IL^{-\alpha_2}h_2|.
\end{equation}
\end{lemma}

\begin{proof}[Proof of Lemma~\ref{lem:u-12}]
Let $\varphi$ be of class $\mathcal{C}^1$ with bounded derivative and $T\in(0,\infty)$. Then using the expression~\eqref{eq:Du-1} gives
\[
|Du(t,x).h|\le \vvvert\varphi\vvvert_1 \E_x[|\eta^h(t)|],
\]
for all $t\ge 0$ and $x,h\in H$. Owing to the definition~\eqref{eq:eta} of $\eta^h(t)$, one obtains
\begin{align*}
\frac12\frac{d|\eta^h(t)|^2}{dt}&=\langle \frac{d\eta^h(t)}{dt}\rangle\\
&=-\langle \IL\eta^h(t),\eta^h(t)\rangle+\langle DF(X(t)).\eta^h(t),\eta^h(t)\rangle\\
&\le (\Lf-\lambda_1)|\eta^h(t)|^2.
\end{align*}
Applying Gronwall's lemma gives the upper bound $|\eta^h(t)|\le e^{(\Lf-\lambda_1)t}|h|$ for all $t\ge 0$. As a consequence the inequality~\eqref{eq:lem-u-1} holds when $\alpha=0$. To deal with the case $\alpha\in(0,1)$, it is convenient to introduce the auxiliary process defined by
\begin{equation}\label{eq:etahat}
\hat{\eta}^h(t)=\eta^h(t)-e^{-t\IL}h,
\end{equation}
which is solution of the evolution equation
\begin{equation}\label{eq:etahat-equation}
\frac{d\hat{\eta}^h(t)}{dt}=-\IL\hat{\eta}^h(t)+DF(X(t)).\hat{\eta}^h(t)+DF(X(t)).e^{-t\IL}h,
\end{equation}
with initial value $\hat{\eta}^h(0)=0$. The mild formulation of the evolution equation~\eqref{eq:etahat-equation} gives for all $t\ge 0$
\begin{align*}
|\hat{\eta}^h(t)|&=\Big|\int_0^t e^{-(t-s)\IL}DF(X(s)).\hat{\eta}^h(s)ds+\int_0^t e^{-(t-s)\IL}DF(X(s)).e^{-s\IL}h ds\Big|\\
&\le \Lf\int_0^t|\hat{\eta}^h(s)|ds+\Lf\int_0^t|e^{-(t-s)\IL}h|ds\\
&\le \Lf\int_0^t|\hat{\eta}^h(s)|ds+\Lf \int_0^t s^{-\alpha}ds |\IL^{-\alpha}h|,
\end{align*}
owing to the smoothing inequality~\eqref{eq:smoothing}. Since $\int_0^Ts^{-\alpha}ds<\infty$ for all $T\in(0,\infty)$ and $\alpha\in(0,1)$, using Gronwall's lemma gives the upper bound
\[
|\hat{\eta}^h(t)|\le C_\alpha(T)|\IL^{-\alpha}h|
\]
for all $t\in[0,T]$. Using~\eqref{eq:Du-1} and the smoothing inequality~\eqref{eq:smoothing}, one then obtains
\[
|Du(t,x).h|\le \vvvert\varphi\vvvert_1 |\IL^{-\alpha}e^{-t\IL}h|+\vvvert\varphi\vvvert_1 \E_x[|\hat{\eta}^h(t)|]\le C_\alpha(T)t^{-\alpha}|\IL^{-\alpha}h|,
\]
which concludes the proof of the first inequality~\eqref{eq:lem-u-1}. It remains to prove the second inequality~\eqref{eq:lem-u-2}. Let $\varphi$ be of class $\mathcal{C}^2$ with bounded first and second order derivatives and $T\in(0,\infty)$. Then using the expression~\eqref{eq:D2u} gives
\[
|D^2u(t,x).(h_1,h_2)|\le \vvvert\varphi\vvvert_1\E_x[|\zeta^{h_1,h_2}(t)|]+\vvvert\varphi\vvvert_2\E_x[|\eta^{h_1}(t)| |\eta^{h_2}(t)|],
\]
for all $t\ge 0$ and $x,h_1,h_2\in H$. Using the inequalities above, one obtains, for all $t\in(0,T]$
\[
\E_x[|\eta^{h_1}(t)| |\eta^{h_2}(t)|]\le C_\alpha(T)t^{-\alpha_1-\alpha_2}|\IL^{-\alpha_1}h_1||\IL^{-\alpha_2}h_2|,
\]
for the second term in the right-hand side above. To deal with the first term, observe that using~\eqref{eq:zeta} one has for all $t\in[0,T]$
\begin{align*}
|\zeta^{h_1,h_2}(t)|&=\Big|\int_0^t e^{-(t-s)\IL}DF(X(s)).\zeta^{h_1,h_2}(s)ds+\int_0^t e^{-(t-s)\IL}D^2F(X(s)).(\eta^{h_1}(s),\eta^{h_2}(s))ds\Big|\\
&\le \Lf\int_0^t|\zeta^{h_1,h_2}(s)|ds+\int_{0}^{t}\|\IL^{\alpha_F}e^{-(t-s)\IL}\|_{\mathcal{L}(H)}|\IL^{-\alpha_F}D^2F(X(s)).(\eta^{h_1}(s),\eta^{h_2}(s))|ds\\
&\le \Lf\int_0^t|\zeta^{h_1,h_2}(s)|ds+C\int_{0}^{t}(t-s)^{-\alpha_F}s^{-\alpha_1-\alpha_2}ds |\IL^{-\alpha_1}h_1||\IL^{-\alpha_2}h_2|,
\end{align*}
using the regularity Assumption~\ref{ass:Fregul2}, the smoothing inequality~\ref{eq:smoothing}, and the inequalities above. Since the conditions $\alpha_F<1$ and $\alpha_1+\alpha_2<1$ give the upper bound
\[
\underset{t\in[0,T]}\sup~\int_{0}^{t}(t-s)^{-\alpha_F}s^{-\alpha_1-\alpha_2}ds<\infty
\]
for all $T\in(0,\infty)$, applying Gronwall's lemma gives
\[
|\zeta^{h_1,h_2}(t)|\le C_{\alpha_1,\alpha_2}(T)|\IL^{-\alpha_1}h_1||\IL^{-\alpha_2}h_2|.
\]
Gathering the estimates then concludes the proof of the second inequality~\eqref{eq:lem-u-2} and of Lemma~\ref{lem:u-12}.
\end{proof}

\begin{lemma}\label{lem:u-ergo}
Let Assumptions~\ref{ass:Lambda},~\ref{ass:F},~\ref{ass:ergo} and~\ref{ass:Fregul2} be satisfied. For all $\varphi$ of class $\mathcal{C}^1$ with bounded derivative and for all $t\ge 0$, one has
\begin{equation}
\vvvert P_t\varphi\vvvert_1=\vvvert u(t,\cdot)\vvvert_1\le e^{-(\lambda_1-\Lf)t}\vvvert\varphi\vvvert_1.
\end{equation}\label{eq:lem-u-ergo1}
Moreover, for all $\kappa\in(0,\lambda_1-\Lf)$, there exists $C_\kappa\in(0,\infty)$ such that for all $\varphi$ of class $\mathcal{C}^2$ with bounded first and second derivatives and for all $t\ge 0$, one has
\begin{equation}\label{eq:lem-u-ergo2}
\vvvert P_t\varphi\vvvert_2=\vvvert u(t,\cdot)\vvvert_2\le C_\kappa e^{-\kappa t}\bigl(\vvvert\varphi\vvvert_1+\vvvert\varphi\vvvert_2\bigr).
\end{equation}
\end{lemma}

\begin{proof}[Proof of Lemma~\ref{lem:u-ergo}]
The proof of the first inequality~\eqref{eq:lem-u-ergo1} is straightforward: for all $t\ge 0$ and all $x,h\in H$, the expression~\eqref{eq:Du-1} gives
\[
|Du(t,x).h|\le \vvvert\varphi\vvvert_1 \E_x[|\eta^h(t)|],
\]
with the inequality
\[
\frac12\frac{d|\eta^h(t)|^2}{dt}\le -(\lambda_1-\Lf)|\eta^h(t)|^2,
\]
see the proof of Lemma~\ref{lem:u-12} above. Applying Gronwall's lemma gives
\[
|\eta^h(t)|\le \exp(-(\lambda_1-\Lf)t)|h|
\]
for all $t\ge 0$ (almost surely), and one obtains~\eqref{eq:lem-u-ergo1}.

To prove the inequality~\eqref{eq:lem-u-ergo2}, note that the expression~\eqref{eq:D2u} gives
\[
|D^2u(t,x).(h_1,h_2)|\le \vvvert\varphi\vvvert_1\E[|\zeta^{h_1,h_2}(t)|]+\vvvert\varphi\vvvert_2 e^{-2(\lambda_1-\Lf)}|h_1||h_2|
\]
using the inequality above. In addition, using the definition~\eqref{eq:zeta} of the process $\bigl(\zeta^{h_1,h_2}(t)\bigr)_{t\ge 0}$ and Assumption~\ref{ass:Fregul2}, one has for all $t\ge 0$
\begin{align*}
|\zeta^{h_1,h_2}(t)|&=\Big|\int_0^t e^{-(t-s)\IL}DF(X(s)).\zeta^{h_1,h_2}(s)ds+\int_0^t e^{-(t-s)\IL}D^2F(X(s)).(\eta^{h_1}(s),\eta^{h_2}(s))ds\Big|\\
&\le \Lf\int_0^t e^{-\lambda_1(t-s)}|\zeta^{h_1,h_2}(s)|ds+C\int_0^t \|\IL^{\alpha_F}e^{-(t-s)\IL}\|_{\mathcal{L}(H)}|\eta^{h_1}(s)||\eta^{h_2}(s)|ds\\
&\le \Lf\int_0^t e^{-\lambda_1(t-s)}|\zeta^{h_1,h_2}(s)|ds+C\int_0^t \min\bigl(t-s,1)^{-\alpha_F} e^{-\lambda_1(t-s)}e^{-2s\lambda_1}ds |h_1||h_2|,
\end{align*}
using the smoothing property~\eqref{eq:smoothing} of the semigroup $\bigl(e^{-t\IL}\bigr)_{t\ge 0}$. For any $\kappa\in(0,\lambda_1-\Lf)$, one has $2\lambda_1\ge \Lf+\kappa$ and one obtains, for all $t\ge 0$,
\[
|\zeta^{h_1,h_2}(t)|\le \Lf\int_0^t e^{-\lambda_1(t-s)}|\zeta^{h_1,h_2}(s)|ds+Ce^{-(\Lf+\kappa)t}\int_0^\infty \min(s,1)^{-\alpha_F}e^{-(\lambda_1-\Lf-\kappa)s}ds|h_1||h_2|.
\]
Applying Gronwall's lemma then gives
\[
e^{(\Lf+\kappa)t}|\zeta^{h_1,h_2}(t)|\le C_\kappa\int_0^t e^{\Lf(t-s)}ds|h_1||h_2|,
\]
which gives $|\zeta^{h_1,h_2}(t)|\le C_\kappa\exp\bigl(-\kappa t\bigr)|h_1||h_2|$ for all $t\ge 0$. Gathering the estimates then gives the upper bound
\[
|D^2u(t,x).h|\le C_\kappa e^{-\kappa t}\vvvert\varphi\vvvert_1+e^{-2(\lambda_1-\Lf)t}\vvvert\varphi\vvvert_2,
\]
for all $t\ge 0$, which concludes the proof of the second inequality~\eqref{eq:lem-u-ergo2} and of Lemma~\ref{lem:u-ergo}.
\end{proof}

\begin{lemma}\label{lem:u-0}
Let Assumptions~\ref{ass:Lambda},~\ref{ass:F} and~\ref{ass:Fregul2} be satisfied. For all $T\in(0,\infty)$, there exists $C(T)\in(0,\infty)$, such that for all bounded and continuous functions $\varphi:H\to\R$, for all $t\in(0,T]$, $P_t\varphi=u(t,\cdot)$ is of class $\mathcal{C}^1$, with
\begin{equation}\label{eq:lem-u-0}
\vvvert P_t\varphi\vvvert_1=\vvvert u(t,\cdot)\vvvert_1\le C(T)t^{-\frac12}\vvvert\varphi\vvvert_0.
\end{equation}
\end{lemma}

\begin{proof}[Proof of Lemma~\ref{lem:u-0}]
The proof of the upper bound~\eqref{eq:lem-u-0} is straightforward: using the expression~\eqref{eq:Du-0} and It\^o's isometry formula, one has for all $t\in(0,T]$ and all $x,h\in H$,
\begin{align*}
|Du(t,x).h|&\le \frac{\vvvert\varphi\vvvert_0}{t}\bigl(\int_0^t\E_x[|\eta^h(s)|^2]ds\bigr)^{\frac12}\\
&\le \frac{\vvvert\varphi\vvvert_0}{t}\bigl(\int_0^t e^{2(\Lf-\lambda_1)s}ds\bigr)^{\frac12}|h|\\
&\le \frac{\vvvert\varphi\vvvert_0}{\sqrt{t}}C(T)|h|
\end{align*}
owing to the inequality $|\eta^h(s)|\le e^{(\Lf-\lambda_1)s}|h|\le C(T)|h|$, see the proof of Lemma~\ref{lem:u-12} above. This concludes the proof of Lemma~\ref{lem:u-0}.
\end{proof}
\section{Proofs of the main results}\label{sec:proofs}

This section is organized as follows. The details of the proof of Theorem~\ref{theo:weakinv} (and of Proposition~\ref{propo:utau-regularity}) are given in Subsection~\ref{sec:proofs-weakinv}. This requires to consider the approximation of the invariant distribution $\mu_\infty=\mu_\star$ and to let Assumption~\ref{ass:gradient} to be satisfied, as a result the function $\varphi$ is not required to be of class $\mathcal{C}^2$. Subsection~\ref{sec:proofs-weak} is devoted to the proof of Theorem~\ref{theo:weak}, in a more general context, with functions $\varphi$ assumed to be of class $\mathcal{C}^2$. Subsection~\ref{sec:proofs-weak-ergo} finally gives a sketch of the proof of Theorem~\ref{theo:weak-ergo}.

\subsection{Proof of Theorem~\ref{theo:weakinv}}\label{sec:proofs-weakinv}

This section is devoted to prove the major result of this article. The important feature is to prove a weak error estimate where the right-hand side depends on $\vvvert\varphi\vvvert$ (instead of $\vvvert\varphi\vvvert_2$ in more standard approaches). For that purpose, it is first necessary to provide the proof of Proposition~\ref{propo:utau-regularity}, which gives the required upper bound for the first order derivative $Du_\tau(t,x).h$ of the solution $u_\tau$ of the Kolmogorov equation~\eqref{eq:IKolmogorov}.

\begin{proof}[Proof of Proposition~\ref{propo:utau-regularity}]
Let $\alpha=\frac12-\delta\in[0,\frac12)$, $\tau\in(0,\tau_0)$ and let $\varphi:H\to\R$ be a bounded and continuous function. The proof of the inequality~\eqref{eq:utau-regularity} requires to separate the cases $t\in(2\tau,2)$ and $t\in[2,\infty)$.

On the one hand, assume that $t\in(2\tau,2)$. Owing to the semigroup property~\eqref{eq:Ptau}, for all $x,h\in H$ one has
\begin{align*}
|Du_\tau(t,x).h|&=|D\IP_{\tau,t}\varphi(x).h|=|D\IP_{\tau,\frac{t}{2}}\bigl(\IP_{\tau,\frac{t}{2}}\varphi\bigr)(x).h|\\
&\le C_\alpha \vvvert\IP_{\tau,\frac{t}{2}}\varphi\vvvert_1\Bigl(\sqrt{\tau}|h|+(\frac{t}{2}-\tau)^{-\alpha}|\Lambda^{-\alpha}h|\Bigr)\\
&\le C_\alpha \vvvert\varphi\vvvert_0 (\frac{t}{2}-\tau)^{-\frac12}\Bigl(\sqrt{\tau}|h|+(\frac{t}{2}-\tau)^{-\alpha}|\Lambda^{-\alpha}h|\Bigr)
\end{align*}
where the first inequality is a consequence of Lemma~\ref{lem:utau-1} and the second inequality is a consequence of Lemma~\ref{lem:utau-0}. This yields the inequality~\eqref{eq:utau-regularity} in the first case $t\in(2\tau,2)$.

On the other hand, assume that $t\in[2,\infty)$.Owing to the semigroup property~\eqref{eq:Ptau}, for all $x,h\in H$ one has
\begin{align*}
|Du_\tau(t,x).h|&=|D\IP_{\tau,t}\varphi(x).h|=|D\IP_{\tau,1}\bigl(\IP_{\tau,t-1}\varphi\bigr)(x).h|\\
&\le C_\alpha \vvvert\IP_{\tau,t-1}\varphi\vvvert_1\Bigl(\sqrt{\tau}|h|+(1-\tau)^{-\alpha}|\Lambda^{-\alpha}h|\Bigr)\\
&\le C_\alpha e^{-\kappa(t-2)}\vvvert\IP_{\tau,1}\varphi\vvvert_1\Bigl(\sqrt{\tau}|h|+(1-\tau)^{-\alpha}|\Lambda^{-\alpha}h|\Bigr)\\
&\le C_\alpha e^{-\kappa t}\bigl(1\wedge (1-\tau)\bigr)^{-\frac12}\vvvert\varphi\vvvert_0\Bigl(\sqrt{\tau}|h|+(1-\tau)^{-\alpha}|\Lambda^{-\alpha}h|\Bigr),
\end{align*}
using succesively Lemma~\ref{lem:utau-1}, the equality $\IP_{\tau,t-1}=\IP_{\tau,t-2}\IP_{\tau,1}$ (from the semigroup property~\eqref{eq:Ptau}) and Lemma~\ref{lem:utau-ergo}, and Lemma~\ref{lem:utau-0}. Using the condition $1-\tau\ge 1-\tau_0>0$ for all $t\in(0,\tau_0)$, one obtains~\eqref{eq:utau-regularity} in the second case $t\in[2,\infty)$.

Note that the constant $C_\alpha$ does not depends on $\tau$. This concludes the proof of Proposition~\ref{propo:utau-regularity}.
\end{proof}

We are now in position to provide the details of the proof of Theorem~\ref{theo:weakinv}.
\begin{proof}[Proof of Theorem~\ref{theo:weakinv}]
The objective is to prove that the weak error estimate~\eqref{eq:theo-weakinv_weakerror} for bounded and continuous functions $\varphi\in\mathcal{C}^0(H,\R)$. Owing to the equality~\eqref{eq:distances}, this is sufficient to establish the error estimate~\eqref{eq:theo-weakinv_dTV-time} in the total variation distance, which is equivalent to~\eqref{eq:theo-weakinv_weakerror}. To obtain the inequality~\eqref{eq:theo-weakinv_dTVinvar}, it suffices to let $N\to\infty$, for instance using the initial value $x_0=0$, since $\mu_{\infty}^{\tau}$ is the unique invariant distribution of the modified Euler scheme, see Proposition~\ref{propo:invar-scheme}.

Let $\varphi:H\to\R$ be bounded and continuous, $\tau_0\in(0,1)$ and $\delta\in(0,\frac12)$. For any time-step size $\tau\in(0,\tau_0)$, any initial value $x_0\in H$, and all $N\in\N$, the error is decomposed as
\[
\E[\varphi(X_N^\tau)]-\int\varphi d\mu_\star=\E[\varphi(\IX_\tau(t_N))]-\int\varphi d\mu_\star+\E[\varphi(\IX_{\tau,N})]-\E[\varphi(\IX_\tau(t_N))],
\]
see Equation~\eqref{eq:decomperror}, where $u_\tau$ is defined by~\eqref{eq:utau}.

Recall that $\mu_\star$ is the unique invariant distribution for the modified Euler scheme~\eqref{eq:scheme}, see Proposition~\ref{propo:mu_star_modified}, since Assumption~\ref{ass:gradient} is satisfied. If $\IX_\star$ is a $H$-valued random variable with distribution $\rho_{\IX_\star}=\mu_\star$ and is independent of the Wiener process $\bigl(W(t)\bigr)_{t\ge 0}$, then for all $N\in\N$, such that $t_N=N\tau\ge 1$, one has
\begin{align*}
\big|\E[\varphi(\IX_\tau(t_N))]-\int\varphi d\mu_\star\big|&=\big|u_\tau(t_N,x_0)-\E[u_\tau(0,\IX_\star)]\big|\\
&=\big|u_\tau(t_N,x_0)-\E[u_\tau(t_N,\IX_\star)]\big|\\
&\le \vvvert u_\tau(t_N,\cdot)\vvvert_1 \E[|x_0-\IX_\star|]\\
&\le Ce^{-\kappa(t_N-1)}\vvvert \IP_{\tau,1}\varphi\vvvert_1(1+|x_0|)\\
&\le Ce^{-\kappa t_N}(1-\tau)^{-\frac12}\vvvert\varphi\vvvert_0 (1+|x_0|),
\end{align*}
using the equality $u_\tau(t_N,\cdot)=\IP_{\tau,t_N}\varphi=\IP_{\tau,t_N-1}\bigl(\IP_{\tau,1}\varphi\bigr)$ (semigroup property~\eqref{eq:Ptau}), Lemma~\ref{lem:utau-ergo} and the bound $\E[|\IX_\star|]=\int |x|d\mu_\star(x)<\infty$ (see Equation~\eqref{eq:mu_star-bound}) to obtain the first inequality, and Lemma~\ref{lem:utau-0} to obtain the second inequality.

It remains to deal with the other error term in the decomposition~\eqref{eq:decomperror} of the error. Recall that the auxiliary process $\bigl(\tilde{\IX}_\tau(t)\bigr)_{t\ge 0}$ is defined by~\eqref{eq:tildeIX} in Section~\ref{sec:auxiliary-scheme}. Using the definition~\eqref{eq:utau} of the function $u_\tau$ and a standard telescoping sum argument, one has
\begin{align*}
\E[\varphi(\IX_{\tau,N})]-\E[\varphi(\IX_\tau(t_N))]&=\E[u_\tau(0,\IX_{\tau,N})]-\E[u_\tau(N\tau,\IX_{\tau,0})]\\
&=\E[u_\tau(0,\tilde{\IX}_{\tau}(t_N))]-\E[u_\tau(t_N,\tilde{\IX}_{\tau}(0))]\\
&=\sum_{n=0}^{N-1}\Bigl(\E[u_\tau(t_N-t_{n+1},\tilde{\IX}_{\tau}(t_{n+1}))]-\E[u_\tau(t_N-t_{n},\tilde{\IX}_{\tau}(t_{n}))]\Bigr)\\
&=\sum_{n=0}^{N-1}e_n^\tau,
\end{align*}
with $e_n^\tau=\E[u_\tau(t_N-t_{n+1},\tilde{\IX}_{\tau}(t_{n+1}))]-\E[u_\tau(t_N-t_{n},\tilde{\IX}_{\tau}(t_{n}))]$ for all $n\in\{0,\ldots,N-1\}$. Applying It\^o's formula, using the expression~\eqref{eq:dtildeIX} for the evolution of the auxiliary process and the fact that $u_\tau$ solves the Kolmogorov equation~\eqref{eq:IKolmogorov}, one obtains, for all $n\in\{0,\ldots,N-1\}$, the expression
\[
e_n^\tau=\int_{t_n}^{t_{n+1}}\E\bigl[Du_\tau(t_N-t,\tilde{\IX}_\tau(t)).\bigl(Q_{\tau}F(\tilde{\IX}_\tau(t_n))-Q_{\tau}F(\tilde{\IX}_\tau(t))\bigr)\bigr]dt.
\]
The cases $n\in\{0,N-2,N-1\}$ and $n\in\{1,\ldots,N-3\}$ are treated separately. On the one hand, using the inequality~\eqref{eq:utau0-bad}, the bound~\eqref{eq:Q_tau-bound}, the Lipschitz continuity of $F$ and the moment bound~\eqref{eq:tildeIX-bound-ergo} (with $\alpha=0$), one obtains
\begin{align*}
|e_0^\tau|+|e_{N-2}^\tau|+|e_{N-1}^\tau|&\le C\vvvert\varphi\vvvert_0\int_{0}^{\tau}(t_N-t)^{-\frac12}\E[\big|Q_\tau^{\frac12}\bigl(F(\tilde{\IX}_\tau(t_n))-F(\tilde{\IX}_\tau(t))\bigr)\big|]dt\\
&+C\vvvert\varphi\vvvert_0\int_{t_{N-2}}^{t_N-1}(t_N-t)^{-\frac12}\E[\big|Q_\tau^{\frac12}\bigl(F(\tilde{\IX}_\tau(t_n))-F(\tilde{\IX}_\tau(t))\bigr)\big|]dt\\
&+C\vvvert\varphi\vvvert_0\int_{t_{N-1}}^{t_N}(t_N-t)^{-\frac12}\E[\big|Q_\tau^{\frac12}\bigl(F(\tilde{\IX}_\tau(t_n))-F(\tilde{\IX}_\tau(t))\bigr)\big|]dt\\
&\le C\tau^{\frac12}\vvvert\varphi\vvvert_0(1+|x_0|).
\end{align*}
On the other hand, using the inequality~\eqref{eq:utau-regularity} from Proposition~\ref{propo:utau-regularity}, with $\alpha=\frac12-\frac{\delta}{4}$, with $t_N-t-2\tau=t_{N-2}-\tau$, for all $n\in\{1,\ldots,N-3\}$, one obtains
\[
|e_n^\tau|\le {\bf e}_{n,1}^\tau+{\bf e}_{n,2}^\tau
\]
where the error terms on the right-hand side above are defined by
\begin{align*}
{\bf e}_{n,1}^\tau&=C\sqrt{\tau}\vvvert\varphi\vvvert_0\int_{t_n}^{t_{n+1}}\frac{e^{-\kappa(t_N-t)}}{(t_{N-2}-t)^{\frac12}}\E[|Q_\tau\bigl(F(\tilde{\IX}_\tau(t_n))-F(\tilde{\IX}_\tau(t))\bigr)|]dt\\
{\bf e}_{n,2}^\tau&=C_\delta\vvvert\varphi\vvvert_0\int_{t_n}^{t_{n+1}}\frac{e^{-\kappa(t_N-t)}}{(t_{N-2}-t)^{1-\frac{\delta}{4}}}\E[|\IL^{-\frac12+\frac{\delta}{4}}Q_\tau\bigl(F(\tilde{\IX}_\tau(t_n))-F(\tilde{\IX}_\tau(t))\bigr)|]dt.
\end{align*}
Using the bound~\eqref{eq:Q_tau-bound}, the Lipschitz continuity of $F$ and the moment bound~\eqref{eq:tildeIX-bound-ergo} (with $\alpha=0$), one obtains
\[
\sum_{n=1}^{N-3}{\bf e}_{n,1}^\tau\le C\sqrt{\tau}\vvvert\varphi\vvvert_0\int_{\tau}^{t_{N-2}}\frac{e^{-\kappa(t_{N-2}-t)}}{(t_{N-2}-t)^{\frac12}}dt(1+|x_0|)\le C\sqrt{\tau}\vvvert\varphi\vvvert_0\int_0^\infty\frac{e^{-\kappa t}}{t^{\frac12}}dt(1+|x_0|).
\]
The treatment of the error term ${\bf e}_{n,2}^\tau$ exploits Assumption~\ref{ass:Fregul1} on the regularity of the nonlinearity $F$: using the bound~\eqref{eq:Q_tau-bound} and the Cauchy--Schwarz inequality, one obtains
\begin{align*}
{\bf e}_{n,2}^\tau\le C_\delta\vvvert\varphi\vvvert_0\int_{t_n}^{t_{n+1}}\frac{e^{-\kappa(t_N-t)}}{(t_{N-2}-t)^{1-\frac{\delta}{4}}}&\bigl(\E[\bigl(1+|\tilde{\IX}_\tau(t)|_{\frac{1}{4}-\frac{\delta}{8}}^2+|\tilde{\IX}_\tau(t_n)|_{\frac{1}{4}-\frac{\delta}{8}}^2\bigr)]\bigr)^{\frac12}\\
&\bigl(\E[\big|\IL^{-\frac14+\frac{\delta}{8}}(\tilde{\IX}_\tau(t_n)-\tilde{\IX}_\tau(t))\big|]\bigr)^{\frac12} dt.
\end{align*}
Using the inequalities~\eqref{eq:tildeIX-bound-ergo} and~\eqref{eq:tildeIX-increment-ergo} from Lemma~\ref{lem:tildeIX} then yields the upper bound
\[
\sum_{n=1}^{N-3}{\bf e}_{n,2}^\tau\le C_\delta\tau^{\frac12-\delta}\vvvert\varphi\vvvert_0\int_0^\infty\frac{e^{-\kappa t}}{t^{1-\frac{\delta}{8}}}dt(1+|x_0|_{\frac14-\frac{\delta}{8}}^2).
\]
Gathering the estimates, one obtains
\[
\big|\E[\varphi(\IX_{\tau,N})]-\E[\varphi(\IX_\tau(t_N))]\big|\le C_\delta\tau^{\frac12-\delta}\vvvert\varphi\vvvert_0(1+|x_0|_{\frac14-\frac{\delta}{8}}^2).
\]

This concludes the proof of Theorem~\ref{theo:weakinv}.
\end{proof}

\subsection{Proof of Theorem~\ref{theo:weak}}\label{sec:proofs-weak}

The objective of this section is to prove Theorem~\ref{theo:weak}. More precisely, we prove that the weak error estimates~\eqref{eq:theo-weak-weakerror} holds, for any function $\varphi:H\to\R$ of class $\mathcal{C}^2$ with bounded first and second order derivatives. This assumption is the main difference with respect to the proof of Theorem~\ref{theo:weakinv} presented above. As already explained in Section~\ref{sec:results_weak} and as will be clear below, the motivation for the more restrictive condition on $\varphi$ is the treatment of an error term, which vanishes in the large time regime when the nonlinearity satisfies Assumption~\ref{ass:gradient}, and cannot be treated without assuming that $\varphi$ is of class $\mathcal{C}^2$ by the approach considered in this work. Whether the regularity conditions on $\varphi$ may be weakened in the framework of Theorem~\ref{theo:weak} is an open question.

In this section, without loss of generality it is assumed that the time-step size $\tau\in(0,\tau_0)$ satisfies the equality $T=N\tau$ with $N\in\N$, where $T\in(0,\infty)$ is fixed. Recall that the solution $X_N^\tau$ of the modified Euler scheme~\eqref{eq:scheme} is equal to $\IX_{\tau,N}$ defined by~\eqref{eq:scheme-IX}, see Section~\ref{sec:scheme-3rd}. As a consequence (see Equation~\eqref{eq:decomperror-gen}) the weak error can be decomposed as
\begin{equation}\label{eq:decomperrorproof}
\E[\varphi(X_N^\tau)]-\E[\varphi(X(T))]=E_{N,1}^\tau+E_{2}^\tau(T)
\end{equation}
with
\begin{equation}\label{eq:EN}
\begin{aligned}
E_{N,1}^\tau&=\E[\varphi(\IX_{\tau,N})]-\E[\varphi(\IX_\tau(T))]=\E[u_\tau(0,\IX_{\tau,N})]-\E[u_\tau(T,\IX_{\tau,0})]\\
E_{2}^\tau(T)&=\E[\varphi(\IX_\tau(T))]-\E[\varphi(X(T))]=\E[u(0,\IX_\tau(T))]-\E[u(T,\IX_\tau(0))],
\end{aligned}
\end{equation}
where the functions $u_\tau$ and $u$ are defined by~\eqref{eq:utau} and~\eqref{eq:u} respectively.

To simplify the notation, it is assumed that the function $\varphi$ satisfies the inequality $\vvvert\varphi\vvvert_1+\vvvert\varphi\vvvert_2\le 1$, and the general case follows by a straightforward argument.

\begin{proof}[Proof of Theorem~\ref{theo:weak}]
Owing to the decomposition~\eqref{eq:decomperrorproof} of the weak error, the weak error estimate~\eqref{eq:theo-weak-weakerror} is an immediate consequence of the two weak error estimates
\begin{subequations}
\begin{align}
|E_{N,1}^\tau|&\le C_{\delta}(T)\tau^{\frac12-\delta}(1+|x_0|_{\frac14-\frac{\delta}{8}}^2),\label{eq:weak1}\\
|E_{N,2}^\tau|&\le C_{\delta}(T)\tau^{\frac12-\delta}(1+|x_0|_{\delta}).\label{eq:weak2} 
\end{align}
\end{subequations}

Let us first establish the weak error estimate~\eqref{eq:weak1}. The arguments are similar to those used in the proof of Theorem~\ref{theo:weakinv} above, in particular it is not necessary to assume that $\varphi$ is of class $\mathcal{C}^2$ to prove~\eqref{eq:weak1}. Since this condition is necessary below to prove~\eqref{eq:weak2}, the inequality~\eqref{eq:weak1} is proved under this condition. Recall that the auxiliary process $\bigl(\tilde{\IX}_\tau(t)\bigr)_{t\ge 0}$ is defined by~\eqref{eq:tildeIX} in Section~\ref{sec:auxiliary-scheme}. Using the definition~\eqref{eq:utau} of the function $u_\tau$ and a standard telescoping sum argument, one has
\begin{align*}
E_{N,1}^\tau&=\E[u_\tau(0,\tilde{X}_{\tau}(T))]-\E[u_\tau(T,\tilde{X}_{\tau}(0))]\\
&=\sum_{n=0}^{N-1}\Bigl(\E[u_\tau(T-t_{n+1},\tilde{X}_{\tau}(t_{n+1}))]-\E[u_\tau(T-t_{n},\tilde{X}_{\tau}(t_{n}))]\Bigr)\\
&=\sum_{n=0}^{N-1}e_n^\tau,
\end{align*}
with $e_n^\tau=\E[u_\tau(T-t_{n+1},\tilde{X}_{\tau}(t_{n+1}))]-\E[u_\tau(T-t_{n},\tilde{X}_{\tau}(t_{n}))]$ for all $n\in\{0,\ldots,N-1\}$. Applying It\^o's formula, using the expression~\eqref{eq:dtildeIX} for the evolution of the auxiliary process and the fact that $u_\tau$ solves the Kolmogorov equation~\eqref{eq:IKolmogorov}, one obtains, for all $n\in\{0,\ldots,N-1\}$, the expression
\[
e_n^\tau=\int_{t_n}^{t_{n+1}}\E\bigl[Du(T-t,\tilde{\IX}_\tau(t)).\bigl(Q_{\tau}F(\tilde{\IX}_\tau(t_n))-Q_{\tau}F(\tilde{\IX}_\tau(t))\bigr)\bigr]dt.
\]
The cases $n=N-1$ and $n\in\{0,\ldots,N-2\}$ are treated separately. On the one hand, using the inequality~\eqref{eq:lem-utau1-0}, the bound~\eqref{eq:Q_tau-bound}, the Lipschitz continuity of $F$ and the moment bound~\eqref{eq:tildeIX-bound} (with $\alpha=0$), one obtains
\begin{align*}
|e_{N-1}^\tau|&\le C(T)\vvvert\varphi\vvvert_1\int_{t_{N-1}}^{t_N}\E[\big|Q_\tau\bigl(F(\tilde{\IX}_\tau(t_n))-F(\tilde{\IX}_\tau(t))\bigr)\big|]dt\\
&\le C(T)\tau^{\frac12}\vvvert\varphi\vvvert_1(1+|x_0|).
\end{align*}
On the other hand, using the inequality~\eqref{eq:lem-utau1-alpha} from Lemma~\ref{lem:utau-1}, with $\alpha=\frac12-\frac{\delta}{4}$, for all $n\in\{0,\ldots,N-2\}$, one obtains
\[
|e_n^\tau|\le {\bf e}_{n,1}^\tau+{\bf e}_{n,2}^\tau
\]
where the error terms on the right-hand side above are defined by
\begin{align*}
{\bf e}_{n,1}^\tau&=C(T)\sqrt{\tau}\int_{t_n}^{t_{n+1}}\E[|Q_\tau\bigl(F(\tilde{\IX}_\tau(t_n))-F(\tilde{\IX}_\tau(t))\bigr)|]dt\\
{\bf e}_{n,2}^\tau&=C(T)\int_{t_n}^{t_{n+1}}\frac{1}{(t_{N-1}-t)^{\frac12-\frac{\delta}{8}}}\E[|\IL^{-\frac12+\frac{\delta}{8}}Q_\tau\bigl(F(\tilde{\IX}_\tau(t_n))-F(\tilde{\IX}_\tau(t))\bigr)|]dt.
\end{align*}
Using the bound~\eqref{eq:Q_tau-bound}, the Lipschitz continuity of $F$ and the moment bound~\eqref{eq:tildeIX-bound} (with $\alpha=0$), one obtains
\[
\sum_{n=1}^{N-2}{\bf e}_{n,1}^\tau\le C(T)\sqrt{\tau}.
\]
The treatment of the error term ${\bf e}_{n,2}^\tau$ exploits Assumption~\ref{ass:Fregul1} on the regularity of the nonlinearity $F$: using the bound~\eqref{eq:Q_tau-bound} and the Cauchy--Schwarz inequality, one obtains
\begin{align*}
{\bf e}_{n,2}^\tau\le C(T)\int_{t_n}^{t_{n+1}}\frac{e^{-\kappa(t_N-t)}}{(t_{N-2}-t)^{\frac12-\frac{\delta}{8}}}&\bigl(\E[\bigl(1+|\tilde{\IX}_\tau(t)|_{\frac{1}{4}-\frac{\delta}{8}}^2+|\tilde{\IX}_\tau(t_n)|_{\frac{1}{4}-\frac{\delta}{8}}^2\bigr)]\bigr)^{\frac12}\\
&\bigl(\E[\big|\IL^{-\frac14+\frac{\delta}{2}}(\tilde{\IX}_\tau(t_n)-\tilde{\IX}_\tau(t))\big|]\bigr)^{\frac12} dt.
\end{align*}
Using the inequalities~\eqref{eq:tildeIX-bound} and~\eqref{eq:tildeIX-increment} from Lemma~\ref{lem:tildeIX} then yields the upper bound
\[
\sum_{n=0}^{N-2}{\bf e}_{n,2}^\tau\le C_\delta(T)\tau^{\frac12-\delta}(1+|x_0|_{\frac14-\frac{\delta}{8}}^2).
\]
Gathering the estimates, one obtains the error estimate~\eqref{eq:weak1}. 

It remains to prove the weak error estimate~\eqref{eq:weak2}. Using It\^o's formula associated with the modified stochastic evolution equation~\eqref{eq:modifiedSPDE}, one obtains
\begin{align*}
E_{2}^{\tau}(T)&=\E[u(0,\IX_\tau(N\tau))]-\E[u(N\tau,\IX_\tau(0))]\\
&=\int_0^T\E[-\partial_tu(T-t,\IX_\tau(t))]dt\\
&+\int_0^T\E[Du(T-t,\IX_\tau(t).\bigl(-\IL_\tau \IX_\tau(t)+Q_\tau F(\IX_\tau(t))\bigr)]dt\\
&+\frac12\sum_{j\in\N}\int_0^T \E[D^2u(T-t,\IX_\tau(t)).\bigl(Q_\tau^{\frac12}e_j,Q_\tau^{\frac12}e_j\bigr)]dt\\
&=E_{2,1}^\tau(T)+E_{2,2}^\tau(T)+E_{2,3}^\tau(T)
\end{align*}
where, owing to the identity $\partial_tu=\mathcal{L}u$ ($u$ is solution of the Kolmogorov equation~\eqref{eq:Kolmogorov}), the error terms $E_{2,1}(T),E_{2,2}(T),E_{2,3}(T)$ are defined by
\begin{align*}
E_{2,1}^\tau(T)&=\int_0^T\E[Du(T-t,\IX_\tau(t)).\bigl((\IL-\IL_\tau)\IX_\tau(t)\bigr)]dt\\
E_{2,2}^\tau(T)&=\int_0^T\E[Du(T-t,\IX_\tau(t)).\bigl((Q_\tau-I)F(\IX_\tau(t))\bigr)]dt\\
E_{2,3}^\tau(T)&=\frac12\sum_{j\in\N}\Bigl(q_{\tau,j}-1\Bigr)\E[D^2u(T-t,\IX_\tau(t)).(e_j,e_j)]dt.
\end{align*}
The weak error estimate~\eqref{eq:weak2} is then a straightforward consequence of the three auxiliary estimates
\begin{subequations}
\begin{align}
|E_{2,1}^\tau(T)|&\le C_{\delta}(T)\tau^{\frac12-\delta}(1+|x_0|_{\delta}),\label{eq:weak2-1}\\
|E_{2,2}^\tau(T)|&\le C_\delta(T)\tau^{1-\delta}(1+|x_0|),\label{eq:weak2-2}\\
|E_{2,3}^\tau(T)|&\le C_{\delta}(T)\tau^{\frac12-\delta}.\label{eq:weak2-3}
\end{align}
\end{subequations}
The inequalities~\eqref{eq:weak2-2} and~\eqref{eq:weak2-3} are proved using straightforward arguments, using the auxiliary results from Section~\ref{sec:auxiliary}. However, the proof of the inequality~\eqref{eq:weak2-1} requires additional arguments to obtain the weak order of convergence $1/2$.

$\bullet$ Proof of the inequality~\eqref{eq:weak2-1}. Using the mild formulation~\eqref{eq:modifiedSPDE-mild} associated with the modified stochastic evolution equation~\eqref{eq:modifiedSPDE}, the term $E_{2,1}^\tau(T)$ can be written as
\[
E_{2,1}^\tau(T)=E_{2,1,1}^\tau(T)+E_{2,1,2}^\tau(T)+E_{2,1,3}^\tau(T),
\]
with
\begin{align*}
E_{2,1,1}^\tau(T)&=\int_0^T\E[Du(T-t,\IX_\tau(t)).\bigl((\IL-\IL_\tau)e^{-t\IL_\tau}x_0\bigr)]dt\\
E_{2,1,2}^\tau(T)&=\int_0^T\E[Du(T-t,\IX_\tau(t)).\bigl((\IL-\IL_\tau)\int_0^t e^{-(t-s)\IL_\tau}Q_\tau F(\IX_\tau(s))ds\bigr)]dt\\
E_{2,1,3}^\tau(T)&=\int_0^T\E[Du(T-t,\IX_\tau(t)).\bigl((\IL-\IL_\tau)\int_0^t e^{-(t-s)\IL_\tau}Q_\tau^{\frac12}dW(s)\bigr)]dt.
\end{align*}
First, using the regularity estimate~\eqref{eq:lem-u-1} from Lemma~\ref{lem:u-12} with $\alpha=1-\frac{\delta}{2}$, one has
\begin{align*}
|E_{2,1,1}^\tau(T)|&\le C_\delta(T)\int_0^T(T-t)^{-1+\frac{\delta}{2}}|\IL^{-1+\frac{\delta}{2}}(\IL-\IL_\tau)e^{-t\IL_\tau}x_0|dt\\
&\le C\int_0^\tau (T-t)^{-1+\frac{\delta}{2}}|\IL^{-1+\frac{\delta}{2}}(\IL-\IL_\tau)e^{-t\IL_\tau}x_0|dt\\
&+C\int_\tau^T(T-t)^{-1+\frac{\delta}{2}}|\IL^{-1+\frac{\delta}{2}}(\IL-\IL_\tau)e^{-t\IL_\tau}x_0|dt.
\end{align*}
Using the inequalities~\eqref{eq:IL_tau-bound} from Lemma~\ref{lem:Q_tauIL_tau-bound} and~\eqref{eq:IL_tau-error} from Lemma~\ref{lem:Q_tauIL_tau-error} with $\alpha=0$, one has
\[
|\IL^{-1+\frac{\delta}{2}}(\IL-\IL_\tau)e^{-t\IL_\tau}x_0|\le C|x_0|_{\frac{\delta}{2}},
\]
for all $t\in(0,\tau)$. In addition, using the smoothing inequality~\eqref{eq:IL_tau-smoothing} from Lemma~\ref{lem:Q_tauIL_tau-bound} with $\alpha=1-\delta$, one has
\begin{align*}
\int_\tau^T(T-t)^{-1+\frac{\delta}{2}}|\IL^{-1+\frac{\delta}{2}}(\IL-\IL_\tau)e^{-t\IL_\tau}x_0|dt&\le C_\delta\int_\tau^T (T-t)^{-1+\frac{\delta}{2}}\tau^{1-\delta}|\IL^{1-\frac{\delta}{2}}e^{-t\IL_\tau}x_0|dt\\
&\le C_\delta\tau^{1-\delta}\int_\tau^T (T-t)^{-1+\frac{\delta}{2}}(t-\tau)^{-1+\frac{\delta}{2}}dt|x_0|\\
&\le C_\delta\tau^{1-\delta}\int_0^{T}(T-t)^{-1+\frac{\delta}{2}}t^{-1+\frac{\delta}{2}}ds|x_0|.
\end{align*}
Therefore one obtains
\[
|E_{2,1,1}^\tau(T)|\le C_\delta(T)\bigl(\tau|x_0|_\delta+\tau^{1-2\delta}|x_0|\bigr).
\]

Second, using similar arguments an upper bound for the error $|E_{2,1,2}^\tau(T)|$ is obtained. Using the regularity estimate~\eqref{eq:lem-u-1} from Lemma~\ref{lem:u-12} and the inequality~\eqref{eq:IL_tau-error} from Lemma~\ref{lem:Q_tauIL_tau-error}, applied with $\alpha=1-\frac{\delta}{2}$ and $\alpha=1-\delta$ respectively, one has
\begin{align*}
|E_{2,1,2}^\tau(T)|&\le C\int_0^T (T-t)^{-1+\frac{\delta}{2}}\big|\IL^{-1+\frac{\delta}{2}}(\IL-\IL_\tau)\int_0^te^{-(t-s)\IL_\tau}Q_\tau F(\IX_\tau(s))ds\big|dt\\
&\le C\tau^{1-\delta}\int_0^T(T-t)^{-1+\frac{\delta}{2}}\int_0^t|\IL^{1-\frac{\delta}{2}}e^{-(t-s)\IL_\tau}Q_\tau F(\IX_\tau(s))|dsdt.
\end{align*}
Using the bound~\eqref{eq:Q_tau-bound} from Lemma~\ref{lem:Q_tauIL_tau-bound}, the Lipschitz continuity of $F$ (Assumption~\ref{ass:F}) and the moment bound~\eqref{eq:modified-wellposed-bound} from Proposition~\ref{propo:modified-wellposed-bound}, one obtains for all $t\in[0,T]$
\begin{align*}
\int_0^t|\IL^{1-\frac{\delta}{2}}e^{-(t-s)\IL_\tau}Q_\tau F(\IX_\tau(s))|ds&\le C(T)(1+|x_0|)\int_0^t\|\IL^{1-\frac{\delta}{2}}e^{-s\IL_\tau}Q_\tau^{\frac12}\|_{\mathcal{L}(H)}ds\\
&\le C(T)(1+|x_0|)\int_0^\tau\|\IL^{1-\frac{\delta}{2}}e^{-s\IL_\tau}Q_\tau^{\frac12}\|_{\mathcal{L}(H)}ds\\
&+C(T)\mathds{1}_{t>\tau}(1+|x_0|)\int_\tau^t\|\IL^{1-\frac{\delta}{2}}e^{-s\IL_\tau}\|_{\mathcal{L}(H)}ds\\
&\le C(T)(1+|x_0|)\tau^{\frac{\delta}{2}}+C_\delta(T)(1+|x_0|)\int_\tau^t(s-\tau)^{-1+\frac{\delta}{2}}ds\\
&\le C_\delta(T)(1+|x_0|),
\end{align*}
using the smoothing inequalities~\eqref{eq:IL_tau-smoothing-bis} for $s\in(0,\tau)$ and~\eqref{eq:IL_tau-smoothing} for $s>\tau$ respectively, and the condition $\tau\le \tau_0$. Therefore one obtains
\[
|E_{2,1,2}^\tau(T)|\le C_\delta(T)\tau^{1-\delta}(1+|x_0|).
\]

Finally, it remains to prove an upper bound for the error $|E_{2,1,3}^\tau(T)|$. Combining the regularity estimate~\eqref{eq:lem-u-1}, the inequality~\eqref{eq:IL_tau-error} and the smoothing inequalities~\eqref{eq:IL_tau-smoothing} and~\eqref{eq:IL_tau-smoothing-bis} like for the treatment of the other terms is not sufficient: these techniques only provide a weak order of convergence equal to $1/4$. In order to obtain the weak order of convergence $1/2$, another approach is necessary, using Malliavin calculus techniques, following~\cite{Debussche:11}. We refer to~\cite[Section~2.2]{BrehierDebussche} for notation and useful results.

For all $s\ge 0$ and all $h\in H$, the Malliavin derivative $\bigl(\mathcal{D}_s^h\IX_\tau(t)\bigr)_{t\ge 0}$ is solution of the evolution equation 
\[
d\mathcal{D}_s^h\IX_\tau(t)=-\IL_\tau\mathcal{D}_s^h\IX_\tau(t)dt+Q_\tau DF(\IX_\tau(t)).\mathcal{D}_s^h\IX_\tau(t)dt,
\]
for $t\ge s$, with the initial value $\mathcal{D}_s^h\IX_\tau(t)=Q_\tau^{\frac12}h$, and $\mathcal{D}_s^h\IX_\tau(t)=0$ if $t<s$. Using arguments similar to the proof of Lemma~\ref{lem:utau-1} (the inequalities~\eqref{eq:Q_tau-bound} and~\eqref{eq:IL_tau-gap} and the Lipschitz continuity of $F$), one obtains for all $t\ge s$
\[
\frac12\frac{d|\mathcal{D}_s^h\IX_\tau(t)|^2}{dt}\le (\Lf-\lambda_1)\frac{\log(1+\tau\lambda_1)}{\tau\lambda_1}|\mathcal{D}_s^h\IX_\tau(t)|^2\le \Lf |\mathcal{D}_s^h\IX_\tau(t)|^2,
\]
and applying Gronwall's lemma yields the almost sure inequality $|\mathcal{D}_s^h\IX_\tau(t)|\le e^{\Lf(t-s)}|Q^{\frac12}h|\le e^{\Lf T}|Q^{\frac12}h|$ for all $0\le s\le t\le T$.

Owing to the Malliavin integration by parts formula~\eqref{eq:MalliavinIBP} and to the chain rule, the error term $E_{2,1,3}^\tau(T)$ is written as
\begin{align*}
E_{2,1,3}^\tau(T)&=\int_0^T\E[Du(T-t,\IX_\tau(t)).\bigl((\IL-\IL_\tau)\int_0^t e^{-(t-s)\IL_\tau}Q_\tau^{\frac12}dW(s)\bigr)]dt\\
&=\sum_{j\in\N}\int_0^T\E[Du(T-t,\IX_\tau(t)). e_j \langle (\IL-\IL_\tau)\int_0^t e^{-(t-s)\IL_\tau}Q_\tau^{\frac12}e_jd\beta_j(s),e_j\rangle]dt\\
&=\sum_{j\in\N}\int_0^T\int_{0}^{t}\E[D^2u(T-t,\IX_\tau(t)).(e_j,\mathcal{D}_s^{e_j}\IX_\tau(t))]\langle (\IL-\IL_\tau)e^{-(t-s)\IL_\tau}Q_\tau^{\frac12}e_j,e_j\rangle dsdt.
\end{align*}
Let $\delta\in(0,\frac12)$. Using the equalities $\IL^{-1+\frac{\delta}{2}}e_j=\lambda_j^{-1+\frac{\delta}{2}}e_j$ and $Q_\tau^{\frac12}e_j=q_{\tau,j}e_j$, and the regularity estimate~\eqref{eq:lem-u-2} from Lemma~\ref{lem:u-12} (with $\alpha_1=1-\frac{\delta}{2}$ and $\alpha_2=0$), one obtains 
\[
\big|D^2u(T-t,\IX_\tau(t)).(e_j,\mathcal{D}_s^{e_j}\IX_\tau(t))\big|\le C_\delta(T)(T-t)^{-1+\frac{\delta}{2}}\lambda_j^{-1+\frac{\delta}{2}}q_{\tau,j}^{\frac12}.
\]
As a consequence, one obtains the inequalities
\begin{align*}
|E_{2,1,3}^\tau(T)|&\le C_\delta(T)\int_0^T(T-t)^{-1+\frac{\delta}{2}}\sum_{j\in\N}\lambda_j^{-1+\frac{\delta}{2}}|\lambda_j-\lambda_{\tau,j}|q_{\tau,j}\int_{0}^{t}e^{-(t-s)\lambda_{\tau,j}}ds dt\\
&\le C_\delta(T)\int_0^T(T-t)^{-1+\frac{\delta}{2}} dt\sum_{j\in\N}\lambda_j^{\frac{\delta}{2}} \frac{q_{\tau,j}}{\lambda_{\tau,j}}\Big|1-\frac{\lambda_{\tau,j}}{\lambda_j}\Big|\\
&\le C_\delta(T)\sum_{j\in\N}\lambda_j^{-1+\frac{\delta}{2}}|1-\frac{\log(1+\tau\lambda_j)}{\lambda_j\tau}|,
\end{align*}
using the identities $\frac{q_{\tau,j}}{\lambda_{\tau,j}}=\frac{1}{\lambda_j}$ and $\frac{\lambda_{\tau,j}}{\lambda_j}=\frac{\log(1+\tau\lambda_j)}{\lambda_j\tau}$, see~\eqref{eq:modifiedILQ-eigen}. Using the inequality~\eqref{eq:boundproofQILerror} with $\alpha=\frac12-\delta$, one then obtains the upper bound
\[
|E_{2,1,3}^\tau(T)|\le C_\delta(T)\sum_{j\in\N}\lambda_j^{-1+\frac{\delta}{2}}|1-\frac{\log(1+\tau\lambda_j)}{\lambda_j\tau}|\le  C_\delta(T)\sum_{j\in\N}\lambda_j^{-\frac12-\frac{\delta}{2}}\tau^{\frac12-\delta}.
\]
Gathering the upper bounds on the error terms $E_{2,1,1}^\tau(T)$, $E_{2,1,2}^\tau(T)$ and $E_{2,1,3}^\tau(T)$,  the inequality~\eqref{eq:weak2-1} for $E_{2,1}^\tau(T)$ is thus proved.

$\bullet$ Proof of the inequality~\eqref{eq:weak2-2}. Let $\delta\in(0,1)$ be an arbitrarily small parameter. Owing to the inequality~\eqref{eq:lem-u-1} from Lemma~\ref{lem:u-12}, and to the inequality~\eqref{eq:Q_tau-error} from Lemma~\ref{lem:Q_tauIL_tau-error}, both applied with $\alpha=1-\delta$, one obtains 
\begin{align*}
|E_{N,2,2}|&\le C_\delta(T)\int_0^t(T-t)^{-1+\delta}\|\IL^{-1+\delta}(Q_\tau-I)\|_{\mathcal{L}(H)}\E[|F(\IX_\tau(t))|]dt\\
&\le C_\delta(T)\tau^{1-\delta}\int_0^t(T-t)^{-1+\delta}\E[|F(\IX_\tau(t))|]dt\\
&\le C_\delta(T)\tau^{1-\delta}(1+|x_0|),
\end{align*}
using the Lipschitz continuity of $F$ (Assumption~\ref{ass:F}) and the moment bound~\eqref{eq:modified-wellposed-bound} from Proposition~\ref{propo:modified-wellposed-bound} in the last inequality. The inequality~\eqref{eq:weak2-2} is thus proved.

$\bullet$ Proof of the inequality~\eqref{eq:weak2-2}. Let $\delta\in(0,\frac12)$ be an arbitrarily small parameter. Owing to the inequality~\eqref{eq:boundproofQILerror},applied with $\alpha=\frac12-\delta$ (see the proof of Lemma~\ref{lem:Q_tauIL_tau-error}) and to the regularity estimate~\eqref{eq:lem-u-1} from Lemma~\ref{lem:u-12},  applied with $\alpha_1=\alpha_2=\frac12-\frac{\delta}{4}$, one obtains
\begin{align*}
|E_{N,2,3}|&\le C_\delta(T)\sum_{j\in\N}(\lambda_j\tau)^{\frac12-\delta}\lambda_j^{-1+\frac{\delta}{2}} \int_0^T (T-t)^{-1+\frac{\delta}{2}}dt\\
&\le C_\delta(T)\tau^{\frac12-\delta}\sum_{j\in\N}\lambda_j^{-\frac12-\frac{\delta}{2}}\int_0^T(T-t)^{-1+\frac{\delta}{2}}dt,
\end{align*}
Since $\sum_{j\in\N}\lambda_j^{-\frac12-\frac{\delta}{2}}<\infty$ and $\int_0^T(T-t)^{-1+\frac{\delta}{2}}dt<\infty$, the inequality~\eqref{eq:weak2-3} is thus proved.

$\bullet$ Conclusion: the three inequalities~\eqref{eq:weak2-1},~\eqref{eq:weak2-2} and~\eqref{eq:weak2-3} hold, thus the inequality~\ref{eq:weak2} is proved.

Since the weak error estimate~\eqref{eq:theo-weak-weakerror} is a straightforward consequence of the inequalities~\eqref{eq:weak1} and~\eqref{eq:weak2}, this concludes the proof of Theorem~\ref{theo:weak}.
\end{proof}

\subsection{Proof of Theorem~\ref{theo:weak-ergo}}\label{sec:proofs-weak-ergo}

The details of the proof of Theorem~\ref{theo:weak-ergo} are standard and are omitted. The main changes compared with the proof of Theorem~\ref{theo:weak} written above is the need to exploit the results of Lemma~\ref{lem:utau-ergo} and Lemma~\ref{lem:u-ergo} to obtain factors of the type $\exp(-\kappa t)$ in the upper bounds for $Du_\tau(t,x).h$, $Du(t,x).h$ and $D^2u(t,x).(h_1,h_2)$. As a consequence, the constants $C_\delta(T)$ appearing in the proof of Theorem~\ref{theo:weak} above are replaced by constants $C_\delta$ which are independent of the final time. This approach is standard and is already used in the proof of Theorem~\ref{theo:weakinv} above. We also refer to~\cite{B:2014} for a similar analysis to prove error estimates for the approximation of the invariant distribution $\mu_\infty$ using the standard Euler scheme. The details of the proof are thus left to the interested reader.

\section{Results on the accelerated exponential Euler scheme}\label{sec:expo}

The section is organized as follows. Instrumental auxiliary results are stated and proved in Subsection~\ref{sec:expo-aux} and the details of the proof of Theorem~\ref{theo:weak-expo} are given in Subsection~\ref{sec:expo-proof}. The arguments are similar to those used in Section~\ref{sec:proofs-weakinv} to prove Theorem~\ref{theo:weakinv}, but the analysis of the accelerated exponential Euler scheme does not require Assumption~\ref{ass:gradient} to be satisfied or to consider only the approximation of the invariant distribution. Finally, in Subsection~\ref{sec:expo-strong} the strong rate of convergence of the accelerated exponential Euler scheme is studied: we check that in this case the strong order is $1/2$ and coincides with the weak order of convergence exhibited in Theorem~\ref{theo:weak-expo}.

Recall that the accelerated exponential Euler scheme
\begin{equation}\label{eq:exponentialscheme-expo}
X_{n+1}^{\tau,\e}=e^{-\tau\IL}X_n^{\tau,\e}+\IL^{-1}(I-e^{-\tau\IL})F(X_n^{\tau,\e})+\int_{t_n}^{t_{n+1}}e^{-(t_{n+1}-s)\IL}dW(s),\quad X_0^{\tau,\e}=x_0,
\end{equation}
see Equation~\eqref{eq:res-exponentialscheme} in Section~\ref{sec:results}.

\subsection{Auxiliary results}\label{sec:expo-aux}

The first auxiliary result is a variant of Proposition~\ref{propo:utau-regularity}, concerning the regularity properties of the function $u$ instead of the function $u_\tau$. Recall that $u(t,x)=\E_x[\varphi(X(t))]$, see Equation~\ref{eq:u} in Section~\ref{sec:auxiliary-kolmogorov-original}.
\begin{propo}\label{propo:u-regularity}
Let Assumption~\ref{ass:Fregul2} be satisfied and $\varphi:H\to \R$ be a bounded and continuous function. 

For all $t>0$, $u(t,\cdot)$ is differentiable and one has the following estimate: for all $T\in(0,\infty)$ and $\delta\in(0,\frac12)$, there exists $C_{\delta}(T)\in(0,\infty)$ such that for $t\in(0,T]$ and all $x,h\in H$, one has 
\begin{equation}\label{eq:u-regularity}
\big|Du(t,x).h\big|\le C_{\delta}(T) \vvvert\varphi\vvvert_0 t^{1-\delta}|\Lambda^{-\frac12+\delta}h|.
\end{equation}
In addition, if Assumption~\ref{ass:ergo} is satisfied, there exists $C_\delta\in(0,\infty)$ such that for all $t\in(0,\infty)$ and all $x,h\in H$, one has 
\begin{equation}\label{eq:u-regularity-ergo}
\big|Du(t,x).h\big|\le C_{\delta}(T)e^{-(\lambda_1-\Lf)t} \vvvert\varphi\vvvert_0 t^{-1+\delta}|\IL^{-\frac12+\delta}h|.
\end{equation}
\end{propo}

The inequality~\eqref{eq:u-regularity-ergo} is required to prove Theorem~\ref{theo:weak-ergo-expo}.

\begin{proof}[Proof of Proposition~\ref{propo:u-regularity}]
The inequality~\eqref{eq:u-regularity} is a straightforward consequence of Lemma~\ref{lem:u-12} and Lemma~\ref{lem:u-0}, and of the semigroup property~\eqref{eq:P}: for all $t\in(0,T]$ and all $x,h\in H$,
\begin{align*}
|Du(t,x).h|&=|DP_t\varphi(x).h|=|DP_{\frac{t}{2}}\bigl(P_{\frac{t}{2}}\varphi\bigr)(x).h|\\
&\le C_\delta(T)t^{-\frac12+\delta}\vvvert P_{\frac{t}{2}}\varphi\vvvert_1 |\IL^{-\frac12+\delta}h|\\
&\le C_\delta(T)t^{-1+\delta}\vvvert\varphi\vvvert_0 |\IL^{-\frac12+\delta}h|,
\end{align*}
using successively the inequalities~\eqref{eq:lem-u-1} and~\eqref{eq:lem-u-0} from Lemma~\ref{lem:u-12} and Lemma~\ref{lem:u-0} respectively.

The inequality~\eqref{eq:u-regularity-ergo} is proved using a similar argument and Lemma~\ref{lem:u-ergo}: if $t\ge 1$, the semigroup property~\eqref{eq:P} yields
\begin{align*}
|Du(t,x).h|&=|DP_t\varphi(x).h|=|DP_{\frac12}\bigl(P_{t-\frac12}\varphi\bigr)(x).h|\\
&\le C_\delta \vvvert P_{t-\frac12}\varphi\vvvert_1 |\IL^{-\frac12+\delta}h|\\
&\le C_\delta e^{-(\lambda_1-\Lf)(t-1)} \vvvert P_{\frac12}\varphi\vvvert_1 |\IL^{-\frac12+\delta}h|\\
&\le C_\delta e^{-(\lambda_1-\Lf)t} \vvvert \varphi\vvvert_0 |\IL^{-\frac12+\delta}h|.
\end{align*}
using successively the inequalities~\eqref{eq:lem-u-1},~\eqref{eq:lem-u-ergo1} and~\eqref{eq:lem-u-0} from Lemma~\ref{lem:u-12}, Lemma~\ref{lem:u-ergo} and Lemma~\ref{lem:u-0} respectively. The case $t\in(0,1)$ is covered by the first inequality~\eqref{eq:u-regularity}.

This concludes the proof of Proposition~\ref{propo:u-regularity}.
\end{proof}

The other main ingredient in the proof of Theorem~\ref{theo:weak-expo} is the introduction of an auxiliary process $\bigl(\tilde{X}^{\tau,\e}(t)\bigr)_{t\ge 0}$, defined as follows: for all $t\ge 0$,
\begin{equation}\label{eq:tildeXe}
\tilde{X}^{\tau,\e}(t)=e^{-t\IL}x_0+\int_{0}^{t}e^{-(t-s)\IL}F(X_{\ell(s)}^{\tau,\e})ds+W^{\IL}(t),
\end{equation}
where $W^{\IL}(t)=\int_0^t e^{-(t-s)\IL}dW(s)$ is given by~\eqref{eq:StochasticConvolution}, $\ell(s)=n$ if $t_n\le s<t_{n+1}$, with $t_n=n\tau$. For every $n\in\N_0$ and all $t\in[t_n,t_{n+1}]$, one has
\begin{equation}\label{eq:dtildeXe}
d\tilde{X}^{\tau,\e}(t)=-\IL\tilde{X}^{\tau,\e}(t)dt+F(X_{n}^{\tau,\e})dt+dW(t).
\end{equation}
Finally, note that by construction of the auxiliary process, one has $\tilde{X}^{\tau,\e}(t_n)=X_{n}^{\tau,\e}$ for all $n\in\N_0$.

The following variant of Lemma~\ref{lem:tildeIX} is used in Section~\ref{sec:expo-proof} below.
\begin{lemma}\label{lem:tildeXe}
Let Assumptions~\ref{ass:Lambda} and~\ref{ass:F} be satisfied.

For all $T\in(0,\infty)$ and $\alpha\in[0,\frac14)$, one has 
\begin{equation}\label{eq:tildeXe-bound}
\underset{x_0\in H^\alpha}\sup~\underset{t\in[0,T]}\sup~\frac{\E[|\tilde{X}^{\tau,\e}(t)|_\alpha^2]}{1+|x_0|_\alpha^2}<\infty
\end{equation}
and
\begin{equation}\label{eq:tildeXe-increment}
\underset{x_0\in H^\alpha}\sup~\underset{t\in[0,T]}\sup~\frac{\E[\big|\IL^{-\alpha}\bigl(\tilde{X}^{\tau,\e}(t)-\tilde{X}^{\tau,\e}(t_{\ell(t)})\bigr)\big|^2]}{\tau^{2\alpha}(1+|x_0|_\alpha^2)}<\infty.
\end{equation}

Moreover, if Assumption~\ref{ass:ergo} is satisfied, then for all $\alpha\in[0,\frac14)$ and $x_0\in H^\alpha$, one has
\begin{equation}\label{eq:tildeXe-bound-ergo}
\underset{x_0\in H^\alpha}\sup~\underset{t\ge 0}\sup~\frac{\E[|\tilde{X}^{\tau,\e}(t)|_\alpha^2]}{1+|x_0|_\alpha^2}<\infty
\end{equation}
and
\begin{equation}\label{eq:tildeXe-increment-ergo}
\underset{x_0\in H^\alpha}\sup~\underset{t\ge 0}\sup~\frac{\E[\big|\IL^{-\alpha}\bigl(\tilde{X}^{\tau,\e}(t)-\tilde{X}^{\tau,\e}(t_{\ell(t)})\bigr)\big|^2]}{\tau^{2\alpha}(1+|x_0|_\alpha^2)}<\infty.
\end{equation}
\end{lemma}

The proofs of the inequalities~\eqref{eq:tildeXe-bound-ergo} and~\eqref{eq:tildeXe-increment-ergo} are omitted, these inequalities are required only to prove Theorem~\ref{theo:weak-ergo-expo}.

\begin{proof}[Proof of Lemma~\ref{lem:tildeXe}]
Introduce the auxiliary random variables
\begin{align*}
Y_n^{\tau,\e}&=X_n^{\tau,\e}-W^{\IL}(t_n)\\
\tilde{Y}^{\tau,\e}(t)&=\tilde{X}^{\tau,\e}(t)-W^{\IL}(t)
\end{align*}
for all $t\ge 0$ and $n\in\N_0$.

On the one hand, for all $\alpha\in[0,\frac14)$, one has
\[
\underset{t\ge 0}\sup~\E[|W^{\IL}(t)|_\alpha^2]=\int_0^\infty \|\IL^{\alpha}e^{-s\IL}\|_{\mathcal{L}_2(H)}^2ds=\sum_{j\in\N}\lambda_j^{2\alpha}\int_0^\infty e^{-2s\lambda_j}ds\le \sum_{j\in\N}\frac{1}{2\lambda_j^{1-2\alpha}}<\infty.
\]
On the other hand, for all $n\in\N_0$ one has
\[
Y_{n+1}^{\tau,\e}=e^{-\tau\IL}\bigl(Y_n^{\tau,\e}+\tau F(Y_n^{\tau,e}+W^{\IL}(t_n))\bigr),
\]
therefore, using the Lipschitz continuity of $F$ (Assumption~\ref{ass:F}), one obtains
\[
|Y_{n+1}^{\tau,\e}|\le e^{-\tau\lambda_1}\bigl(1+\Lf\tau|Y_n^\tau|\bigr)+\Lf\tau e^{-\tau\lambda_1}\bigl(1+|W^{\IL}(t_n)|\bigr).
\]
A straightforward argument then yields the moment bound
\[
\underset{n\in\N_0,n\tau\le T}\sup~\bigl(\E[|X_n^{\tau,\e}|^2]\bigr)^{\frac12}\le C(T)\bigl(1+|x_0|\bigr).
\]
Then, for all $t\in[0,T]$, one has
\begin{align*}
\bigl(\E[|\tilde{X}^{\tau,\e}(t)|_\alpha^2]\bigr)^{\frac12}&\le |e^{-t\IL} x_0|_\alpha+\int_0^t \|\IL^\alpha e^{-(t-s)\IL}\|_{\mathcal{L}(H)}\bigl(\E[|F(X_{\ell(s)}^{\tau,\e})|^2]\bigr)^{\frac12}ds\\
&+\bigl(\E[|W^{\IL}(t)|_\alpha^2]\bigr)^{\frac12}\\
&\le C(T)(1+|x_0|_\alpha),
\end{align*}
using the smoothing inequality~\eqref{eq:smoothing}, the Lipschitz continuity of $F$ and the moment bound above. Thus the inequality~\eqref{eq:tildeXe-bound} is proved.

It remains to prove the inequality~\eqref{eq:tildeXe-increment}. Observe that for all $n\in\N_0$ and $t\in[t_n,t_{n+1})$, one has
\[
\tilde{X}^{\tau,\e}(t)-\tilde{X}^{\tau,\e}(t_n)=\bigl(e^{-(t-t_n)\IL}-I\bigr)X_n^{\tau,\e}+\int_{t_n}^{t}e^{-(t-s)\IL} F(X_n^{\tau,\e})ds+\int_{t_n}^{t}e^{-(t-s)\IL}dW(s),
\]
and, using the moment bound~\eqref{eq:tildeXe-bound} proved above, and It\^o's isometry formula, one obtains
\begin{align*}
\bigl(\E[\big|\IL^{-\alpha}\bigl(\tilde{X}^{\tau,\e}(t)-\tilde{X}^{\tau,\e}(t_{\ell(t)})\bigr)\big|^2]\bigr)^{\frac12}&\le C(1+|x_0|_\alpha)\|\IL^{-2\alpha}\bigl(e^{-(t-t_n)\IL_\tau}-I\bigr)\|_{\mathcal{L}(H)}+C\tau(1+|x_0|)\\
&+\bigl(\int_{t_n}^{t}\|\IL^{-\alpha}e^{-(t-s)\IL}\|_{\mathcal{L}_2(H)}^2ds\bigr)^{\frac12}.
\end{align*}
Using the inequality~\eqref{eq:regularity}, one has $C(1+|x_0|_\alpha)\|\IL^{-2\alpha}\bigl(e^{-(t-t_n)\IL_\tau}-I\bigr)\|_{\mathcal{L}(H)}\le C_\alpha\tau^{2\alpha}(1+|x_0|_\alpha)$. In addition, one has
\begin{align*}
\int_{t_n}^{t}\|\IL^{-\alpha}e^{-(t-s)\IL}\|_{\mathcal{L}_2(H)}^2ds&\le \sum_{j\in\N}\lambda_j^{-2\alpha}\int_0^\tau e^{-2s\lambda_j}ds\\
&\le \sum_{j\in\N}\lambda_j^{-2\alpha}\frac{1-e^{-2\tau\lambda_j}}{2\lambda_j}\\
&\le C_\alpha\tau^{4\alpha}\sum_{j\in\N}\lambda_j^{-1+2\alpha},
\end{align*}
and one has $\sum_{j\in\N}\lambda_j^{-1+2\alpha}$ for all $\alpha\in[0,\frac14)$, using the inequality $\underset{z\in(0,\infty)}\sup~z^{-\alpha}|1-e^{-z}|<\infty$. Gathering the estimates gives the inequality~\eqref{eq:tildeXe-increment}.

This concludes the proof of Lemma~\ref{lem:tildeXe}.
\end{proof}

\subsection{Proof of Theorem~\ref{theo:weak-expo}}\label{sec:expo-proof}

Using the auxiliary results presented in Subsection~\ref{sec:expo-aux} above, we are now in position to prove Theorem~\ref{theo:weak-expo}.

\begin{proof}[Proof of Theorem~\ref{theo:weak-expo}]
Like in the proof of Theorem~\ref{theo:weakinv}, it suffices to establish the weak error estimate~\eqref{eq:theo-weak-weakerror-expo} when the function $\varphi$ is bounded and continuous. The weak error is written and decomposed as follows:
\begin{align*}
\E[\varphi(X_N^{\tau,\e})]-\E[\varphi(X(T))]&=E[u(0,X_N^{\tau,\e})]-\E[u(T,X_0^{\tau,\e})]\\
&=E[u(0,\tilde{X}^{\tau,\e}(t_N))]-\E[u(T,\tilde{X}^{\tau,\e}(0))]\\
&=\sum_{n=0}^{N-1}\Bigl(\E[u(t_N-t_{n+1},\tilde{X}^{\tau,\e}(t_{n+1}))]-\E[u(t_N-t_n,\tilde{X}^{\tau,\e}(t_n))]\Bigr)\\
&=\sum_{n=0}^{N-1}e_{n}^{\tau,\e},
\end{align*}
with $e_{n}^{\tau,\e}=\E[u(t_N-t_{n+1},\tilde{X}^{\tau,\e}(t_{n+1}))]-\E[u(t_N-t_n,\tilde{X}^{\tau,\e}(t_n))]$ for all $n\in\{0,\ldots,N-1\}$, using a standard telescoping sum argument, where $u$ is defined by~\eqref{eq:u}. Applying It\^o's formula, using the expression~\eqref{eq:dtildeXe} for the evolution of the auxiliary process and the fact that $u$ solves the Kolmogorov equation~\eqref{eq:Kolmogorov}, one obtains, for all $n\in\N_0$, the expression
\[
e_n^{\tau,\e}=\int_{t_n}^{t_{n+1}}\E\bigl[Du(t_N-t,\tilde{X}^{\tau,\e}(t)).\bigl(F(\tilde{X}^{\tau,\e}(t_n))-F(\tilde{X}^{\tau,\e}(t))\bigr)\bigr]dt.
\]
Using the regularity estimate~\eqref{eq:u-regularity} from Proposition~\ref{propo:u-regularity}, then Assumption~\ref{ass:Fregul1} and the Cauchy--Schwarz inequality, one obtains for all $n\in\N_0$
\begin{align*}
|e_n^{\tau,\e}|&\le C_\delta(T)\vvvert\varphi\vvvert_0\int_{t_n}^{t_{n+1}}\frac{1}{(t_N-t)^{1-\frac{\delta}{8}}}\E[\big|\IL^{-\frac12+\frac{\delta}{2}}\bigl(F(\tilde{X}^{\tau,\e}(t_n))-F(\tilde{X}^{\tau,\e}(t))\bigr)\big|]dt\\
&\le C_\delta(T)\vvvert\varphi\vvvert_0\int_{t_n}^{t_{n+1}}\frac{1}{(t_{N}-t)^{1-\frac{\delta}{8}}}\bigl(\E[\bigl(1+|\tilde{X}^{\tau,\e}(t)|_{\frac{1}{4}-\frac{\delta}{8}}^2+|\tilde{X}^{\tau,\e}(t_n)|_{\frac{1}{4}-\frac{\delta}{8}}^2\bigr)]\bigr)^{\frac12}\\
&\hspace{6cm}\bigl(\E[\big|\IL^{-\frac14+\frac{\delta}{2}}(\tilde{X}^{\tau,\e}(t_n)-\tilde{X}^{\tau,\e}(t))\big|]\bigr)^{\frac12} dt.
\end{align*}
Finally, using the inequalities~\eqref{eq:tildeXe-bound} and~\eqref{eq:tildeXe-increment-ergo}, one obtains
\[
\big|\E[\varphi(X_N^{\tau,\e})]-\E[\varphi(X(T))]\big|\le \sum_{n=0}^{N-1}|e_n^{\tau,\e}|\le C_\delta(T)\tau^{\frac12-\delta}\vvvert\varphi\vvvert_0\int_0^\infty\frac{1}{t^{1-\frac{\delta}{8}}}dt(1+|x_0|_{\frac14-\frac{\delta}{8}}^2).
\]
This concludes the proof of Theorem~\ref{theo:weak-expo}.
\end{proof}

\subsection{Strong convergence of the exponential Euler scheme}\label{sec:expo-strong}

Theorem~\ref{theo:weak-expo} states that the weak order of convergence of the (accelerated) exponential Euler scheme~\eqref{eq:exponentialscheme-expo} is equal to $1/2$. If the function $\varphi$ is assumed to be of class $\mathcal{C}^1$ with bounded derivative, instead of being bounded and measurable, a proof of the weak error estimate is obtained using Lemma~\ref{lem:u-12} directly instead of Proposition~\ref{propo:u-regularity} above. Note that contrary to the analysis of the standard Euler scheme or of the modifed Euler scheme (see Theorem~\ref{theo:weak}), it is not needed to assume that $\varphi$ is of class $\mathcal{C}^2$ with bounded first and second order derivatives. This difference is due to the construction of the accelerated exponential Euler scheme, which is exact when $F=0$, in other words the stochastic convolution is computed with no error.

In fact, when the conditions of Theorem~\ref{theo:weak-expo} are satisfied, the accelerated exponential Euler scheme~\eqref{eq:exponentialscheme-expo} has strong order of convergence equal to $1/2$.
\begin{propo}\label{propo:strong-expo}
Let the nonlinearity $F$ satisfy Assumptions~\ref{ass:F} and~\ref{ass:Fregul1}. For all $T\in(0,\infty)$, $\delta\in(0,\frac12)$ and $\tau_0\in(0,1)$, there exists $C_{\delta}(T)\in(0,\infty)$ such that for all $\tau=\frac{T}{N}\in(0,\tau_0)$ with $N\in\N$, and all $x_0\in H^{\frac14-\frac{\delta}{8}}$ one has
\begin{equation}\label{eq:strong-expo}
\underset{n\in\N;n\tau\le T}\sup~\E[|X(n\tau)-X_n^{\tau,\e}|]\le C_\delta(T)\tau^{\frac12-\delta}\bigl(1+|x_0|_{\frac14-\frac{\delta}{8}}\bigr).
\end{equation}
\end{propo}
The strong convergence estimate~\eqref{eq:strong-expo} is a variant of the results of~\cite{Jentzen}, under slightly different assumptions on the linearity $F$. A proof is given below for completeness. It is thus not surprising that it is sufficient to assume that $\varphi$ is Lipschitz continuous to obtain weak error estimates with order $1/2$ for this integrator, instead of assuming $\varphi$ is of class $\mathcal{C}^2$ in standard results. Theorem~\ref{theo:weak-expo} is a substantial improvement of this straightforward result, since $\varphi$ is only assumed to be bounded and measurable, in particular it is not Lipschitz continuous.

\begin{proof}[Proof of Proposition~\ref{propo:strong-expo}]
The strong error estimate~\eqref{eq:strong-expo} is a straightforward consequence of the more precise result
\[
\underset{0\le t\le t}\sup~\E[|X(T)-\tilde{X}^{\tau,\e}(t)|]\le C_\delta(T)\tau^{\frac12-\delta}\bigl(1+|x_0|_{\frac14-\frac{\delta}{8}}\bigr),
\]
where $\bigl(\tilde{X}^{\tau,\e}(t)\bigr)_{t\ge 0}$ is the auxiliary process defined by~\eqref{eq:tildeXe}. Indeed, $X_n^{\tau,\e}=\tilde{X}^{\tau,\e}(t_n)$ with $t_n=n\tau$, for all $n\in\N_0$ and $\tau\in(0,\tau_0)$.

The strong error estimate above is obtained using the following arguments. Using the mild formulations for the processes $X$ and $\tilde{X}^{\tau,\e}$, one obtains for all $t\in[0,T]$,
\begin{align*}
X(t)-\tilde{X}^{\tau,\e}(t)&=\int_0^t e^{-(t-s)\IL}F(X(s))ds-\int_0^t e^{-(t-s)\IL}F(\tilde{X}^{\tau,\e}(t_{\ell(s)}))ds\\
&=\int_0^t e^{-(t-s)\IL}\bigl(F(X(s))-F(\tilde{X}^{\tau,\e}(s)\bigr)ds\\
&+\int_0^t e^{-(t-s)\IL}\bigl(F(\tilde{X}^{\tau,\e}(s))-F(\tilde{X}^{\tau,\e}(t_{\ell(s)}))\bigr)ds.
\end{align*}
Therefore, for all $t\in[0,T]$, one obtains
\begin{align*}
\E[|X(t)-\tilde{X}^{\tau,\e}(t)|]&\le \Lf\int_0^t\E[|X(s)-\tilde{X}^{\tau,\e}(s)|]ds\\
&+\int_0^t \E[\big|e^{-(t-s)\IL}\bigl(F(\tilde{X}^{\tau,\e}(s))-F(\tilde{X}^{\tau,\e}(t_{\ell(s)}))\big|]ds\\
&\le \Lf\int_0^t\E[|X(s)-\tilde{X}^{\tau,\e}(s)|]ds\\
&+\int_0^t \frac{1}{(t-s)^{\frac12-\frac{\delta}{8}}}\E[|\IL^{-\frac12+\frac{\delta}{8}}\bigl(F(\tilde{X}^{\tau,\e}(s))-F(\tilde{X}^{\tau,\e}(t_{\ell(s)}))\big|]ds
\end{align*}
using the smoothing inequality~\eqref{eq:smoothing}. Finally, using Assumption~\ref{ass:Fregul1}, and then the auxiliary bounds~\eqref{eq:tildeXe-bound} and~\eqref{eq:tildeXe-increment}, like in the proof of Theorem~\ref{theo:weak-expo} above (see Subsection~\ref{sec:expo-proof}) then gives
\begin{align*}
\E[|X(t)-\tilde{X}^{\tau,\e}(t)|]&\le \Lf\int_0^t\E[|X(s)-\tilde{X}^{\tau,\e}(s)|]ds\\
&+C_\delta(T)\tau^{\frac12-\delta}\int_0^\infty\frac{1}{t^{1-\frac{\delta}{8}}}dt(1+|x_0|_{\frac14-\frac{\delta}{8}}^2),
\end{align*}
for all $t\in[0,T]$. Applying Gronwall's lemma then concludes the proof of Proposition~\ref{propo:strong-expo}.
\end{proof}

\section{Results on the standard Euler scheme}\label{sec:standard}

This section is devoted to the proof of Theorem~\ref{theo:weakinv-standard} and is organized as follows. The proof of the weak error estimate~\eqref{eq:theo-weakinv_weakerror-standard} is based on the decomposition~\eqref{eq:decomperror-standard} of the error, where the auxiliary processes $\bigl(\IX^{\tau,\s}(t)\bigr)_{t\ge 0}$ and $\bigl(\IX_\star^{\tau,\s}(t)\bigr)_{t\ge 0}$ are the solutions of the modified stochastic evolution equations~\eqref{eq:modifiedSPDEstandard} and~\eqref{eq:modifiedSPDEstandardstar} respectively, which we recall for convenience:
\begin{equation}\label{eq:modified-standard}
\begin{aligned}
d\IX^{\tau,\s}(t)&=-\IL_\tau \IX^{\tau,\s}(t)dt+Q_{\tau}F(\IX^{\tau,\s}(t))dt+R_{\tau}^{\frac12}dW(t),\\
d\IX_\star^{\tau,\s}(t)&=-\IL_\tau \IX_\star^{\tau,\s}(t)dt+R_{\tau}F(\IX_\star^{\tau,\s}(t))dt+R_{\tau}^{\frac12}dW(t).
\end{aligned}
\end{equation}
Subsection~\ref{sec:standard-aux} provides several auxiliary results, concerning the linear operator $R_\tau$ (defined by Equation~\eqref{eq:Rtau}), the two auxiliary processes introduced above, and an additional auxiliary process $\bigl(\tilde{\IX}^{\tau,\s}(t)\bigr)_{t\ge 0}$ defined below. Subsection~\ref{sec:standard-kolmogorov} is then devoted to the study of regularity properties of the solutions $u^{\tau,\s}$ and $u_\star^{\tau,\s}$ of the Kolmogorov equations associated with the stochastic evolution equations~\eqref{eq:modified-standard}: the statements and arguments are similar to those presented in Section~\eqref{sec:auxiliary-kolmogorov-modified}, however some new difficulties need to be dealt with. Finally, the details of the proof of Theorem~\ref{theo:weakinv-standard} are presented in Subsection~\ref{sec:standard-proof}.

\subsection{Auxiliary results}\label{sec:standard-aux}

Let us first state the three main properties of the linear operator $R_\tau$ which are used in the analysis. Recall that
\[
R_\tau=Q_\tau(I+\frac{\tau\IL}{2})^{-1}
\]
for all $\tau\in(0,\tau_0)$.
\begin{lemma}\label{lem:Rtau}
For all $\tau\in(0,\tau_0)$ and $x\in H$, one has
\begin{equation}\label{eq:lemRtau-bound}
|R_\tau x|\le |Q_\tau x|,\quad |R_\tau^{\frac12}x|\le |Q_\tau^{\frac12}x|.
\end{equation}
Moreover, one has
\begin{align}
&\underset{\tau\in(0,\tau_0)}\sup~\|R_\tau^{-\frac12}e^{-\tau\IL_\tau}\|_{\mathcal{L}(H)}<\infty,\label{eq:IL_tau-IL_tauR_tau1/2}\\
&\underset{\tau\in(0,\tau_0)}\sup~\underset{t\in(0,\infty)}\sup~\frac{\sqrt{t}}{\sqrt{\tau}}\|R_\tau^{-\frac12}e^{-t\IL_\tau}Q_\tau\|_{\mathcal{L}(H)}<\infty.\label{eq:IL_tau-IL_tauR_tau1/2Qtau}
\end{align}

Finally, for all $\delta\in(0,\frac12)$, one has
\begin{align}
&\underset{\tau\in(0,\tau_0)}\sup~\tau^{-\frac12+\delta}\|\IL^{-\frac12+\delta}(Q_\tau-R_\tau)\|_{\mathcal{L}(H)}<\infty,\label{eq:QtauRtau}\\
&\underset{\tau\in(0,\tau_0)}\sup~\tau^{-\delta}\|\IL^{-\delta}R_\tau^{-\frac12}(Q_\tau-R_\tau)\|_{\mathcal{L}(H)}<\infty.\label{eq:QtauRtauRtau1/2}
\end{align}
\end{lemma}

\begin{proof}[Proof of Lemma~\ref{lem:Rtau}]
The inequality~\eqref{eq:lemRtau-bound} is a straightforward consequence of the definition of $R_\tau$ and of the inequality $\|(I+\frac{\tau\IL}{2})^{-1}\|_{\mathcal{L}(H)}\le 1$.

The inequality~\eqref{eq:IL_tau-IL_tauR_tau1/2} is proved as follows (see the proof of inequality~~\eqref{eq:IL_tau-IL_tauQ_tau1/2} from Lemma~\ref{lem:Q_tauIL_tau-bound} for similar arguments): for all $\tau\in(0,\tau_0)$, one has
\begin{align*}
\|R_\tau^{-\frac12}e^{-\tau\IL_\tau}\|_{\mathcal{L}(H)}&=\underset{j\in\N}\sup~\frac{(1+\frac{\tau\lambda_j}{2})^{\frac12}e^{-\tau\lambda_{\tau,j}}}{q_{\tau,j}^{\frac12}}\\
&=\underset{j\in\N}\sup~\frac{(1+\frac{\tau\lambda_j}{2})^{\frac12}(\tau\lambda_j)^{\frac12}}{(1+\tau\lambda_j)\sqrt{\log(1+\tau\lambda_j)}}\\
&\le \underset{z\in(0,\infty)}\sup~\frac{\bigl((1+\frac{z}{2})z\bigr)^{\frac12}}{(1+z)\sqrt{\log(1+z)}}<\infty.
\end{align*}

Let us now prove the inequality~\eqref{eq:IL_tau-IL_tauR_tau1/2Qtau}: for all $\tau\in(0,\tau_0)$ and $t\in(0,\infty)$, one has
\begin{align*}
\|R_\tau^{-\frac12}e^{-t\IL_\tau}Q_\tau\|_{\mathcal{L}(H)}&=\underset{j\in\N}\sup~\frac{(1+\frac{\tau\lambda_j}{2})^{\frac12}}{q_{\tau,j}^{\frac12}}e^{-t\lambda_{\tau,j}}q_{\tau,j}\\
&=\underset{j\in\N}\sup~(1+\frac{\tau\lambda_j}{2})^{\frac12}\frac{\sqrt{\log(1+\tau\lambda_j)}}{(\tau\lambda_j)^{\frac12}}e^{-t\frac{\log(1+\tau\lambda_j)}{\tau}}\\
&\le \frac{\tau^{\frac12}}{t^{\frac12}}\underset{z\in(0,\infty)}\sup~ z^{\frac12}e^{-z}\underset{j\in\N}\sup~\frac{(1+\frac{\tau\lambda_j}{2})^{\frac12}}{(\tau\lambda_j)^{\frac12}}\\
&\le \frac{\tau^{\frac12}}{t^{\frac12}}\underset{z\in(0,\infty)}\sup~ z^{\frac12}e^{-z}\underset{z\in(0,\infty)}\sup~\frac{(1+\frac{z}{2})^{\frac12}}{z^{\frac12}}.
\end{align*}

It remains to prove the inequalities~\eqref{eq:QtauRtau} and~\eqref{eq:QtauRtauRtau1/2}. Note that for all $\tau\in(0,\tau_0)$, one has
\[
Q_\tau-R_\tau=Q_\tau(I-(I+\frac{\tau\IL}{2})^{-1})=Q_\tau(I+\frac{\tau\IL}{2})^{-1}\frac{\tau\IL}{2}.
\]
One the one hand, one has
\begin{align*}
\underset{\tau\in(0,\tau_0)}\sup~\tau^{-\frac12+\delta}\|\IL^{-\frac12+\delta}(Q_\tau-R_\tau)\|_{\mathcal{L}(H)}&=\underset{\tau\in(0,\tau_0)}\sup~\underset{j\in\N}\sup~\frac{q_{\tau,j}(\tau\lambda_j)^{\frac12+\delta}}{2+\tau\lambda_j}\\
&\le \underset{z\in(0,\infty)}\sup~\frac{z^{\frac12+\delta}}{2+z}<\infty,
\end{align*}
using the inequality~\eqref{eq:Q_tau-bound} from Lemma~\ref{lem:Q_tauIL_tau-bound} to have $q_{\tau,j}\le 1$.

On the other hand, the linear operator
\[
R_\tau^{-\frac12}(Q_\tau-R_\tau)=Q_\tau^{\frac12}(I+\frac{\tau\IL}{2})^{-\frac12}\frac{\tau\IL}{2}=\frac12\log(1+\tau\IL)^{\frac12}(I+\frac{\tau\IL}{2})^{-\frac12}(\tau\IL)^{\frac12},
\]
is unbounded for all $\tau\in(0,\tau_0)$. However, for all $\delta\in(0,\frac12)$, one has
\begin{align*}
\underset{\tau\in(0,\tau_0)}\sup~\tau^{-\delta}\|\IL^{-\delta}R_\tau^{-\frac12}(Q_\tau-R_\tau)\|_{\mathcal{L}(H)}&=\underset{\tau\in(0,\tau_0)}\sup~\underset{j\in\N}\sup~\frac{\sqrt{\log(1+\tau\lambda_j)}}{2(\tau\lambda_j)^{\delta}}\frac{(\tau\lambda_j)^{\frac12}}{(1+\frac{\tau\lambda_j}{2})^{\frac12}}\\
&\le \underset{z\in(0,\infty)}\sup~\frac{\sqrt{\log(1+z)}}{2z^{\delta}}\frac{z^{\frac12}}{(1+\frac{z}{2})^{\frac12}}<\infty.
\end{align*}

The proof of the inequalities~\eqref{eq:QtauRtau} and~\eqref{eq:QtauRtauRtau1/2} is thus completed.

This concludes the proof of Lemma~\ref{lem:Rtau}.
\end{proof}

Note that combining the inequalities~\eqref{eq:lemRtau-bound} and~\eqref{eq:IL_tau-IL_tauR_tau1/2Qtau} also gives the inequality
\begin{equation}\label{eq:IL_tau-IL_tauR_tau1/2+}
\underset{\tau\in(0,\tau_0)}\sup~\underset{t\in(0,\infty)}\sup~\frac{\sqrt{t}}{\sqrt{\tau}}\|e^{-t\IL_\tau}R_\tau\|_{\mathcal{L}(H)}<\infty.
\end{equation}

The next result states moment estimates for the solutions $\bigl(\IX^{\tau,\s}(t)\bigr)_{t\ge 0}$ and $\bigl(\IX_\star^{\tau,\s}(t)\bigr)_{t\ge 0}$ of the modified stochastic evolution equations~\eqref{eq:modified-standard}. These moment bounds are uniform with respect to $t\in(0,\infty)$ and $\tau\in(0,\tau_0)$. We refer to Proposition~\ref{propo:modified-wellposed-bound} from Section~\ref{sec:auxiliary-modified} for a similar result.

\begin{lemma}\label{lem:modifieds}
Let Assumptions~\ref{ass:Lambda},~\ref{ass:F} and~\ref{ass:ergo} be satisfied, and let the linear operators $\IL_\tau$, $Q_\tau$ and $R_\tau$ be defined by~\eqref{eq:modifiedILQ} and~\eqref{eq:Rtau}, for all $\tau\in(0,\tau_0)$.

For any initial value $x_0\in H$, the modified stochastic evolution equations~\eqref{eq:modified-standard} admit unique mild solutions $\bigl(\IX^{\tau,\s}(t)\bigr)_{t\ge 0}$ and $\bigl(\IX_\star^{\tau,\s}(t)\bigr)_{t\ge 0}$, with initial values $\IX^{\tau,\s}(0)=\IX_\star^{\tau,\s}(0)=x_0$. In addition, for all $\alpha\in[0,\frac14)$, one has
\begin{equation}\label{eq:modified-standard-bound}
\underset{x_0\in H^\alpha}\sup~\underset{\tau\in(0,\tau_0)}\sup~\underset{t\ge 0}\sup~\frac{\E[|\IX^{\tau,\s}(t)|_\alpha^2]+\E[|\IX_\star^{\tau,\s}(t)|_\alpha^2]}{1+|x_0|_\alpha^2}<\infty.
\end{equation}
\end{lemma}

\begin{proof}[Proof of Lemma~\ref{lem:modifieds}]
The proof follows the same steps as the proof of the moment bound~\eqref{eq:modified-wellposed-bound-ergo} from Proposition~\ref{propo:modified-wellposed-bound}, using the inequality~\eqref{eq:lemRtau-bound} from Lemma~\ref{lem:Rtau}. The details are omitted.
\end{proof}

The final tool studied in this section is the auxiliary process $\bigl(\tilde{\IX}^{\tau,\s}(t)\bigr)_{t\ge 0}$ defined as follows: for all $\tau\in(0,\tau_0$ and for all $t\ge 0$,
\begin{equation}\label{eq:tildeXs}
\tilde{X}^{\tau,\s}(t)=e^{-t\IL_\tau}x_0+\int_{0}^{t}e^{-(t-s)\IL_\tau}Q_\tau F(X_{\ell(s)}^{\tau,\s})ds+\int_0^te^{-(t-s)\IL_\tau}R_{\tau}^{\frac12}dW(s),
\end{equation}
where $\ell(s)=n$ if $t_n\le s<t_{n+1}$, with $t_n=n\tau$. For every $n\in\N_0$ and all $t\in[t_n,t_{n+1}]$, one has
\begin{equation}\label{eq:dtildeXs}
d\tilde{X}^{\tau,\s}(t)=-\IL\tilde{X}^{\tau,\s}(t)dt+Q_\tau F(X_{n}^{\tau,\s})dt+R_\tau^{\frac12}dW(t).
\end{equation}
Finally, note that by construction of the auxiliary process, one has $\tilde{X}^{\tau,\s}(t_n)=X_{n}^{\tau,\s}$ for all $n\in\N_0$.

The following variant of Lemma~\ref{lem:tildeIX} and Lemma~\ref{lem:tildeXe} holds.
\begin{lemma}\label{lem:tildeIXs}
Let Assumptions~\ref{ass:Lambda},~\ref{ass:F} and~\ref{ass:ergo} be satisfied. For all $\tau_0\in(0,1)$, $\alpha\in[0,\frac14)$, one has
\begin{equation}\label{eq:tildeIXs-bound-ergo}
\underset{x_0\in H^\alpha}\sup~\underset{\tau\in(0,\tau_0)}\sup~\underset{t\ge 0}\sup~\frac{\E[|\tilde{\IX}^{\tau,\s}(t)|_\alpha^2]}{1+|x_0|_\alpha^2}<\infty
\end{equation}
and
\begin{equation}\label{eq:tildeIXs-increment-ergo}
\underset{x_0\in H^\alpha}\sup~\underset{\tau\in(0,\tau_0)}\sup~\underset{t\ge\tau}\sup~\frac{\E[\big|\IL^{-\alpha}\bigl(\tilde{\IX}^{\tau,\s}(t)-\tilde{\IX}^{\tau,\s}(t_{\ell(t)})\bigr)\big|^2]}{\tau^{2\alpha}(1+|x_0|_\alpha^2)}<\infty.
\end{equation}
\end{lemma}

\begin{proof}[Proof of Lemma~\ref{lem:tildeIXs}]
Observe that for all $t\ge 0$, one has
\[
\int_0^te^{-(t-s)\IL_\tau}R_{\tau}^{\frac12}dW(s)=(I+\frac{\tau\IL}{2})^{-\frac12}W^{\IL}(t),
\]
where $\IW_\tau(t)=\int_0^t e^{-(t-s)\IL_\tau}Q_{\tau}^{\frac12}dW(s)$ is given by~\eqref{eq:modifiedStochasticConvolution}. Since $\|(I+\frac{\tau\IL}{2})^{-\frac12}\|_{\mathcal{L}(H)}<\infty$, the inequalities~\eqref{eq:tildeIXs-bound-ergo} and~\eqref{eq:tildeIXs-increment-ergo} are obtained using the same arguments as the inequalities~\eqref{eq:tildeIX-bound-ergo} and~\eqref{eq:tildeIX-increment-ergo} from Lemma~\ref{lem:tildeIX}. The details are omitted.
\end{proof}

\subsection{Kolmogorov equation associated with the modified equations}\label{sec:standard-kolmogorov}

The objective of this section is to state and prove regularity results for the functions $u^{\tau,\s}$ and $u_\star^{\tau,\s}$ defined by
\begin{equation}\label{eq:utaus}
u^{\tau,\s}(t,x)=\E_x[\varphi(\IX^{\tau,\s}(t))],
\end{equation}
and
\begin{equation}\label{eq:utausstar}
u_\star^{\tau,\s}(t,x)=\E_x[\varphi(\IX_\star^{\tau,\s}(t))],
\end{equation}
for all $t\ge 0$, $x\in H$, and $\tau\in(0,\tau_0)$, where $\varphi$ is a bounded and measurable function from $H$ to $\R$. In the above definitions, $\bigl(\IX^{\tau,\s}(t)\bigr)_{t\ge 0}$ and $\bigl(\IX_\star^{\tau,\s}(t)\bigr)_{t\ge 0}$ are the unique solutions of the modified stochastic evolution equation~\eqref{eq:modified-standard}, with initial values $\IX^{\tau,\s}(0)=\IX_\star^{\tau,\s}(0)=x$. Like in Section~\ref{sec:auxiliary-kolmogorov-modified}, in order to study the regularity properties of the functions $u^{\tau,\s}$ and $u_\star^{\tau,\s}$, it is convenient to rely on the convention introduced in Section~\ref{sec:Galerkin}. An auxiliary finite dimensional approximation is applied, in order to justify the regularity properties and the computations, and all the upper bounds do not depend on the auxiliary discretization parameter, which is omitted to simplify the notation.

It is convenient to introduce the families of linear operators $\bigl(\IP_t^{\tau,\s}\bigr)_{t\ge 0}$ and $\bigl(\IP_{\star,t}^{\tau,\s}\bigr)_{t\ge 0}$, such that $u^{\tau,s}(t,\cdot)=\IP_t^{\tau,\s}\varphi(\cdot)$ and $u_\star^{\tau,s}(t,\cdot)=\IP_{\star,t}^{\tau,\s}\varphi(\cdot)$ for all $t\ge 0$. The Markov property for the solutions of the modified stochastic evolution equations~\eqref{eq:modified-standard} yields the semigroup property: for all $t,s\ge 0$ and all $\tau\in(0,\tau_0)$, one has
\begin{equation}\label{eq:Ptausstar}
\IP_{t+s}^{\tau,\s}\varphi=\IP_t^{\tau,\s}\bigl(\IP_s^{\tau,\s}\varphi\bigr),\quad \IP_{\star,t+s}^{\tau,\s}\varphi=\IP_{\star,t}^{\tau,\s}\bigl(\IP_{\star,s}^{\tau,\s}\varphi\bigr)
\end{equation}
for any $\varphi\in\mathcal{B}_b(H)$.

Under appropriate regularity conditions on the function $\varphi$, the function $(t,x)\in \R^+\times H\mapsto u^{\tau,\s}(t,x)=\IP_t^{\tau,\s}\varphi(x)$ is solution of the Kolmogorov equation
\begin{equation}\label{eq:IKolmogorovs}
\partial_t u^{\tau,\s}=\mathcal{L}^{\tau,\s} u^{\tau,\s}
\end{equation} 
with initial value $u^{\tau,\s}(0,\cdot)=\varphi$, where the infinitesimal generator $\mathcal{L}^{\tau,\s}$ is defined by
\[
\mathcal{L}^{\tau,\s}\phi(x)=D\phi(x).\bigl(-\IL_\tau x+Q_\tau F(x)\bigr)+\frac12\sum_{j\in\N}D^2\phi(x).(R_\tau^{\frac12}e_j,R_\tau^{\frac12}e_j).
\]
Similarly, under appropriate regularity conditions on the function $\varphi$, the function $(t,x)\in \R^+\times H\mapsto u_\star^{\tau,\s}(t,x)=\IP_{\star,t}^{\tau,\s}\varphi(x)$ is solution of the Kolmogorov equation
\begin{equation}\label{eq:IKolmogorovsstar}
\partial_t u_\star^{\tau,\s}=\mathcal{L}_\star^{\tau,\s} u^{\tau,\s}
\end{equation} 
with initial value $u_\star^{\tau,\s}(0,\cdot)=\varphi$, where the infinitesimal generator $\mathcal{L}_\star^{\tau,\s}$ is defined by
\[
\mathcal{L}_\star^{\tau,\s}\phi(x)=D\phi(x).\bigl(-\IL_\tau x+R_\tau F(x)\bigr)+\frac12\sum_{j\in\N}D^2\phi(x).(R_\tau^{\frac12}e_j,R_\tau^{\frac12}e_j).
\]

The main result of this section is the following version of Proposition~\ref{propo:utau-regularity} stated in Section~\ref{sec:results_invar}: the functions $u^{\tau,\s}$ and $u_\star^{\tau,\s}$ satisfy the same regularity estimates as the function $u_\tau$. However, the proof of Lemma~\ref{lem:utausutausstar-regularity} is more technical than the proof of Proposition~\ref{propo:utau-regularity}, therefore detailed proofs are given.
\begin{lemma}\label{lem:utausutausstar-regularity}
Let Assumptions~\ref{ass:ergo} and~\ref{ass:Fregul2} be satisfied and $\varphi:H\to \R$ be a bounded and continuous function. For all $t>0$ and $\tau\in(0,\tau_0)$, $u^{\tau,\s}(t,\cdot)$ and $u_\star^{\tau,\s}$ are differentiable and one has the following estimates: for all $\delta\in(0,\frac12]$, there exists $C_{\delta}\in(0,\infty)$ such that for all $\tau\in(0,\tau_0)$ and for all $t\in(2\tau,\infty)$, one has 
\begin{equation}\label{eq:utausustausstar-regularity}
\begin{aligned}
\big|&Du^{\tau,\s}(t,x).h\big|+\big|Du_\star^{\tau,\s}(t,x).h\big|\\
&\le C_{\delta} e^{-\kappa t}\vvvert\varphi\vvvert_0 \bigl(1\wedge (t-2\tau)\bigr)^{-\frac12}\Bigl(\sqrt{\tau}|h|+\bigl(1\wedge (t-2\tau)\bigr)^{-\frac12+\delta}|\IL^{-\frac12+\delta}h|\Bigr)
\end{aligned}
\end{equation}
for all $x,h\in H$, with $\kappa=\frac{\log(1+\tau_0\lambda_1)}{\tau\lambda_1}(\lambda_1-\Lf)>0$.
\end{lemma}

\begin{proof}[Proof of Lemma~\ref{lem:utausutausstar-regularity}]
Like the proof of Proposition~\ref{propo:utau-regularity}, the inequality~\eqref{eq:utausustausstar-regularity} for $u^{\tau,\s}$ and $u_\star^{\tau,\s}$ is obtained by the combination of three estimates using the semigroup property~\eqref{eq:Ptausstar}. For the function $u^{\tau,\s}$, the three estimates are of the type
\begin{align}
\big|Du^{\tau,\s}(t,x).h\big|&\le C_\alpha\vvvert\varphi\vvvert_1\Bigl(\sqrt{\tau}|h|+(t-\tau)^{-\alpha}|\Lambda^{-\alpha}h|\Bigr),\quad t\in(0,1], x,h\in H,\label{eq:utaus-alpha}\\
\vvvert u^{\tau,\s}(t,\cdot)\vvvert_1&\le e^{-\kappa t} \vvvert\varphi\vvvert_1\label{eq:utaus-ergo},\quad t\ge 0,\\
\vvvert u^{\tau,\s}(t,\cdot)\vvvert_1&\le \frac{C}{(t-\tau)^{\frac12}}\vvvert\varphi\vvvert_0,\quad t\in(\tau,1],\label{eq:utaus-0}
\end{align}
where $\alpha\in[0,\frac12)$, and the function $\varphi$ is assumed to be of class $\mathcal{C}^1$ with bounded derivative in~\eqref{eq:utaus-alpha} and~\eqref{eq:utaus-ergo}, and bounded and continuous in~\eqref{eq:utaus-0}. Similarly, for the function $u_\star^{\tau,\s}$, the three estimates are of the type
\begin{align}
\big|Du_\star^{\tau,\s}(t,x).h\big|&\le C_\alpha\vvvert\varphi\vvvert_1\Bigl(\sqrt{\tau}|h|+(t-\tau)^{-\alpha}|\Lambda^{-\alpha}h|\Bigr),\quad t\in(0,1], x,h\in H,\label{eq:utausstar-alpha}\\
\vvvert u_\star^{\tau,\s}(t,\cdot)\vvvert_1&\le e^{-\kappa t} \vvvert\varphi\vvvert_1\label{eq:utausstar-ergo},\quad t\ge 0\\
\vvvert u_\star^{\tau,\s}(t,\cdot)\vvvert_1&\le \frac{C}{(t-\tau)^{\frac12}}\vvvert\varphi\vvvert_0,\quad t\in(\tau,1],\label{eq:utausstar-0}
\end{align}
with same conditions as above.

To prove the inequalities~\eqref{eq:utaus-alpha} and~\eqref{eq:utaus-ergo}, the following expression for the first order derivative $Du^{\tau,\s}(t,x).h$ is used, when $\varphi$ is of class $\mathcal{C}^1$ with bounded derivative: one has
\[
Du^{\tau,\s}(t,x).h=\E_x[D\varphi(\IX^{\tau,\s}(t)).\xi_{\tau}^h(t)],
\]
where $\bigl(\xi_\tau^h(t)\bigr)_{t\ge 0}$ is solution of
\[
d\xi_\tau^h(t)=-\IL_\tau\xi_\tau^h(t)dt+Q_\tau DF(\IX^{\tau,\s}(t)).\xi_\tau^h(t),
\]
with initial value $\xi_\tau^h(0)=h$. Similarly, to prove the inequalities~\eqref{eq:utausstar-alpha} and~\eqref{eq:utausstar-ergo}, the following expression for the first order derivative $Du_\star^{\tau,\s}(t,x).h$ is used, when $\varphi$ is of class $\mathcal{C}^1$ with bounded derivative:
\[
Du_\star^{\tau,\s}(t,x).h=\E_x[D\varphi(\IX_\star^{\tau,\s}(t)).\xi_{\tau,\star}^h(t)],
\]
where $\bigl(\xi_{\tau,\star}^h(t)\bigr)_{t\ge 0}$ is solution of
\[
d\xi_{\tau,\star}^h(t)=-\IL_\tau\xi_{\tau,\star}^h(t)dt+R_\tau DF(\IX_\star^{\tau,\s}(t)).\xi_{\tau,\star}^h(t),
\]
with initial value $\xi_{\tau,\star}^h(0)=h$. The proofs of the inequalities~\eqref{eq:utaus-alpha} and~\eqref{eq:utaus-ergo} are identical to the proofs of~\eqref{eq:lem-utau1-alpha} from Lemma~\ref{lem:utau-1} and~\eqref{eq:lem-utau-ergo} from Lemma~\ref{lem:utau-ergo} respectively. Using the inequality~\eqref{eq:lemRtau-bound}, the proofs of the inequalities~\eqref{eq:utausstar-alpha} and~\eqref{eq:utausstar-ergo} follow from the same arguments. The details are omitted.

It remains to prove the inequalities~\eqref{eq:utaus-0} and~\eqref{eq:utausstar-0}. The proofs are not straightforward modifications of the proof of~\eqref{eq:lem-utau0} from Lemma~\ref{lem:utau-0}, therefore it is worth giving the details. The reason for the additional difficulties is the behavior of the unbounded linear operator $R_\tau^{-\frac12}=Q_\tau^{-\frac12}(I+\frac{\tau\IL}{2})^{\frac12}$ which differs from the behavior of the operator $Q_\tau^{-\frac12}$ appearing in the proof of Lemma~\ref{lem:utau-0}.

To prove the inequalities~\eqref{eq:utaus-0} and~\eqref{eq:utausstar-0}, the following expressions for the first order derivatives $Du^{\tau,\s}(t,x).h$ and $Du_\star^{\tau,\s}(t,x).h$ are used, when $\varphi$ is assumed to be only bounded and continuous: for all $t\in(0,\infty)$ and $x,h\in H$, one has
\begin{align*}
Du^{\tau,\s}(t,x).h&=\frac1t \E_x[\varphi(\IX^{\tau,\s}(t))\int_{0}^{t}\langle R_\tau^{-\frac12}\xi_\tau^h(s),dW(s)\rangle],\\
Du_\star^{\tau,\s}(t,x).h&=\frac1t \E_x[\varphi(\IX_\star^{\tau,\s}(t))\int_{0}^{t}\langle R_\tau^{-\frac12}\xi_{\tau,\star}^h(s),dW(s)\rangle].\\
\end{align*}

The proof of the inequality~\eqref{eq:utausstar-0} employs simpler arguments than the proof of the inequality~\eqref{eq:utaus-0}. Like the proof of Lemma~\ref{lem:utau-0}, two steps are required. First, observe that
\begin{align*}
\frac12\frac{d|R_\tau^{-\frac12}\xi_{\tau,\star}^h(t)|^2}{dt}&=-\langle \IL_\tau\xi_{\tau,\star}^h(t),R_\tau\xi_{\tau,\star}^h(t)\rangle+\langle DF(\IX^{\tau,\s}(t))\xi_{\tau,\star}^h(t),\xi_{\tau,\star}^h(t)\rangle\\
&=-\langle \IL(I+\frac{\tau\IL}{2})\xi_{\tau,\star}^h(t),\xi_{\tau,\star}^h(t)\rangle+\Lf|\xi_{\tau,\star}^h(t)|^2\\
&\le -(\lambda_1-\Lf)|\xi_\tau^h(t)|^2\le 0,
\end{align*}
using Assumption~\ref{ass:ergo}. Therefore, for all $\tau\in(0,\tau_0)$, $t\ge 0$ and $h\in H$, one has
\[
|R_\tau^{-\frac12}\xi_{\tau,\star}^h(t)|\le |R_\tau^{-\frac12}h|.
\]
Applying It\^o's formula and using the expression above for $Du_\star^{\tau,\s}(t,x).h$, one obtains the following inequality
\begin{equation}\label{eq:utausstar0-bad}
|D\IP_{\star,t}^{\tau,\s}\varphi(x).h|=|Du_\star^{\tau,\s}(t,x).h|\le \frac{\vvvert\varphi\vvvert_0}{\sqrt{t}}|R_\tau^{-\frac12}h|.
\end{equation}
Second, let $t\in(\tau,1]$. The semigroup property~\eqref{eq:Ptausstar} yields the identity
\[
u_\star^{\tau,\s}(t,\cdot)=\IP_{\tau,\star}^{\tau,\s}\bigl(\IP_{t-\tau,\star}^{\tau,\s}\varphi\bigr),
\]
which gives the equality
\[
Du_\star^{\tau,\s}(t,x).h=\E_x[D\IP_{t-\tau,\star}^{\tau,\s}\varphi(\IX_\star^{\tau,\s}(t)).\xi_{\tau,\star}^h(\tau)].
\]
Applying the inequality~\eqref{eq:utausstar0-bad} then gives
\[
|Du_\tau(t,x).h|\le \frac{\vvvert\varphi\vvvert_0}{\sqrt{t-\tau}}\E_x[|R_\tau^{-\frac12}\xi_{\tau,\star}^h(\tau)|].
\]
Using the mild formulation
\[
\xi_{\tau,\star}^h(\tau)=e^{-\tau\IL_\tau}h+\int_{0}^{\tau}e^{-(\tau-s)\IL_\tau}R_\tau DF(\IX_\star^{\tau,\s}(s))\xi_{\tau,\star}^h(s)ds,
\]
one obtains the inequality for all $\tau\in(0,\tau_0)$
\begin{align*}
|R_\tau^{-\frac12}\xi_{\tau,\star}^h(\tau)|&\le |R_\tau^{-\frac12}e^{-\tau\IL_\tau}h|+\Lf\int_{0}^{t}\|e^{-(\tau-s)\IL_\tau}R_\tau^{\frac12}\|_{\mathcal{L}(H)}|\xi_{\tau,\star}^h(s)|ds\\
&\le |R_\tau^{-\frac12}e^{-\tau\IL_\tau}h|+C|h|\\
&\le C|h|,
\end{align*}
using the upper bound $\|e^{-(\tau-s)\IL_\tau}R_\tau^{\frac12}\|_{\mathcal{L}(H)}\le \|R_\tau^{\frac12}\|_{\mathcal{L}(H)}\le 1$ (owing to~\eqref{eq:lemRtau-bound} and to~\eqref{eq:Q_tau-bound} from Lemma~\ref{lem:Q_tauIL_tau-bound}), the inequality~\eqref{eq:IL_tau-IL_tauR_tau1/2} and the inequality $|\xi_{\tau,\star}^h(s)|\le e^{-\kappa s}|h|$ (which is used in the proof of the inequality~\eqref{eq:utausstar-ergo}). As a consequence, one obtains 
\[
|Du_\star^{\tau,\s}(t,x).h|\le \frac{\vvvert\varphi\vvvert_0}{\sqrt{t-\tau}}|h|,
\]
for all $t\in(\tau,1]$ and $x,h\in H$. This concludes the proof of the inequality~\eqref{eq:utausstar-0}.

It remains to prove the inequality~\eqref{eq:utaus-0}. Like in the proof of the inequality~\eqref{eq:utausstar-0} given above, two steps are required but different arguments need to be used since $R_\tau\neq Q_\tau$. First, it is straightforward to prove that $|\xi_\tau^h(t)|\le e^{-\kappa t}|h|$ for all $t\ge 0$ and $\tau\in(0,\tau_0)$ (this inequality is used in the proof of the inequality~\eqref{eq:utaus-ergo}). Using the mild formulation
\[
\xi_\tau^h(t)=e^{-t\IL_\tau}h+\int_{0}^{t}e^{-(t-s)\IL_\tau}R_\tau DF(\IX^{\tau,\s})\xi_\tau^h(s)ds,
\]
one obtains, for all $t\in(0,1]$,
\begin{align*}
|R_\tau^{-\frac12}\xi_\tau^h(t)|&\le|R_\tau^{-\frac12}e^{-t\IL_\tau}h|+\Lf\int_{0}^{t}\big\|R_\tau^{-\frac12}e^{-(t-s)\IL_\tau}Q_\tau DF(\IX^{\tau,\s}(s))\|_{\mathcal{L}(H)}.|\xi_\tau^h(s)|ds\\
&\le |R_\tau^{-\frac12}h|+C\int_{0}^{t}(1+\frac{\sqrt{\tau}}{\sqrt{t-s}})ds|h|\\
&\le C|R_\tau^{-\frac12}h|+C|h|\\
&\le C|R_\tau^{-\frac12}h|,
\end{align*}
using the inequality $\|R_\tau^{\frac12}\|_{\mathcal{L}(H)}\le 1$ (as a consequence of the inequalities~\eqref{eq:lemRtau-bound} from Lemma~\ref{lem:Rtau} and~\eqref{eq:Q_tau-bound} from Lemma~\ref{lem:Q_tauIL_tau-bound}) in the last step. Applying It\^o's formula and using the expression above for $Du^{\tau,\s}(t,x).h$, one obtains the following inequality
\begin{equation}\label{eq:utaus0-bad}
|D\IP_{t}^{\tau,\s}\varphi(x).h|=|Du^{\tau,\s}(t,x).h|\le \frac{\vvvert\varphi\vvvert_0}{\sqrt{t}}|R_\tau^{-\frac12}h|.
\end{equation}
Second, let $t\in(\tau,1]$. The semigroup property~\eqref{eq:Ptausstar} yields the identity
\[
u^{\tau,\s}(t,\cdot)=\IP_{\tau}^{\tau,\s}\bigl(\IP_{t-\tau}^{\tau,\s}\varphi\bigr),
\]
which gives the equality
\[
Du^{\tau,s}(t,x).h=\E_x[D\IP_{t-\tau}^{\tau,\s}\varphi(\IX^{\tau,\s}(t)).\xi_{\tau}^h(\tau)].
\]
Using again the mild formulation and the same arguments as above, one obtains for all $\tau\in(0,\tau_0)$
\begin{align*}
|R_\tau^{-\frac12}\xi_\tau^h(\tau)|&\le |R_\tau^{-\frac12}e^{-\tau\IL_\tau}h|+C|h|\\
&\le C|h|,
\end{align*}
owing to the inequality~\eqref{eq:IL_tau-IL_tauR_tau1/2}. As a consequence, one obtains 
\[
|Du^{\tau,\s}(t,x).h|\le \frac{\vvvert\varphi\vvvert_0}{\sqrt{t-\tau}}|h|,
\]
for all $t\in(\tau,1]$ and $x,h\in H$. This concludes the proof of the inequality~\eqref{eq:utaus-0}.

Like in the proof of Proposition~\ref{propo:utau-regularity} (see Section~\ref{sec:proofs-weakinv}), the inequalities in~\eqref{eq:utausustausstar-regularity} are obtained by combining the three estimates~\eqref{eq:utaus-alpha},~\eqref{eq:utaus-ergo} and~\eqref{eq:utaus-0}, and the three estimates~\eqref{eq:utausstar-alpha},~\eqref{eq:utausstar-ergo} and~\eqref{eq:utausstar-0}, using the semigroup property~\eqref{eq:Ptausstar}. The details are omitted.

The proof of Lemma~\ref{lem:utausutausstar-regularity} is thus completed.
\end{proof}

\subsection{Proof of Theorem~\ref{theo:weakinv-standard}}\label{sec:standard-proof}

We are now in position to provide the proof of Theorem~\ref{theo:weakinv-standard}, using the auxiliary results presented in Subsections~\ref{sec:standard-aux} and~\ref{sec:standard-kolmogorov} above.

\begin{proof}[Proof of Theorem~\ref{theo:weakinv-standard}]
Like for the proof of Theorem~\ref{theo:weakinv} (see Section~\ref{sec:proofs-weakinv}), it suffices to establish the weak error estimate~\eqref{eq:theo-weakinv_weakerror-standard} for all functions $\varphi:H\to\R$ which are bounded and continuous. Indeed, this property then yields
\[
d_{\rm TV}(\rho_{X_N^{\tau,\s}},\mu_\star^\infty)=d_0(\rho_{X_N^{\tau,\s}},\mu_\star^\infty)\le C_{\delta}\Bigl(\tau^{\frac12-\delta}(1+|x_0|_{\frac14-\frac{\delta}{2}}^2)+e^{-\kappa N\tau}(1+|x_0|)\Bigr).
\]
As a consequence, the inequality~\eqref{eq:theo-weakinv_weakerror-standard} also holds for functions $\varphi$ which are bounded and measurable. Moreover, it suffices to let $N\to\infty$ to obtain the inequality~\eqref{eq:theo-weakinv_dTVinvar-standard}.

Let $\varphi$ be bounded and continuous, then the weak error can be decomposed as
\begin{align*}
\E[\varphi(X_N^{\tau,\s})]-\int\varphi d\mu_\star^\tau
&=\E[\varphi(\IX_N^{\tau,\s})]-\E[\varphi(\IX^{\tau,\s}(t_N)]\\
&+\E[\varphi(\IX^{\tau,\s}(t_N))]-\E[\varphi(\IX_\star^{\tau,\s}(t_N)]\\
&+\E[\varphi(\IX_\star^{\tau,\s}(t_N)]-\int\varphi d\mu_\star^\tau\\
&=\E[u^{\tau,\s}(0,\IX_N^{\tau,\s})]-\E[u^{\tau,\s}(t_N,\IX_0^{\tau,\s})]\\
&+\E[u_\star^{\tau,\s}(0,\IX^{\tau,\s}(t_N))]-\E[u_\star^{\tau,s}(t_N,\IX^{\tau,\s}(0))]\\
&+\E[u_\star^{\tau,s}(t_N,x_0)]-\int\varphi d\mu_\star^\tau,
\end{align*}
see Equation~\eqref{eq:decomperror-standard}, where the functions $u^{\tau,\s}$ and $u_\star^{\tau,\s}$ are defined by~\eqref{eq:utaus} and~\eqref{eq:utausstar} respectively. It thus suffices to prove the three following weak error estimates:
\begin{align}
\big|\E[u^{\tau,\s}(0,\IX_N^{\tau,\s})]-\E[u^{\tau,\s}(t_N,\IX_0^{\tau,\s})]\big|&\le C_\delta\tau^{\frac12-\delta}\vvvert\varphi\vvvert_0(1+|x_0|_{\frac14-\frac{\delta}{8}}^2),\label{eq:error-standard1}\\
\big|\E[u_\star^{\tau,\s}(0,\IX^{\tau,\s}(t_N))]-\E[u_\star^{\tau,s}(t_N,\IX^{\tau,\s}(0))]\big|&\le C_\delta\vvvert\varphi\vvvert_0\tau^{\frac12-\delta}(1+|x_0|_{2\delta}),\label{eq:error-standard2}\\
\big|\E[u_\star^{\tau,s}(t_N,x_0)]-\int\varphi d\mu_\star^\tau\big|&\le Ce^{-\kappa t_N}\vvvert\varphi\vvvert_0 (1+|x_0|)\label{eq:error-standard3}.
\end{align}
Above and in the sequel, $\delta\in(0,\frac12)$ is an arbitrarily small positive real number and $C_\delta\in(0,\infty)$ is independent of $\tau$, $N$ and $x_0$.

The proofs of the inequalities~\eqref{eq:error-standard1} and~\eqref{eq:error-standard3} employs the same argument as in the proof of Theorem~\ref{theo:weakinv}, however one needs additional arguments to prove the inequality~\eqref{eq:error-standard2}.

$\bullet$ Proof of the inequality~\eqref{eq:error-standard1}.

Recall that the auxiliary process $\bigl(\tilde{\IX}^{\tau,\s}(t)\bigr)_{t\ge 0}$ is defined by Equation~\eqref{eq:tildeXs}, and satisfies $\tilde{\IX}^{\tau,\s}(t_n)=\IX_n^{\tau,\s}$ for all $n\in\N_0$. Therefore, using a standard telescoping sum argument, one has
\begin{align*}
\E[u^{\tau,\s}(0,\IX_{N}^{\tau,\s})]&-\E[u^{\tau,\s}(N\tau,\IX_0^{\tau,\s})]\\
&=\E[u^{\tau,\s}(0,\tilde{\IX}^{\tau,\s}(t_N))]-\E[u^{\tau,\s}(N\tau,\tilde{\IX}^{\tau,\s}(0))]\\
&=\sum_{n=0}^{N-1}\Bigl(\E[u^{\tau,\s}(t_N-t_{n+1},\tilde{\IX}^{\tau,\s}(t_{n+1}))]-\E[u^{\tau,\s}(t_N-t_{n},\tilde{\IX}_{\tau,\s}(t_{n}))]\Bigr)\\
&=\sum_{n=0}^{N-1}e_n^{\tau,\s},
\end{align*}
with $e_n^{\tau,\s}=\E[u^{\tau,s}(t_N-t_{n+1},\tilde{\IX}^{\tau,\s}(t_{n+1}))]-\E[u^{\tau,\s}(t_N-t_{n},\tilde{\IX}^{\tau,\s}(t_{n}))]$ for all $n\in\{0,\ldots,N-1\}$. Applying It\^o's formula, using the expression~\eqref{eq:dtildeXs} for the evolution of the auxiliary process and the fact that $u^{\tau,\s}$ solves the Kolmogorov equation~\eqref{eq:IKolmogorovs}, one obtains, for all $n\in\N_0$, the expression
\[
e_n^{\tau,\s}=\int_{t_n}^{t_{n+1}}\E\bigl[Du^{\tau,\s}(t_N-t,\tilde{\IX}^{\tau,\s}(t)).\bigl(Q_{\tau}F(\tilde{\IX}^{\tau,\s}(t_n))-Q_{\tau}F(\tilde{\IX}^{\tau,\s}(t))\bigr)\bigr]dt.
\]
The cases $n\in\{0,N-2,N-1\}$ and $n\in\{1,\ldots,N-3\}$ are treated separately. On the one hand, using the inequality~\eqref{eq:utaus0-bad}, the bound~\eqref{eq:Q_tau-bound}, the Lipschitz continuity of $F$ and the moment bound~\eqref{eq:tildeIXs-bound-ergo} (with $\alpha=0$), one obtains
\begin{align*}
|e_0^{\tau,\s}|+|e_{N-2}^{\tau,\s}|+|e_{N-1}^{\tau,s}|&\le C\vvvert\varphi\vvvert_0\int_{0}^{\tau}(t_N-t)^{-\frac12}\E[\big|Q_\tau^{\frac12}\bigl(F(\tilde{\IX}^{\tau,\s}(t_n))-F(\tilde{\IX}^{\tau,\s}(t))\bigr)\big|]dt\\
&+C\vvvert\varphi\vvvert_0\int_{t_{N-2}}^{t_N-1}(t_N-t)^{-\frac12}\E[\big|Q_\tau^{\frac12}\bigl(F(\tilde{\IX}^{\tau,\s}(t_n))-F(\tilde{\IX}^{\tau,\s}(t))\bigr)\big|]dt\\
&+C\vvvert\varphi\vvvert_0\int_{t_{N-1}}^{t_N}(t_N-t)^{-\frac12}\E[\big|Q_\tau^{\frac12}\bigl(F(\tilde{\IX}^{\tau,\s}(t_n))-F(\tilde{\IX}^{\tau,\s}(t))\bigr)\big|]dt\\
&\le C\tau^{\frac12}\vvvert\varphi\vvvert_0(1+|x_0|).
\end{align*}
On the other hand, using the inequality~\eqref{eq:utausustausstar-regularity} from Lemma~\ref{lem:utausutausstar-regularity}, with $\alpha=\frac12-\frac{\delta}{4}$, for all $n\in\{1,\ldots,N-3\}$, one obtains
\[
|e_n^{\tau,\s}|\le {\bf e}_{n,1}^{\tau,\s}+{\bf e}_{n,2}^{\tau,\s}
\]
where the error terms on the right-hand side above are defined by
\begin{align*}
{\bf e}_{n,1}^{\tau,\s}&=C\sqrt{\tau}\vvvert\varphi\vvvert_0\int_{t_n}^{t_{n+1}}\frac{e^{-\kappa(t_N-t)}}{(t_{N-2}-t)^{\frac12}}\E[|Q_\tau\bigl(F(\tilde{\IX}^{\tau,\s}(t_n))-F(\tilde{\IX}^{\tau,\s}(t))\bigr)|]dt\\
{\bf e}_{n,2}^{\tau,\s}&=C_\delta\vvvert\varphi\vvvert_0\int_{t_n}^{t_{n+1}}\frac{e^{-\kappa(t_N-t)}}{(t_{N-2}-t)^{1-\frac{\delta}{4}}}\E[|\IL^{-\frac12+\frac{\delta}{4}}Q_\tau\bigl(F(\tilde{\IX}^{\tau,\s}(t_n))-F(\tilde{\IX}^{\tau,\s}(t))\bigr)|]dt.
\end{align*}
Using the bound~\eqref{eq:Q_tau-bound}, the Lipschitz continuity of $F$ and the moment bound~\eqref{eq:tildeIXs-bound-ergo} (with $\alpha=0$), one obtains
\[
\sum_{n=1}^{N-3}{\bf e}_{n,1}^{\tau,\s}\le C\sqrt{\tau}\vvvert\varphi\vvvert_0\int_{\tau}^{t_{N-2}}\frac{e^{-\kappa(t_{N-2}-t)}}{(t_{N-2}-t)^{\frac12}}dt(1+|x_0|)\le C\sqrt{\tau}\vvvert\varphi\vvvert_0\int_0^\infty\frac{e^{-\kappa t}}{t^{\frac12}}dt(1+|x_0|).
\]
The treatment of the error term ${\bf e}_{n,2}^{\tau,\s}$ exploits Assumption~\ref{ass:Fregul1} on the regularity of the nonlinearity $F$: using the bound~\eqref{eq:Q_tau-bound} and the Cauchy--Schwarz inequality, one obtains
\begin{align*}
{\bf e}_{n,2}^{\tau,\s}&\le C_\delta\vvvert\varphi\vvvert_0\int_{t_n}^{t_{n+1}}\frac{e^{-\kappa(t_N-t)}}{(t_{N-2}-t)^{1-\frac{\delta}{8}}}\bigl(\E[\bigl(1+|\tilde{\IX}^{\tau,\s}(t)|_{\frac{1}{4}-\frac{\delta}{8}}^2+|\tilde{\IX}^{\tau,\s}(t_n)|_{\frac{1}{4}-\frac{\delta}{8}}^2\bigr)]\bigr)^{\frac12}\\
&\hspace{6cm}\bigl(\E[\big|\IL^{-\frac14+\frac{\delta}{2}}(\tilde{\IX}^{\tau,\s}(t_n)-\tilde{\IX}^{\tau,\s}(t))\big|]\bigr)^{\frac12} dt.
\end{align*}
Using the inequalities~\eqref{eq:tildeIXs-bound-ergo} and~\eqref{eq:tildeIXs-increment-ergo} from Lemma~\ref{lem:tildeIXs} then yields the upper bound
\[
\sum_{n=1}^{N-3}{\bf e}_{n,2}^{\tau,\s}\le C_\delta\tau^{\frac12-\delta}\vvvert\varphi\vvvert_0\int_0^\infty\frac{e^{-\kappa t}}{t^{1-\frac{\delta}{8}}}dt(1+|x_0|_{\frac14-\frac{\delta}{8}}^2).
\]
Gathering the estimates, one obtains the inequality~\eqref{eq:error-standard1}.

$\bullet$ Proof of the inequality~\eqref{eq:error-standard2}.

Applying It\^o's formula and using the fact that the function $u_\star^{\tau,\s}$ is solution of the Kolmogorov equation~\eqref{eq:IKolmogorovsstar}, for all $T\in(2\tau_0,\infty)$, one has
\begin{align*}
\E[u_\star^{\tau,\s}(0,\IX^{\tau,\s}(T))]&-\E[u_\star^{\tau,s}(T,\IX^{\tau,\s}(0))]\\
&=\int_0^T\E[\langle  Du_\star^{\tau,s}(T-t,\IX^{\tau,\s}(t)),(Q_\tau-R_\tau)F(\IX^{\tau,\s}(t))\rangle]dt\\
&=\int_{T-2\tau}^T\E[\langle  Du_\star^{\tau,s}(T-t,\IX^{\tau,\s}(t)),(Q_\tau-R_\tau)F(\IX^{\tau,\s}(t))\rangle]dt\\
&+\int_0^{T-2\tau}\E[\langle  Du_\star^{\tau,s}(T-t,\IX^{\tau,\s}(t)),(Q_\tau-R_\tau)F(\IX^{\tau,\s}(t))\rangle]dt.
\end{align*}

On the one hand, using the inequality~\eqref{eq:utausstar0-bad} (see the proof of Lemma~\ref{lem:utausutausstar-regularity} in Subsection~\ref{sec:standard-kolmogorov}), the inequality~\eqref{eq:QtauRtauRtau1/2} from Lemma~\ref{lem:Rtau}, one has
\begin{align*}
\Big|\int_{T-2\tau}^T&\E[\langle  Du_\star^{\tau,s}(T-t,\IX^{\tau,\s}(t)),(Q_\tau-R_\tau)F(\IX^{\tau,\s}(t))\rangle]dt\Big|\\
&\le \int_{T-2\tau}^{T}\frac{\vvvert\varphi\vvvert_0}{\sqrt{T-t}}\E[|R_\tau^{-\frac12}(Q_\tau-R_\tau)F(\IX^{\tau,\s}(t))|]dt\\
&\le C_\delta\tau^\delta\int_{T-2\tau}^{T}\frac{\vvvert\varphi\vvvert_0}{\sqrt{T-t}}\E[|\IL^\delta F(\IX^{\tau,\s}(t))|]dt.
\end{align*}
Using then the additional regularity condition~\eqref{eq:Fregul3} on the nonlinearity $F$, the moment bound~\eqref{eq:modified-standard-bound} from Lemma~\ref{lem:modifieds}, and the inequality
\[
\int_{T-2\tau}^{T}(T-t)^{-\frac12}dt\le \int_0^{2\tau}t^{-\frac12}dt\le 2\sqrt{2\tau},
\]
one obtains the upper bound
\[
\Big|\int_{T-2\tau}^T\E[\langle  Du_\star^{\tau,s}(T-t,\IX^{\tau,\s}(t)),(Q_\tau-R_\tau)F(\IX^{\tau,\s}(t))\rangle]dt\Big|\le C_\delta\tau^{\frac12+\delta}\vvvert\varphi\vvvert_0(1+|x_0|_{2\delta}).
\]

On the other hand, using the inequality~\eqref{eq:utausustausstar-regularity} from Lemma~\ref{lem:utausutausstar-regularity} (with $\alpha=\frac12-\delta$), one obtains
\begin{align*}
\Big|\int_0^{T-2\tau}&\E[\langle  Du_\star^{\tau,s}(T-t,\IX^{\tau,\s}(t)),(Q_\tau-R_\tau)F(\IX^{\tau,\s}(t))\rangle]dt\Big|\\
&\le \sqrt{\tau}\int_0^{T-2\tau}\frac{C_\delta e^{-\kappa (T-t)}}{(T-t-2\tau)^{\frac12}}\E[|(Q_\tau-R_\tau)F(\IX^{\tau,\s}(t))|]dt\\
&+\int_0^{T-2\tau}\frac{C_\delta e^{-\kappa (T-t)}}{(T-t-2\tau)^{1-\delta}}\E[|\IL^{-\frac12+\delta}(Q_\tau-R_\tau)F(\IX^{\tau,\s}(t))|]dt\\
&\le C\vvvert\varphi\vvvert_0\sqrt{\tau}\int_{0}^{\infty}\frac{e^{-\kappa t}}{t^{\frac12}}dt(1+\underset{s\ge 0}\sup~\E[|\IX^{\tau,\s}(s)|])\\
&+C_\delta\vvvert\varphi\vvvert_0\tau^{\frac12-\delta}\int_{0}^{\infty}\frac{e^{-\kappa t}}{t^{1-\delta}}dt(1+\underset{s\ge 0}\sup~\E[|\IX^{\tau,\s}(s)|])
\end{align*}
using the inequality~\eqref{eq:QtauRtau} and the Lipschitz continuity of $F$ in the last step. Using the moment bound~\eqref{eq:modified-standard-bound} from Lemma~\ref{lem:modifieds}, one finally obtains
\[
\Big|\int_0^{T-2\tau}\E[\langle  Du_\star^{\tau,s}(T-t,\IX^{\tau,\s}(t)),(Q_\tau-R_\tau)F(\IX^{\tau,\s}(t))\rangle]dt\Big|\le C_\delta\vvvert\varphi\vvvert_0\tau^{\frac12-\delta}(1+|x_0|).
\]

Gathering the two estimates gives
\[
\big|\E[u_\star^{\tau,\s}(0,\IX^{\tau,\s}(T))]-\E[u_\star^{\tau,s}(T,\IX^{\tau,\s}(0))]\big|\le C_\delta\vvvert\varphi\vvvert_0\tau^{\frac12-\delta}(1+|x_0|_{2\delta}).
\]
and concludes the proof of the inequality~\eqref{eq:error-standard2}.

$\bullet$ Proof of the inequality~\eqref{eq:error-standard3}.

Recall that $\mu_\star^\tau$ is defined by Equation~\eqref{eq:mu_startau} in Section~\ref{sec:results_standard}, and is the invariant distribution of the modified stochastic evolution equation~\eqref{eq:modifiedSPDEstandardstar}, when the nonlinearity $F$ satisfies Assumption~\ref{ass:gradient}, i.\,e. $F=-DV$. Let $\IX_\star^{\tau}$ be a $H$-valued random variable with distribution $\rho_{\IX_\star^\tau}=\mu_\star^\tau$ and assume that it is independent of the Wiener process $\bigl(W(t)\bigr)_{t\ge 0}$. Note that one has
\[
\underset{\tau\in(0,\tau_0)}\sup~\E[|\IX_\star^\tau|]=\underset{\tau\in(0,\tau_0)}\sup~\int |x|d\mu_\star^\tau(x)<\infty,
\]
using the condition $\Lf<\lambda_1\le \lambda_1(1+\frac{\tau\lambda_1}{2})$ from Assumption~\ref{ass:ergo}. Since $\mu\star^\tau$ is the unique invariant distribution of the process $\bigl(\IX_\star^{\tau,\s}(t)\bigr)_{t\ge 0}$, one has
\[
\E[u_\star^{\tau,\s}(t_N,\IX_\star^\tau)]=\E[u_\star^{\tau,\s}(0,\IX_\star^\tau)]=\int\varphi d\mu_\star^\tau
\]
for all $N\in\N$ and $\tau\in(0,\tau_0)$. As a consequence, when $t_N=N\tau\ge 2\tau_0$, one has
\begin{align*}
\big|\E[\varphi(\IX_\star^{\tau,\s}(t_N))]-\int\varphi d\mu_\star^\tau\big|&=\big|u_\star^{\tau,\s}(t_N,x_0)-\E[u_\star^{\tau,\s}(t_N,\IX_\star)]\big|\\
&\le \vvvert u_\star^{\tau,\s}(t_N,\cdot)\vvvert_1 \E[|x_0-\IX_\star|]\\
&\le Ce^{-\kappa t_N}\vvvert\varphi\vvvert_0 (1+|x_0|),
\end{align*}
using the inequality~\eqref{eq:utausustausstar-regularity} from Lemma~\ref{lem:utausutausstar-regularity} with $\delta=\frac12$. This concludes the proof of the inequality~\eqref{eq:error-standard3}.

$\bullet$ Combining the three inequalities~\eqref{eq:error-standard1},~\eqref{eq:error-standard2} and~\eqref{eq:error-standard3}, one obtains the weak error estimate~\eqref{eq:theo-weakinv_weakerror-standard} for all functions $\varphi:H\to\R$ which are bounded and continuous. Owing to the arguments given above, the proof of Theorem~\ref{theo:weakinv-standard} is thus completed.
\end{proof}

\section{Applications and extensions of the modified Euler scheme}\label{sec:extensions}

This section is devoted to the presentation of two applications of the proposed modified Euler scheme~\eqref{eq:scheme-intro}, which aim to illustrate again the superiority of that integrator compared with the standard Euler scheme~\eqref{eq:scheme-standard-intro}. In Subsection~\ref{sec:AP}, an asymptotic preserving scheme is provided to approximate the slow component of a multiscale stochastic evolution system, where the application of the modified Euler scheme is able to capture the so-called averaged coefficient which governs the behavior of the limiting evolution equation (averaging principle). In Subsection~\ref{sec:MCMC}, the modified Euler scheme is employed as a proposal transition kernel for a Markov Chain Monte Carlo method to approximate the Gibbs distribution~\eqref{eq:mu_star}, which is shown to be well defined and to have a spectral gap in infinite dimension. Using the accelerated exponential Euler scheme would give similar results, however the standard Euler scheme fails in both situations. Finally, the range of application of the modified Euler scheme is extended in Subsection~\ref{sec:generalizations}, to encompass stochastic evolution equations with non-globally Lipschitz drift, with multiplicative noise or with colored noise, possibly in higher dimension. The validity of the main results in those situations is discussed, however precise statements and proofs are omitted.

Note that the two applications presented in subsections~\ref{sec:AP} and~\ref{sec:MCMC} are direct consequences of the preservation of the Gaussian invariant distribution $\nu$ in the Ornstein--Uhlenbeck case ($F=0$) when using the modified Euler scheme, see Proposition~\ref{propo:invarGauss}.

\subsection{Asymptotic preserving scheme}\label{sec:AP}

We refer to~\cite{B} for a more detailed analysis of the problem discussed in this section.

The modified Euler scheme can be applied to design an efficient integrator for the slow-fast SPDE system
\begin{equation}\label{eq:SPDE-slowfast}
\left\lbrace
\begin{aligned}
d\XX^\epsilon(t)&=-\IL\XX^\epsilon(t)dt+G\bigl(\XX^\epsilon(t),\YY^\epsilon(t)\bigr)dt\\
d\YY^\epsilon(t)&=-\frac{1}{\epsilon}\IL\YY^\epsilon(t)dt+\frac{\sigma(\XX^\epsilon(t))}{\sqrt{\epsilon}}dW(t),
\end{aligned}
\right.
\end{equation}
where $\epsilon\in(0,\epsilon_0)$ is a small parameter, and where $G:H\times H\to H$ and $\sigma:H\to \R$ are bounded and globally Lipschitz continuous mappings. The linear operator $\IL$ and the cylindrical Wiener process $\bigl(W(t)\bigr)_{t\ge 0}$ satisfy the conditions presented in Section~\ref{sec:setting}. Moreover, the initial values $x_0=\XX^\epsilon(0)$ and $y_0=\YY^\epsilon(0)$ are assumed to be deterministic, and independent of the parameter $\epsilon$.

If the mapping $\sigma$ is constant, the fast component $\YY^\epsilon$ is a $H$-valued Ornstein--Uhlenbeck process, which does not depend on the slow component $\XX^\epsilon$. In that case one has the equality in distribution
\[
\bigl(\YY^\epsilon(t)\bigr)_{t\ge 0}=\bigl(\YY(\frac{t}{\epsilon})\bigr)_{t\ge 0}
\]
for all $\epsilon\in(0,\epsilon_0)$, where $\YY$ is the solution of the stochastic evolution equation
\[
d\YY(t)=-\IL\YY(t)dt+dW(t).
\]
The $H$-valued Ornstein--Uhlenbeck process $\YY$ is ergodic, and its unique invariant distribution is the Gaussian distribution $\nu$ (see~\eqref{eq:nu} from Section~\ref{sec:invariant}). When $\sigma$ is not constant, one needs to consider the invariant distribution $\nu_x=\mathcal{N}(0,\frac{\sigma(x)^2}{2}\IL^{-1})$ of the stochastic evolution equation
\[
d\YY_x(t)=-\IL\YY_x(t)dt+\sigma(x)dW(t)
\]
with frozen slow component $x\in H$.

In this section, we study the behavior of the system~\eqref{eq:SPDE-slowfast} and of numerical schemes in the regime $\epsilon\to 0$. On the one hand, to approximate $\bigl(\XX^\epsilon(T),\YY^\epsilon(T)\bigr)$, in the strong or weak sense, applying an integrator with time-step size $\tau=\frac{T}{N}$, the presence of the small parameter $\epsilon$ in the fast evolution equation requires to choose $\tau={\rm o}(\epsilon)$. On the other hand, the slow component $\XX^\epsilon$ satisfies the averaging principle: $\XX^\epsilon$ converges to $\overline{\XX}$ when $\epsilon\to 0$, where $\overline{\XX}$ is solution of the averaged equation
\begin{equation}\label{eq:SPDEaveraged}
d\overline{\XX}(t)=-\IL\overline{\XX}(t)dt+\overline{G}(\overline{\XX}(t)),
\end{equation}
with initial value $\overline{\XX}(0)=x_0$, where for all $x\in H$
\begin{equation}\label{eq:averagedG}
\overline{G}(x)=\int G(x,y)d\nu_x(y)=\E_{\YY\sim\nu_x}[G(x,\YY_x)]=\E_{\YY\sim\nu}[G(x,\sigma(x)\YY)].
\end{equation}
The convergence above holds in both the strong and weak senses for the system considered in this section. See for instance~\cite{MR2480788} for convergence results and~\cite{B:2012} for strong and weak rates of convergence. As a result of the averaging principle, if the objective is to approximate only the slow component $\XX^\epsilon$, the time-step size restriction $\tau={\rm o}(\epsilon)$ may not be necessary.

The observations above lead to the following question: is it possible to design numerical schemes which are efficient in both regimes $\epsilon\in(0,\epsilon_0)$ and $\epsilon\to 0$, and which do not impose time-step size restrictions? An answer is provided by asymptotic preserving schemes. Rougly, a numerical scheme $\bigl(\XX_n^{\epsilon,\tau},\YY_n^{\epsilon,\tau})_{n\ge 0}$ is asymptotic preserving scheme when the following conditions are satisfied.
\begin{itemize}
\item For any $\epsilon\in(0,\epsilon_0)$, the scheme is consistent with the system~\eqref{eq:SPDE-slowfast}, namely one has $\XX_N^{\epsilon,\tau}\underset{\tau\to 0}\to\XX^\epsilon(T)$ and $\YY_N^{\epsilon,\tau}\underset{\tau\to 0}\to\YY^\epsilon(T)$ (with the condition $T=N\tau$).
\item For any $\tau\in(0,\tau_0)$, there exists a limiting scheme, namely $\XX_n^{\epsilon,\tau}\underset{\epsilon\to 0}\to \XX_n^{0,\tau}$, where the sequence $\bigl(\XX_{n}^{0,\tau}\bigr)_{n\ge 0}$ is a Markov chain.
\item The limiting scheme is consistent with the averaged equation~\eqref{eq:SPDEaveraged}, meaning that one has $\XX_N^{0,\tau}\underset{\tau\to 0}\to \overline{\XX}(T)$ (with the condition $T=N\tau$).
\end{itemize}
The convergence results above may hold either in weak and strong senses. The asymptotic preserving property can be described by the following commutative diagram
\[
\begin{CD}
\XX_N^{\epsilon,\tau}     @>{N \to \infty}>> \XX^\epsilon(T) \\
@VV{\epsilon\to 0}V        @VV{\epsilon\to 0}V\\
\XX_N^{0,\tau}     @>{N \to \infty}>> \overline{\XX}(T)
\end{CD}
\]
We refer to~\cite{BR} for the introduction of asymptotic preserving schemes for (finite dimensional) SDEs. For a more complete overview of asymptotic preserving schemes for deterministic and stochastic systems, we refer to~\cite{B}.

The main result of this section is to check that the following numerical scheme is asymptotic preserving, where the fast Ornstein--Uhlenbeck component is discretized using the modified Euler scheme.
\begin{theo}\label{theo:AP}
Let $\tau\in(0,\tau_0)$ be an arbitrary time-step size. For all $n\in\N_0$ and all $\epsilon\in(0,\epsilon_0)$, set
\begin{equation}\label{eq:APscheme}
\left\lbrace
\begin{aligned}
\XX_{n+1}^{\epsilon,\tau}&=\IA_\tau\bigl(\XX_n^{\epsilon,\tau}+\tau G(\XX_n^{\epsilon,\tau},\YY_{n+1}^{\epsilon,\tau})\bigr)\\
\YY_{n+1}^{\epsilon,\tau}&=\IA_{\frac{\tau}{\epsilon}}\YY_n^{\epsilon,\tau}+\sigma(\XX_n^{\epsilon,\tau})\sqrt{\frac{\tau}{\epsilon}}\IB_{\frac{\tau}{\epsilon},1}\Gamma_{n,1}+\sigma(\XX_n^{\epsilon,\tau})\sqrt{\frac{\tau}{\epsilon}}\IB_{\frac{\tau}{\epsilon},2}\Gamma_{n,2},
\end{aligned}
\right.
\end{equation}
where the linear operators $\IA_\tau$, $\IA_{\frac{\tau}{\epsilon}}$, $\IB_{\frac{\tau}{\epsilon},1}$ and $\IB_{\frac{\tau}{\epsilon},2}$ are defined by~\eqref{eq:operators}, and $\bigl(\Gamma_{n,1}\bigr)_{n\ge 0}$ and $\bigl(\Gamma_{n,2}\bigr)_{n\ge 0}$ are two independent sequences of independent cylindrical Gaussian random variables.

The scheme~\eqref{eq:APscheme} is asymptotic preserving, more precisely the following results hold.
\begin{enumerate}
\item[$(i)$] The limiting scheme is given by
\begin{equation}\label{eq:limitingscheme}
\XX_{n+1}^{0,\tau}=\IA_\tau\Bigl(\XX_n^{0,\tau}+\tau G\bigl(\XX_n^{0,\tau},\sigma(\XX_n^{0,\tau})Q^{\frac12}\Gamma_n\bigr)\Bigr),
\end{equation}
with initial value $\XX_0^{0,\tau}=x_0$, where $Q=\frac12\IL^{-1}$ is the covariance operator of the Gaussian distribution $\nu$ and where $\bigl(\Gamma_n\bigr)_{n\ge 0}$ is a sequence of independent cylindrical Gaussian random variables. Precisely, one has
\begin{equation}\label{eq:cvlimitingscheme}
\underset{\epsilon\to 0}\lim~\E[\varphi(\XX_{n}^{\epsilon,\tau})]=\E[\varphi(\XX_{n}^{0,\tau})]
\end{equation}
for any Lipschitz continuous function $\varphi:H\to \R$ and all $n\in\{0,\ldots,N\}$.
\item[$(ii)$] The limiting scheme is consistent with the averaged equation, in the sense of convergence in distribution, with weak order $1$: for all $\delta\in(0,1)$ and all $T\in(0,\infty)$, there exists $C_\delta(T,x_0)\in(0,\infty)$ such that
\begin{equation}\label{eq:consistencylimitingscheme}
\big|\E[\varphi(\XX_N^{0,\tau})]-\varphi(\overline{\XX}(T))\big|\le C_\delta(T,x_0)\bigl(\vvvert\varphi\vvvert_1+\vvvert\varphi\vvvert_2\bigr)\tau^{1-\delta}.
\end{equation}
\end{enumerate}
\end{theo}

Note that the requirement that the scheme~\eqref{eq:APscheme} is consistent with the system~\eqref{eq:SPDE-slowfast} for any $\epsilon\in(0,\epsilon_0)$ is not stated in Theorem~\ref{theo:AP}, in order to focus on the most relevant properties of the scheme. For completeness, let us state the associated weak error estimate: for all $\delta\in(0,\frac12)$, $\epsilon\in(0,\epsilon_0)$ and $T\in(0,\infty)$, there exists $C_{\delta}(\epsilon,T,x_0,y_0)\in(0,\infty)$ such that for any function $\varphi:H\times H\to\R$ of class $\mathcal{C}^2$, one has
\[
\big|\E[\varphi(\XX_N^{\epsilon,\tau},\YY_N^{\epsilon,\tau})]-\E[\varphi(\XX^{\epsilon}(T)),\YY^{\epsilon}(T))]|\le C_{\delta}(\epsilon,T,x_0,y_0)\vvvert\varphi\vvvert_2\tau^{\frac12-\delta}.
\]
The proof of the weak error estimate above is omitted, it would use the same techniques as the proof of Theorem~\ref{theo:weak}. Due to the discretization of the fast component, one has $C_{\delta}(\epsilon,T,x_0,y_0)\underset{\epsilon\to 0}\to\infty$, therefore the weak error estimate above is not uniform with respect to the parameter $\epsilon$. For details, see~\cite[Proposition~3.6]{B}.

Let us provide the proof of Theorem~\ref{theo:AP}.
\begin{proof}[Proof of Theorem~\ref{theo:AP}]
$\bullet$ Proof of $(i)$.

It is convenient to employ the second interpretation of the modified Euler scheme, see Section~\ref{sec:scheme-2nd}: recall (see~\eqref{eq:distribution_increments}) that one has the equality in distribution 
\[
\IB_{\frac{\tau}{\epsilon},1}\Gamma_{n,1}+\IB_{\frac{\tau}{\epsilon},2}\Gamma_{n,2}=\IB_{\frac{\tau}{\epsilon}}\Gamma_{n}
\]
where the linear operator $\IB_{\frac{\tau}{\epsilon}}$ is defined by~\eqref{eq:IB}. As a consequence, one has the equality in distribution
\[
\bigl(\XX_n^{\epsilon,\tau},\YY_n^{\epsilon,\tau}\bigr)_{n\ge 0}=\bigl(\hat{\XX}_n^{\epsilon,\tau},\hat{\YY}_n^{\epsilon,\tau}\bigr)_{n\ge 0}
\]
where the scheme $\bigl(\hat{\XX}_n^{\epsilon,\tau},\hat{\YY}_n^{\epsilon,\tau}\bigr)_{n\ge 0}$ is defined by
\begin{equation}\label{eq:APscheme2nd}
\left\lbrace
\begin{aligned}
\hat{\XX}_{n+1}^{\epsilon,\tau}&=\IA_\tau\bigl(\hat{\XX}_n^{\epsilon,\tau}+\tau G(\hat{\XX}_n^{\epsilon,\tau},\hat{\YY}_{n+1}^{\epsilon,\tau})\bigr)\\
\hat{\YY}_{n+1}^{\epsilon,\tau}&=\IA_{\frac{\tau}{\epsilon}}\hat{\YY}_n^{\epsilon,\tau}+\sigma(\hat{\XX}_{n}^{\epsilon,\tau})\sqrt{\frac{\tau}{\epsilon}}\IB_{\frac{\tau}{\epsilon}}\Gamma_{n},
\end{aligned}
\right.
\end{equation}
with initial values $\hat{\XX}_0^{\epsilon,\tau}=x_0$ and $\hat{\YY}_0^{\epsilon,\tau}=y_0$. To prove the claim~\eqref{eq:cvlimitingscheme}, it suffices to prove that
\[
\E[|\hat{\XX}_n^{\epsilon,\tau}-\XX_n^{0,\tau}|]\underset{\epsilon\to 0}\to 0
\]
for all $n\in\N$.

Note that for all $n\in\N$, one has
\begin{align*}
\hat{\XX}_n^{\epsilon,\tau}&=\IA_\tau^n x_0+\tau\sum_{k=0}^{n-1}\IA_\tau^{n-k}G(\hat{\XX}_k^{\epsilon,\tau},\hat{\YY}_{k+1}^{\epsilon,\tau}),\\
{\XX}_n^{0,\tau}&=\IA_\tau^n x_0+\tau\sum_{k=0}^{n-1}\IA_\tau^{n-k}G({\XX}_k^{0,\tau},\sigma({\XX}_n^{0,\tau})Q^{\frac12}\Gamma_k),
\end{align*}
therefore for all $n\in\N$, one has
\begin{align*}
\E[|\hat{\XX}_n^{\epsilon,\tau}-\XX_n^{0,\tau}|]&\le \tau\sum_{k=0}^{n-1}\E[|G(\hat{\XX}_k^{\epsilon,\tau},\hat{\YY}_{k+1}^{\epsilon,\tau})-G({\XX}_k^{0,\tau},\sigma({\XX}_k^{0,\tau})Q^{\frac12}\Gamma_k)|]\\
&\le \tau\sum_{k=0}^{n-1}\E[|G(\hat{\XX}_k^{\epsilon,\tau},\hat{\YY}_{k+1}^{\epsilon,\tau})-G({\XX}_k^{0,\tau},\hat{\YY}_{k+1}^{\epsilon,\tau})|]\\
&+\tau\sum_{k=0}^{n-1}\E[|G({\XX}_k^{0,\tau},\hat{\YY}_{k+1}^{\epsilon,\tau})-G({\XX}_k^{0,\tau},\sigma(\hat{\XX}_k^{\epsilon,\tau})Q^{\frac12}\Gamma_k)|]\\
&+\tau\sum_{k=0}^{n-1}\E[|G({\XX}_k^{0,\tau},\sigma(\hat{\XX}_k^{\epsilon,\tau})Q^{\frac12}\Gamma_k)-G({\XX}_k^{0,\tau},\sigma({\XX}_k^{0,\tau})Q^{\frac12}\Gamma_k)|]\\
&\le C\tau\sum_{k=0}^{n-1}\E[|\hat{\XX}_k^{\epsilon,\tau}-\XX_k^{0,\tau}|]\\
&+C\tau\sum_{k=0}^{n-1}\E[|\hat{\YY}_{k+1}^{\epsilon,\tau}-\sigma(\hat{\XX}_k^{\epsilon,\tau})Q^{\frac12}\Gamma_k|],
\end{align*}
since $G$ and $\sigma$ are globally Lipschitz continuous and $\E[|Q^{\frac12}\Gamma_n|]<\infty$. It finally suffices to check that
\[
\E[|\hat{\YY}_{k+1}^{\epsilon,\tau}-\sigma(\hat{\XX}_k^{\epsilon,\tau})Q^{\frac12}\Gamma_k|^2]\underset{\epsilon\to 0}\to 0
\]
for all $k\in\{0,\ldots,N-1\}$. Note that
\[
\hat{\YY}_{k+1}^{\epsilon,\tau}-\sigma({\XX}_k^{0,\tau})Q^{\frac12}\Gamma_k=\IA_{\frac{\tau}{\epsilon}}\hat{\YY}_k^{\epsilon,\tau}+\sigma(\hat{\XX}_k^{\epsilon,\tau})\bigl(\sqrt{\frac{\tau}{\epsilon}}\IB_{\frac{\tau}{\epsilon}}-Q^{\frac12}\bigr)\Gamma_k.
\]
On the one hand, one has
\[
\E[|\IA_{\frac{\tau}{\epsilon}}\hat{\YY}_k^{\epsilon,\tau}|^2]\le \frac{1}{(1+\lambda_1\frac{\tau}{\epsilon})^2}\underset{\epsilon\in(0,\epsilon_0)}\sup~\underset{\ell\ge 0}\sup~\E[|\hat{\YY}_\ell^{\epsilon,\tau}|^2]\le \frac{C}{(1+\lambda_1\frac{\tau}{\epsilon})^2}\underset{\epsilon\to 0}\to 0,
\]
using the moment bound $\underset{\epsilon\in(0,\epsilon_0)}\sup~\underset{\ell\ge 0}\sup~\E[|\hat{\YY}_\ell^{\epsilon,\tau}|^2]<\infty$, which is a variant of Lemma~\ref{lem:modifiedStochasticConvolution}, using the boundedness of $\sigma$.

On the other hand, one has
\begin{align*}
\E[|\bigl(\sqrt{\frac{\tau}{\epsilon}}\IB_{\frac{\tau}{\epsilon}}-Q^{\frac12}\bigr)\Gamma_k|^2]&=
\|\sqrt{\frac{\tau}{\epsilon}}\IB_{\frac{\tau}{\epsilon}}-Q^{\frac12}\|_{\mathcal{L}_2(H)}^2\\
&=\sum_{j\in\N}\bigl(\sqrt{\frac{\tau}{\epsilon}}\frac{\sqrt{2+\lambda_j\frac{\tau}{\epsilon}}}{\sqrt{2}(1+\lambda_j\frac{\tau}{\epsilon})}-\frac{1}{\sqrt{2\lambda_j}}\bigr)^2\\
&=\sum_{j\in\N}\frac{1}{\bigl(\sqrt{\frac{\tau}{\epsilon}}\frac{\sqrt{2+\lambda_j\frac{\tau}{\epsilon}}}{\sqrt{2}(1+\lambda_j\frac{\tau}{\epsilon})}+\frac{1}{\sqrt{2\lambda_j}}\bigr)^2}\bigl(\frac{\frac{\tau}{\epsilon}(2+\lambda_j\frac{\tau}{\epsilon})}{2(1+\lambda_j\frac{\tau}{\epsilon})^2}-\frac{1}{2\lambda_j}\bigr)^2\\
&\le \sum_{j\in\N}2\lambda_j\frac{1}{\bigl(2\lambda_j(1+\lambda_j\frac{\tau}{\epsilon})^2\bigr)^2}\\
&\le \sum_{j\in\N}\frac{1}{2\lambda_j(1+\lambda_j\frac{\tau}{\epsilon})^4}\\
&\underset{\epsilon\to 0}\to 0.
\end{align*}

As a consequence, since $\sigma$ is assumed to be bounded, one obtains
\[
\underset{\epsilon\to 0}\limsup~\E[|\hat{\XX}_n^{\epsilon,\tau}-\XX_n^{0,\tau}|]\le C\tau\sum_{k=0}^{n-1}\underset{\epsilon\to 0}\limsup~\E[|\hat{\XX}_k^{\epsilon,\tau}-\XX_k^{0,\tau}|].
\]
Since $\hat{\XX}_0^{\epsilon,\tau}-\XX_0^{0,\tau}=0$, it is straightforward to obtain
\[
\underset{\epsilon\to 0}\limsup~\E[|\hat{\XX}_n^{\epsilon,\tau}-\XX_n^{0,\tau}|]=0
\]
for all $n\in\{0,\ldots,N\}$. This concludes the proof of $(i)$.

$\bullet$ Proof of $(ii)$.

Let us introduce the auxiliary scheme $\bigl(\overline{\XX}_n^{\tau}\bigr)_{n\ge 0}$, obtained by the application of the standard Euler scheme to the averaged equation~\eqref{eq:SPDEaveraged}: for all $n\ge 0$, set
\[
\overline{\XX}_{n+1}^\tau=\IA_\tau\bigl(\overline{\XX}_n^\tau+\tau \overline{G}(\overline{\XX}_n^\tau)\bigr),
\]
with initial value $\overline{\XX}_0^{\tau}=x_0$.

Let $\varphi:H\to \R$ be a mapping of class $\mathcal{C}^2$, with bounded first and second order derivatives. Then the error can be decomposed as follows:
\[
\big|\E[\varphi(\XX_N^{0,\tau})]-\varphi(\overline{\XX}(T))\big|\le \big|\varphi(\overline{\XX}_N^{0,\tau})-\varphi(\overline{\XX}(T))\big|+\big|\E[\varphi(\XX_N^{0,\tau})]-\varphi(\overline{\XX}_N^{0,\tau})\big|.
\]
On the one hand, the upper bound for the first error term comes from a (deterministic) standard error estimate
\begin{align*}
\big|\varphi(\overline{\XX}_N^{0,\tau})-\varphi(\overline{\XX}(T))\big|&\le \vvvert\varphi\vvvert_1|\XX_N^{0,\tau}-\overline{\XX}(T)|\\
&\le C_\delta(T,x_0)\vvvert\varphi\vvvert_1\tau^{1-\delta}.
\end{align*}
The details are omitted.

On the other hand, the treatment of the second error term is based on the following argument. For all $n\in\N$ and all $x\in H$, set
\[
\overline{u}_n^{\tau}(x)=\varphi(\overline{\XX}_n^{\tau}(x))
\]
where $\bigl(\overline{\XX}_n^{\tau}(x)\bigr)_{n\ge 0}$ is the solution of the auxiliary scheme with initial value $\overline{\XX}_0^{\tau}(x)=x$. The second error term may then be written as
\begin{align*}
\E[\varphi(\XX_N^{0,\tau})]-\varphi(\overline{\XX}_N^\tau)&=\E[\overline{u}_0^\tau(\XX_N^{0,\tau})]-\E[\overline{u}_N^\tau(\XX_0^{0,\tau})]\\
&=\sum_{n=0}^{N-1}\bigl(\E[\overline{u}_{N-n-1}^\tau(\XX_{n+1}^{0,\tau})]-\E[\overline{u}_{N-n}^\tau(\XX_n^{0,\tau})]\bigr)\\
&=\sum_{n=0}^{N-1}\bigl(\E[\overline{u}_{N-n-1}^\tau(\IA_\tau\XX_n^{0,\tau}+\tau \IA_\tau G(\XX_n^{0,\tau},\sigma(\XX_n^{0,\tau})Q^{\frac12}\Gamma_n))]\\
&\hspace{2cm}-\E[\overline{u}_{N-n-1}^\tau(\IA_\tau\XX_n^{0,\tau}+\tau\IA_\tau \overline{G}(\XX_n^{0,\tau}))]\bigr),
\end{align*}
using a telescoping sum argument, the definition of the limiting scheme~\eqref{eq:limitingscheme} and the identity
\[
\overline{u}_{k+1}^\tau(x)=\overline{u}_k^\tau(\IA_\tau x+\tau\overline{G}(x))
\]
in the last step.

It is straightforward to check that the auxiliary functions $u_n^\tau$ are of class $\mathcal{C}^2$, and satisfy the following estimate
\[
\underset{\tau\in(0,\tau_0)}\sup~\underset{n\in\N,0\le n\tau\le T}\sup~\vvvert\overline{u}_n^\tau\vvvert_2\le C(T)(\vvvert\varphi\vvvert_1+\vvvert\varphi\vvvert_2)
\]
for some $C(T)\in(0,\infty)$. We refer to~\cite[Lemma~4.2]{B} for a precise statement and the proof.

The crucial property for the proof of the consistency of the limiting scheme~\eqref{eq:limitingscheme} with the averaged equation~\eqref{eq:SPDEaveraged}, hence of the asymptotic preserving property, is the following inequality: for all $n\in\{0,\ldots,N-1\}$, one has
\[
\E[D\overline{u}_{N-n-1}(\IA_\tau\XX_n^{0,\tau}).\bigl(G(\XX_n^{0,\tau},\sigma(\XX_n^{0,\tau})Q^{\frac12}\Gamma_n)-\overline{G}(\XX_n^{0,\tau})\bigr)]=0,
\]
using a conditional expectation argument and the definition~\eqref{eq:averagedG} of the averaged nonlinearity $\overline{G}$: indeed $\XX_n^{0,\tau}$ and $\Gamma_n$ are independent and $Q^{\frac12}\Gamma_n\sim\nu$.

Using a Taylor expansion argument, one then obtains
\begin{align*}
\big|\E[\varphi(\XX_N^{0,\tau})]-\varphi(\overline{XX}_N^\tau)\big|&\le \tau^2\sum_{n=0}^{N-1}\vvvert\overline{u}_{N-n-1}^\tau\vvvert_2\E[|G(\XX_n^{0,\tau},\sigma(\XX_n^{0,\tau})Q^{\frac12}\Gamma_n)|^2+|\overline{G}(\XX_n)|^2]\\
&\le C(T)\tau(\vvvert\varphi\vvvert_1+\vvvert\varphi\vvvert_2),
\end{align*}
since $G$ is assumed to be bounded.

Gathering the estimates then concludes the proof of the weak error estimate~\eqref{eq:consistencylimitingscheme} and of item $(ii)$.

The proof of Theorem~\ref{theo:AP} is thus completed.
\end{proof}

The asymptotic preserving property stated in Theorem~\ref{theo:AP} can be written as follows: for any bounded and continuous mapping $\varphi:H\to\R$, one has
\begin{equation}\label{eq:APlimits}
\underset{\tau\to 0}\lim~\underset{\epsilon\to 0}\lim~\E[\varphi(\XX_N^{\epsilon,\tau})]=\underset{\epsilon\to 0}\lim~\underset{\tau\to 0}\lim~\E[\varphi(\XX_N^{\epsilon,\tau})],
\end{equation}
with fixed $T=N\tau$. Owing to this property, there is no restriction on the time-step size $\tau$ in terms of $\epsilon$, and the scheme~\eqref{eq:APscheme} provides an accurate approximation of $\XX^\epsilon(T)$ for any fixed $\epsilon$ and of $\overline{\XX}(T)$ when $\epsilon\to 0$. However, the asymptotic preserving property above does not provide relevant information for choosing the time-step size $\tau$ in order to achieve a given tolerance error when applying the scheme, either for a fixed $\epsilon$ or in the regime $\epsilon\to 0$. To study the computational cost of the scheme~\eqref{eq:APscheme} in terms of $\epsilon$, the relevant question to consider is whether the scheme is uniformly accurate, i.e. whether uniform weak error estimates
\begin{equation}\label{eq:UA}
\underset{\epsilon\in(0,\epsilon_0]}\sup~\big|\E[\varphi(\XX_N^{\epsilon,\tau})]-\E[\varphi(\XX^\epsilon(T))]\big|\underset{\tau\to 0}\to 0,
\end{equation}
when $T=N\tau$ is fixed, for a suitable class of functions $\varphi:H\to \R$. When~\eqref{eq:UA} holds, the time-step size $\tau$ may be chosen independently of $\epsilon$ in order to achieve a given accuracy of the approximation. We refer to~\cite{B} for the proof of the uniform accuracy property~\eqref{eq:UA}, when $\sigma$ is constant. More precisely, under appropriate conditions, one obtains in~\cite{B} uniform weak error estimates of the type
\[
\underset{\epsilon\in(0,\epsilon_0]}\sup~\big|\E[\varphi(\XX_N^{\epsilon,\tau})]-\E[\varphi(\XX^\epsilon(T))]\big|\le C_\delta(T,x_0,\varphi)\tau^{\frac13-\delta}
\]
for functions $\varphi$ of class $\mathcal{C}^3$. The proof of~\eqref{eq:UA} requires many additional technical arguments.

To conclude this subsection, let us compare the performances of the modified Euler scheme and of other schemes when applied to discretize the fast Ornstein--Uhlenbeck component $\YY^\epsilon$ in the SPDE system~\eqref{eq:SPDE-slowfast}. On the one hand, employing the accelerated exponential Euler method would result also in an asymptotic preserving scheme: a variant of Theorem~\ref{theo:AP} also holds for the scheme
\begin{equation}\label{eq:APscheme-expo}
\left\lbrace
\begin{aligned}
\XX_{n+1}^{\epsilon,\tau,\e}&=\IA_\tau\bigl(\XX_n^{\epsilon,\tau,\e}+\tau G(\XX_n^{\epsilon,\tau,\e},\YY_{n+1}^{\epsilon,\tau,\e})\bigr)\\
\YY_{n+1}^{\epsilon,\tau,\e}&=e^{-\frac{\tau}{\epsilon}\IL}\YY_n^{\epsilon,\tau,\e}+\frac{\sigma(\XX_n^{\epsilon,\tau,\e})}{\sqrt{\epsilon}}\int_{t_n}^{t_{n+1}}e^{-\frac{t_{n+1}-t}{\epsilon}\IL}dW(t).
\end{aligned}
\right.
\end{equation}
This result is not surprising since the accelerated exponential Euler scheme preserves the distribution of the Ornstein--Uhlenbeck component. The details of the proof of the asymptotic preserving property for the scheme~\eqref{eq:APscheme-expo} are omitted.

On the other hand, if the fast component $\YY^\epsilon$ is discretized using the standard Euler method, the scheme
\begin{equation}\label{eq:APscheme-standard}
\left\lbrace
\begin{aligned}
\XX_{n+1}^{\epsilon,\tau,\s}&=\IA_\tau\bigl(\XX_n^{\epsilon,\tau,\s}+\tau G(\XX_n^{\epsilon,\tau,\s},\YY_{n+1}^{\epsilon,\tau,\s})\bigr)\\
\YY_{n+1}^{\epsilon,\tau,\s}&=\IA_{\frac{\tau}{\epsilon}}\Bigl(\YY_n^{\epsilon,\tau,\s}+\sigma(\XX_n^{\epsilon,\tau,\s})\sqrt{\frac{\tau}{\epsilon}}\Gamma_{n}\Bigr)
\end{aligned}
\right.
\end{equation}
is not asymptotic preserving in general. Indeed, the associated limiting scheme is then given by
\[
\XX_n^{0,\tau,\s}=\IA_\tau\bigl(\XX_n^{0,\tau,\s}+\tau G(\XX_n^{0,\tau,\s},0)\bigr)
\]
which is consistent with the averaged equation~\eqref{eq:SPDEaveraged} if and only if $\overline{G}(x)=G(x,0)$ for all $x\in H$. In general, the identity~\eqref{eq:APlimits} thus does not hold: one has
\[
\underset{\tau\to 0}\lim~\underset{\epsilon\to 0}\lim~\E[\varphi(\XX_N^{\epsilon,\tau,\s})]\neq\underset{\epsilon\to 0}\lim~\underset{\tau\to 0}\lim~\E[\varphi(\XX_N^{\epsilon,\tau,\s})].
\]
Similarly, the uniform accuracy property~\eqref{eq:UA} does not hold when the standard Euler scheme is used: one has
\[
\underset{\tau\to 0}\limsup~\underset{\epsilon\in(0,\epsilon_0]}\sup~\big|\E[\varphi(\XX_N^{\epsilon,\tau})]-\E[\varphi(\XX^\epsilon(T))]\big|>0.
\]
As a consequence, it is not possible to choose $\tau$ independently of $\epsilon$ to achieve a given accuracy for the scheme~\eqref{eq:APscheme-standard}. The construction of the asymptotic preserving scheme~\eqref{eq:APscheme} for the SPDE system~\eqref{eq:SPDE-slowfast} is another illustration of the superiority of the modified Euler scheme compared with the standard method.

\begin{rem}\label{rem:postprocAP}
The scheme~\eqref{eq:APscheme-standard}, which is not asymptotic preserving, can be improved using the postprocessed integrator introduced in Remark~\ref{rem:postproc}:
\begin{equation}\label{eq:APscheme-standard-postproc}
\left\lbrace
\begin{aligned}
\XX_{n+1}^{\epsilon,\tau,{\rm pp}}&=\IA_\tau\bigl(\XX_n^{\epsilon,\tau,{\rm pp}}+\tau G(\XX_n^{\epsilon,\tau,{\rm pp}},\underline{\YY}_{n+1}^{\epsilon,\tau,{\rm pp}})\bigr)\\
\YY_{n+1}^{\epsilon,\tau,{\rm pp}}&=\IA_{\frac{\tau}{\epsilon}}\Bigl(\YY_n^{\epsilon,\tau,{\rm pp}}+\sigma(\XX_n^{\epsilon,\tau,{\rm pp}})\sqrt{\frac{\tau}{\epsilon}}\Gamma_{n}\Bigr)\\
\underline{\YY}_n^{\epsilon,\tau,{\rm pp}}&=\YY_n^{\epsilon,\tau,{\rm pp}}+\frac{\sigma(\XX_n^{\epsilon,\tau,{\rm pp}})}{2}\mathcal{J}_{\frac{\tau}{\epsilon}}\sqrt{\frac{\tau}{\epsilon}}\Gamma_n.
\end{aligned}
\right.
\end{equation}
The resulting scheme is asymptotic preserving: indeed it can be checked that the resulting limiting scheme is given by
\[
\XX_{n+1}^{0,\tau,{\rm pp}}=\IA_\tau\bigl(\XX_n^{0,\tau,{\rm pp}}+\tau G(\XX_n^{0,\tau,{\rm pp}},\sigma(\XX_n^{0,\tau,{\rm pp}})Q^{\frac12}\Gamma_{n+1})\bigr).
\]
Whereas in the context of Remark~\ref{rem:postproc} the computation of the postprocessed variable would be necessary at the final iteration $n=N$, for the scheme~\eqref{eq:APscheme-standard-postproc} it is necessary to compute the variable $\underline{\YY}_n^{\epsilon,\tau,pp}$ for all $n=1,\ldots,N+1$. As a consequence, the costs of each iteration of the schemes~\eqref{eq:APscheme} and~\eqref{eq:APscheme-standard-postproc} are of the same order. This justifies a preference for the application and the analysis of the scheme~\eqref{eq:APscheme} based on the modified Euler scheme.
\end{rem}

\subsection{Markov Chain Monte Carlo method}\label{sec:MCMC}

\subsubsection{Context}

The objective of this subsection is to describe how the modified Euler scheme proposed in this article can be used as a proposal kernel in a Markov Chain Monte Carlo (MCMC) method. Recall that the Gibbs distribution $\mu_\star$ is defined by~\eqref{eq:mu_star} (see Proposition~\ref{propo:mu_star}. It is assumed that the mapping $V:H\to \R$ is bounded and of class $\mathcal{C}^3$ with bounded derivatives. In order to approximate integrals $\int \varphi d\mu_\star$, the MCMC method consists in introducing an $H$-valued Markov chain $\bigl(\X_n\bigr)_{n\ge 0}$ which is ergodic and admits $\mu_\star$ as its unique invariant distribution: an estimator of $\int\varphi d\mu_\star$ is then defined as the temporal average
\[
\frac{1}{N}\sum_{n=1}^{N}\varphi(\X_n),
\]
which converges when $N\to\infty$ almost surely to $\int\varphi d\mu_\star$. To analyze the quality of the approximation, it may be convenient to study the spectral gap of the considered Markov chain, when it is reversible with respect to $\mu_\star$.

A popular strategy to design such Markov chains is the Metropolis--Hastings method, which requires two ingredients:
\begin{itemize}
\item the choice of a proposal kernel,
\item an acceptance-rejection rule,
\end{itemize}
which are needed to ensure that $\mu_\star$ is an invariant distribution. The resulting Markov chain is then reversible with respect to $\mu_\star$, by construction. In an infinite dimensional context, the choice of the proposal kernel is crucial in order to obtain well-defined acceptance-rejection ratios. As explained in~\cite{CotterRobertsStuart}, using the preconditioned Crank--Nicolson (pCN) proposal kernel
\[
\hat{\X}_{n+1}^{\tau,{\rm pCN}}=\X_n^{\tau,{\rm pCN}}-\frac{\tau}{2}(\X_n^{\tau,{\rm pCN}}+\X_{n+1}^{\tau,{\rm pCN}})+\IL^{-\frac12}\sqrt{\tau}\Gamma_n,
\]
leads to a well-defined Metropolis--Hastings MCMC method, for any value of the time-step size $\tau$. On the contrary, for all $\theta\in[0,1]\setminus\{\frac12\}$, using the proposal kernel
\[
\hat{\X}_{n+1}^{\tau,\theta}=\X_n^{\tau,\theta}-\bigl((1-\theta)\tau\X_n^{\tau,\theta}+\theta\tau\X_{n+1}^{\tau,\theta}\bigr)+\IL^{-\frac12}\sqrt{\tau}\Gamma_n,
\]
based on the $\theta$-method, the acceptance-rejection ratio is ill-defined in infinite dimension. For the pCN proposal kernel, the associated acceptance ratio is computed as
\[
a^{{\rm pCN}}(x,\hat{x})=\min(1,e^{2(V(x)-V(\hat{x}))}),
\]
which means that the Markov chain is constructed as follows: for all $n\ge 0$,
\[
\X_{n+1}^{\tau,{\rm pCN}}=\mathds{1}_{U_n\le a^{{\rm pCN}}(\X_n^{\tau,{\rm pCN}},\hat{X}_{n+1}^{\tau,{\rm pCN}})}\hat{X}_{n+1}^{\tau,{\rm pCN}}+\mathds{1}_{U_n>a^{{\rm pCN}}(\X_n^{\tau,{\rm pCN}},\hat{X}_{n+1}^{\tau,{\rm pCN}})}X_{n}^{\tau,{\rm pCN}},
\]
where $\bigl(\Gamma_n\bigr)_{n\ge 0}$ is a sequence of independent cylindrical $H$-valued Gaussian random variables, and $\bigl(U_n\bigr)_{n\ge 0}$ is a sequence of independent random variables which are uniformly distributed on $[0,1]$, and the two sequences are independent. The choice of the auxiliary parameter $\tau$ has an impact on the performance on the method, which is not discussed in this work.

The proposal kernel in the pCN Markov chain $\bigl(\X_n^{\tau,{\rm pCN}}\bigr)_{n\ge 0}$ consists in applying the Crank--Nicolson method to discretize the Ornstein--Uhlenbeck dynamics
\[
dZ^{{\rm p}}(t)=-Z^{{\rm p}}(t)dt+\IL^{-\frac12}dW(t)
\]
which can be interpreted as a preconditioned version of the stochastic evolution equation
\[
dZ(t)=-\IL Z(t)dt+dW(t).
\]
Note that the crucial property is the fact that $\nu$ is the invariant distribution of the Ornstein--Uhlenbeck process $\bigl(Z^{{\rm p}}(t)\bigr)_{t\ge 0}$, and that the Crank--Nicolson scheme preserves this invariant distribution, for any choice of $\tau$. See~\cite{BDV} for the analysis of integrators applied to preconditioned stochastic evolution equations for the approximation of the Gibbs invariant distribution $\mu_\star$.

To the best of our knowledge, the construction of MCMC methods using a numerical discretization of the process $\bigl(Z(t)\bigr)_{t\ge 0}$ (instead of its preconditioned version $\bigl(Z^{{\rm p}}(t)\bigr)_{t\ge 0}$), has not been treated in the literature so far. Note that using the Crank--Nicolson method
\[
Z_{n+1}=Z_n-\frac{\tau}{2}\IL(Z_n+Z_{n+1})+\sqrt{\tau}\Gamma_n=(I-\frac{\tau}{2}\IL)(I+\frac{\tau}{2}\IL)^{-1}Z_n+(I+\frac{\tau}{2}\IL)^{-1}\sqrt{\tau}\Gamma_n
\]
would not be appropriate: even if the acceptance-rejection ratio is well-defined, the Markov chain may not be ergodic, due to the fact that the Crank--Nicolson method is not L-stable: one has
\[
\|(I-\frac{\tau}{2}\IL)(I+\frac{\tau}{2}\IL)^{-1}\|_{\mathcal{L}(H)}=1
\]
for any choice of the time-step size.

\subsubsection{MCMC method based on the modified Euler scheme}

Let us state the main result of this subsection: a Metropolis--Hastings MCMC method is well-defined in infinite dimension when using the modified Euler scheme as the proposal kernel.
\begin{theo}\label{theo:MCMC}
For all $\tau\in(0,\tau_0)$, introduce the Markov chain defined by
\begin{equation}\label{eq:MCMC}
\left\lbrace
\begin{aligned}
\hat{\X}_{n+1}^{\tau}&=\IA_\tau\X_n^\tau+\sqrt{\tau}\IB_{\tau,1}\Gamma_{n,1}+\sqrt{\tau}\IB_{\tau,2}\Gamma_{n,2}\\
\X_{n+1}^{\tau}&=\mathds{1}_{U_n\le a(\X_n^{\tau},\hat{\X}_{n+1}^{\tau})}\hat{\X}_{n+1}^{\tau}+\mathds{1}_{U_n>a(\X_n^{\tau},\hat{\X}_{n+1}^{\tau})}\X_{n}^{\tau},
\end{aligned}
\right.
\end{equation}
where the acceptance-rejection ratio is defined by
\begin{equation}\label{eq:acceptanceMCMC}
a(x,\hat{x})=\min(1,e^{2(V(x)-V(\hat{x}))})
\end{equation}
for all $x,\hat{x}\in H$, and where $\bigl(\Gamma_{n,1}\bigr)_{n\ge 0}$ and $\bigl(\Gamma_{n,2}\bigr)_{n\ge 0}$ are two independent sequences of independent cylindrical $H$-valued Gaussian random variables, and $\bigl(U_n\bigr)_{n\ge 0}$ is a sequence of independent random variables which are uniformly distributed on $[0,1]$, which is independent of the two sequences $\bigl(\Gamma_{n,1}\bigr)_{n\ge 0}$ and $\bigl(\Gamma_{n,2}\bigr)_{n\ge 0}$. Let $\PP^\tau$ denote the transition operator associated with the Markov chain: for all $x\in H$ and $n\in\N$, and any bounded and measurable function $\varphi:H\to \R$,
\[
(\PP^\tau)^n\varphi(x)=\E_x[\varphi(\X_n)].
\]

Assume that $V$ is bounded and globally Lipschitz continuous.

The $H$-valued Markov chain $\bigl(\X_n^{\tau}\bigr)_{n\ge 0}$ is ergodic, and its invariant distribution is the Gibbs distribution $\mu_\star$. In addition, this Markov chain admits a spectral gap in the following sense: for all $\tau\in(0,\tau_0)$, there exists $\kappa(\tau)\in(0,1)$, such that for all $\varphi\in L^2(\mu_\star)$, one has
\begin{equation}\label{eq:MCMC-spectralgap}
\|(\PP^\tau)^n\varphi-\int\varphi d\mu_\star\|_{L^2(\mu_\star}\le e^{-\kappa(\tau)n\tau}\|\varphi-\int\varphi d\mu_\star\|_{L^2(\mu_\star)},
\end{equation}
with $\|\varphi\|_{L^2(\mu_\star)}^2=\int \varphi(x)^2d\mu_\star(x)$.
\end{theo}

Observe that the acceptance-ratio $a(\X_n^{\tau},\hat{\X}_{n+1}^\tau)$ appearing in~\eqref{eq:MCMC} is defined with the same expression as for the pCN Markov chain mentioned above. Before proceeding with the proof of Theorem~\ref{theo:MCMC}, it is worth mentioning that the spectral gap inequality~\eqref{eq:MCMC-spectralgap} from Theorem~\ref{theo:MCMC} and the error estimate~\eqref{eq:theo-weakinv_weakerror} from Theorem~\ref{theo:weakinv} have different formulations. The comparisons of the performances of these two methods to approximate $\int\varphi d\mu_\star$ is out of the scope of this article. However, it is possible to compare the results as follows. On the one hand, using the MCMC method~\eqref{eq:MCMC} instead of the modified Euler scheme~\eqref{eq:scheme}, results in the absence of bias due to the choice of the time-step size $\tau>0$. In addition, the assumptions on the function $V$ are weaker to obtain Theorem~\ref{theo:MCMC}: it is only assumed that $V$ is bounded and globally Lipschitz continuous, whereas it is required that $F=-DV$ satisfies Assumption~\ref{ass:ergo} to obtain Theorem~\ref{theo:weakinv}, meaning that $\vvvert V\vvvert_2<\lambda_1$ needs to be small enough. On the other hand, there are two disadvantages with the spectral gap inequality~\eqref{eq:MCMC-spectralgap} compared with the weak error estimate~\eqref{eq:theo-weakinv_weakerror}. First, the value of $\kappa(\tau)$ obtained in the proof of Theorem~\ref{theo:MCMC} is not explicit and may depend on $\tau$, in particular it may be the case that $\kappa(\tau)\to 0$ when $\tau>0$. On the contrary, the rate $\kappa$ of convergence to equilibrium in~\eqref{eq:theo-weakinv_weakerror} is independent of $\tau$. Second, the weak error estimate~\eqref{eq:theo-weakinv_weakerror} allows one to approximate $\int\varphi d\mu_\star$ starting the modified Euler scheme from an arbitrary initial condition $x\in H$. On the contrary, the spectral gap inequality is a $L^2$ estimate instead of a pointwise bound with respect to $x$. Finally, note that the spectral gap inequality~\eqref{eq:MCMC-spectralgap} is a consequence of estimates in a Wasserstein distance-like function, whereas the weak error estimate~\eqref{eq:theo-weakinv_weakerror} is related to estimates in total variation distance. The construction of the MCMC method in Theorem~\ref{theo:MCMC} is a new result, which provides an alternative to the widely used pCN sampler. It is not clear whether the performances of the MCMC method are better, compared either with using directly the modified Euler scheme, or with using the pCN method. This question may be investigated theoretically and numerically in future works.

The performance of the MCMC method depending on the value of the auxiliary time-step size parameter $\tau$ is not studied in this work. In addition to the analysis of the behavior of the spectral gap $\kappa(\tau)$, it may be appropriate to study whether diffusion limits hold, like in~\cite{MattinglyPillaiStuart}, in order to identify the optimal averaged acceptance probabilities to guide the choice of $\tau$ in practice. These questions are left open for future work.

Proving Theorem~\ref{theo:MCMC} requires two main contributions. First, one needs to check that the Markov chain~\eqref{eq:MCMC} is indeed the result of the Metropolis--Hastings procedure. In particular, $\mu_\star$ is an invariant distribution of the Markov chain. Second, one needs to prove the spectral gap inequality~\eqref{eq:MCMC-spectralgap}. The strategy is a variant of the one used in~\cite{HairerStuartVollmer} to prove the spectral gap property of the pCN method, by applying the weak Harris theorem from~\cite{HairerMattinglyScheutzow}.

\begin{proof}[Proof of Theorem~\ref{theo:MCMC}]
$\bullet$ Verification of the Metropolis--Hastings formulation.

Let us introduce the general formulation of the Metropolis--Hastings MCMC method, see for instance~\cite{CotterRobertsStuart}. Let $\bigl(q_x(\cdot)\bigr)_{x\in H}$ denote the proposal kernel, and introduce two probability distributions $\eta$ and $\eta^{\perp}$ on $H^2$, defined as follows: for any bounded and measurable function $\phi:H^2\to\R$, set
\begin{align*}
\iint \phi(x,y)d\eta(x,y)&=\iint \phi(x,y)dq_x(y)d\mu_\star(x)\\
\iint \phi(x,y)d\eta^\perp(x,y)&=\iint \phi(x,y)dq_y(x) d\mu_\star(y).
\end{align*}
Under the condition that the probability distributions $\eta$ and $\eta^\perp$ are equivalent, define the acceptance probability by
\[
a(x,\hat{x})=\min(1,\frac{d\eta^\perp}{d\eta}(x,\hat{x})),
\]
for all $x,\hat{x}\in H$, where $\frac{d\eta^\perp}{d\eta}:H^2\to\R$ denotes the Radon-Nikodym derivative. Then, the Metropolis--Hastings algorithm is well-defined, and $\mu_\star$ is invariant for the resulting Markov chain.

Let $\tau\in(0,\tau_0)$. The objective is to check that if the proposal kernel is defined using the modified Euler scheme with time-step size $\tau$, then the associated distributions $\eta$ and $\eta^\perp$ are equivalent and to compute the Radon-Nikodym derivative. Owing to the definition~\eqref{propo:mu_star} of the Gibbs distribution $\mu_\star$, one has
\begin{align*}
d\eta(x,y)&=dq_x(y)\mathcal{Z}^{-1}e^{-2V(x)}d\nu(x)=\mathcal{Z}^{-1}e^{-2V(x)}d\eta_0(x,y)\\
d\eta^\perp(x,y)&=dq_y(x)\mathcal{Z}^{-1}e^{-2V(y)}d\nu(y)=\mathcal{Z}^{-1}e^{-2V(y)}d\eta_0^\perp(x,y)
\end{align*}
where $\eta_0$ and $\eta_0^\perp$ are centered Gaussian distributions on $H^2$. Due to the fact that the Gaussian distribution $\nu$ is preserved by the modified Euler scheme in the Ornstein--Uhlenbeck case, for any value of the time-step size $\tau$, it is straightforward to check that the covariance operators of the Gaussian distributions $\eta_0$ and $\eta_0^\perp$ are identical, hence $\eta_0=\eta_0^\perp$. As a consequence, $\eta$ and $\eta^\perp$ are equivalent and one obtains the expression~\eqref{eq:acceptanceMCMC} for the acceptance probability. This concludes the first step of the proof, namely the verification of the Metropolis--Hastings structure of the Markov chain~\eqref{eq:MCMC}. In particular, the chain is reversible with respect to the probability distribution $\mu_\star$, which is thus an invariant distribution.

$\bullet$ Application of the weak Harris theorem.

The objective is to show that the weak Harris theorem~\cite{HairerMattinglyScheutzow} can be applied. The arguments of the proof follow those use in~\cite{HairerStuartVollmer}. To simplify notation, the time-step size parameter $\tau$ is omitted in the sequel. It is worth mentioning that the values of the auxiliary parameters introduced below may depend on $\tau$. For any $x\in H$, let the probability distribution $\PP(x,\cdot)$ be defined by
\[
\int\varphi(y)d\PP(x,y)=\E_{\X_0=x}[\varphi(\X_1)]
\]
for any bounded and measurable function $\varphi:H\to\R$.

For all positive $\varepsilon$, Introduce the auxiliary distance-like function $d_\varepsilon$ defined by
\[
d_\varepsilon(x,y)=\min(1,\frac{|x-y|}{\varepsilon}),
\]
for all $x,y\in H$. Define also
\[
\tilde{d}_\varepsilon(x,y)=\sqrt{d_\varepsilon(x,y)+|x|^2+|y|^2}
\]
for all $x,y\in H$

For any probability distributions $\mu_1,\mu_2$ on $H$, set
\begin{align*}
{d}_\varepsilon(\mu_1,\mu_2)&=\underset{\pi\in \Pi(\mu_1,\mu_2)}\inf\iint~{d}_\varepsilon(x,y)d\pi(x,y)\\
\tilde{d}_\varepsilon(\mu_1,\mu_2)&=\underset{\pi\in \Pi(\mu_1,\mu_2)}\inf~\iint\tilde{d}_\varepsilon(x,y)d\pi(x,y)
\end{align*}
where $\Pi(\mu_1,\mu_2)$ is the set of couplings of the probability distributions $\mu_1,\mu_2$. The functions $d_\varepsilon(\cdot,\cdot)$ and $\tilde{d}_\varepsilon(\cdot,\cdot)$ defined above are referred to as the Wasserstein distance-like functions associated with the functions $d_\varepsilon$ and $\tilde{d}_\varepsilon$ respectively.

To apply the weak Harris theorem, it suffices to check that there exists $\varepsilon\in(0,1)$ such that the three following claims hold.
\begin{itemize}
\item Lyapunov structure: there exists $\ell\in(0,1)$ and $C\in(0,\infty)$ such that for all $x\in H$ one has
\begin{equation}\label{eq:MCMC-Lyapunov}
\PP(|\cdot|^2)(x)\le \ell|x|^2+C.
\end{equation}
\item d-contraction: there exists $c\in(0,1)$ such that for all $x,y\in H$ with $d_\varepsilon(x,y)<1$, one has
\begin{equation}\label{eq:MCMC-contraction}
d_\varepsilon(\PP(x,\cdot),\PP(x,\cdot))\le cd_\varepsilon(x,y).
\end{equation}
\item d-smallness (of balls): for any $R\in(0,\infty)$, there exist $N_R\in\N$ and $s_R\in(0,1)$ such that
\begin{equation}\label{eq:MCMC-smallness}
\underset{|x|\le R,|y|\le R}\sup~d_\varepsilon(\PP^{N_R}(x,\cdot),\PP^{N_R}(y,\cdot))\le s_R.
\end{equation}
\end{itemize}
The application of the weak Harris theorem then provides the following result: there exists $\tilde{n}\in \N$ such that for any probability distribution $\mu$ on $H$, one has
\begin{equation}\label{eq:MCMC-Wasserstein-inequality}
\tilde{d}_\varepsilon(\mu\PP^{\tilde{n}},\mu_\star)\le \frac{1}{2}\tilde{d}_\varepsilon(\mu,\mu_\star).
\end{equation}

$\bullet$ Proof of the Lyapunov structure property~\eqref{eq:MCMC-Lyapunov}.

Let $r\in(0,\infty)$ be an arbitrary positive real number, and let $R\in(0,\infty)$ which will be chosen later. Introduce the auxiliary Gaussian random variable
\[
\xi=\sqrt{\tau}\IB_{\tau,1}\Gamma_{1,1}+\sqrt{\tau}\IB_{\tau,2}\Gamma_{1,2}.
\]
First, assume that $|x|< R$. Then one has
\begin{align*}
\PP(|\cdot|^2)(x)=\E_x[|\X_1|^2]&\le \E_x\bigl[\max\bigl(|x|^2,|\hat{X}_1|^2\bigr)\bigr]\\
&\le |x|^2+\E_x[|\hat{X}_1|^2]\\
&\le 2|x|^2+\E_x[|\xi|^2]\\
&\le C(R)<\infty.
\end{align*}
Second, assume that $|x|\ge R$. Then one has the decomposition
\begin{align*}
\PP(|\cdot|^2)(x)=\E_x[|\X_1|^2]&=\E_x\bigl[\mathds{1}_{|\xi|\le r}\mathds{1}_{\X_1=\hat{\X}_1}|\hat{\X}_1|^2\bigr]\\
&+\E_x\bigl[\mathds{1}_{|\xi|\le r}\mathds{1}_{\X_1\neq\hat{\X}_1}|\hat{\X}_1|^2\bigr]\\
&+\E_x\bigl[\mathds{1}_{|\xi|> r}\max\bigl(|x|^2,|\hat{\X}_1|^2\bigr)\bigr].
\end{align*}
Under the condition $|\xi|\le r$, one has
\begin{align*}
|\hat{X}_1|&\le |\IA_\tau x|+|\xi|\le \frac{1}{1+\lambda_1\tau}|x|+r\le (1-\frac{\lambda_1\tau}{2(1+\lambda_1\tau)})|x|+r-\frac{\lambda_1\tau}{2(1+\lambda_1\tau)}|x|\\
&\le (1-\frac{\lambda_1\tau}{2(1+\lambda_1\tau)})|x|+r-\frac{\lambda_1\tau}{2(1+\lambda_1\tau)}R\\
&\le (1-\frac{\lambda_1\tau}{2(1+\lambda_1\tau)})|x|,
\end{align*}
if $R$ is chosen such that $\frac{\lambda_1\tau}{2(1+\lambda_1\tau)}R\ge r$. Let $\theta=1-(1-\frac{\lambda_1\tau}{2(1+\lambda_1\tau)})^2$. Therefore, one obtains
\begin{align*}
\PP(|\cdot|^2)(x)&\le \E_x\bigl[\mathds{1}_{|\xi|\le r}\bigl((1-\theta)\mathds{1}_{\X_1=\hat{\X}_1}+\mathds{1}_{\X_1=\hat{\X}_1}\bigr)\bigr]|x|^2\\
&+\E_x\bigl[\mathds{1}_{|\xi|> r}\max\bigl(|x|^2,|\hat{\X}_1|^2\bigr)\bigr]\\
&\le \E_x\bigl[\mathds{1}_{|\xi|\le r}\bigl(1-\theta\mathds{1}_{\X_1=\hat{\X}_1}\bigr)]|x|^2\\
&+\E_x\bigl[\mathds{1}_{|\xi|> r}\max\bigl(|x|^2,|\hat{\X}_1|^2\bigr)\bigr].
\end{align*}
The acceptance probability $a$ defined by~\eqref{eq:acceptanceMCMC} satisfies
\[
\underset{x,\hat{x}\in H}\inf~a(x,\hat{x})\ge e^{\min(V)-\max(V)}=a_m>0,
\]
owing to the assumption that the function $V$ is bounded. As a consequence, by a conditioning argument, one obtains
\[
\E_x\bigl[\mathds{1}_{|\xi|\le r}\bigl(1-\theta\mathds{1}_{\X_1=\hat{\X}_1}\bigr)]|x|^2\le \mathbb{P}_x\bigl(\mathds{1}_{|\xi|\le r}\bigr)\bigl(1-\theta a_m\bigr)|x|^2.
\]
In addition, there exists $C\in(0,\infty)$ such that for all $x\in H$ one has
\[
\E_x\bigl[\mathds{1}_{|\xi|> r}\max\bigl(|x|^2,|\hat{\X}_1|^2\bigr)\bigr]\le \mathbb{P}_x\bigl(\mathds{1}_{|\xi|>r}\bigr)|x|^2+C.
\]
Finally, one obtains for all $x\in H$, such that $|x|\ge R$, the inequality
\begin{align*}
\PP(|\cdot|^2)(x)&\le \bigl(1-\theta a_m C\mathbb{P}_x(\mathds{1}_{|\xi|\le r})\bigr)|x|^2+C\\
&\le \ell|x|^2+C
\end{align*}
where $\ell\in(0,1)$, owing to the property that the $H$-valued Gaussian random variable $\xi$ satisfies $\mathbb{P}_x(\mathds{1}_{|\xi|\le r})\bigr)>0$ for all $r\in(0,\infty)$. Gathering the upper bounds in the two cases $|x|<R$ and $|x|\ge R$ concludes the proof of~\eqref{eq:MCMC-Lyapunov}.

$\bullet$ Proof of the d-contraction property~\eqref{eq:MCMC-contraction}.

Let $x,y$ be two arbitrary elements of $H$. Let the proposals $\hat{\X}_1$ and $\hat{\Y}_1$ be defined by~\eqref{eq:MCMC} using the same cylindrical Gaussian random variables $\Gamma_{1,1}$ and $\Gamma_{1,2}$, then let $\X_1$ and $\Y_1$ be defined using the same uniformly distribution random variable $U_1$:
\begin{align*}
\hat{\X}_{1}&=\IA_\tau x+\sqrt{\tau}\IB_{\tau,1}\Gamma_{1,1}+\sqrt{\tau}\IB_{\tau,2}\Gamma_{1,2}\\
\hat{\Y}_{1}&=\IA_\tau y+\sqrt{\tau}\IB_{\tau,1}\Gamma_{1,1}+\sqrt{\tau}\IB_{\tau,2}\Gamma_{1,2}
\end{align*}
and
\begin{align*}
\X_{1}&=\mathds{1}_{U_1\le a(x,\hat{\X}_{1})}\hat{\X}_{1}+\mathds{1}_{U_n>a(x,\hat{\X}_{1})}x\\
\Y_{1}&=\mathds{1}_{U_1\le a(y,\hat{\Y}_{1})}\hat{\Y}_{1}+\mathds{1}_{U_n>a(x,\hat{\Y}_{1})}y.
\end{align*}
By definition of the Wasserstein distance-like function $d_\varepsilon$, one has
\[
d_\varepsilon(\PP(x,\cdot),\PP(x,\cdot))\le \E[d_\varepsilon(\X_1,\Y_1)].
\]
A decomposition according to the different acceptance or rejection events for $\X_1$ and $\Y_1$, one obtains
\begin{align*}
\E[d_\varepsilon(\X_1-\Y_1)]&=\E[\mathds{1}_{U_1\le \min(a(x,\hat{\X}_{1}),a(y,\hat{\Y}_{1}))}d_\varepsilon(\hat{\X}_1,\hat{\Y}_1)]+\E[\mathds{1}_{U_1\ge \max (a(x,\hat{\X}_{1}),a(y,\hat{\Y}_{1}))}]d_\varepsilon(x,y)\\
&+\E[\mathds{1}_{a(y,\hat{\Y}_{1}) \le U\le a(x,\hat{\X}_{1})}d_\varepsilon(\hat{\X}_1,y)]+\E[\mathds{1}_{a(x,\hat{\X}_{1}) \le U\le a(y,\hat{\Y}_{1})}d_\varepsilon(x,\hat{\Y}_1)].
\end{align*}
Assume that $x,y$ satisfy $d_\varepsilon(x,y)<1$. Then by construction, one has $d_\varepsilon(x,y)=\frac{|x-y|}{\epsilon}$. In addition, one has $\hat{\X}_1-\hat{\Y}_1=\IA_\tau(x-y)$, hence the condition $d_\varepsilon(x,y)<1$ implies
\[
d_\varepsilon(\hat{\X}_1,\hat{\Y}_1)=\min\Bigl(1,\frac{|\hat{\X}_1-\hat{\Y}_1|}{\epsilon}\Bigr)\le \min\Bigl(1,\frac{|x-y|}{\epsilon(1+\lambda_1\tau)}\Bigr)=\frac{|x-y|}{\epsilon(1+\lambda_1\tau)}=\frac{d_\varepsilon(x,y)}{1+\lambda_1\tau}.
\]
On the one hand, one obtains, with $\rho=\frac{1}{1+\lambda_1\tau}$,
\begin{align*}
\E[\mathds{1}_{U_1\le \min(a(x,\hat{\X}_{1}),a(y,\hat{\Y}_{1}))}d_\varepsilon(\hat{\X}_1,\hat{\Y}_1)]&+\E[\mathds{1}_{U_1\ge \max (a(x,\hat{\X}_{1}),a(y,\hat{\Y}_{1}))}]d_\varepsilon(x,y)\\
&\le \frac{1}{1+\lambda_1\tau}\mathbb{P}(U_1\le \min(a(x,\hat{\X}_{1}),a(y,\hat{\Y}_{1})))d_{\varepsilon}(x,y)\\
&+\mathbb{P}(U_1\ge  \max(a(x,\hat{\X}_{1}),a(y,\hat{\Y}_{1})))d_\varepsilon(x,y)\\
&\le (1-\rho\mathbb{P}(U_1\le \min(a(x,\hat{\X}_{1}),a(y,\hat{\Y}_{1}))))d_\varepsilon(x,y)\\
&\le (1-\rho a_m)d_\varepsilon(x,y),
\end{align*}
using the lower bound on the acceptance probability above.

On the other hand, using the bound $d_\varepsilon(\cdot,\cdot)\le 1$, one has
\begin{align*}
\E[\mathds{1}_{a(y,\hat{\Y}_{1}) \le U\le a(x,\hat{\X}_{1})}d_\varepsilon(\hat{\X}_1,y)]&+\E[\mathds{1}_{a(x,\hat{\X}_{1}) \le U\le a(y,\hat{\Y}_{1})}d_\varepsilon(x,\hat{\Y}_1)]\\
&\le \E[|a(x,\hat{\X}_{1})-a(y,\hat{\Y}_{1})|]\\
&\le C|V(x)-V(y)|+C\E[|V(\hat{\X}_1)-V(\hat{\Y}_1|]\\
&\le C|x-y|+C\E[|\hat{\X}_1-\hat{\Y}_1]\\
&\le C|x-y|\\
&\le C\epsilon d_\varepsilon(x,y),
\end{align*}
using the assumptions that $V$ is bounded and globally Lipschitz continuous, and the observation that $|x-y|=\epsilon d_\varepsilon(x,y)$ owing to the condition $d_\varepsilon(x,y)<1$.

Gathering the estimates, one obtains
\[
d_\varepsilon(\PP(x,\cdot),\PP(x,\cdot))\le \E[d_\varepsilon(\X_1,\Y_1)]\le (1-\rho a_m+C\epsilon)d_\varepsilon(x,y)
\]
for all $x,y\in H$ such that $d_\varepsilon(x,y)<1$. It suffices to choose $\epsilon$ sufficiently small, to have $C\epsilon<\rho a_m$ and obtain the required estimate~\eqref{eq:MCMC-contraction}.

$\bullet$ Proof of the $d$-smallness property~\eqref{eq:MCMC-smallness}.

Let $R\in(0,\infty)$, and define $N$ as the smallest integer such that
\[
\frac{2R}{\epsilon(1+\tau\lambda_1)^N}\le \frac12.
\]
Let $x,y\in H$ be such that $|x|\le R$ and $|y|\le R$. Like in the proof of the d-contraction property above, introduce the sequences $\bigl(\X_n\bigr)_{n\in\N}$, $\bigl(\hat{X}_n\bigr)_{n\in\N}$, $\bigl(\Y_n\bigr)_{n\in\N}$ and $\bigl(\hat{\Y}_n\bigr)_{n\in\N}$, using the basic coupling strategy: for all $n\ge 0$,
\begin{align*}
\hat{\X}_{n+1}&=\IA_\tau\X_n+\sqrt{\tau}\IB_{\tau,1}\Gamma_{n,1}+\sqrt{\tau}\IB_{\tau,2}\Gamma_{n,2}\\
\X_{n+1}&=\mathds{1}_{U_n\le a(\X_n,\hat{\X}_{n+1})}\hat{\X}_{n+1}+\mathds{1}_{U_n>a(\X_n,\hat{\X}_{n+1})}\X_{n},\\
\hat{\Y}_{n+1}&=\IA_\tau\Y_n+\sqrt{\tau}\IB_{\tau,1}\Gamma_{n,1}+\sqrt{\tau}\IB_{\tau,2}\Gamma_{n,2}\\
\Y_{n+1}&=\mathds{1}_{U_n\le a(\Y_n,\hat{\Y}_{n+1})}\hat{\Y}_{n+1}+\mathds{1}_{U_n>a(\Y_n,\hat{\Y}_{n+1})}\Y_{n},
\end{align*}
with initial values $\X_0=x$ and $\Y_0=y$. By construction of the Wasserstein distance-like function $d_\varepsilon$, one has
\[
d_\varepsilon(\PP^{N_R}(x,\cdot),\PP^{N_R}(y,\cdot))\le \E_x[d_\varepsilon(\X_N,\Y_N)].
\]
Introduce the event
\[
\mathbf{A}=\left\{\X_1=\hat{\X}_1,\Y_1=\hat{\Y}_1,\ldots,\X_N=\hat{\X}_N,\Y_N=\hat{\Y}_N\right\},
\]
such that the proposals are accepted up to time $N$, for both chains. Note that the lower bound on the acceptance probability above gives the lower bound
\[
\mathbf{A}\ge a_m^N.
\]
As a consequence, using the definition of the distance-like function $d_\varepsilon$, for all $x,y\in H$, such that $|x|\le R$ and $|y|\le R$, one obtains
\begin{align*}
d_\varepsilon(\PP^{N_R}(x,\cdot),\PP^{N_R}(y,\cdot))&\le \E_x[d_\varepsilon(\X_N,\Y_N)]\\
&\le \E_x[\frac{|X_N-\Y_N|}{\epsilon}\mathds{1}_\mathbf{A}]+1-\mathbb{P}(\mathbf{A})\\
&\le \frac{2R}{\epsilon(1+\tau\lambda_1)^N}\mathds{1}_\mathbf{A}]+1-\mathds{1}_\mathbf{A}]\\
&\le 1-\frac12\mathds{1}_\mathbf{A}]\\
&\le 1-\frac{a_m^N}{2}=s_R\in(0,1),
\end{align*}
owing to the definition of $N$. This concludes the proof of the d-smallness property~\eqref{eq:MCMC-smallness}.

$\bullet$ Having proved the three claims~\eqref{eq:MCMC-Lyapunov},~\eqref{eq:MCMC-contraction} and~\eqref{eq:MCMC-smallness}, the weak Harris theorem can be applied and yiels the estimate~\eqref{eq:MCMC-Wasserstein-inequality}. The spectral gap inequality~\eqref{eq:MCMC-spectralgap} is obtained as a consequence of the estimate~\eqref{eq:MCMC-Wasserstein-inequality} above using the same arguments as in~\cite{HairerStuartVollmer}. The details are omitted.

$\bullet$ The proof of Theorem~\ref{theo:MCMC} is thus completed.
\end{proof}

\begin{rem}
The Markov chain $\bigl(\X_n^\tau\bigr)_{n\ge 0}$ may also be seen as a metropolized integrator in order to approximate the solution of the SPDE~\eqref{eq:SPDE} at any time $T$ (with $N\tau=T$), when the nonlinearity $F$ satisfies Assumption~\ref{ass:gradient}. For this interpretation to be valid, one would need to prove an error estimate of type
\[
\big|\E[\varphi(\X_N^\tau)]-\E[\varphi(X(T))]\big|\le C(T,x_0,\varphi)\tau^\alpha
\]
where $\alpha$ would be the order of convergence. Note that computing $\bigl(\X_n^\tau\bigr)_{n\ge 0}$ (given by the MCMC method~\eqref{eq:MCMC}) requires to evaluate the mapping $V$, whereas computing $\bigl(X_n^\tau\bigr)_{n\ge 0}$ (given by the modified Euler scheme~\eqref{eq:scheme}) requires to evaluate $F=-DV$.

Metropolized integrators for SDEs have been studied in~\cite{MR2583309} for instance. The analysis in the infinite dimensional case is not a straightforward extension of the arguments above and is thus left open for future work.
\end{rem}

\subsubsection{Comparisons with the standard and exponential Euler schemes}

To conclude this subsection concerning MCMC methods to approximate integrals $\int\varphi d\mu_\star$, let us study the behavior of the Metropolis--Hastings Monte Carlo Markov Chain algorithms obtained using either the exponential or the standard Euler scheme, instead of the modified Euler scheme, as the proposal kernel. Let $\bigl(\Gamma_n\bigr)_{n\ge 0}$ be a sequence of independent cylindrical $H$-valued Gaussian random variables, and $\bigl(U_n\bigr)_{n\ge 0}$ be a sequence of independent random variables which are uniformly distributed on $[0,1]$, such that the two sequences are independent.

On the one hand, a well-defined Metropolis--Hastings MCMC method is obtained when the proposal kernel is the accelerated exponential Euler scheme. For all $\tau\in(0,\tau_0)$, set
\begin{equation}\label{eq:MCMC-expo}
\left\lbrace
\begin{aligned}
\hat{\X}_{n+1}^{\tau,\e}&=e^{-\tau\IL}\X_n^{\tau,\e}+\bigl(\frac{1}{2}(I-e^{-2\tau\IL})\bigr)^{\frac12}\Gamma_n \\
\X_{n+1}^{\tau,\e}&=\mathds{1}_{U_n\le a(\X_n^{\tau,\e},\hat{X}_{n+1}^{\tau,\e})}\hat{X}_{n+1}^{\tau,\e}+\mathds{1}_{U_n>a(\X_n^{\tau,\e},\hat{X}_{n+1}^{\tau,\e})}X_{n}^{\tau,\e},
\end{aligned}
\right.
\end{equation}
where the acceptance probability is defined by~\eqref{eq:acceptanceMCMC}. The Markov chain $\bigl(\X_n^{\tau,\e}\bigr)_{n\ge 0}$ satisfies the results stated in Theorem~\ref{theo:MCMC} for the modified Euler scheme, for any value $\tau\in(0,\tau_0)$ of the time-step size. The proof is omitted, since the arguments are similar. The result is not suprising: indeed the accelerated exponential Euler scheme preserves the Gaussian invariant distribution $\nu$ in the Ornstein--Uhlenbeck case, since it is exact in distribution at all times. The convergence to the equilibrium is also exponentially fast (contrary to the Crank--Nicolson method which preserves the invariant distribution but is not L-stable). The comparison of the modified Euler scheme and of the (accelerated) exponential Euler scheme leads to the same conclusions as in the other parts of this article: the convergence results are identical for the two methods, however the modified Euler scheme does not require the knowledge of the eigendecomposition of the linear operator $\IL$ and may thus be applied in greater generality than the exponential Euler method.

On the other hand, using the standard Euler scheme is not appropriate. More precisely, set 
\begin{equation}\label{eq:MCMC-standard}
\left\lbrace
\begin{aligned}
\hat{\X}_{n+1}^{\tau,\s}&=e^{-\tau\IL}\bigl(\X_n^{\tau,\s}+\sqrt{\tau}\Gamma_n\bigr) \\
\X_{n+1}^{\tau,\s}&=\mathds{1}_{U_n\le a(\X_n^{\tau,\s},\hat{X}_{n+1}^{\tau,\s})}\hat{X}_{n+1}^{\tau,\s}+\mathds{1}_{U_n>a(\X_n^{\tau,\s},\hat{X}_{n+1}^{\tau,\s})}X_{n}^{\tau,\s},
\end{aligned}
\right.
\end{equation}
where the acceptance probability is defined by~\eqref{eq:acceptanceMCMC}. The scheme above cannot be interpreted as a Metropolis--Hastings MCMC method which targets the Gibbs distribution $\mu_\star$: in fact, the general rule which provides the acceptance ratio is ill-defined in the infinite dimensional situation, due to singularity of the Gaussian distributions appearing in its definition. This is another illustration of the superiority of the modified Euler scheme over the standard Euler method.

Observe that the Markov chain defined by~\eqref{eq:MCMC-standard} can be interpreted as a Metropolis--Hastings MCMC method which targets the modified Gibbs distribution $\mu_\star^{\tau}$ defined by~\eqref{eq:mu_startau} (see Section~\ref{sec:results_standard} and in particular Theorem~\ref{theo:weakinv-standard}). Since the targetted distribution is not independent of the auxiliary time-step size parameter $\tau$, the benefits of using a MCMC method compared with a standard integrator are not recovered. Note that a variant of Theorem~\ref{theo:MCMC} is expected to hold for the method defined by~\eqref{eq:MCMC-standard} if one considers the target distribution $\mu_\star^{\tau}$.

\subsection{Application to other SPDE systems}\label{sec:generalizations}

The main results stated in Section~\ref{sec:results} concerning the modified Euler scheme are stated and proved in the framework described in Section~\ref{sec:setting}, and are restricted to a particular class of stochastic evolution equations of the type~\eqref{eq:SPDE}: parabolic semilinear stochastic PDEs, in a bounded one-dimensional domain with homogeneous Dirichlet boundary conditions, driven by additive Gaussian space-time white noise. The objective of this subsection is to suggest possible extensions of the definition of the modified Euler scheme. In the more general  framework, the proposed scheme satisfies the following results.
\begin{itemize}
\item Theorem~\ref{theo:regularity} is satisfied by construction of the scheme, which means that the spatial regularity of the solution is preserved by the numerical approximation, for any value of the time-step size. However, Theorem~\ref{theo:equivalence} does not always hold.
\item Theorem~\ref{theo:weakinv} may not hold. Indeed, this result requires that the invariant distribution $\mu_\infty$ of~\eqref{eq:SPDE} is equal to the Gibbs distribution $\mu_\star$, and that this is also the invariant distribution of the modified equation~\eqref{eq:modifiedSPDE}. This crucial property is not satisfied for instance when the equation is driven by colored noise.
\item Theorem~\ref{theo:weak} always hold, with an order of convergence which depends on the considered problem.
\end{itemize}

Three generalizations are studied below, they may of course be combined to consider other generalizations which are omitted. Note also that we only consider homogeneous Dirichlet boundary conditions, however Neumann or periodic boundary conditions may be also considered.

\subsubsection{SPDEs with one-sided Lipschitz nonlinearities}

A first possible generalization is to weaken Assumption~\ref{ass:F}, which requires the nonlinearity $F$ to be globally Lipschitz continuous. In this section, it is assumed only that $F(x)=f(x(\cdot))$ is defined as a Nemytskii operator, such that the real-valued function $f$ satisfies a one-sided Lipschitz condition (but is not globally Lipschitz continuous):
\begin{equation}\label{eq:onesided}
\underset{z\in\R}\sup~f'(z)<\infty.
\end{equation}
It is also required to assume that $f$ has at most polynomial growth. In this setting, the stochastic evolution equation 
\begin{equation}\label{eq:SPDE-nonLip}
dX(t)=-\IL X(t)dt+F(X(t))dt+dW(t),\quad X(0)=x_0,
\end{equation}
is of the same form as~\eqref{eq:SPDE}. When $f(x)=x-x^3$, this gives the stochastic Allen--Cahn equation.

The modified Euler scheme~\eqref{eq:scheme} cannot be applied to the stochastic evolution equation~\eqref{eq:SPDE-nonLip}: since the nonlinearity $F$ is not globally Lipschitz continuous and may have superlinear growth, a standard explicit discretization of the nonlinearity leads to a scheme which does not satisfy moment bounds as given in Lemma~\ref{lem:scheme-bound}.

When the flow $(t,z)\mapsto \phi_t(z)$ of the nonlinear ordinary differential equation $\dot{z}=f(z)$ is known, which is the case for the Allen--Cahn equation, a splitting scheme can be designed: with the same notation as in the definition of the modified Euler scheme~\eqref{eq:scheme}, set
\begin{equation}\label{eq:splitting}
X_{n+1}^{\tau}=\IA_\tau \Phi_{\tau}(X_n^{\tau})+\IB_{\tau,1}\sqrt{\tau}\Gamma_{n,1}+\IB_{\tau,2}\sqrt{\tau}\Gamma_{n,2},
\end{equation}
where $\Phi_\tau(x)=\phi_\tau(x)$ for all $x\in H$.

When the flow of the nonlinear ordinary differential equation $\dot{z}=f(z)$ is not known, a split-step scheme may be used:
\begin{equation}\label{eq:splitstep}
\left\lbrace
\begin{aligned}
\hat{X}_n^\tau&=X_n+\tau F(\hat{X}_n^\tau)\\
X_{n+1}^{\tau}&=\IA_\tau \hat{X}_n^{\tau}+\IB_{\tau,1}\sqrt{\tau}\Gamma_{n,1}+\IB_{\tau,2}\sqrt{\tau}\Gamma_{n,2}.
\end{aligned}
\right.
\end{equation}
Alternatively, a fully implicit scheme may also be used:
\begin{equation}\label{eq:implicit}
X_{n+1}^{\tau}=\IA_\tau\bigl(X_n^{\tau}+\tau F(X_{n+1}^{\tau})\bigr)+\IB_{\tau,1}\sqrt{\tau}\Gamma_{n,1}+\IB_{\tau,2}\sqrt{\tau}\Gamma_{n,2}.
\end{equation}

Finally, a taming strategy may be used, to define an explicit integrator which satisfies moment bounds:
\begin{equation}\label{eq:tamed}
X_{n+1}^{\tau}=\IA_\tau\bigl(X_n^{\tau}+\frac{\tau}{1+\tau|F(X_n^\tau)|}F(X_n^{\tau})\bigr)+\IB_{\tau,1}\sqrt{\tau}\Gamma_{n,1}+\IB_{\tau,2}\sqrt{\tau}\Gamma_{n,2}.
\end{equation}

The schemes defined above are natural generalizations of schemes already studied in the literature, where the discretization of the stochastic convolution (linear part and noise) is performed using the modified Euler scheme instead of the standard Euler scheme or of the accelerated exponential Euler scheme. For instance, the splitting scheme~\eqref{eq:splitting} is a generalization of the scheme studied in~\cite{BrehierGoudenege}. The fully implicit scheme~\eqref{eq:implicit} is a generalization of the scheme studied in~\cite{CuiHongSun:21}. The tamed scheme~\eqref{eq:tamed} is a generalization of the scheme studied in~\cite{B:2022}. For all the schemes, the result of Theorem~\ref{theo:regularity} still holds, for any value of the time-step size $\tau$. In addition, note that the splitting scheme~\eqref{eq:splitting} and the split-step scheme~\eqref{eq:splitstep} can both be interpreted in terms of the accelerated exponential Euler scheme applied to a modified stochastic evolution equation of the type
\[
d\IX_\tau(t)=-\IL_\tau \IX_\tau(t)dt+Q_{\tau}\Psi_\tau(\IX_\tau(t))dt+Q_{\tau}^{\frac12}dW(t)
\]
with a modified nonlinearity $\Psi_\tau$. In the splitting scheme case, $\Psi_\tau(x)=\tau^{-1}(\Phi_\tau(x)-x)$. Similarly, the tamed scheme~\eqref{eq:tamed} can be interpreted in terms of the tamed accelerated exponential Euler scheme applied to the modified stochastic evolution equation~\eqref{eq:modifiedSPDE}, i.\,e. with $\Psi_\tau=F$. The analysis of the fully implicit scheme~\eqref{eq:implicit} would require different arguments.

To prove a version of Theorem~\ref{theo:weak}, giving weak error estimates of the type
\[
\big|\E[\varphi(X_N^\tau)]-\E[\varphi(X(T))]\big|\le C_{\epsilon}(T,x_0)\tau^{\frac12-\epsilon}\bigl(\vvvert\varphi\vvvert_1+\vvvert\varphi\vvvert_2\bigr)
\]
for functions $\varphi$ of class $\mathcal{C}^2$ and all fixed $T\in(0,\infty)$, one needs to modify the proof of Lemma~\ref{lem:utau-1} and of Lemma~\ref{lem:u-12}, which give regularity estimates for the solutions $u_\tau$ and $u$ of Kolmogorov equations. We refer to \cite[Theorems~4.1 and~4.2]{BrehierGoudenege}, see also~\cite{CuiHong:19}.

In the ergodic case, a version of Theorem~\ref{theo:weakinv}, which gives weak error estimates of the type
\[
\big|\E[\varphi(X_N^\tau)]-\int\varphi d\mu_\star\big|\le C_{\epsilon}(x_0)\vvvert\varphi\vvvert\Bigl(\tau^{\frac12-\epsilon}+e^{-\kappa N\tau}\Bigr)
\]
for functions $\varphi$ which are only bounded and continuous (or equivalently an error estimate in the total variation distance), can be proved for the splitting scheme~\eqref{eq:splitting} and for the split-step scheme~\eqref{eq:splitstep}. This generalization requires to replace Assumption~\ref{ass:ergo} by the condition
\[
\underset{z\in\R}\sup~f'(z)<\lambda_1
\]
and to prove versions of Lemma~\ref{lem:utau-ergo} and of Lemma~\ref{lem:u-ergo}: this is straightforward, we refer for instance to~\cite[Proposition~6.1]{B:2022}. One also needs to prove uniform moment bounds of the type~\eqref{eq:lem-scheme-bound-ergo}, in $L^\infty$ norms instead of $H=L^2$ norms. For the tamed scheme~\eqref{eq:tamed}, the arguments from~\cite{B:2022} maye be generalized. Note that the Gibbs distribution $\mu_\star$ defined by~\eqref{eq:mu_star} is also the unique invariant distribution of the stochastic evolution equation~\eqref{eq:SPDE-nonLip} even if $f$ is only one-sided Lipschitz continuous (and satisfies the ergodicity condition above). This is why a version of Theorem~\ref{theo:weakinv} is expected to hold also in the non globally Lipschitz case described above.

This concludes the description of the non-globally Lipschitz case.

\subsubsection{SPDEs with colored noise}

The framework described in Section~\ref{sec:setting} is restricted to consider stochastic evolution equations~\eqref{eq:SPDE} where $-\IL$ is an elliptic second-order operator in dimension $1$ and where $\bigl(W(t)\bigr)_{t\ge 0}$ is a cylindrical Wiener process, i.\,e. the   system is driven by space-time white noise. In this subsection, we explain how the modified Euler scheme can be applied to equations in higher dimension and/or driven by colored noise, and what are the expected results in those situations.

On the one hand, for any dimension $d\in\N$, $-\IL$ can be defined such that
\[
-\IL x(\cdot)={\rm div}\bigl(a(\cdot)\nabla x(\cdot)\bigr)
\]
for all $x\in D(\IL)=H_0^1((0,1)^d)\cup H^2((0,1)^d)$, where $a:[0,1]^d\to a(x)=(a_{i_1,i_2}(x)\bigr)_{1\le i_1,i_2\le d}$ is smooth and the ellipticity condition $\underset{z\in[0,1]^d}\min~\underset{\xi\in\R^d}\min~\frac{a(x)\xi\cdot\xi}{\xi\cdot\xi}>0$ is satisfied, where $\cdot$ denotes the inner product in $\R^d$. In this setting, Assumption~\ref{ass:Lambda} needs to be modified: one has $\lambda_j\sim cj^{\frac{2}{d}}$ when $j\to\infty$.

On the other hand, let the $Q$-Wiener process $\bigl(W^Q(t)\bigr)_{t\ge 0}$ be defined as follows. Let $\bigl(q_j\bigr)_{j\in\N}$ be a sequence of non-negative real numbers and $\bigl({\bf e}_j\bigr)_{j\in\N}$ be a complete orthonormal system of $H$. The linear operator $Q$ and $Q^{\frac12}$ are given by
\[
Qx=\sum_{j\in\N}q_j\langle x,{\bf e_j}\rangle {\bf e_j},\quad Q^{\frac12}x=\sum_{j\in\N}\sqrt{q_j}\langle x,{\bf e_j}\rangle {\bf e_j},
\]
and for all $t\ge 0$ set
\[
W^Q(t)=\sum_{j\in\N}\sqrt{q_j}\beta_j(t){\bf e_j}
\]
where $\bigl(\beta_j\bigr)_{j\in\N}$ is a sequence of independent standard real-valued Wiener processes.

The stochastic evolution equation driven by additive colored noise
\begin{equation}\label{eq:SPDE-colored}
dX(t)=-\IL X(t)dt+F(X(t))dt+dW^Q(t),\quad X(0)=x_0,
\end{equation}
is well-posed when the covariance operator $Q$ satisfies a condition of the type
\[
\int_0^T\|e^{-t\IL}Q^{\frac12}\|_{\mathcal{L}_2(H)}^2dt<\infty
\]
is satisfied, where we recall that $\|\cdot\|_{\mathcal{L}_2(H)}$ denotes the Hilbert-Schmidt norm. Owing to the smoothing property~\eqref{eq:smoothing}, a sufficient condition is the existence of $\alpha>0$ such that
\begin{equation}\label{eq:conditioncolored}
\|\IL^{\frac{2\alpha-1}{2}}Q^{\frac12}\|_{\mathcal{L}_2(H)}<\infty.
\end{equation}
When $Q=I$ (cylindrical Wiener process/space-time white noise), the condition above holds when $d=1$ (with $\alpha\in[0,\frac14)$), but is not satisfied if $d\ge 2$. To consider equations in dimension $d\ge 2$, the system needs to be driven colored noise, i.\,e. $Q\neq I$. In the trace-class noise case, meaning that ${\rm Tr}(Q)=\|Q^{\frac12}\|_{\mathcal{L}_2(H)}^2=\sum_{j\in\N}q_j<\infty$, the condition~\eqref{eq:conditioncolored} holds for $\alpha=\frac12$. In general, the range of values of $\alpha$ such that~\eqref{eq:conditioncolored} holds depends both on the covariance operator $Q$ and on the dimension $d$.

To discretize the stochastic evolution equation~\eqref{eq:SPDE-colored} driven by additive colored noise, the definition of the modified Euler scheme is modified as follows:
\begin{equation}\label{eq:scheme-colored}
X_{n+1}^{\tau}=\IA_\tau\bigl(X_n^{\tau}+\tau F(X_n^{\tau})\bigr)+\IB_{\tau,1}\sqrt{\tau}\Gamma_{n,1}^Q+\IB_{\tau,2}\sqrt{\tau}\Gamma_{n,2}^Q,
\end{equation}
where the operators $\IA_\tau$, $\IB_{\tau,1}$ and $\IB_{\tau,2}$ are defined by~\eqref{eq:operators}, and the Gaussian random variables $\Gamma_{n,1}^Q$ and $\Gamma_{n,2}^Q$ are defined as follows:
\[
\Gamma_{n,i}^Q=\sum_{j\in\N}\sqrt{q_j}\gamma_{n,i,j}{\rm e_j},
\]
where $\bigl(\gamma_{n,i,j}\bigr)_{n\in\N_0,i\in\{1,2\},j\in\N}$ are independent standard real-valued random variables.

If the covariance operator $Q$ and the linear operator $\IL$ commute (\emph{commutative noise case}), the interpretations of the modified Euler scheme presented in Subsections~\ref{sec:scheme-2nd} and~\ref{sec:scheme-3rd} are valid also for the scheme~\eqref{eq:scheme-colored}. In particular, the scheme~\eqref{eq:scheme-colored} can be interpreted as the accelerated exponential Euler scheme
\begin{equation}\label{eq:scheme-IX-colored}
\IX_{\tau,n+1}=e^{-\tau\IL_\tau}\IX_{\tau,n}+\IL_\tau^{-1}(I-e^{-\tau\IL_\tau})Q_\tau F(\IX_{\tau,n})+\int_{t_n}^{t_{n+1}}e^{-(t_{n+1}-s)\IL_\tau}Q_\tau^{\frac12}dW^Q(s)
\end{equation}
applied to the modified stochastic evolution equation
\begin{equation}\label{eq:modifiedSPDE-colored}
d\IX_\tau(t)=-\IL_\tau \IX_\tau(t)dt+Q_{\tau}F(\IX_\tau(t))dt+Q_{\tau}^{\frac12}dW^Q(t),
\end{equation}
where the linear operators $\IL_\tau$, $Q_\tau$ and $Q_\tau^{\frac12}$ are given by~\eqref{eq:modifiedILQ}.

However, when the operators $Q$ and $\IL$ do not commute, the interpretations of the modified Euler scheme presented in Subsections~\ref{sec:scheme-2nd} and~\ref{sec:scheme-3rd} are not valid for the scheme~\eqref{eq:scheme-colored}. Indeed, the covariance operator of the Gaussian random variable $\IB_{\tau,1}\Gamma_{n,1}^Q+\IB_{\tau,2}\Gamma_{n,2}^Q$ is equal to
\[
\bigl(\IB_{\tau,1}Q^{\frac12}\bigr)\bigl(\IB_{\tau,1} Q^{\frac12}\bigr)^{\star}+\bigl(\IB_{\tau,2}Q^{\frac12}\bigr)\bigl(\IB_{\tau,2}Q^{\frac12}\bigr)^{\star}=\IB_{\tau,1}Q\IB_{\tau,1}+\IB_{\tau,2}Q\IB_{\tau,2}
\]
and is different from $\bigl(\IB_{\tau}Q^{\frac12}\bigr)\bigl(\IB_\tau Q^{\frac12}\bigr)^{\star}=\IB_\tau Q\IB_\tau$, where the linear operator $\IB_\tau$ is defined by~\eqref{eq:IB} and satisfies $\IB_\tau^2=\IB_{\tau,1}^2+\IB_{\tau,2}^2$.

Let us now describe how the results of Section~\ref{sec:results} need to be modified in the case of stochastic evolution equations driven by additive colored noise.

First, Theorem~\ref{theo:equivalence} remains valid in the commutative noise case, with straightforward modifications of the proof, however it may not be satisfied in the non-commutative noise case. Second, Theorem~\ref{theo:regularity} holds in the general case, with a modification of the range of values $\alpha\in[0,\frac14)$ in (iii): instead, one needs to consider the interval of values of $\alpha$ such that the condition~\eqref{eq:conditioncolored} is satisfied. 

Concerning error estimates in the total variation distance, Theorem~\ref{theo:weakinv} does not hold in general: indeed, when Assumption~\ref{ass:gradient} is satisfied, the invariant distribution of~\eqref{eq:SPDE-colored} is not the Gibbs distribution $\mu_\star$. It is conjectured that a version of Theorem~\ref{theo:weakinv} holds if $Q$ is assumed to commute with $\IL$, if the nonlinearity is assumed to satisfy the condition $F=-QDV$ and if a suitable non-degeneracy condition is satisfied. A modification of Assumption~\ref{ass:Fregul1} may also be needed, as discussed below. As explained above, considering the commutative noise case is required to interpret the modified Euler scheme~\eqref{eq:scheme-colored} in terms of the accelerated exponential Euler scheme~\eqref{eq:scheme-IX-colored} applied to the modified stochastic evolution equation~\eqref{eq:modifiedSPDE-colored}. The condition $F=-QDV$ implies that the invariant distribution of~\eqref{eq:SPDE-colored} and of~\eqref{eq:modifiedSPDE-colored} is equal to a Gibbs distribution
\[
d\mu_{\star,Q}(x)=\mathcal{Z}_Q^{-1}e^{-2V(x)}d\nu_Q(x)
\]
where $\nu_Q$ is the Gaussian distribution with mean $0$ and covariance operator $\frac{Q\IL}{2}$ (which is the invariant distribution of~\eqref{eq:SPDE-colored} when $F=0$). Finally, a non-degeneracy condition is required to prove versions of Lemmas~\ref{lem:utau-0} and~\ref{lem:u-0}, which give regularity results for the derivatives $Du_\tau(t,\cdot)$ and $Du(t,\cdot)$, for $t>0$, when the initial value $\varphi=u_\tau(0,\cdot)=u(0,\cdot)$ is only assumed to be bounded and continuous. Precise statement and proofs are omitted and left for future work.

In the general case, a version of Theorem~\ref{theo:weak} is conjectured to hold for the modified Euler scheme~\eqref{eq:modifiedSPDE-colored} applied to the stochastic evolution equation~\eqref{eq:SPDE-colored} driven by additive colored noise. Note that the order of convergence depends on the values of $\alpha$ such that the condition~\eqref{eq:conditioncolored} is satisfied. For instance, the weak order of convergence is expected to be equal to $1$ in the trace-class noise case. Similarly, a version of Theorem~\ref{theo:weak-ergo} is conjectured to hold in the ergodic case (when Assumption~\ref{ass:ergo} is satisfied).

Note that a tool of the proof of weak error estimates in Sections~\ref{sec:proofs} is Assumption~\ref{ass:Fregul1}, which gives a regularity condition on the nonlinearity $F$, in order to exploit the temporal regularity, with H\"older exponent $2\alpha$ of the solutions in the norm $|\IL^{-\alpha}\cdot|$, for all $\alpha\in[0,\frac14)$, see for instance~\eqref{eq:tildeIX-increment} from Lemma~\ref{lem:tildeIX}. This argument is well-suited for stochastic evolution equations in dimension $1$ driven by space-time white noise. On the one hand, the arguments in Section~\ref{sec:example} to check that Assumption~\ref{ass:Fregul1} is satisfied for the example of Nemytskii operators, exploit Sobolev type inequalities which are valid only in dimension $1$. On the other hand, the argument is not sufficient to exhibit orders of convergence larger than $1/2$, since the H\"older regularity of the solutions is smaller than $1/2$. As a consequence, other arguments are needed, for instance to treat the trace-class noise case.

It is worth mentioning that in the non commutative noise case, applying the accelerated exponential Euler scheme to the stochastic evolution equation~\eqref{eq:SPDE-colored} is not feasible: it is not sufficient to know the eigenvalues and eigenfunctions of the operators $\IL$ and $Q$ to sample exactly Gaussian random variables
\[
\int_{t_n}^{t_{n+1}}e^{-(t_{n+1}-t)\IL}dW^Q(t).
\]
Using an approximation of the type $\sqrt{\tau}e^{-\tau\IL}\Gamma_n^{Q}$ leads to define a non-accelerated exponential Euler scheme of the type
\[
X_{n+1}^{\tau,\e}=e^{-\tau\IL}\bigl(X_n^{\tau,\e}+\tau F(X_n^{\tau,\e}+\sqrt{\tau}\Gamma_n^Q\bigr)
\]
For that scheme, Theorems~\ref{theo:regularity} and~\ref{theo:weak-expo} are not valid, even when $F=0$, and even in the commutative case: the resulting scheme does not preserve the regularity of the solution. On the contrary, the modified Euler scheme~\eqref{eq:scheme-colored} is applicable in the non-commutative noise case and Theorem~\ref{theo:regularity} is satisfied. In that case, the modified Euler scheme is thus qualitatively superior to the (non-accelerated) exponential Euler scheme.

This concludes the description of the colored noise case.

\subsubsection{SPDEs with non-additive noise}

Finally, it is possible to generalize the definition of the modified Euler scheme, to be applied to stochastic evolution equations driven by multiplicative (or non-additive) noise:
\begin{equation}\label{eq:SPDE-multi}
dX(t)=-\IL X(t)dt+F(X(t))dt+\sigma(X(t))dW(t),\quad X(0)=x_0,
\end{equation}
where $\sigma$ is a function from $H$ to $\mathcal{L}(H)$, assumed to be globally Lipschitz continuous. For instance, $\sigma$ may be defined as a Nemytskii operator, in the setting of Section~\ref{sec:example}. In addition, $\bigl(W(t)\bigr)_{t\ge 0}$ is a cylindrical Wiener process, however it would also be possible to consider $Q$-Wiener processes, under a condition of the type~\eqref{eq:conditioncolored}. In that situation, the modified Euler scheme applied to~\eqref{eq:SPDE-multi} is defined as
\begin{equation}\label{eq:scheme-multi}
X_{n+1}^{\tau}=\IA_\tau\bigl(X_n^{\tau}+\tau F(X_n^{\tau})\bigr)+\IB_{\tau,1}\sqrt{\tau}\sigma(X_n^\tau)\Gamma_{n,1}+\IB_{\tau,2}\sqrt{\tau}\sigma(X_n^\tau)\Gamma_{n,2}.
\end{equation}
First of all, even if the nonlinearity $F$ satisfies Assumption~\ref{ass:gradient}, there is no known expression for the invariant distribution~\eqref{eq:SPDE-multi} (which is unique when a version of Assumption~\ref{ass:ergo} is satisfied). Moreover, like in the case of equations driven by colored noise when the covariance operator does not commute with $\IL$, in general the modified Euler scheme~\eqref{eq:scheme-multi} cannot be interpreted in terms of the accelerated exponential Euler scheme applied to a modified stochastic evolution equation of the type~\eqref{eq:modifiedSPDE}. As a consequence of the two observations above, Theorem~\ref{theo:weakinv} is not expected to hold in the multiplicative noise case, since two of the main arguments of the proof are not applicable. Note also that Theorem~\ref{theo:weak-expo} does not hold in general for equations driven by multiplicative noise. Indeed, in the multiplicative noise case, the accelerated exponential Euler method cannot be implemented and the non-accelerated exponential Euler method suffers from the same issues as the standard linear Euler method. Whether it is possible to prove error estimates in the total variation distance for either the modified Euler scheme~\eqref{eq:scheme-multi} or an exponential Euler scheme when applied to~\eqref{eq:SPDE-multi} is an open question.

Like in the other situations described above, the main benefit of applying the modified Euler scheme~\eqref{eq:scheme-multi} over existing methods -- standard Euler scheme and (non-accelerated) exponential Euler scheme -- is the validity of Theorem~\ref{theo:regularity}: the modified Euler scheme preserves the spatial regularity of the solution, for any choice of the time-step size $\tau$.

Version of Theorems~\ref{theo:weak} and~\ref{theo:weak-ergo}, to state weak error estimates in the $d_2$ distance, i.\,e. for functions $\varphi$ of class $\mathcal{C}^2$, could be obtained also for the modified Euler scheme~\eqref{eq:scheme-multi} applied to~\eqref{eq:SPDE-multi}. As explained above, the interpretation of the modified Euler scheme in terms of an exponential Euler scheme applied to a modified stochastic evolution equation is not valid in the multiplicative noise case, and the analysis of the weak error needs to be performed the same approaches as used in the analysis of the standard Euler scheme. Note that the main difficulty in the analysis of the multiplicative noise case compared with the additive noise case is the proof of Lemma~\ref{lem:u-12}, which gives regularity results for the first and second order derivatives of the solution $u$ of the Kolmogorov equation associated with~\eqref{eq:SPDE-multi}: we refer to~\cite{BrehierDebussche}. The detailed analysis of the weak error for the scheme~\eqref{eq:scheme-multi} is not considered.

This concludes the description of the multiplicative noise case.

\section*{Acknowledgments}

The author warmly thanks Gilles Vilmart for crucial discussions about the construction of the proposed method at an early stage of this work, and Arnaud Debussche for the suggestion to state and prove Theorem~\ref{theo:weakinv-standard}. This work is partially supported by the following projects operated by the French National Research Agency: ADA (ANR-19-CE40-0019-02) and SIMALIN (ANR-19-CE40-0016).

\newpage

\end{document}